\documentclass[preprint,a4paper,11pt]{elsarticle}

\makeatletter
\def\ps@pprintTitle{%
 \let\@oddhead\@empty
 \let\@evenhead\@empty
 \def\@oddfoot{}%
 \let\@evenfoot\@oddfoot}
\makeatother

\usepackage[utf8]{inputenc}
\usepackage{amsfonts}
\usepackage{amssymb}
\usepackage{amsmath}
\usepackage{amsthm}
\usepackage{mathtools}
\usepackage{todonotes}
\usepackage{enumitem}
\usepackage{esint}
\usepackage{dsfont}

\setenumerate[0]{label=(\arabic*)}

\numberwithin{equation}{section}

\usepackage{geometry}
\def \R {\mathbb{R}}
\def \CC {\mathbb{C}}
\def \N {\mathbb{N}}
\def \Z {\mathbb{Z}}
\def \eps {\varepsilon}
\def \with {\, : \,}
\def \CalQ {\mathcal{Q}}
\def \CalB {\mathcal{B}}
\def \CalP {\mathcal{P}}
\def \CalD {\mathcal{D}}
\def \CalH {\mathcal{H}}

\def \BoldW {\mathbb{W}}
\def \emptyset {\varnothing}
\def \identity {{\mathrm{id}}}
\def \DecompSp {\mathcal{D}}
\def \Schwartz {\mathcal{S}}
\def \Fourier {\mathcal{F}}
\def \GL {{\mathrm{GL}}}
\def \supp {\operatorname{supp}}
\newcommand{\mybullet}{\bullet}

\newcommand{\PacketSpace}{\mathcal{W}}
\newcommand{\Indicator}{{\mathds{1}}}

\usepackage[pdfencoding=auto]{hyperref}

\theoremstyle{plain}
\newtheorem{thm}{Theorem}[section]
\theoremstyle{definition}
\newtheorem{defn}[thm]{Definition}
\theoremstyle{plain}
\newtheorem{lem}[thm]{Lemma}
\theoremstyle{plain}
\newtheorem{prop}[thm]{Proposition}
\theoremstyle{plain}
\newtheorem{cor}[thm]{Corollary}
\theoremstyle{remark}
\newtheorem*{rem*}{Remark}

\usepackage{etoolbox}
\patchcmd{\pprintMaketitle}
  {\fi\hrule}
  {\fi\ifvoid\extrainfobox\else\unvbox\extrainfobox\par\vskip10pt\fi\hrule}
  {}{}

\newsavebox\extrainfobox

\newcommand\nnfootnote[1]{%
  \begin{NoHyper}
  \renewcommand\thefootnote{}\footnote{#1}%
  \addtocounter{footnote}{-1}%
  \end{NoHyper}
}

\makeatletter
\newcommand*{\itemequation}[3][]{%
  \item
  \begingroup
    \refstepcounter{equation}%
    \ifx\\#1\\%
    \else
      \label{#1}%
    \fi
    \sbox0{#2}%
    \sbox2{$\displaystyle#3\m@th$}%
    \sbox4{\@eqnnum}%
    \dimen@=.5\dimexpr\linewidth-\wd2\relax
    \ifcase
        \ifdim\wd0>\dimen@
          \z@
        \else
          \ifdim\wd4>\dimen@
            \z@
          \else
            \@ne
          \fi
        \fi
      \@latex@warning{Equation is too large}%
    \fi
    \noindent
    \rlap{\copy0}%
    \rlap{\hbox to \linewidth{\hfill\copy2\hfill}}%
    \hbox to \linewidth{\hfill\copy4}%
    \hspace{0pt}
  \endgroup
  \ignorespaces
}
\makeatother

\begin{document}

\begin{frontmatter}

  \title{Design and properties of wave packet smoothness spaces}

  \author{Dimitri Bytchenkoff $^{1, 2, *}$}
  \author{Felix Voigtlaender $^{1,3, *}$}

  \address{$^{1}$Technische Universit\"at Berlin, Institut f\"ur Mathematik, Stra{\ss}e des 17. Juni 136, 10623 Berlin, Germany}
  \address{$^{2}$Universit{\'e} de Lorraine, Laboratoire d'Energ{\'e}tique et de M{\'e}canique Th{\'e}orique et Appliqu{\'e}e, 2 avenue de la For{\^e}t de Haye, 54505 Vandoeuvre-l{\`e}s-Nancy, France}
  \address{$^{3}$Katholische Universität Eichstätt--Ingolstadt, Lehrstuhl für wissenschaftliches Rechnen, Ostenstraße 26, 85072 Eichstätt, Germany}

   \begin{abstract}
     We introduce a family of quasi-Banach spaces --- which we call
     wave packet smoothness spaces --- that includes those function spaces which
     can be characterised by the sparsity of their expansions in Gabor frames, wave atoms, and many other frame constructions.
     We construct Banach frames for and atomic decompositions of the
     wave packet smoothness spaces and study their embeddings in each other
     and in a few more classical function spaces such as Besov and Sobolev spaces.

\section*{Résumé}

\noindent Nous introduisons une famille d’espaces affines quasi-normés complets – que nous appellerons espaces de paquets d’ondelettes réguliers – qui incluent de nombreux espaces de fonctions caractérisés par leurs transformées, comme celle de Gabor ou en ondelettes, clairsemées. Nous construisons des cadres de Banach et des décompositions atomiques pour ces espaces et étudions leurs inclusions l’un dans l’autre ainsi que dans quelques espaces de fonctions devenus classiques tels que ceux de Sobolev ou de Besov.

   \end{abstract}

   \begin{keyword}
     Wave packets \sep
     Banach frames \sep
     atomic decompositions \sep
     embeddings \sep
     analysis and synthesis sparsity \sep
     decomposition spaces \sep
     $\alpha$-modulation spaces \sep
     Sobolev and Besov spaces
     \MSC[2010]{42B35, 42C15, 46E15, 46E35, 42C40}
   \end{keyword}

\end{frontmatter}

\nnfootnote{* Corresponding authors.
E-mail addresses: \url{dimitri.bytchenkoff@univ-lorraine.fr} (D.\@ Bytchenkoff),
\url{felix@voigtlaender.xyz} (F.\@ Voigtlaender).}

\section{Introduction}
\label{sec:Introduction}

\noindent A large number of different frame constructions are used in harmonic analysis
both in practical applications such as image denoising
\cite{NonlinearWaveletImageProcessing,ImageDenoisingOrthonormalWavelet,BytchenkoffCardinalSeriesOne,BytchenkoffCardinalSeriesTwo},
restoring truncated signals \cite{BytchenkoffCardinalSeries,BytchenkoffExtrapolationPhaseCorrection},
edge detection \cite{ReisenhoferShearletFlame,ShearletClassificationOfEdges}
and compressed sensing \cite{HansenQuestForOptimalSampling}
and in pure mathematics for characterising
function spaces in terms of the frame coefficients
\cite{UllrichNewCharacterizations,AlphaShearletSparsity,
GroechenigTimeFrequencyAnalysis,FrazierJawerthDiscreteTransform,
TriebelTheoryOfFunctionSpaces3},
studying the boundedness of operators
\cite{GroechenigPsiDOOverview,MeyerWaveletsVolume1,MeyerWaveletsVolume2,KatoPseudoDifferentialOperatorsOnAlphaModulation}
or for characterising the wave front set of distributions
\cite{FellFuehrWaveFront,ShearletWavefront,CurveletWavefront}.
The most important of these frame constructions are wavelets \cite{DaubechiesTenLectures},
Gabor frames \cite{GroechenigTimeFrequencyAnalysis,GroechenigGaborFramesReview},
shearlets \cite{Kittipoom,GittaOptimallySparse},
curvelets \cite{CandesCurveletSecondGeneration},
ridgelets \cite{DonohoRidgelets,GrohsRidgelet}
and wave atoms \cite{DemanetWaveAtoms}.

All of these frames are constructed by applying dilations, modulations and translations
to a finite set of prototype functions.
Since the seminal work of Cordoba and Fefferman \cite{FeffermanWavePackets} --- who
used such systems to study the mapping properties of pseudo-differential operators --- it has become
customary to refer to such systems as \emph{wave packet systems}.
While the first papers mainly considered wave packet systems with continuous index sets
\cite{FeffermanWavePackets,NazaretInterpolationFamily}, nowadays the focus lies on
\emph{discrete} wave packet systems, which are special generalised shift invariant systems
\cite{HernandezLabateWeiss,LemvigVelthovenCriteriaForGSISystems,FuehrSystemBandwidth,
VelthovenLICForGSISystems}.
As particular highlights of the theory of wave packets, we mention the characterisation
of the Parseval property \cite{OversamplingQuasiAffineFramesWavePackets,ApproachToWavePacketSystems}
for such systems and the use of Gaussian wave packets for approximating solutions
of the homogeneous wave equation \cite{RomeroMultiscaleGaussianBeam}.

In the present paper, we will concentrate on the case of functions on $\R^2$
and consider the class of \emph{$(\alpha,\beta)$ wave packet systems} as introduced in
\cite{DemanetWaveAtoms}.
Here, the parameter $\alpha \in [0,1]$ describes the \emph{frequency bandwidth relationship}
of the system, while $\beta \in [0,1]$ describes its \emph{directional selectivity}.

More precisely, if $(\psi_i)_{i \in I}$ is a system of $(\alpha,\beta)$ wave packets
and if $\psi_i$ is concentrated at frequency $\xi \in \R^2$,
then the bandwidth of $\psi_i$ is approximately $(1 + |\xi|)^{\alpha}$.
The parameter $\alpha$ describes how multi-scale the system is.
For instance, for Gabor systems, the bandwidth of the
frame elements is independent of the frequency ($\alpha = 0$), while for
wavelets, the bandwidth is proportional to the frequency ($\alpha=1$).

The parameter $\beta$ determines how many different directions the
wave packet system can distinguish at each frequency scale; that is,
on the dyadic frequency ring $\{\xi \in \R^2 \,:\, |\xi| \asymp 2^j \}$,
an $(\alpha,\beta)$-wave packet system distinguishes approximately
$2^{(1-\beta) j}$ different directions.
For instance, wavelets are directionally insensitive
($\beta = 1$), while Gabor frames have high directional sensitivity
($\beta = 0$).
Figure~\ref{fig:WavePacketParametrization} shows how wave packet systems,
including their most important examples, relate to $\alpha$ and $\beta$.

\begin{figure}[h]
  \vspace{0.4cm}
  \begin{center}
    \includegraphics[width=0.7\textwidth]{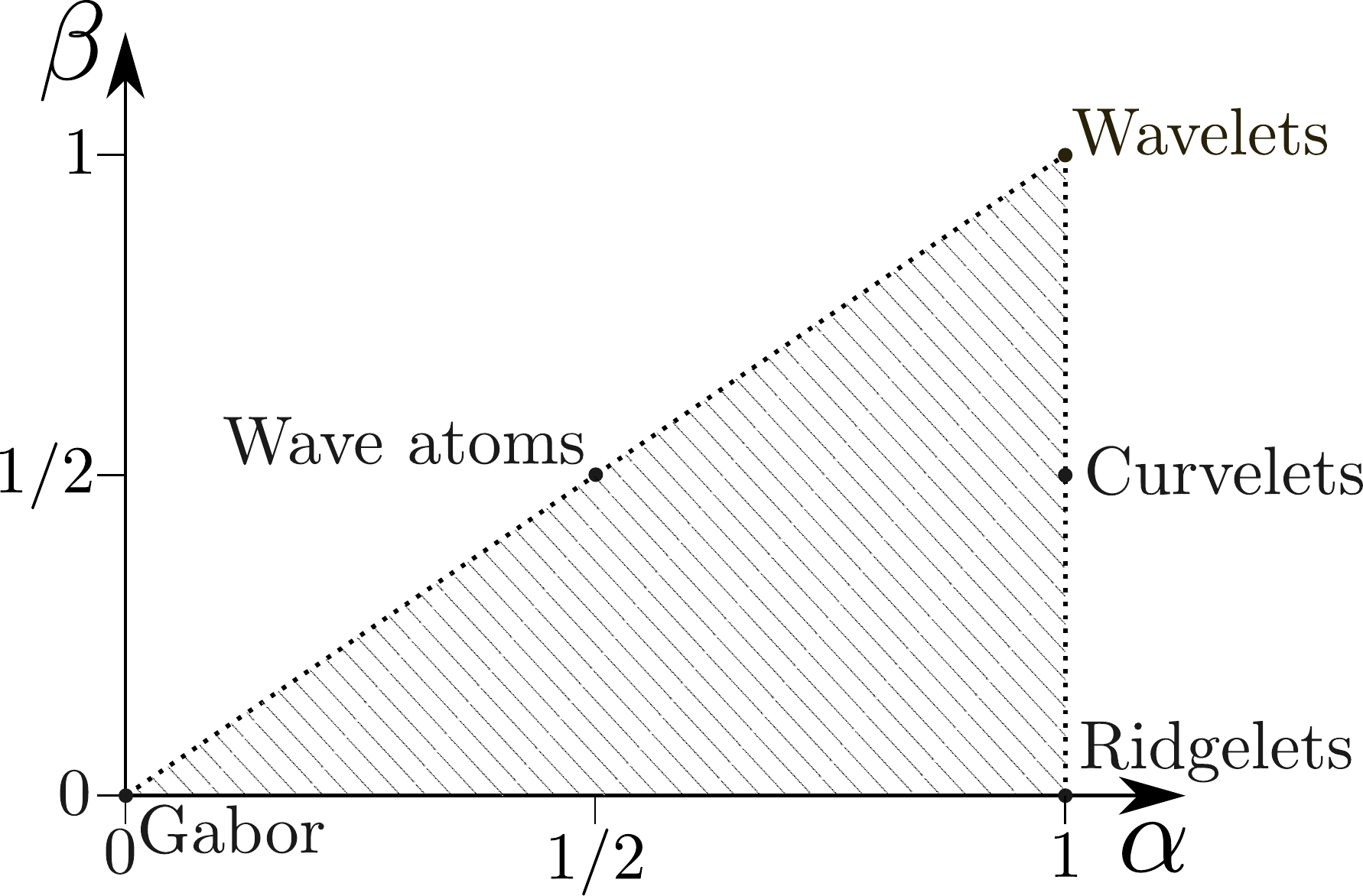}
  \end{center}
  \caption{\label{fig:WavePacketParametrization}
           Parametrisation of $(\alpha,\beta)$ wave packet systems
           including their most important special cases.
           In this work we focus on the regime where $0  \leq \beta \leq \alpha \leq 1$
           (hatched in the figure).}
\end{figure}

\paragraph{Our contribution}
In this work, we provide a rigorous mathematical framework for studying the
properties of $(\alpha,\beta)$ wave packet systems.
Specifically, for given $0 \leq \beta \leq \alpha \leq 1$, we define
a family of \emph{wave packet smoothness spaces}
$\PacketSpace_s^{p,q}(\alpha,\beta)$, parametrised by the integrability
exponents $p,q \in (0,\infty]$ and the smoothness parameter $s \in \R$,
and investigate properties of these spaces.

One of our main results is that if $\mathbb{W}(\alpha,\beta) = (\psi_i)_{i \in I}$
is a sufficiently regular frame of $(\alpha,\beta)$-wave packets, then
$\mathbb{W}(\alpha,\beta)$ constitutes a \emph{Banach frame} and an
\emph{atomic decomposition} for a whole family of wave packet smoothness spaces.
We would like to emphasise that the wave packet system is \emph{not} required
to be band-limited; on the contrary, we show that if the generators of the wave
packet system are \emph{compactly supported} and smooth enough, then the
resulting wave packet system will constitute a Banach frame and an atomic
decomposition for a family of wave packet smoothness spaces provided that the
sampling density of the wave packet system is fine enough.

In a nutshell, this means that \emph{the wave packet smoothness space is
characterised by the decay of the frame coefficients} with respect to the wave packet system.
More precisely, there is an explicitly given \emph{coefficient space}
$\mathcal{C}_s^{p,q}$ such that
\begin{equation}
  \PacketSpace_s^{p,q} (\alpha,\beta)
  = \big\{
      f
      \,:\,
      (\langle f \mid \psi_i \rangle_{L^2})_{i \in I} \in \mathcal{C}_s^{p,q}
    \big\}
  = \Big\{
      f = \sum_{i \in I} c_i \, \psi_i
      \,:\,
      c = (c_i)_{i \in I} \in \mathcal{C}_s^{p,q}
    \Big\} \, .
  \label{eq:IntroductionWavePacketSpaceFrameCharacterization}
\end{equation}
Moreover, a function $f \in \PacketSpace_s^{p,q}(\alpha,\beta)$ can be
continuously reconstructed from its \emph{analysis coefficients}
$(\langle f \mid \psi_i \rangle_{L^2})_{i \in I}$, and the
\emph{synthesis coefficients} $c(f) = (c_i)_{i \in I} \in \mathcal{C}_s^{p,q}$
satisfying $f = \sum_{i \in I} c_i \, \psi_i$ can be chosen to depend
linearly and continuously on $f$.

In a less technical terminology, the
identity~\eqref{eq:IntroductionWavePacketSpaceFrameCharacterization} means that
\emph{analysis} and \emph{synthesis sparsity} are equivalent for sufficiently
regular wave packet systems, where sparsity is quantified by the
coefficient space $\mathcal{C}_s^{p,q}$.
We note that $\mathcal{C}_{s}^{p,p} = \ell^p$ for a suitable choice of
$s = s(p,\alpha,\beta)$.
For non-tight frames, this equivalence between analysis- and synthesis sparsity
is nontrivial, but often useful.
For instance, it is usually relatively simple to verify that a certain class of
functions has sparse --- or quickly decaying --- analysis coefficients, which
amounts to estimating the inner products $\langle f \mid \psi_i \rangle_{L^2}$.
In contrast, it can be quite difficult to construct coefficients
$c = (c_i)_{i \in I}$ such that $f = \sum_{i \in I} c_i \, \psi_i$, even without
requiring that the sequence $c$ has good decay properties.
For applications in approximation theory or for studying the boundedness
properties of operators, however, it is usually much more useful to know that
$f = \sum_{i \in I} c_i \, \psi_i$ with sparse coefficients $c$,
rather than that the analysis coefficients of $f$ are sparse.

\smallskip{}

The second of our main findings are several useful results
concerning \emph{embeddings} of the wave packet smoothness spaces.
First, we study the existence of embeddings
\begin{equation}
  \PacketSpace_{s_1}^{p_1, q_1} (\alpha, \beta) \hookrightarrow
  \PacketSpace_{s_2}^{p_2, q_2}(\alpha', \beta')
  \label{eq:IntroductionWavePacketEmbedding}
\end{equation}
between wave packet spaces with different parameters.
Given~\eqref{eq:IntroductionWavePacketSpaceFrameCharacterization},
this amounts to asking whether sparsity of a function $f$ with
respect to an $(\alpha,\beta)$ wave packet system implies some, possibly worse,
sparsity with respect to an $(\alpha',\beta')$ wave packet system.
If $\beta \leq \beta'$ and $\alpha \leq \alpha'$ or if $\beta' \leq \beta$
and $\alpha' \leq \alpha$, we can \emph{completely characterise} the
existence of the embedding~\eqref{eq:IntroductionWavePacketEmbedding}.
Furthermore, we show that distinct parameter choices yield distinct wave packet smoothness spaces;
that is,
${\PacketSpace_{s_1}^{p_1, q_1}(\alpha, \beta) \neq \PacketSpace_{s_2}^{p_2, q_2}(\alpha', \beta')}$
unless $(p_1, q_1, s_1, \alpha, \beta) = (p_2, q_2, s_2, \alpha', \beta')$ or
$(p_1, q_1) = (2, 2) = (p_2, q_2)$ and $s_1 = s_2$.

Finally, we also consider embeddings between wave packet smoothness spaces
on the one hand and Besov- or Sobolev spaces on the other hand.
For the case of Besov spaces, we again obtain a complete characterisation
of the existence of the embeddings
$\PacketSpace_{s_1}^{p_1, q_1}(\alpha,\beta) \hookrightarrow B^{s_2}_{p_2,q_2}(\R^2)$
and of the reverse embedding; as a corollary, we show that
$B^s_{p,q}(\R^2) = \PacketSpace_s^{p,q}(1,1)$.
For the case of Sobolev spaces, we can completely characterise
the existence of the embedding
$\PacketSpace_s^{p,q} (\alpha,\beta) \hookrightarrow W^{k,r}(\R^2)$ for
$r \in [1,2] \cup \{\infty\}$.
For $r \in (2,\infty)$ we establish certain necessary and certain sufficient
conditions for the existence of the embedding, which are not equivalent.

In particular, we show that if $s \geq k + c(p)$, then
$\PacketSpace_s^{p,q}(\alpha,\beta) \hookrightarrow C_b^k (\R^2)$, so that
all functions in the wave packet smoothness space are $k$-times continuously differentiable.
This is one of the reasons for calling the spaces $\PacketSpace_s^{p,q}$
\emph{smoothness spaces}.


\medskip{}

One can in principle define
wave packet systems for arbitrary $\alpha,\beta \in [0,1]$.
In this work, however, we restrict ourselves to the case where
$0 \leq \beta \leq \alpha \leq 1$ for defining the wave packet smoothness spaces
and to $0 \leq \beta \leq \alpha < 1$ for constructing Banach frames and atomic decompositions
for these spaces.
This restriction $\alpha < 1$ is mainly done for convenience,
since the case $\alpha = 1$ was already explored in \cite{AlphaShearletSparsity},
which studies $\alpha$-shearlet systems
and the associated smoothness spaces for $\alpha \in [0,1]$.
In our terminology, $\alpha$-shearlets are $(1,\alpha)$ wave packets.

In contrast --- at least with our construction of the wave packet smoothness spaces --- the restriction
$\beta \leq \alpha$ seems to be unavoidable.
Precisely, we define the wave packet spaces as \emph{decomposition spaces}
\cite{DecompositionSpaces1} with respect to a certain covering of the frequency space,
which we call the $(\alpha,\beta)$ wave packet covering.
To show that this construction yields well-defined spaces, the wave packet covering
needs to satisfy a bounded overlap property; for this, the assumption $\beta \leq \alpha$
seems to be essential.
Finally, we are unaware of any frame construction that results in
$(\alpha,\beta)$ wave packets with $\beta > \alpha$; as seen in
Figure~\ref{fig:WavePacketParametrization}, all commonly used
frame constructions fall into the regime $\beta \leq \alpha$.

%
%

\paragraph{Structure of the paper}

To define the wave packet smoothness spaces
$\PacketSpace_s^{p,q}(\alpha,\beta)$, we shall use the formalism of
\emph{decomposition spaces},
originally introduced in \cite{DecompositionSpaces1}.
In order to define such a decomposition space
$\DecompSp(\CalQ, L^p, \ell_w^q)$, one needs a \emph{covering}
$\CalQ = (Q_i)_{i \in I}$ of the frequency domain
which has to satisfy certain regularity criteria;
namely, it has to be \emph{admissible} and, preferably,
\emph{almost-structured}.

In Section~\ref{sec:DecompositionSpaces}, we recall those parts of the existing
theory of decomposition spaces that are essential for our work.
In particular, we recapitulate the notions of admissible and almost-structured
coverings, the existing theory concerning the existence of embeddings between different
decomposition spaces, and the recent theory of structured Banach frame
decompositions of decomposition spaces.

In Section~\ref{sec:CoveringDefinition}, we introduce the
\emph{wave packet coverings} $\CalQ^{(\alpha,\beta)}$ that we shall use
to define the wave packet smoothness spaces and verify that
$\CalQ^{(\alpha,\beta)}$ indeed covers the whole frequency plane.
That the covering $\CalQ^{(\alpha,\beta)}$ is admissible and almost-structured
is shown in Sections~\ref{sec:Admissibility} and \ref{sec:Structuredness},
respectively.

The wave packet smoothness spaces will be defined in
Section~\ref{sec:WavePacketSpaces}, where we also study many of their
properties.
First, we show that they are indeed well-defined quasi-Banach spaces.
Second, we study the existence of embeddings
\(
  \PacketSpace_{s_1}^{p_1,q_1}(\alpha,\beta)
  \hookrightarrow \PacketSpace_{s_2}^{p_2,q_2}(\alpha',\beta')
\)
between
wave packet smoothness spaces with different parameters
and show that distinct parameters yield distinct spaces.
Third, we will completely characterise the existence of embeddings
between inhomogeneous Besov spaces and wave packet smoothness spaces.
Fourth, we study the conditions under which the wave packet smoothness spaces
embed into the classical Sobolev spaces $W^{k,p}(\R^2)$.
For the range $p \in [1,2] \cup \{\infty\}$ we characterise these conditions
completely.
Finally, we show that the $(\alpha,\alpha)$ wave packet smoothness spaces
are identical to $\alpha$-modulation spaces.

Since our construction of the covering $\CalQ^{(\alpha,\beta)}$
involves some non-canonical choice of parameters,
the spaces $\PacketSpace_s^{p,q}(\alpha,\beta)$ might appear rather esoteric.
In Section~\ref{sec:Universality}, we show that this is not the case.
Precisely, we introduce the \emph{natural} class of $(\alpha,\beta)$ coverings,
and show that any two $(\alpha,\beta)$ coverings give rise to the same family
of decomposition spaces. We also verify that $\CalQ^{(\alpha,\beta)}$
is indeed an $(\alpha,\beta)$ covering.
This shows that the wave packet smoothness spaces are natural objects,
and it allows us to show that the wave packet smoothness spaces are invariant
under dilation with respect to arbitrary invertible matrices
(see Section~\ref{sec:DilationInvariance}).

Finally, in Section~\ref{sec:SeriesConvergence}, we formally define the
notion of \emph{$(\alpha,\beta)$ wave packet systems}.
We then show that the wave packet smoothness spaces can be described using
the decay of the analysis or synthesis coefficients with respect to
such a system. More formally, we show that if the generators of the
wave packet system are sufficiently smooth and decay fast enough,
then the associated wave packet system constitutes a Banach frame and an atomic
decomposition for a whole range of wave packet smoothness spaces provided
that the sampling density is fine enough.

The proofs of some particularly lengthy auxiliary statements were transferred
to Appendices~\ref{sec:MainLemmaProof} -- \ref{sec:IntervalInclusionProofs}.
All general mathematical notions used in this manuscript
are summarised in Appendix~\ref{sec:Notation}.

\section{Decomposition spaces and their relation to frames and sparsity}
\label{sec:DecompositionSpaces}

\noindent Decomposition spaces, originally introduced by Feichtinger and Gröbner
\cite{DecompositionSpaces1}, provide a unified generalisation of modulation
and Besov spaces and were used to introduce the $\alpha$-modulation
spaces \cite{GroebnerAlphaModulationSpaces}, which have recently received
great attention \cite{KatoAlphaModulationSobolev,
SpeckbacherAlphaModulationTransform,GuoAlphaModulationEmbeddingCharacterization,
BorupNielsenAlphaModulationSpaces,FornasierFramesForAlphaModulation,
AlphaModulationNotInterpolation,HanWangAlphaModulationEmbeddings,
QuotientCoorbitTheoryAndAlphaModulationSpaces}.

The essential element of the decomposition space
$\DecompSp(\CalQ,L^p,\ell_w^q)$ is the covering $\CalQ = (Q_i)_{i \in I}$
of the frequency domain $\R^d$. Given this covering, the Fourier transform
$\widehat{g}$ of a given function $g$ can be decomposed into the components
$\varphi_i \cdot \widehat{g}$, where $(\varphi_i)_{i \in I}$ is a suitable
partition of unity subordinate to $\CalQ$. The Fourier inverse
$g_i := \Fourier^{-1} (\varphi_i \cdot \widehat{g})$ of the components
$\varphi_i \cdot \widehat{g}$ are frequency-localised components
of the function $g$. The contribution of each of these components $g_i$ to the
decomposition space quasi-norm of $g$ is measured by the $L^p$-norm, in the
time domain, and the total quasi-norm of $g$ is obtained by using the weighted
$\ell^q$ space $\ell_w^q$ as follows:
\begin{equation}
   \| g \|_{\DecompSp (\CalQ,L^p,\ell_w^q)}
  = \left\| \left( \| g_i \|_{L^p} \right)_{i \in I} \right\|_{\ell_w^q}
  = \left\|
      \left(
        w_i \cdot
        \| \Fourier^{-1}(\varphi_i \cdot \widehat{g}) \|_{L^p}
      \right)_{i \in I}
    \right\|_{\ell^q} \, .
  \label{eq:DecompositionSpaceIntroductionNormDefinition}
\end{equation}
To ensure that the decomposition space $\DecompSp(\CalQ,L^p,\ell_w^q)$
is indeed a well-defined quasi-Banach space, certain conditions must be imposed
on the covering $\CalQ$, the partition of unity $(\varphi_i)_{i \in I}$
and the weight $w = (w_i)_{i \in I}$.
These conditions and elementary properties of the decomposition spaces
$\DecompSp(\CalQ,L^p,\ell_w^q)$ will be reminded in
Section~\ref{sub:DecompositionDefinition}.

\medskip{}

An attractive feature of decomposition spaces is the recently developed theory
of \emph{structured Banach frame decompositions of decomposition spaces}
\cite{StructuredBanachFrames1}, which shows that there is a
close connexion between the existence of a sparse expansion of a given function
$f$ in terms of a given frame, and the membership of $f$ in a certain
decomposition space, which depends on the frame under consideration.

This theory will be formally introduced in Section~\ref{sub:BFD}; here, we
outline the underlying intuition.
We first note that most frame constructions used in harmonic analysis
have two crucial properties:
\begin{itemize}
  \item The frame is a \emph{generalised shift-invariant system}
        (see \cite{HernandezLabateWeiss,VelthovenLICForGSISystems,ChristensenConstructionAndPropertiesOfGSISystems,
        FuehrSystemBandwidth,LemvigVelthovenCriteriaForGSISystems} for more about these systems),
        i.e., it is of the form $\Psi = (L_x \, \psi_j)_{j\in I, x\in\Gamma_j}$ for
        suitable generators $(\psi_j)_{j \in I}$ and certain lattices
        $\Gamma_j = \delta B_j \Z^d$ where the matrices $B_j \in \GL (\R^d)$
        are determined by the frame construction
        and $\delta > 0$ globally stands for the sampling density.

        The matrices $B_j$ determine the \emph{relative step size} of the
        translations that are applied to each of the generators $\psi_j$.
        For example, in a Gabor system, $B_j = \identity$
        for all $j \in I$, while in a wavelet system,
        $B_j = 2^{-j} \cdot \identity$, so that the wavelets on higher scales
        have a much smaller step size than the wavelets on lower scales.

  \item The generators $\psi_j$ of the GSI system $\Psi$ have a characteristic
        \emph{frequency concentration}.
        For instance, in a Gabor frame,
        $I = \Z^d$ and $\psi_j (x) = e^{2\pi i \langle j, x \rangle} \cdot\psi(x)$ where $\psi$ is a given window function .
        Hence, if $\widehat{\psi}$ is concentrated in a subset $Q$ of the
        frequency domain $\R^d$, then $\widehat{\psi}_j$ is concentrated in
        $Q_j = Q + j$; the frequency tiling associated with a Gabor frame is thus
        uniform.
        Similarly, the frequency tiling associated with a wavelet system is dyadic.

        For most frame constructions, the Fourier transforms $\widehat{\psi}_j$
        of the generators $\psi_j$ are concentrated in subsets of the
        frequency domain of the form $Q_j = T_j Q + b_j$.
        Here, $Q \subset \R^d$ is a fixed bounded set, $b_j \in \R^d$ and $T_j = B_j^{-t}$
        where $\psi_j$ is, as mentioned above, translated along the lattice
        $\Gamma_j = \delta B_j \Z^d$.
\end{itemize}

Given such a system $\Psi = (L_x \, \psi_j)_{j \in I, x \in \Gamma_j}$,
the theory of structured Banach frame decompositions \cite{StructuredBanachFrames1} provides
verifiable conditions on the generators $\psi_j$ so that if these
conditions are satisfied and if the sampling density $\delta > 0$ is fine
enough, then $\Psi$ forms a \emph{Banach frame} and an \emph{atomic decomposition} for
the decomposition spaces $\DecompSp (\CalQ,L^p,\ell_w^q)$, where $\CalQ = (Q_j)_{j \in I}$.

In Subsection~\ref{sub:BFD}, we shall discuss in detail what these two notions
mean and what kind of conditions the generators have to satisfy.
Here, we merely note that for $p=q \in (0,2]$ and a suitable choice of the
weight $w = w(p)$, there is the following equivalence for any
$f \in L^2 (\R^d)$:
\begin{equation}
  \begin{split}
  f \in \DecompSp(\CalQ,L^p,\ell_w^p)
  & \quad \Longleftrightarrow \quad
    \big(
      \langle f \,\mid\, L_x \, \psi_j \rangle_{L^2}
    \big)_{j \in I, x \in \Gamma_j} \in \ell^p \\
  & \quad \Longleftrightarrow \quad
    \exists \, (c_{j,x})_{j \in I, x \in \Gamma_j} \in \ell^p :
      \,\, f = \smash{\sum_{j \in I, x \in \Gamma_j}}\vphantom{\sum_i}
                  \big( c_{j,x} \cdot L_x \, \psi_j \big) \, .
  \end{split}
  \label{eq:DecompositionSparsityCharacterizationGeneral}
\end{equation}
In other words, a function $f$ is sparse with respect to the frame $\Psi$ if and
only if $f$ belongs to the decomposition space $\DecompSp(\CalQ,L^p,\ell_w^p)$.

We mention that the theory of structured Banach frames is related to the findings
in \cite{NielsenRasmussenCompaclySupportedFrames}.
In that paper, Nielsen and Rasmussen establish the existence of compactly supported
Banach frames for certain decomposition spaces.
The main difference between these results and those in \cite{StructuredBanachFrames1}
is that the theory of structured Banach frame decompositions does not just establish the
\emph{existence} of Banach frames; rather, it allows to verify whether a \emph{given}
set of prototype functions generates a Banach frame or an atomic decomposition.
Moreover, the theory of structured Banach frames applies to more general coverings $\CalQ$
than those considered in \cite{NielsenRasmussenCompaclySupportedFrames}.


\medskip{}

Finally, as we shall be interested in a whole family of wave packet systems,
parametrised by $0 \leq \beta \leq \alpha < 1$, it is sensible to ask oneself whether
sparsity of a function $f$ in a given wave packet system
$\BoldW_{\alpha_1,\beta_1}(\psi_1;\delta_1)$ entails
some, albeit worse, sparsity in another wave packet system
$\BoldW_{\alpha_2,\beta_2}(\psi_2;\delta_2)\vphantom{\sum_j}$.
Given \eqref{eq:DecompositionSparsityCharacterizationGeneral},
this is equivalent to the question of whether
$\DecompSp\big(\CalQ^{(\alpha_1,\beta_1)},L^{p_1},\ell_{w_1}^{p_1}\big)$ is a subset
of $\DecompSp\big(\CalQ^{(\alpha_2,\beta_2)}, L^{p_2}, \ell_{w_2}^{p_2}\big)$
where the frequency covering $\CalQ^{(\alpha,\beta)}$ is the one
associated with $(\alpha,\beta)$-wave packet systems.
In many cases this question can be answered using the recently developed theory
of embeddings for decomposition spaces
\cite{DecompositionEmbeddings,VoigtlaenderPhDThesis},
which we shall briefly present in Subsection \ref{sub:DecompositionEmbedding}.

\subsection{Definition of decomposition spaces}
\label{sub:DecompositionDefinition}

\noindent As explained after \eqref{eq:DecompositionSpaceIntroductionNormDefinition},
one needs to impose certain conditions on the covering $\CalQ$, the weight $w$,
and the partition of unity $(\varphi_i)_{i \in I}$ in order to obtain
well-defined decomposition spaces.
Precisely, the covering $\CalQ$ should be \textbf{almost structured},
the partition of unity $(\varphi_i)_{i \in I}$ should be \textbf{regular},
and the weight $(w_i)_{i \in I}$ should be \textbf{$\CalQ$-moderate}.
%
%
Let us now give the definitions of these notions:
%
%

\begin{defn}\label{def:AlmostStructuredCovering}%
  (Definition~2.5 in \cite{DecompositionEmbeddings};
  inspired by \cite{BorupNielsenDecomposition})

\noindent A family $\CalQ = (Q_i)_{i \in I}$ is called an \textbf{almost structured
  covering} of $\R^d$, if there is an \textbf{associated family}
  $(T_i \bullet + b_i)_{i \in I}$ of invertible affine-linear maps
  such that the following properties hold:
  \begin{enumerate}
    \item $\CalQ$ is \textbf{admissible}; that is, the sets
          \begin{equation}
            i^\ast := \left\{
                        \ell \in I
                        \with
                        Q_\ell \cap Q_i \neq \emptyset
                      \right\}
            \quad \text{for} \quad i \in I
            \label{eq:IndexClusterDefinition}
          \end{equation}
          have uniformly bounded cardinality.

    \item There is $C > 0$ such that $\|T_i^{-1}T_\ell\| \leq C$ for all
          $i, \ell \in I$ for which $Q_i \cap Q_\ell \neq \emptyset$.

    \item There are $n \in \N$ and open, non-empty, bounded sets
          $Q_1^{(0)}, \dots, Q_n^{(0)}, P_1, \dots, P_n \subset \R^d$ such that
          \begin{itemize}
            \item for each $i \in I$ there is some $k_i \in \{1,\dots,n\}$ such that
                  $Q_i = T_i Q_{k_i}^{(0)} + b_i$;

            \item $P_k$ is compactly contained in $Q_k^{(0)}$; that is,
                  $\overline{P_k} \subset Q_k^{(0)}$ for all $k \in \{ 1,\dots,n \}$;


            \item $\R^d = \bigcup_{i \in I} (T_i P_{k_i} + b_i)$.
          \end{itemize}
  \end{enumerate}
  If it is possible to choose $n = 1$, then the covering becomes \textbf{structured},
  as it was defined in \cite{BorupNielsenDecomposition}.
\end{defn}



\begin{defn}\label{def:RegularPartitionOfUnity}%
  (Definition~2.4 in \cite{DecompositionIntoSobolev};
  inspired by \cite{BorupNielsenDecomposition})

\noindent Let $\CalQ = (Q_i)_{i \in I}$ be an almost structured covering of $\R^d$
  with associated family $(T_i \bullet + b_i)_{i \in I}$.
  A family of functions $\Phi = (\varphi_i)_{i \in I}$ is called a
  \textbf{regular partition of unity subordinate to $\CalQ$},
  if
  \begin{enumerate}
    \item $\varphi_i \in C_c^\infty (\R^d)$ with
          $\supp \varphi_i \subset Q_i$ for all $i \in I$;

    \item $\sum_{i \in I} \varphi_i \equiv 1$ on $\R^d$; and

    \item $\sup_{i \in I}
               \| \, \partial^\alpha \varphi_i^\natural \, \|_{L^\infty}
            < \infty$ for all $\alpha \in \N_0^d$, where
            $\varphi_i^\natural :\R^d\to\CC,\xi\mapsto\varphi_i(T_i \,\xi+b_i)$.
  \end{enumerate}
\end{defn}

\begin{rem*}
  The classical definition of decomposition spaces in \cite{DecompositionSpaces1}
  uses so-called BAPUs (bounded admissible partitions of unity) to define the decomposition spaces.
  The notion of regular partitions of unity is a modification of the concept of a BAPU and
  is necessary for handling the spaces $L^p$ for $p \in (0,1)$, which are not considered
  in \cite{DecompositionSpaces1}.
\end{rem*}

\begin{defn}\label{def:ModerateWeight}
  (Definition~3.1 in \cite{DecompositionSpaces1})

\noindent Let $\CalQ = (Q_i)_{i \in I}$ be an almost structured covering of $\R^d$.
  A \textbf{weight} on $I$ is a sequence $w = (w_i)_{i \in I}$
  where $w_i \in (0,\infty)$ for all $i \in I$. Such a weight is called
  \textbf{$\CalQ$-moderate}, if there is $C > 0$
  such that $w_i \leq C \cdot w_\ell$ for all $i,\ell \in I$ for which
  $Q_i \cap Q_\ell \neq \emptyset$.
  We write $C_{\CalQ,w}$ for the smallest constant $C$ for which this holds;
  that is, $C_{\CalQ,w} = \sup_{i \in I} \sup_{\ell \in i^\ast} w_i / w_\ell$.
\end{defn}

The following theorem ensures that, given an almost structured covering,
one can always find an associated regular partition of unity:
\begin{thm}\label{thm:RegularPartitionOfUnityExistence}%
  (Theorem~2.8 in \cite{DecompositionIntoSobolev};
  inspired by Proposition~1 in \cite{BorupNielsenDecomposition})

\noindent Let $\CalQ = (Q_i)_{i \in I}$ be an almost structured covering of $\R^d$.
  Then the index set $I$ is countably infinite and there exists a regular
  partition of unity subordinate to $\CalQ$.
\end{thm}


In principle, the decomposition space $\DecompSp (\CalQ, L^p, \ell_w^q)$
could be defined as the set of all tempered distributions
$g \in \Schwartz'(\R^d)$ for which
$\|g\|_{\DecompSp (\CalQ, L^p, \ell_w^q)}<\infty$ with the quasi-norm
$\|\bullet\|_{\DecompSp (\CalQ, L^p, \ell_w^q)}$ as defined in
\eqref{eq:DecompositionSpaceIntroductionNormDefinition}.
However, the decomposition space defined in this way
would not necessary be complete (see the example in Section~5 in
\cite{FuehrVoigtlaenderCoorbitSpacesAsDecompositionSpaces}).
To avoid this possible incompleteness, we shall use a slightly different
set than the space of tempered distributions for defining the decomposition
spaces:


\begin{defn}\label{def:ReservoirDefinition}(inspired by
  \cite{TriebelFourierAnalysisAndFunctionSpaces})

\noindent Let us define the set
  $Z := \Fourier (C_c^\infty (\R^d)) \subset \Schwartz (\R^d)$ and equip
  it with the unique topology that makes the Fourier transform
  $\Fourier : C_c^\infty (\R^d) \to Z$ into a homeomorphism.
  The topological dual space $Z'$ of $Z$ will be called the \textbf{reservoir}.
  We write $\langle \phi ,g \rangle_{Z'} := \langle \phi , g \rangle := \phi(g)$
  for the bilinear dual pairing between $Z'$ and $Z$.

  As in the space of tempered distributions, the
  \textbf{Fourier transform in the reservoir $Z'$}
  can be defined by using its duality with the space $Z$, i.e.
  \[
    \Fourier : Z'   \to    \CalD' (\R^d),
               \phi \mapsto \Fourier \phi
                            := \widehat{\phi}
                            := \phi \circ \Fourier \,
    \quad \text{and therefore} \quad
    \langle \Fourier \phi, g \rangle_{\CalD'}
    = \langle \phi, \widehat{g} \rangle_{Z'} \text{ for } g \in C_c^\infty (\R^d) \, .
  \]
  When $Z'$ and $\CalD'(\R^d)$ are both equipped with their
  respective weak-$\ast$-topologies, this Fourier transform
  is a homeomorphism with its inverse being given by
  $\Fourier^{-1} : \CalD'(\R^d) \to Z', \phi \mapsto \phi \circ \Fourier^{-1}$.
\end{defn}

We can now define our decomposition spaces.

\begin{defn}\label{def:DecompositionSpace}

\noindent Let $\CalQ = (Q_i)_{i \in I}$, $\Phi = (\varphi_i)_{i \in I}$,
  and $w = (w_i)_{i \in I}$ be an almost structured covering of $\R^d$,
  a regular partition of unity subordinate to $\CalQ$ and a $\CalQ$-moderate
  weight, respectively, and let $p,q \in (0,\infty]$.

  The \textbf{decomposition space} with the covering $\CalQ$, the weight $w$,
  and the integrability exponents $p$ and $q$ is defined as
  \[
    \DecompSp (\CalQ, L^p, \ell_w^q)
    := \left\{
         g \in Z' \with \| g \|_{\DecompSp(\CalQ,L^p,\ell_w^q)} < \infty
       \right\} \, ,
  \]
  where the decomposition space quasi-norm
  $\| g \|_{\DecompSp (\CalQ, L^p, \ell_w^q)}$ of any $g \in Z'$
  is defined as
  \begin{equation}
    \| g \|_{\DecompSp (\CalQ, L^p, \ell_w^q)}
    := \left\|
          \left(
            w_i \cdot \| \Fourier^{-1} (\varphi_i \cdot \widehat{g}) \|_{L^p}
          \right)_{i \in I}
       \right\|_{\ell^q} \in [0,\infty] \, .
    \label{eq:DecompositionSpaceNorm}
  \end{equation}
\end{defn}

\begin{rem*}
  At this point a few comments regarding the the notions
  introduced in Definition \ref{def:DecompositionSpace} are appropriate.

  First, we see from Definition~\ref{def:ReservoirDefinition} of the Fourier
  transform $\Fourier : Z' \to \CalD'(\R^d)$ that
  $\widehat{g} \in \CalD'(\R^d)$ for $g \in Z'$, whence $\varphi_i \cdot \widehat{g}$
  is a \emph{tempered} distribution with compact support.
  Furthermore the Paley-Wiener theorem (Theorem~7.23 in \cite{RudinFA})
  allows us to infer that $\Fourier^{-1}(\varphi_i \cdot \widehat{g})$ is
  (given by integration against) a smooth function of moderate growth so that
  the expression
  $\| \Fourier^{-1} (\varphi_i \cdot \widehat{g}) \|_{L^p} \in [0,\infty]$
  makes sense.
  Given the convention that
  $\| (c_i)_{i \in I} \|_{\ell^q} = \infty$ if $c_i = \infty$ for some
  $i \in I$, we can now conclude that
  ${\| g \|_{\DecompSp (\CalQ,L^p,\ell_w^q)} \in [0,\infty]}$ is indeed
  well-defined.

  Second, the combination of Corollary~2.7 in \cite{DecompositionIntoSobolev}
  and Corollary~3.18 in \cite{DecompositionEmbeddings} allows us to conclude
  that any two regular partitions of unity will yield equivalent quasi-norms
  as in Equation~\eqref{eq:DecompositionSpaceNorm}.
  Therefore, the space $\DecompSp(\CalQ,L^p,\ell_w^q)$ is independent
  of the choice of the regular partition of unity.

  Third, we chose to use the somewhat unusual reservoir $Z'$ to make sure
  that our decomposition space is complete.
  Indeed, in Theorem~3.12 in \cite{DecompositionEmbeddings},
  which uses the same definition of decomposition spaces as we do here,
  it is shown that
  $\big(\DecompSp(\CalQ,L^p,\ell_w^q), \|\bullet\|_{\DecompSp(\CalQ,L^p,\ell_w^q)}\big)$
  is a quasi-Banach space; that is, a complete quasi-normed vector space.

\end{rem*}

\subsection{Structured Banach frame decompositions for decomposition spaces}
\label{sub:BFD}

\noindent Let us select a particular almost structured covering
$\CalQ = (Q_i)_{i \in I}$ of $\R^d$ with associated family
${(T_i \bullet + b_i)_{i \in I}}$.
By definition, there are non-empty, open, bounded sets
$Q_1^{(0)}, \dots, Q_n^{(0)} \subset \R^d$ and for each $i \in I$
some $k_i \in \{1,\dots,n\}$ such that $Q_i = T_i Q_{k_i}^{(0)} + b_i$.
Let us select one particular such family $(k_i)_{i \in I}$,
which we shall use in the rest of this subsection.
Furthermore, let us define $Q_i ' := Q_{k_i}^{(0)}$ for $i \in I$.

\medskip{}

In this section, we consider generalised shift-invariant systems of the form
\begin{equation}
  \Gamma^{(\delta)}
  := \big( \, \gamma^{[i,k;\delta]} \, \big)_{i \in I, k\in \Z^d}
  := \big( \,
        L_{\delta \cdot T_i^{-t} \cdot k} \, \gamma^{[i]}
     \, \big)_{i \in I, k \in \Z^d}
  \label{eq:GSISystemDefinition}
\end{equation}
where the \emph{generators} $\gamma^{[i]}$ are given by
\begin{equation}
  \gamma^{[i]}
  := | \det T_i \, |^{1/2}\cdot M_{b_i} [\gamma_i (T_i^{t} \bullet)]
  \label{eq:GeneratorDefinition}
\end{equation}
with a suitable \emph{prototype functions} $\gamma_i \in L^2(\R^d)$.
The suitable choice of these prototype functions $\gamma_i$ will ensure that the system
$\Gamma^{(\delta)}$ is compatible with the frequency covering $\CalQ$.
Indeed, if the Fourier transform $\widehat{\gamma_i}$ of $\gamma_i$
decays rapidly outside the set $Q_i'$, then
\(
  \vphantom{\sum_j}
  \widehat{\gamma^{[i]}}
  = |\det T_i|^{-1/2} \cdot \widehat{\gamma_i} \big( T_i^{-1} ( \bullet - b_i) \big)
\),
so that the Fourier transform of $\gamma^{[i]}$ decays rapidly
outside the set $Q_i = T_i Q_i ' + b_i$.
Therefore, except for their normalisation,
the $\widehat{\gamma^{[i]}}$ are similar to a regular partition of unity
$\Phi = (\varphi_i)_{i \in I}$ subordinate to $\CalQ$.
Thus, with the generalised shift invariant system $\Gamma^{(\delta)}$ defined in
Equation~\eqref{eq:GSISystemDefinition},
one would intuitively expect that the membership of a function $g$ in the decomposition space
$\DecompSp(\CalQ,L^p,\ell_w^q)$ could be characterised in terms of the decay
of its coefficients
$\big( \, \langle g, \gamma^{[i,k;\delta]} \rangle \, \big)_{i \in I,k\in\Z^d}$.

The theory of structured Banach frames \cite{StructuredBanachFrames1},
whose elements essential to this work we shall remind here, makes this intuition
precise and provides criteria on the prototype functions $\gamma_i$ that,
if satisfied, will guarantee that the system $\Gamma^{(\delta)}$ constitutes
a \textbf{Banach frame} or an \textbf{atomic decomposition} for the
decomposition space $\DecompSp(\CalQ,L^p,\ell_w^q)$
provided the sampling density $\delta > 0$ is sufficiently fine.



The concept of Banach frames and atomic decompositions
\cite{GroechenigDescribingFunctions} are generalisations of
the notion of \textbf{frames} in Hilbert spaces.
By definition, a frame $(\psi_j)_{j \in J}$ in a Hilbert space $\CalH$ satisfies
${\|x\|_{\CalH}^2 \asymp \sum_{j \in J} |\langle x \,\mid\, \psi_j \rangle_{\CalH}|^2}$
for all $x \in \CalH$.
In other words, the norm of an element $x$ of the Hilbert space
can be characterised in terms of its coefficients
$(\langle x \,\mid\, \psi_j \rangle_{\CalH})_{j \in J}$.
This, given the rich structure of Hilbert spaces, has far-reaching consequences.
In particular $(\psi_j)_{j \in J}$ has a \textbf{dual frame}
$(\widetilde{\psi_j})_{j \in J}$ (see Theorem~5.1.6 in \cite{ChristensenFramesRieszBases})
which satisfies
\[
  x
  = \sum_{j \in J} \langle x \,\mid\, \widetilde{\psi_j} \rangle_{\CalH} \, \psi_j
  = \sum_{j \in J} \langle x \,\mid\, \psi_j \rangle_{\CalH} \, \widetilde{\psi_j}
  \qquad \forall \, x \in \CalH \, .
\]
Thus, any $x \in \CalH$ can be, on the one hand,
stably recovered from its coefficients $(\langle x \,\mid\, \psi_j \rangle)_{j \in J} \in \ell^2(J)$
and, on the other hand, represented as a series
$x = \sum_{j \in J} c_j \psi_j$, where the coefficients
$(c_j)_{j \in J} \in \ell^2(J)$ depend linearly and continuously on $x$.
Each of these properties can be shown to be equivalent to $(\psi_j)_{j \in J}$
being a frame for $\CalH$ and thus are equivalent to each other.
In Banach spaces, however, these properties are no longer equivalent,
thus leading to the introduction of the concepts of Banach frames and
atomic decompositions.

In the space $\ell^2 (J)$, if $c = (c_j)_{j \in J} \in \ell^2(J)$ and
$e = (e_j)_{j \in J}$ are such that $| e_j | \leq | c_j |$ for all $j \in J$,
then $e \in \ell^2 (J)$ and $\|e\|_{\ell^2} \leq \|c\|_{\ell^2}$.
More generally, a quasi-Banach space $X \subset \CC^J$ --- which consists of
sequences with index set $J$ --- with the analogous property is called \textbf{solid}.


\begin{defn}\label{def:AtomicDecomposition}

\noindent A family $\Psi = (\psi_i)_{i \in I}$ in a quasi-Banach space $Y$ is
called an \textbf{atomic decomposition} of $Y$ with coefficient space $X$, if
\begin{enumerate}[label=(\arabic*)]
  \item $X \subset \CC^J$ is a solid quasi-Banach space;

  \item the \textbf{synthesis map}
        $S_\Psi : X \to Y, (c_j)_{j \in J} \mapsto \sum_{j \in J} c_j \, \psi_j$
        is well-defined and bounded, with convergence of the series in a
        suitable topology; and

  \item there is such a bounded linear \textbf{coefficient map} $C_\Psi : Y \to X$
        that $S_\Psi \circ C_\Psi = \identity_Y$.
\end{enumerate}
\end{defn}

\begin{defn}\label{def:BanachFrame}

\noindent A family $\Theta = (\theta_j)_{j \in J}$ in the dual space $Y'$
of a quasi-Banach space $Y$ is called a \textbf{Banach frame} for $Y$ with
coefficient space $X$, if
\begin{itemize}
  \item $X \subset \CC^J$ is a solid quasi-Banach space;

  \item the \textbf{analysis map}
        $A_\Psi : Y \to X, f \mapsto (\langle f, \theta_j \rangle_{Y,Y'})_{j \in J}$
        is well-defined and bounded; and

  \item there is such a bounded linear \textbf{reconstruction map}
        $R_\Psi : X \to Y$ that $R_\Psi \circ A_\Psi = \identity_Y$.
\end{itemize}
\end{defn}

We now introduce the associated sequence spaces which we shall use
in the theory of structured Banach frames for decomposition spaces.
\begin{defn}\label{def:StructuredBFDSequenceSpace}
  (Definition~2.8 in \cite{AlphaShearletSparsity})

\noindent For $p,q \in (0,\infty]$ and $w = (w_i)_{i \in I}$,
  the \textbf{associated coefficient space}
  $C_w^{p,q} \subset \CC^{I \times \Z^d}$ is defined as
  \[
    C_w^{p,q}
    := \left\{
         c = (c_k^{(i)})_{i \in I, k \in \Z^d}
         \with
         \| c \|_{C_w^{p,q}}
         := \left\|
              \left(
                | \det T_i \, |^{\frac{1}{2} - \frac{1}{p}}
                \cdot w_i
                \cdot \| (c_k^{(i)})_{k \in \Z^d} \|_{\ell^p}
              \right)_{i \in I}
            \right\|_{\ell^q}
         < \infty
       \right\} \, .
  \]
\end{defn}

The following theorem on structured atomic decompositions for decomposition
spaces is a combination of Theorem 2.10 and Proposition 2.11 in \cite{AlphaShearletSparsity},
which provide simplified versions of the results obtained in \cite{StructuredBanachFrames1}.

\begin{thm}\label{thm:StructuredBFDAtomicDecomposition}

\noindent Let $\eps,p_0,q_0 \in (0,1]$, $p,q \in (0,\infty]$ such that $p \geq p_0$
  and $q \geq q_0$ and $w = (w_i)_{i \in I}$ be $\CalQ$-moderate.
  Let $\gamma_1^{(0)}, \dots, \gamma_n^{(0)} \in L^1(\R^d)$ and
  $\gamma_i := \gamma_{k_i}^{(0)}$ for $i \in I$.
  Let us define
  \[
    \Lambda := 1 + \frac{d}{\min \{1, p\}}
    \qquad \text{and} \qquad
    N := \left\lfloor \frac{d+\eps}{\min \{1,p\}} \right\rfloor
  \]
  and assume that, for each $k \in \{1,\dots,n\}$, there is a non-negative function
  $\varrho_k \in L^1(\R^d)$ such that the following hold:
  \begin{enumerate}
    \item \label{enu:AtomicDecompositionFourierTransformDecent}
          $\Fourier \gamma_k^{(0)} \in C^\infty (\R^d)$ and all partial derivatives
          of $\Fourier \gamma_k^{(0)}$ are of polynomial growth at most;


    \item \label{enu:AtomicDecompositionNonVanishing}
          $\Fourier \gamma_k^{(0)} (\xi) \neq 0$ for all
          $\xi \in \overline{Q_k^{(0)}}$;

    \item \label{enu:AtomicDecompositionSpaceSideDecay}
          \(
            \displaystyle
            \sup_{x \in \R^d}
              \left[
                 (1+|x|)^{\Lambda} \cdot | \gamma_k^{(0)}(x) |
              \right]
            < \infty \, ;
          \)

    \item \label{enu:AtomicDecompositionRhoDefinition}
          \(
            \Big|
              \partial^\alpha \big[\Fourier \gamma_k^{(0)}\big] (\xi)
            \Big|
            \leq \varrho_k (\xi) \cdot (1+|\xi|)^{-(d+1+\eps)}
          \)
          for all $\xi \in \R^d$ and $\alpha \in \N_0^d$ with $|\alpha| \leq N$.

  \end{enumerate}
  Finally, let us define
  \[
    \tau := \min \{1,p,q\},
    \qquad
    \vartheta := (p^{-1} - 1)_+
    \qquad \text{and} \qquad
    \sigma
    := \begin{cases}
         \tau \cdot (d+1),
         & \text{if } p \in [1,\infty] \, , \\
         \tau \cdot \big(p^{-1}\cdot d + \lceil p^{-1}\cdot (d+\eps)\rceil\big),
         & \text{if } p \in (0,1)
       \end{cases}
  \]
  and
  \[
    N_{i,j}
    := \left(
         \frac{w_i}{w_j} \cdot
         \left(
           |\det T_j| \,\big/\, |\det T_i|
         \right)^{\vartheta}
       \right)^{\!\tau}
       \!\! \cdot (1 + \| T_j^{-1} T_i \|)^{\sigma}
       \cdot \left(
               |\det T_i|^{-1}
               \,\, \int_{Q_i}
                        \varrho_{k_j} \big(T_j^{-1}(\xi - b_j)\big)
                    \, d \xi
             \right)^{\! \tau}
  \]
  for $i,j \in I$ and assume that
  \begin{equation}
    K_1 := \sup_{i \in I} \sum_{j \in I} N_{i,j} < \infty
    \qquad \text{and} \qquad
    K_2 := \sup_{j \in I} \sum_{i \in I} N_{i,j} < \infty.
    \label{eq:AtomicDecompositionConstantDefinition}
  \end{equation}
  Then there is a $\delta_0 > 0$ such that for any $\delta \in (0,\delta_0]$,
  the family $\Gamma^{(\delta)}$ as defined by \eqref{eq:GeneratorDefinition}
  and \eqref{eq:GSISystemDefinition} constitutes an atomic decomposition for the
  decomposition space $\DecompSp(\CalQ,L^p,\ell_w^q)$ with associated
  coefficient space $C_w^{p,q}$ as introduced in
  Definition \ref{def:StructuredBFDSequenceSpace}.

  More specifically,
  \begin{enumerate}[label=(\arabic*)]
    \item there is a constant
          $C = C(p_0,q_0,\eps,d,\CalQ,\gamma_1^{(0)},\dots,\gamma_n^{(0)}) > 0$
          that allows us to choose
          \[
            \delta_0
            = \min \left\{
                      1,
                      \Big[
                        C \cdot \big(K_1^{1/\tau} + K_2^{1/\tau} \, \big)
                      \Big]^{-1}
                   \right\} \, ;
          \]

    \item the synthesis map
          \begin{equation}
            S_{\Gamma^{(\delta)}} :
            C_w^{p,q} \to \DecompSp(\CalQ,L^p,\ell_w^q),
            (c_k^{(i)})_{i \in I, k \in \Z^d}
            \mapsto \sum_{i \in I}
                      \sum_{k \in \Z^d}
                        \left[
                          c_k^{(i)}
                          \cdot L_{\delta \cdot T_i^{-t}k} \gamma^{[i]}
                        \right]
            \label{eq:SynthesisMapDefinition}
          \end{equation}
          is well-defined and bounded for all $\delta \in (0,1]$.
          Moreover, for each $i \in I$ the inner series
          $\sum_{k \in \Z^d} [c_k^{(i)}\cdot L_{\delta\cdot T_i^{-t}k}\gamma^{[i]}]$
          converges absolutely to a function
          $g_i \in L_{\mathrm{loc}}^1 (\R^d) \cap \Schwartz'(\R^d)$
          and the series
          ${S_{\Gamma^{(\delta)}} \, (c_k^{(i)})_{i\in I,k\in \Z^d} = \sum_{i\in I} g_i}$
          converges unconditionally in the weak-$\ast$-sense in $Z'$; and

   \item for $0 < \delta \leq \delta_0$, there is a coefficient operator
         $C^{(\delta)} = C^{(\delta)}_{p,q,w}:
         \DecompSp(\CalQ,L^p,\ell_w^q) \to C_w^{p,q}$
         such that $S_{\Gamma^{(\delta)}} \circ C^{(\delta)}
         = \identity_{\DecompSp(\CalQ,L^p,\ell_w^q)}$.
         Furthermore, the action of $C^{(\delta)}_{p,q,w}$ on a given
         $f \in \DecompSp(\CalQ,L^p,\ell_w^q)$ is \emph{independent} of $p,q$
         and $w$, thus justifying the notation $C^{(\delta)}$.
  \end{enumerate}
\end{thm}

\begin{rem*}
  1) The description of the convergence of the series in \eqref{eq:SynthesisMapDefinition}
     might appear quite technical.
     Luckily, for $p,q < \infty$, this description can be simplified.
     Indeed, the finitely supported sequences are dense in $C_w^{p,q}$ if $p,q < \infty$.
     Combined with the boundedness of the synthesis map $S_{\Gamma^{(\delta)}}$,
     this implies that the series
     $\sum_{(i,k) \in I \times \Z^d} c_k^{(i)} L_{\delta T_i^{-t} k} \, \gamma^{[i]}$
     converges unconditionally in $\DecompSp(\CalQ,L^p,\ell_w^q)$.

  2) We note that the conditions
     \ref{enu:AtomicDecompositionFourierTransformDecent}
     and \ref{enu:AtomicDecompositionSpaceSideDecay} are satisfied as long as
     all $\gamma_k^{(0)}$ are bounded and have compact supports.
     In the case of the condition \ref{enu:AtomicDecompositionFourierTransformDecent},
     this is a consequence of the Paley-Wiener theorem.
\end{rem*}



For the next theorem --- which is a combination of Theorem~2.9
and Lemma~5.12 in \cite{AlphaShearletSparsity}),
we shall need a GSI system $\widetilde{\Gamma}^{(\delta)}$ that differs slightly from
the system $\Gamma^{(\delta)}$ given by \eqref{eq:GeneratorDefinition} and
\eqref{eq:GSISystemDefinition}.
Precisely, let us define
\begin{equation}
   \widetilde{\Gamma}^{(\delta)}
  := \left(
       L_{\delta \cdot T_i^{-t}k} \, \widetilde{\gamma^{[i]}}
     \right)_{i \in I, k \in \Z^d} \, ,
  \quad \text{where} \quad \widetilde{g}(x) = g(-x).
  \label{eq:BanachFrameFlippedGSISystem}
\end{equation}

\begin{thm}\label{thm:StructuredBFDBanachFrame}
  Let $\eps,p_0,q_0 \in (0,1]$ and $p,q \in (0,\infty]$ such that $p \geq p_0$ and $q \geq q_0$. Let $\Phi = (\varphi_i)_{i \in I}$ and $w = (w_i)_{i \in I}$ be a
  regular partition of unity for $\CalQ$ and $\CalQ$-moderate weight, respectively. Let $\gamma_1^{(0)}, \dots, \gamma_n^{(0)} \in L^1(\R^d)$ and let us define
  $\gamma_i := \gamma_{k_i}^{(0)}$ for $i \in I$.
  Let us assume that ,for all $k \in \{ 1,\dots,n \}$,
  \begin{enumerate}
    \item \label{enu:BanachFrameFourierTransformDecent}
          $\Fourier\gamma_k^{(0)}\in C^\infty(\R^d)$
          and all partial derivatives of $\Fourier \gamma_k^{(0)}$ are of polynomial growth at most;

    \item \label{enu:BanachFrameNonVanishing}
          $\Fourier \gamma_k^{(0)}(\xi) \neq 0$ for all $\xi \in \overline{Q_k^{(0)}}$;

    \item \label{enu:BanachFrameSpaceSideDecay}
          $\gamma_k^{(0)} \in C^1(\R^d)$ and $\nabla \gamma_k^{(0)} \in L^1(\R^d) \cap L^\infty(\R^d)$.

  \end{enumerate}
  Finally, let us define
  \[
    N := \left\lfloor \frac{d+\eps}{\min\{ 1,p \}} \right\rfloor,
    \qquad
    \tau := \min \{ 1,p,q \},
    \qquad
    \theta := \tau \cdot \Big( N + \frac{d}{\min \{ 1,p \}} \Big)
  \]
  and
  \[
    M_{j,i}
    := \left( \frac{w_j}{w_i} \right)^\tau
       \cdot (1 + \| T_j^{-1} T_i \|)^\theta
       \cdot \max_{|\beta| \leq 1}
               \left(
                 |\det T_i|^{-1} \,
                 \int_{Q_i}
                   \max_{|\alpha| \leq N}
                   \left|
                     \left[
                       \partial^\alpha
                         \widehat{\partial^\beta \gamma_j}
                     \right] \big( T_j^{-1} (\xi - b_j) \big)
                   \right|
                 \, d \xi
               \right)^\tau
  \]
  for $i,j \in I$ and assume that
  \[
    C_1 := \sup_{i \in I} \sum_{j \in I} M_{j,i} < \infty
    \qquad \text{and} \qquad
    C_2 := \sup_{j \in I} \sum_{i \in I} M_{j,i} < \infty.
  \]
  Then there is a $\delta_0 \in (0,1]$ such that for any
  $\delta \in (0,\delta_0]$, the family $\widetilde{\Gamma}^{(\delta)}$
  defined in Equations~\eqref{eq:GeneratorDefinition} and
  \eqref{eq:BanachFrameFlippedGSISystem} constitutes a Banach frame for the
  decomposition space $\DecompSp(\CalQ,L^p,\ell_w^q)$ with associated
  coefficient space $C_w^{p,q}$ as introduced in
  Definition~\ref{def:StructuredBFDSequenceSpace}.

  More specifically,
  \begin{enumerate}[label=(\arabic*)]
    \item There is a constant
          $C = C(p_0,q_0,\eps,d,\CalQ,\gamma_1^{(0)},\dots,\gamma_n^{(0)}) > 0$
          that allows us to choose
          \[
            \delta_0
            = \delta_0(p_0,q_0,w)
            = 1 \Big/ \Big[
                        1
                        + C
                          \cdot C_{\CalQ,w}^4
                          \cdot \left( C_1^{1/\tau} + C_2^{1/\tau} \right)^2
                      \Big] \, .
          \]

    \item The \textbf{analysis map}
          \[
            A_{\widetilde{\Gamma}^{(\delta)}} :
            \DecompSp(\CalQ,L^p,\ell_w^q) \to C_w^{p,q},
            f \mapsto \Big(
                        [\gamma^{[i]} \ast f] (\delta \cdot T_i^{-t}k)
                      \Big)_{i \in I, k \in \Z^d} \,\, ,
          \]
          where the convolution $\gamma^{[i]} \ast f$, i.e.
          \begin{equation}
            \big( \gamma^{[i]} \ast f \big)(x)
            = \sum_{\ell \in I}
                \Fourier^{-1}
                \Big(
                  \,\widehat{\gamma^{[i]}} \cdot \varphi_\ell \cdot \widehat{f} \,
                \Big)(x) \, ,
            \label{eq:BanachFrameConvolutionDefinition}
          \end{equation}
          is well-defined and bounded for all $\delta \in (0,1]$
          and the series in \eqref{eq:BanachFrameConvolutionDefinition} converges
          normally in $L^\infty (\R^d)$.

          Moreover, if
          $f \in L^2(\R^d) \hookrightarrow \Schwartz'(\R^d) \hookrightarrow Z'$,
          then the convolution defined by
          \eqref{eq:BanachFrameConvolutionDefinition} agrees with its usual
          definition and
          \begin{equation}
            A_{\widetilde{\Gamma}^{(\delta)}} f
            = \Big(
                \langle
                   f, L_{\delta \cdot T_i^{-t} \cdot k} \widetilde{\gamma^{[i]}}
                \rangle
              \Big)_{i \in I, k\in \Z^d}
            \quad
            \forall \, f \in L^2 (\R^d) \cap \DecompSp(\CalQ,L^p,\ell_w^q) \, .
            \label{eq:AnalysisOperatorConsistency}
          \end{equation}

    \item For $0 < \delta \leq \delta_0$, there is such a bounded linear
          reconstruction map
          $R^{(\delta)}_{p,q,w} : C_w^{p,q} \to \DecompSp(\CalQ,L^p,\ell_w^q)$ that
          \(
            R^{(\delta)}_{p,q,w} \circ A_{\widetilde{\Gamma}^{(\delta)}}
            = \identity_{\DecompSp(\CalQ,L^p,\ell_w^q)}
          \).

    \item If the assumptions of the current theorem are valid for
          $(p,q,w) = (p_\ell,q_\ell,w^{(\ell)})$ for $\ell \in \{1,2\}$ and
          $0 < \delta \leq \min \{\delta_0 (p_0,q_0,w^{(1)}),
          \delta_0(p_0,q_0,w^{(2)}) \}$, then
          \[
            \forall \, f \in \DecompSp(\CalQ,L^{p_2},\ell_{w^{(2)}}^{q_2})
            \quad : \quad
            f \in \DecompSp(\CalQ,L^{p_1},\ell_{w^{(1)}}^{q_1})
            \Longleftrightarrow
            A_{\widetilde{\Gamma}^{(\delta)}} f \in C_{w^{(1)}}^{p_1,q_1} \, .
          \]
  \end{enumerate}
\end{thm}


\subsection{Embeddings of decomposition spaces}
\label{sub:DecompositionEmbedding}

\noindent In this subsection, we recall from \cite{DecompositionEmbeddings} the
results concerning the existence of embeddings between two
decomposition spaces $\DecompSp(\CalQ, L^{p_1}, \ell^{q_1}_w)$ and
$\DecompSp(\CalP, L^{p_2}, \ell^{q_2}_u)$ which we shall need in the following.
Furthermore, we recall a few notions and results established by Feichtinger and
Gröbner \cite{DecompositionSpaces1} on which we shall rely in this work.

\begin{defn}\label{def:Embeddings}

\noindent Let $\CalQ$ and $\CalP$ be two almost-structured coverings of $\R^d$ and $w,u$ be a $\CalQ$-moderate weight and a $\CalP$-moderate weight, respectively
  and let $p_1, p_2, q_1, q_2 \in (0,\infty]$.

  We shall write $\DecompSp(\CalQ, L^{p_1}, \ell^{q_1}_w)
  \hookrightarrow \DecompSp(\CalP, L^{p_2}, \ell^{q_2}_u)$ and say that
  $\DecompSp(\CalQ, L^{p_1}, \ell^{q_1}_w)$ \emph{embeds} in
  $\DecompSp(\CalP, L^{p_2}, \ell^{q_2}_u)$, if
  $\DecompSp(\CalQ, L^{p_1}, \ell^{q_1}_w) \subset \DecompSp(\CalP, L^{p_2}, \ell^{q_2}_u)$
  and if the identity map
  \(
    \DecompSp(\CalQ, L^{p_1}, \ell^{q_1}_w)
    \to \DecompSp(\CalP, L^{p_2}, \ell^{q_2}_u),
    f \mapsto f
  \)
  is bounded.
\end{defn}

\begin{rem*}
  From the closed graph theorem (see Theorem~2.15 in \cite{RudinFA}),
  in combination with the embeddings
  ${\CalD(\CalQ,L^{p_1}, \ell_w^{q_1}) \hookrightarrow Z'}$
  and $\CalD(\CalP,L^{p_2}, \ell_u^{q_2}) \hookrightarrow Z'$
  (see Theorem~3.21 in \cite{DecompositionEmbeddings}), we infer that, if
  $\DecompSp(\CalQ,L^{p_1}, \ell_w^{q_1}) \subset \DecompSp(\CalP,L^{p_2}, \ell_u^{q_2})$, then
  $\DecompSp(\CalQ,L^{p_1}, \ell_w^{q_1}) \hookrightarrow \DecompSp(\CalP,L^{p_2}, \ell_u^{q_2})$;
  that is, the identity map is always bounded if the decomposition spaces are included in each other.
\end{rem*}

To be able to provide meaningful criteria allowing to decide whether such an embedding holds,
one needs a certain compatibility between the coverings $\CalQ$ and $\CalP$.
The required type of compatibility is discussed in the following definition.

\begin{defn}\label{def:Subordinateness}
  (Definition~3.3 in \cite{DecompositionSpaces1})

\noindent Let $\CalQ = (Q_i)_{i \in I}$ be an admissible covering of $\R^d$.
  Using the notation $i^\ast$ as introduced in \eqref{eq:IndexClusterDefinition},
  let us define $L^\ast := \bigcup_{\ell \in L} \ell^\ast \subset I$ for
  any $L \subset I$.
  Moreover, let us inductively define $L^{0\ast} := L$, and
  $L^{(n+1)\ast} := (L^{n\ast})^\ast$ for $n \in \N_0$.
  Finally, let us write $i^{n \ast} := \{i\}^{n \ast}$ and
  $Q_i^{n \ast} := \bigcup_{\ell \in i^{n \ast}} Q_\ell$ for $i \in I$ and
  $n \in \N$.

  \medskip{}

  Now, let $\CalQ = (Q_i)_{i \in I}$ and $\CalP = (P_j)_{j \in J}$ be two
  admissible coverings of $\R^d$. Let us define
  \begin{equation}
    I_j := \{ i \in I \,:\, Q_i \cap P_j \neq \emptyset \} \text{ for } j \in J
    \qquad \text{and} \qquad
    J_i := \{ j \in J \,:\, P_j \cap Q_i \neq \emptyset \} \text{ for } i \in I
    \, .
    \label{eq:IntersectionSets}
  \end{equation}

  \noindent
  We shall say that
  \begin{enumerate}[label=(\arabic*)]
    \item $\CalQ$ is \textbf{weakly subordinate} to $\CalP$ if
          $\sup_{i \in I} |J_i|$ is finite, that is, if the number of elements
          of the sets $J_i$ is uniformly bounded;

    \item $\CalQ$ is \textbf{almost subordinate} to $\CalP$ if
          \[
            \exists \, N \in \N_0 \quad
              \forall \, i \in I \quad
                \exists \, j_i \in J : \quad
                  Q_i \subset P_{j_i}^{N\ast} \, ;
          \]

    \item $\CalQ$ and $\CalP$ are \textbf{weakly equivalent} if
          $\CalQ$ is weakly subordinate to $\CalP$ and if also $\CalP$ is weakly
          subordinate to $\CalQ$; and

    \item $\CalQ$ and $\CalP$ are \textbf{equivalent}, if
          $\CalQ$ is almost subordinate to $\CalP$ and if also $\CalP$ is almost
          subordinate to $\CalQ$.
  \end{enumerate}
\end{defn}

Most of the results in \cite{DecompositionEmbeddings} concerning
embeddings of decomposition spaces will require $\CalQ$ to be almost
subordinate to $\CalP$, or vice versa.
However, this almost subordinateness is often quite difficult to verify.
Since it is often easier to verify that one covering is \emph{weakly}
subordinate to another, the following lemma will be useful.

\begin{lem}\label{lem:WeakAndAlmostSubordinateness}
  (slightly corrected version of Proposition~3.6 in \cite{DecompositionSpaces1};
   see also Lemma~2.12 in \cite{DecompositionEmbeddings})

\noindent Let $\CalQ = (Q_i)_{i \in I}$ and $\CalP = (P_j)_{j \in J}$ be two admissible
  coverings of $\R^d$ such that each $Q_i$ is path-connected and each
  $P_j$ is open.

  Then $\CalQ$ is weakly subordinate to $\CalP$ if and only if
  $\CalQ$ is almost subordinate to $\CalP$.
\end{lem}

In addition to the different concepts of subordinateness,
we shall also need the following two notions of
relative moderateness.

\begin{defn}\label{def:RelativeModerateness}

\noindent Let $\CalQ = (Q_i)_{i \in I}$ and $\CalP = (P_j)_{j \in J}$ be two almost
  structured coverings of $\R^d$ with associated families
  $(T_i \mybullet + b_i)_{i \in I}$ and $(S_j \mybullet + c_j)_{j \in J}$
  and let $w = (w_i)_{i \in I}$ be a weight.
  We shall say that
  \begin{enumerate}[label=(\arabic*)]
    \item $w$ is \textbf{relatively $\CalP$-moderate} if there is a constant
          $C > 0$ such that
          \[
            w_i \leq C \cdot w_\ell
            \qquad \text{for all } \, i,\ell \in I
                   \text{ and all } j \in J
                   \text{ for which } Q_i \cap P_j
                                      \neq \emptyset
                                      \neq Q_\ell \cap P_j \, ;
          \]

    \item $\CalQ$ is relatively $\CalP$-moderate if the weight
          $\big(|\det T_i|\big)_{i \in I}$ is relatively $\CalP$-moderate.
  \end{enumerate}
\end{defn}

We now state the two embedding results on which we shall rely.
In the first, we assume $\CalP$ to be almost subordinate to $\CalQ$,
while in the second we will assume $\CalQ$ to be almost subordinate to $\CalP$.

\begin{thm}\label{thm:EmbeddingPFinerThanQ}
  (special case of Theorem~7.2 in \cite{DecompositionEmbeddings})

\noindent Let $p_1,p_2, q_1, q_2 \in (0,\infty]$.
  Let $\CalQ = (Q_i)_{i \in I}$ and $\CalP = (P_j)_{j \in J}$ be two almost
  structured coverings of $\R^d$ with associated families
  $(T_i \mybullet + b_i)_{i \in I}$ and $(S_j \mybullet + c_j)_{j \in J}$.
  Let $w = (w_i)_{i \in I}$ and $v = (v_j)_{j \in J}$ be $\CalQ$-moderate and
  $\CalP$-moderate, respectively.

  Assume that $\CalP$ is almost subordinate to $\CalQ$ and that $\CalP$ and $v$
  are relatively $\CalQ$-moderate.
  Finally, for each $i \in I$, let us choose an index $j_i \in J$ that
  $Q_i \cap P_{j_i} \neq \emptyset$.
  Then $\DecompSp(\CalQ, L^{p_1}, \ell_w^{q_1}) \hookrightarrow
  \DecompSp(\CalP, L^{p_2}, \ell_v^{q_2})$ if and only if
  \[
    p_1 \leq p_2
    \quad \text{and} \quad
    \left\|
      \left(
        \frac{v_{j_i}}{w_i}
        \cdot |\det T_i|^{\nu}
        \cdot |\det S_{j_i}|^{p_1^{-1} - p_2^{-1} - \nu}
      \right)_{i \in I}
    \right\|_{\ell^{q_2 \cdot (q_1 / q_2)'}} < \infty
  \]
  where
  \[
    \nu := \left( q_2^{-1} - p_1^\ast \right)_+
    \qquad \text{with} \qquad
    x_+ := \max \{0, x\}
    \qquad \text{and} \qquad
    p_1^\ast := \min \big\{ p_1^{-1}, 1 - p_1^{-1} \big\}
    \, ,
  \]
  and where the exponent $q_2 \cdot (q_1 / q_2)' \in (0,\infty]$ is
  defined by
  \begin{equation}
    \frac{1}{q_2 \cdot (q_1 / q_2)'}
    = \left(q_2^{-1} - q_1^{-1}\right)_+ \,\, .
    \label{eq:EmbeddingSequenceSpaceExponent}
  \end{equation}
  In particular, $q_2 \cdot (q_1 / q_2)' = \infty$ if and only if $q_1 \leq q_2$.
\end{thm}

\begin{rem*}
  The definition~\eqref{eq:EmbeddingSequenceSpaceExponent} results in
  the same value as when computing $q_2 \cdot (q_1 / q_2)'$ as usual
  (with the conjugate exponent as defined in Appendix~\ref{sec:Notation})
  \emph{if the latter expression is defined};
  the advantage of \eqref{eq:EmbeddingSequenceSpaceExponent} is that it is defined
  in some cases where $q_2 \cdot (q_1 / q_2)'$ is not --- for instance if $q_2 = \infty$.
\end{rem*}


\begin{thm}\label{thm:EmbeddingQFinerThanP}
  (special case of Theorem~7.4 in \cite{DecompositionEmbeddings})

\noindent Let $p_1,p_2, q_1, q_2 \in (0,\infty]$, let $\CalQ = (Q_i)_{i \in I}$
  and $\CalP = (P_j)_{j \in J}$ be two almost
  structured coverings of $\R^d$ with associated families
  $(T_i \mybullet + b_i)_{i \in I}$ and $(S_j \mybullet + c_j)_{j \in J}$,
  and let $w = (w_i)_{i \in I}$ and $v = (v_j)_{j \in J}$ be $\CalQ$-moderate and
  $\CalP$-moderate, respectively.

  Let us assume that $\CalQ$ is almost subordinate to $\CalP$ and that $\CalQ$ and $w$
  are relatively $\CalP$-moderate.
  Finally, for each $j \in J$, let us choose $i_j \in I$ such that
  $Q_{i_j} \cap P_j \neq \emptyset$.
  Then ${\DecompSp(\CalQ, L^{p_1}, \ell_w^{q_1}) \hookrightarrow \DecompSp(\CalP, L^{p_2}, \ell_v^{q_2})}$
  if and only if
  \[
    p_1 \leq p_2
    \quad \text{and} \quad
    \left\|
      \left(
        \frac{v_{j}}{w_{i_j}}
        \cdot |\det T_{i_j}|^{p_1^{-1} - p_2^{-1} - \mu}
        \cdot |\det S_{j}|^{\mu}
      \right)_{j \in J}
    \right\|_{\ell^{q_2 \cdot (q_1 / q_2)'}} < \infty
  \]
  where the exponent $q_2 \cdot (q_1 / q_2)' \in (0,\infty]$ is as defined in
  \eqref{eq:EmbeddingSequenceSpaceExponent} and where
  \[
    \mu := \left( p_2^{\ast \ast} - q_1^{-1} \right)_+
    \qquad \text{with} \qquad
    p_2^{\ast \ast} := \big( \min \big\{p_2, p_2 ' \big\} \big)^{-1} \, .
  \]
  Here, $p_2 '$ is the conjugate exponent of $p_2 \in (0,\infty]$,
  as defined in Appendix~\ref{sec:Notation}.
\end{thm}

Finally, we shall also need the following rigidity result, which shows that if
two decomposition spaces are identical, then the ``ingredients'' used to define the
decomposition spaces are closely related.

\begin{thm}\label{thm:RigidityTheorem}(Theorem~6.9 in \cite{DecompositionEmbeddings})

\noindent Let $p_1, p_2, q_1, q_2 \in (0,\infty]$, $\CalQ = (Q_i)_{i \in I}$ and $\CalP := (P_j)_{j \in J}$ be two almost structured coverings
  of $\R^d$ and $w = (w_i)_{i \in I}$ and $v = (v_j)_{j \in J}$ be $\CalQ$-moderate
  and $\CalP$-moderate, respectively.

  \smallskip{}

  If $\DecompSp(\CalQ,L^{p_1},\ell_w^{q_1}) = \DecompSp(\CalP,L^{p_2},\ell_v^{q_2})$, then
  $(p_1, q_1) = (p_2, q_2)$ and there is a constant $C > 0$ such that
  \[
    C^{-1} \cdot w_i \leq v_j \leq C \cdot w_i
    \qquad \forall \, i \in I \text{ and } j \in J \text{ for which } Q_i \cap P_j \neq \emptyset.
  \]
  If furthermore $(p_1, q_1) \neq (2, 2)$ then $\CalQ$ and $\CalP$ are weakly equivalent.
\end{thm}


\section{Defining the wave packet covering \texorpdfstring{$\CalQ^{(\alpha,\beta)}$}{}}
\label{sec:CoveringDefinition}

\noindent In order to define the wave packet smoothness spaces,
we shall need suitable coverings of the frequency plane $\R^2$, which we now introduce.
We recall that $\N = \{ 1,2,3,\dots \}$, $\N_0 = \{0\} \cup \N$
and $B_r (x)$ is the Euclidean ball of radius $r$ with its centre in $x \in \R^d$.


\begin{defn}\label{defn:CoveringSets}

\noindent Let $0 \leq \beta \leq \alpha \leq 1$.
  First, let
  \begin{equation}
      N := 10,
      \qquad
      m_j^{\max}
      := m_j^{\max,\alpha}
      := \left \lceil 2^{\left(1-\alpha\right)j-1} \right \rceil,
      \qquad \text{and} \qquad
      \ell_j^{\max}
      := \ell_j^{\max,\beta}
      := \left\lceil
           N \cdot 2^{\left(1-\beta\right)j}
         \right\rceil
      \label{eq:MaxValues}
  \end{equation}
  and furthermore $I := I^{(\alpha,\beta)} := \{ 0 \} \cup I_0^{(\alpha,\beta)}$, where
  \begin{equation}
      I_0
      := I_0^{(\alpha,\beta)}
      := \left\{
            (j,m,\ell) \in \N \times \N_0 \times \N_0
            \with
            m \leq m_{j}^{\max} \text{ and } \ell \leq \ell_j^{\max}
         \right\} .
      \label{eq:IndexSet}
  \end{equation}
  Second, let us choose $\eps \in (0, 1/32)$ and define
  \begin{equation}
    Q := (-\eps, 1 + \eps) \times (-1-\eps, 1+\eps)
    \qquad \text{and} \qquad
    P := [0,1] \times [-1,1] \, .
    \label{eq:BaseSets}
  \end{equation}
  Third, for all $j \in \N$ and all $m \in \N_0$ such that $m \leq m_j^{\max}$, let us define
  \begin{equation}
      A_{j} := \left(
                 \begin{matrix}
                    2^{\alpha j} & 0 \\
                    0            & 2^{\beta j}
                 \end{matrix}
               \right)
      \qquad \text{and} \qquad
      c_{j,m} := \left(
                   \begin{matrix}
                     2^{j-1} \!+\! m \, 2^{ \alpha j} \\ 0
                   \end{matrix}
                 \right).
      \label{eq:DilationMatrixAndTranslation}
  \end{equation}
  Fourth, for all $\ell \in \N_0$ such that $\ell \leq \ell_j^{\max}$, let us define
  \begin{equation}
      R_{j,\ell} := \left(
                      \begin{matrix}
                         \cos \Theta_{j,\ell} & -\sin \Theta_{j,\ell} \\
                         \sin \Theta_{j,\ell} & \cos \Theta_{j,\ell}
                      \end{matrix}
                    \right)
      \quad \text{where} \quad
      \Theta_{j,\ell} := \Theta_{j,\ell}^{(\beta)} := 2 \ell \cdot \phi_j
      \quad \text{and} \quad
      \phi_j := \phi_j^{(\beta)} := \frac{\pi}{N} \cdot 2^{(\beta - 1) j} .
      \label{eq:RotationMatrix}
  \end{equation}
  Finally, for all $(j,m,\ell) \in I_0^{(\alpha,\beta)}$, let us define
  \begin{equation}
      Q_{j,m,\ell} := R_{j,\ell} \, (A_{j} \, Q + c_{j,m})
      \quad \text{and} \quad
      P_{j,m,\ell} := R_{j,\ell} \, (A_{j} \, P + c_{j,m}) ,
      \label{eq:CoveringSets}
  \end{equation}
  and set $Q_0 := B_4 (0)$ and $P_0 := B_3 (0)$.
  The family $\CalQ^{(\alpha,\beta)} := (Q_i)_{i \in I}$
  will be called the \textbf{$(\alpha,\beta)$ wave packet covering} of $\R^2$.
\end{defn}

\begin{figure}[h]
  \begin{center}
    \vspace{-0.1cm}
    \includegraphics[width=0.45\textwidth]{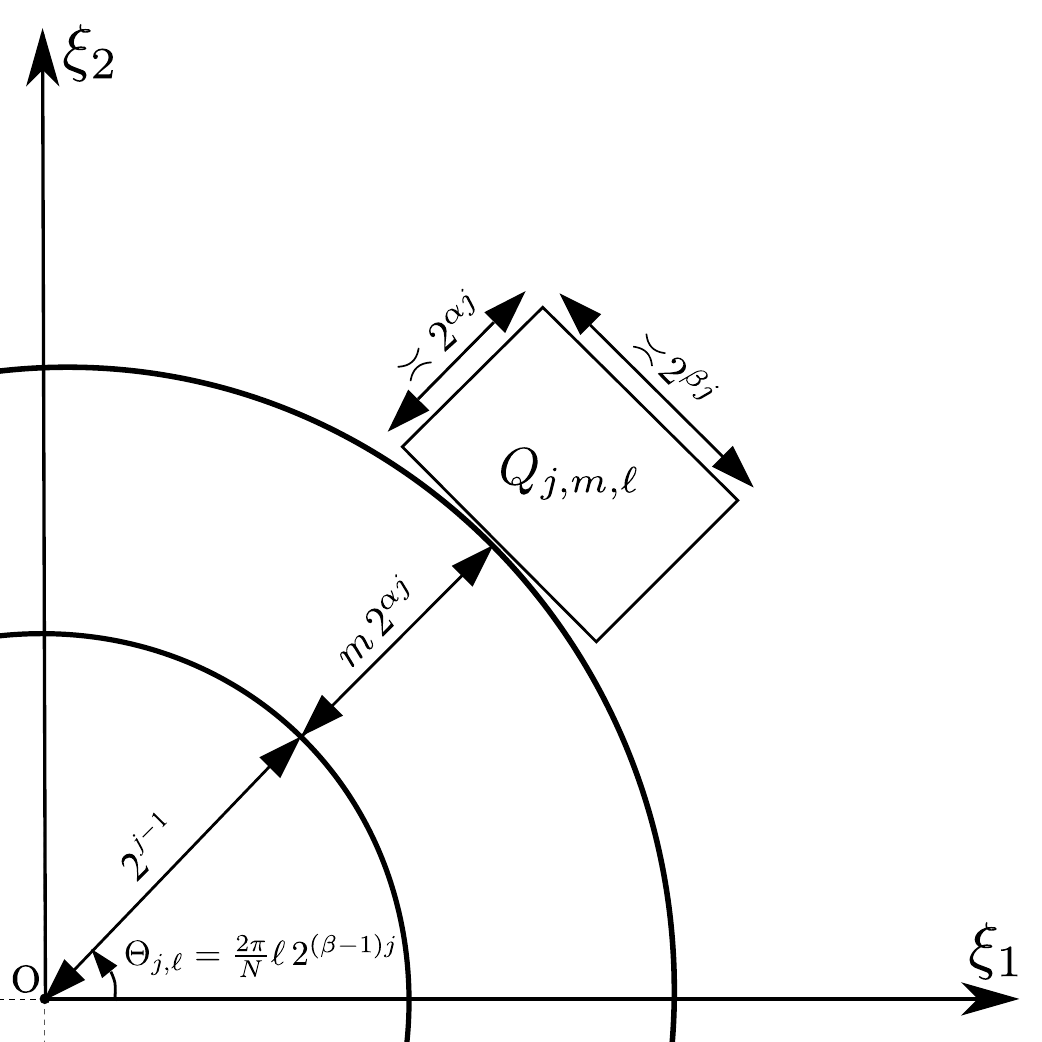}
    \vspace{-0.1cm}
  \end{center}
  \caption{\label{fig:RectangleSketch}
  An element $Q_{j,m,\ell}$ of the wave packet covering introduced
  in Definition~\ref{defn:CoveringSets}.
  The distance between the rectangle $Q_{j,m,\ell}$ and the origin $O$
  of the frequency plane is approximately $2^{j-1} + m 2^{\alpha j}$.
  The length of $Q_{j,m,\ell}$, in the radial direction, is approximately
  $2^{\alpha j}$ while its width, in the angular direction, is approximately $2^{\beta j}$.
  Its axis of symmetry intersecting the origin $O$ deviates from the
  $\xi_1$-axis by the angle
  $\Theta_{j,\ell} = \frac{2 \pi}{N} \ell \cdot 2^{(\beta-1)j}$.
  For a given $j$, all rectangles $Q_{j,m,\ell}$ are contained in the dyadic
  ring $\{ \xi \in \R^2 : |\xi| \asymp 2^j \}$.
  We also note that $m_j^{\max} \asymp 2^{(1-\alpha)j}$ and $\ell_j^{\max} \asymp 2^{(1-\beta)j}$.}
\end{figure}

\noindent In other words, the elements $Q_{j,m,\ell}$ of the covering
are generated from the rectangle $Q$ by scaling, shifting and
rotating --- the corresponding operators being represented by
$A_{j}$, $c_{j,m}$ and $R_{j,\ell}$ --- as schematically shown in
Figure~\ref{fig:RectangleSketch}.
For a given $j$, all rectangles $Q_{j,m,\ell}$ are contained in the dyadic
ring $\{\xi \in \R^2 \,:\, |\xi| \asymp 2^j\}$. Moreover, the length
of the rectangle $Q_{j,m,\ell}$, in the radial direction, is approximately
$2^{\alpha j}$ while its width, in the angular direction, is approximately
$2^{\beta j}$.



We shall now prove that the family $\CalQ^{(\alpha,\beta)}$ introduced
in Definition~\ref{defn:CoveringSets} is indeed a covering of $\R^2$.
Indeed, we shall prove the following stronger statement.

\begin{lem}\label{lem:CoveringCovers}
    Let $0 \leq \beta \leq \alpha \leq 1$.
    The sets $(P_i)_{i \in I^{(\alpha,\beta)}}$ and $(Q_i)_{i \in I^{(\alpha,\beta)}}$ introduced
    in Definition~\ref{defn:CoveringSets} satisfy
    \begin{equation}
        \R^2
        = P_0 \,\,
          \cup \,\, \bigcup_{j=1}^\infty \,\,
                      \bigcup_{m=0}^{m_j^{\max}} \,\,
                        \bigcup_{\ell = 0}^{\ell_j^{\max}} \,\,
                            P_{j,m,\ell}
        = Q_0 \,\,
          \cup \,\, \bigcup_{j=1}^\infty \,\,
                      \bigcup_{m=0}^{m_j^{\max}} \,\,
                        \bigcup_{\ell = 0}^{\ell_j^{\max}} \,\,
                            Q_{j,m,\ell} \, .
        \label{eq:Covering}
    \end{equation}
\end{lem}

\begin{proof}
First of all we note that $P_0 \subset Q_0$ and $P \subset Q$, whence
$P_{j,m,\ell} \subset Q_{j,m,\ell} \subset \R^2$ for all ${(j,m,\ell) \in I_0^{(\alpha,\beta)}}$.
Therefore, the second equality in \eqref{eq:Covering} indeed holds, provided
that the first holds.

\smallskip{}

Second we note that
\begin{equation}
    S_{j,m,\ell}
    := R_{j,\ell} \, S_{j,m,0}
    \subset P_{j, m-1, \ell} \cup P_{j,m,\ell}
    \qquad \text{for all} \quad (j,m,\ell) \in I_0^{(\alpha,\beta)}
           \quad \text{for which} \quad m \geq 1 \, ,
    \label{eq:SectorsIncludedInPSets}
\end{equation}
where\vspace{-0.3cm} 
\begin{align}
    S_{j,m,0}
    & := \left\lbrace
            \xi = r \cdot \left(
                            \begin{matrix}
                              \cos \phi \\ \sin \phi
                            \end{matrix}
                          \right)
            \, : \,
            2^{j-1} + m \, 2^{\alpha j}
            \leq r
            \leq 2^{j-1} + \left( m+1 \right) 2^{\alpha j}
            \hspace{0.25cm} \text{and} \hspace{0.25cm}
            \vert \phi \vert \leq \phi_j
       \right\rbrace
    \label{eq:NonRotatedSector}
\end{align}
and $\phi_j = \frac{\pi}{N} \cdot 2^{(\beta-1)j}$ as defined in
\eqref{eq:RotationMatrix}.

Indeed, from the definitions in \eqref{eq:CoveringSets} and
\eqref{eq:SectorsIncludedInPSets} 
we see that $P_{j,m,\ell}$ and $S_{j,m,\ell}$ can be obtained
by rotating $P_{j,m,0}$ and $S_{j,m,0}$ through the angle
$\Theta_{j,\ell} = 2\ell \cdot \phi_j$, respectively.
Therefore, we would prove~\eqref{eq:SectorsIncludedInPSets} in general,
should we prove it for $\ell = 0$.
To do so, we first note from~\eqref{eq:CoveringSets} and
\eqref{eq:DilationMatrixAndTranslation} that
\begin{equation}
       P_{j,m,0}
    =  [2^{j-1} +m 2^{\alpha j} , 2^{j-1} + (m+1) 2^{\alpha j}]
       \times [-2^{\beta j}, 2^{\beta j} \, ]
    =: P_{j,m,0}^{(1)} \times P_{j,m,0}^{(2)} \, ,
\label{eq:Pjm0}
\end{equation}
and therefore $\xi \in P_{j, m-1, 0} \cup P_{j,m,0}$ if and only if
$2^{j-1} + (m-1) 2^{\alpha j} \leq \xi_1 \leq 2^{j-1} + (m+1) 2^{\alpha j}$
and $|\xi_2| \leq 2^{\beta j}$.
We now verify that these conditions hold for $\xi \in S_{j,m,0}$.

Indeed, from~\eqref{eq:NonRotatedSector} we see that
if $\xi = r \cdot (\cos \phi, \sin \phi) \in S_{j,m,0}$, then
$\xi_1 \leq \vert \xi_1 \vert \leq r \leq 2^{j-1} + \left( m+1 \right) 2^{\alpha j}$
and
\[
  \xi_1
  = r \cos \phi
  = r \cos \vert \phi \vert
  \geq r \cos \phi_j
  \geq (2^{j-1} + m 2^{\alpha j})
       \cdot \left(1-\frac{2}{N} 2^{(\beta-1)j}\right) \, ,
\]
where we noticed that $\vert \phi \vert \leq \phi_j \leq \pi / 2$
as $N = 10$ and $0 \leq \beta \leq \alpha \leq 1$
and that the cosine is a decreasing function on
$[0, \frac{\pi}{2}]$ that satisfies
\begin{equation}
    1- \frac{2}{\pi} \phi
    \leq \cos \phi
    \leq \frac{\pi}{2} \left( 1 - \frac{2}{\pi} \phi \right)
    =         \frac{\pi}{2} - \phi
    \qquad \text{for} \qquad 0 \leq \phi \leq \frac{\pi}{2} \, ;
    \label{eq:COS}
\end{equation}
see Appendix~\ref{sec:TrigonometricLinearBounds} for a proof.

Furthermore, since $m \leq m_j^{\max} \leq 1+2^{(1-\alpha )j-1}$,
and noting that $\beta \leq \alpha \leq 1$ and hence
$2^{(\beta-\alpha)j} \leq 1$ and $2^{(\beta - 1)j + 1} \leq 2$,
we establish the following chain of implications:
\begin{align*}
    & \phantom{\Longleftrightarrow}
        (2^{j-1} + m \, 2^{\alpha j})
        \left( 1-\frac{2}{N} \cdot 2^{(\beta-1)j} \right)
        \overset{!}{\geq} 2^{j-1} + (m-1)2^{\alpha j} \\
    & \Longleftrightarrow
        1 - \frac{2}{N} \cdot 2^{(\beta-1)j}
        \overset{!}{\geq}
        \frac{2^{j-1} + (m-1)2^{\alpha j} }{2^{j-1} + m \, 2^{\alpha j}}
        = 1 - \left(2^{(1-\alpha) j -1} + m \right)^{-1} \\
    & \Longleftrightarrow
        N
        \overset{!}{\geq} 2 \cdot 2^{(\beta -1) j}
                          \cdot \left( 2^{(1-\alpha) j -1} + m \right) \\
    & \Longleftarrow
        N
        \overset{!}{\geq} 2 \cdot \left(
                                     2^{(\beta-\alpha) j -1}
                                     + \left(
                                         1+2^{(1-\alpha) j-1}
                                       \right) \cdot 2^{(\beta -1) j}
                                  \right)
          =         2^{(\beta-\alpha) j}
                    + 2^{(\beta -1) j+1}
                    + 2^{(\beta - \alpha ) j} \\
    & \Longleftarrow
        N \overset{!}{\geq} 4.
\end{align*}
The last inequality does indeed hold, since $N = 10$ by Definition~\ref{defn:CoveringSets}.
Thus we have demonstrated that
$2^{j-1} + (m-1) \, 2^{\alpha j} \leq \xi_1 \leq 2^{j-1} + (m+1) \, 2^{\alpha j}$
if $\xi \in S_{j,m,0}$.

Now we estimate $\xi_2$ for $\xi \in S_{j,m,0}$.
Write $\xi = r \cdot (\cos \phi, \sin \phi)^t$ with $r,\phi$ as in
Equation~\eqref{eq:NonRotatedSector}.
Next, note as a consequence of the definition of $m_j^{\max}$ in Equation~\eqref{eq:MaxValues}
that $m+1 \leq 2 + 2^{(1 - \alpha) j - 1}$, and recall that $\alpha - 1 \leq 0$ and $N = 10$.
In combination with the estimate $|\sin \phi| \leq |\phi|$, this implies
\begin{align*}
  |\xi_2|
  & = r \cdot |\sin \phi|
    \leq r \cdot |\phi|
    \leq (2^{j-1} + (m+1) \, 2^{\alpha j}) \cdot \frac{\pi}{N} \cdot 2^{(\beta - 1) j} \\
  & \leq \frac{2^{\beta j} \pi}{N} \cdot \Big( \frac{1}{2} + (m+1) \cdot 2^{(\alpha - 1) j} \Big)
    \leq \frac{2^{\beta j} \pi}{N} \cdot \Big( 1 + 2^{1 + (\alpha - 1) j} \Big)
    \leq \frac{3 \pi}{N} \cdot 2^{\beta j}
    \leq 2^{\beta j} .
\end{align*}
Overall, we have thus shown $\xi \in P_{j,m-1,0} \cup P_{j,m,0}$ for all $\xi \in S_{j,m,0}$.
As discussed above, we have thus proven Equation~\eqref{eq:SectorsIncludedInPSets}.

\medskip{}

Third, we note that
\begin{equation}
  S_{j,0,\ell}
  := R_{j,\ell} \, S_{j,0}
  \subset P_{j,0,\ell}
  \qquad \text{for} \quad
  j \in \N \quad \text{and} \quad
  \ell \in \N_0
  \,\, \text{ with } \,\, \ell \leq \ell_j^{\max},
  \label{eq:SectorInclusionNoShift}
\end{equation}
where
\begin{equation}
    S_{j,0}
    := \! \left\lbrace
          r \cdot  \left(
                     \begin{matrix}
                        \cos \phi \\ \sin \phi
                     \end{matrix}
                   \right)
          \, : \,
          2^{j-1}+ 2^{\alpha (j-1)} \leq r \leq 2^{j-1} + 2^{\alpha j}
          \hspace{0.25cm} \text{and} \hspace{0.25cm}
          \vert \phi \vert \leq \phi_j \!
       \right\rbrace  .
\label{eq:Sj00}
\end{equation}
Indeed, since $S_{j,0,\ell}$ and $P_{j,0,\ell}$ can be obtained by rotating
$S_{j,0} = S_{j,0,0}$ and $P_{j,0,0}$ 
using the matrix $R_{j,\ell}$, we would prove \eqref{eq:SectorInclusionNoShift}
in general, for any $\ell$, if we prove it for $\ell = 0$.

To do so, we infer from \eqref{eq:DilationMatrixAndTranslation}
and \eqref{eq:CoveringSets} that
\[
  P_{j,0,0}
  = [2^{j-1} , 2^{j-1} + 2^{\alpha j}]
    \times [-2^{\beta j} , 2^{\beta j}]
  =: P_{j,0,0}^{(1)} \times P_{j,0,0}^{(2)}.
\]
Furthermore, from~\eqref{eq:Sj00} we deduce that,
if $\xi = r \cdot (\cos \phi, \sin \phi)^t \in S_{j,0,0}$,
then on the one hand $\xi_1 \leq \vert \xi_1 \vert \leq r \leq 2^{j-1} + 2^{\alpha j}$,
but on the other hand, thanks to \eqref{eq:COS},
\[
  \xi_1
  = r \cos \phi
  = r \cos \vert \phi \vert \geq r \cos \phi_j
  \geq (2^{j-1} + 2^{\alpha (j-1)}) \left( 1-\frac{2}{N} 2^{(\beta-1)j} \right)
  \geq 2^{j-1} \, ,
\]
where the last inequality is justified by the following chain of implications:
\[
    \frac{2}{N} \cdot 2^{(\beta-1)j} \cdot (2^{j-1} + 2^{\alpha (j-1)})
    \overset{!}{\leq} 2^{\alpha (j-1)}
    \quad \! \Longleftrightarrow \quad \!
    N \overset{!}{\geq} 2^{1+\alpha}
                             \cdot \left(
                                      2^{(\beta-\alpha)j-1}
                                      +2^{-\alpha} 2^{(\beta-1)j}
                                   \right)
    \quad \! \overset{0 \leq \beta \leq \alpha \leq 1}{\Longleftarrow} \quad \!
    N \overset{!}{\geq} 8.
\]
The last inequality does indeed hold, since $N = 10$.
Thus we have shown that $\xi_1 \in P_{j,0,0}^{(1)}$ if $\xi \in S_{j,0,0}$.

Furthermore, if $\xi = r \cdot (\cos \phi, \sin \phi)^t \in S_{j,0,0}$, then
\[
  |\xi_2|
  = r \, |\sin \phi|
  \leq r \cdot |\phi|
  \leq (2^{j-1} + 2^{\alpha j}) \cdot \frac{\pi}{N} \, 2^{(\beta - 1) j}
  = \frac{\pi}{N} \, 2^{\beta j} \cdot (2^{-1} + 2^{(\alpha - 1) j})
  \leq \frac{3 \pi}{2 N} \cdot 2^{\beta j}
  \leq 2^{\beta j}
\]
and hence $\xi_2 \in P_{j,0,0}^{(2)}$ and $\xi \in P_{j,0,0}$.
This completes the proof of \eqref{eq:SectorInclusionNoShift} for $\ell = 0$
and hence in general, for any $\ell$.

\medskip{}

Finally, from~\eqref{eq:SectorsIncludedInPSets} we deduce that
\begin{equation}
  \bigcup_{j=1}^{\infty} \,\,
      \bigcup_{m=0}^{m_j^{\max}} \,\,
          \bigcup_{\ell=0}^{\ell_j^{\max}} P_{j,m,\ell}
  \supset \bigcup_{j=1}^{\infty} \,\,
              \bigcup_{m=1}^{m_j^{\max}} \,\,
                  \bigcup_{\ell=0}^{\ell_j^{\max}} S_{j,m,\ell}
  \supset \bigcup_{j=1}^{\infty}
              \left\lbrace
                  \xi \in {\mathbb{R}}^2
                  \, : \,
                  2^{j-1} + 2^{\alpha j}
                  \leq \vert \xi \vert
                  \leq 2^j + 2^{\alpha j}
              \right\rbrace .
\label{eq:U1}
\end{equation}
Here, we noted that $\ell_j^{\max} \geq \pi / \phi_j$ and
$m_j^{\max} = \lceil 2^{(1-\alpha) j - 1} \rceil$
thanks to \eqref{eq:MaxValues} and \eqref{eq:RotationMatrix} and therefore
$
  \bigcup_{\ell=0}^{\ell_j^{\max}} [\phi_j (2 \ell - 1), \phi_j (2 \ell + 1)]
  = [- \phi_j , \phi_j (2 \cdot \ell_j^{\max}+1)] \supset [0, 2 \pi]
$ 
and
\begin{align*}
    \bigcup_{m=1}^{m_j^{\max}}
       [2^{j-1} \! + m 2^{\alpha j} , 2^{j-1} \! + (m+1) 2^{\alpha j}]
    & =       [2^{j-1} \! +2^{\alpha j},
               2^{j-1} \! + (m_j^{\max} +1) 2^{\alpha j}] \\
    & \supset [2^{j-1} \! +2^{\alpha j},
               2^{j-1} \! + (2^{(1-\alpha) j -1} \! + 1) 2^{\alpha j}]
      =       [2^{j-1} \! +2^{\alpha j} , 2^{j} \! + 2^{\alpha j}].
\end{align*}

Similarly, from \eqref{eq:SectorInclusionNoShift}, we infer that
\begin{equation}
    \begin{split}
        \bigcup_{j=2}^{\infty} \,\,
            \bigcup_{\ell=0}^{\ell_j^{\max}} \,\,
                P_{j,0,\ell}
    &\supset \bigcup_{j=2}^{\infty} \,\,
                \bigcup_{\ell=0}^{\ell_j^{\max}}
                    S_{j,0,\ell}
    \supset \bigcup_{j=2}^{\infty}
                \left\lbrace
                    \xi \in {\mathbb{R}}^2
                    \, : \,
                    2^{j-1} + 2^{\alpha (j-1)}
                    \leq \vert \xi \vert
                    \leq 2^{j-1} + 2^{\alpha j}
                \right\rbrace \\
    &= \bigcup_{j=1}^{\infty}
         \left\lbrace
            \xi \in {\mathbb{R}}^2
            \,:\,
            2^{j} + 2^{\alpha j}
            \leq \vert \xi \vert
            \leq 2^{j} + 2^{\alpha (j+1)}
         \right\rbrace \, .
    \end{split}
    \label{eq:U2}
\end{equation}
Combining~\eqref{eq:U1} and~\eqref{eq:U2} results in
\begin{align*}
    \bigcup_{j=1}^{\infty} \,
        \bigcup_{m=0}^{m_j^{\max}} \,\,
            \bigcup_{\ell=0}^{\ell_j^{\max}} \,\,
                P_{j,m,\ell}
    &\supset \bigcup_{j=1}^{\infty}
                \left\lbrace
                    \xi \in {\mathbb{R}}^2
                    \,:\,
                    2^{j-1} \! + 2^{\alpha ((j-1)+1)}
                    \! \leq \! \vert \xi \vert
                    \! \leq \! 2^j \! + 2^{\alpha (j+1)}
                \right\rbrace 
    = \left\lbrace
        \xi \in {\mathbb{R}}^2
        :
        \vert \xi \vert \geq 1 \! + \! 2^{\alpha}
      \right\rbrace .
\end{align*}
Since $1 + 2^\alpha \leq 3$, this implies
\(
    B_{3} (0)
    \,\, \cup \,\, \bigcup_{j=1}^{\infty} \,\,
                       \bigcup_{m=0}^{m_j^{\max}} \,\,
                         \bigcup_{\ell=0}^{\ell_j^{\max}} \,\,
                            P_{j,m,\ell}
    = {\mathbb{R}}^2
\).
\end{proof}

\section{Proving admissibility of the wave packet covering}
\label{sec:Admissibility}


Our next lemma will clarify in more detail the geometric structure of the
wave packet covering and will be useful in proving its admissibility.
The lemma makes clear how the Euclidean length $|\xi|$ and the angle
$\angle(\xi)$ of the vectors $\xi \in Q_{j,m,\ell}$ are influenced by the
indices $j,m$ and $\ell$, respectively.

%

\begin{lem}\label{lem:CoveringGeometry}
  Let $0 \leq \beta \leq \alpha \leq 1$.
  With notation as in Definition~\ref{defn:CoveringSets}, let
  $(j,m,\ell) \in I_0^{(\alpha,\beta)}$ and $\xi \in Q_{j,m,\ell}$.
  Then
  \begin{equation}
    2^{j-2}
    < 2^{j-1} + 2^{\alpha j} \, (m - \eps)
    \leq |\xi|
    \leq 2^{j-1} + 2^{\alpha j} (m + 2 + 2 \eps)
    \leq 2^j + 2^{\alpha j} (3 + 2 \eps)
    < 2^{j + 3}
    \label{eq:AbsoluteValueEstimate}
  \end{equation}
  and
  \begin{equation}
    \exists \, \varphi \in \R : \qquad
      \xi = |\xi| \cdot e^{i \varphi}
      \quad \text{and} \quad
      \left| \varphi - \Theta_{j,\ell}\right|
      \leq 4 \, (1 + \eps) \cdot 2^{(\beta-1) j},
    \label{eq:AngleControl}
  \end{equation}
  where the vector $(\cos \varphi, \sin \varphi)^t \in \R^2$ is identified
  with the complex number $e^{i \varphi}$.
\end{lem}

\begin{proof}
  Since $Q_{j,m,\ell} = R_{j,\ell} \, Q_{j,m,0}$ can be obtained
  from $Q_{j,m,0}$ by rotation through the angle $\Theta_{j,\ell}$
  and since rotations preserve the Euclidean norm, we would prove
  \eqref{eq:AbsoluteValueEstimate}, in general, for $\xi \in Q_{j,m,\ell}$,
  should we prove it for $\xi \in Q_{j,m,0}$.
  To do so, directly from Definition~\ref{defn:CoveringSets} we infer that
  \begin{equation}
    Q_{j,m,0}
    = \left(
        2^{j-1} + 2^{\alpha j} (m - \eps) ,
        \, 2^{j-1} + 2^{\alpha j}(m+1+\eps)
      \right)
      \times
      \left(
        \vphantom{2^{\alpha j}}
        - (1+\eps) 2^{\beta j}, \,
        (1+\eps) 2^{\beta j}
      \right) .
    \label{eq:UnrotatedQExplicit}
  \end{equation}
  As $\varepsilon \in (0, 1/32)$ and $\alpha \leq 1$, we conclude that
  \begin{equation}
    |\xi|
    \geq \xi_1
    \geq 2^{j-1} + 2^{\alpha j} (m-\eps)
    \geq 2^{j-1} - \eps 2^{\alpha j}
    >    2^{j-2} > 0
    \qquad \forall \, \xi \in Q_{j,m,0} .
    \label{eq:UnrotatedSetXi1LowerBound}
  \end{equation}
  This completes the proof of the lower bound in \eqref{eq:AbsoluteValueEstimate}.

  Similarly, since $\xi_1 \geq 0$ for $\xi \in Q_{j,m,0}$ and since $\beta \leq \alpha$,
  we infer from \eqref{eq:UnrotatedQExplicit} that, for any $\xi \in Q_{j,m,0}$,
  \begin{align*}
    |\xi|
    \leq \xi_1 + |\xi_2|
    & \leq 2^{j-1} + 2^{\alpha j} (m + 1 + \eps) + (1 + \eps) 2^{\beta j}
    \leq 2^{j-1} + 2^{\alpha j} (m_j^{\max} + 2 + 2\eps) \\
    ({\scriptstyle{\text{Definition of } m_j^{\max}, \text{ see Eq.~}\eqref{eq:MaxValues}}})
    & \leq 2^{j-1} + 2^{\alpha j} \cdot \big( 2^{(1-\alpha)j - 1} + 3 + 2\eps \big)
      =    2^{j} + 2^{\alpha j} (3 + 2\eps) < 2^{j+3} \, .
  \end{align*}
  This completes the proof of the upper bound in
  \eqref{eq:AbsoluteValueEstimate}.

  \medskip{}


  To prove \eqref{eq:AngleControl}, let us first consider the case where
  $\xi \in Q_{j,m,0}$ and choose $\varphi \in [-\pi, \pi)$ such that
  $\xi = |\xi| \cdot e^{i \varphi}$.
  Since $\xi_1 = |\xi| \cdot \cos(\varphi)$ and $\xi_1 > 0$ for
  $\xi \in Q_{j,m,0}$ (see Equation~\eqref{eq:UnrotatedSetXi1LowerBound}),
  we conclude that $\varphi \in (-\pi/2, \pi/2)$.
  Since the derivative $\tan ' (\varphi) = 1 + \tan^2 (\varphi)$ of
  $\tan \varphi$ is not less than one for $\varphi \in (-\pi/2, \pi/2)$
  and since $\tan (0) = 0$, we conclude that $\tan (\varphi) \geq \varphi \geq 0$ for
  $\varphi \in [0,\pi/2)$ and
  $|\tan(\varphi)| = \tan(|\varphi|) \geq |\varphi| \geq 0$ for
  $\varphi \in (-\pi/2, \pi/2)$.
  Therefore,
  \begin{align*}
    |\varphi|
    \leq |\tan (\varphi)|
    & = \frac{|\xi_2|}{|\xi_1|} \\
    ({\scriptstyle{
                  \xi_1 \geq 2^{j-2} \text{ and }
                  |\xi_2|\leq (1+\eps) 2^{\beta j}
                  \text{ for } \xi \in Q_{j,m,0}}
                })
    & \leq \frac{(1+\eps) 2^{\beta j}}{2^{j-2}}
      \leq 4 \cdot (1+\eps) \cdot 2^{(\beta - 1) j} \, .
  \end{align*}
  This completes the proof of \eqref{eq:AngleControl} for $\xi \in Q_{j,m,0}$.

  In general, if $\xi \in Q_{j,m,\ell} = R_{j,\ell} \, Q_{j,m,0}$,
  there is $\xi' = |\xi'| \cdot e^{i \varphi_0} \in Q_{j,m,0}$
  such that $|\varphi_0| \leq 4 (1+\varepsilon) \cdot 2^{(\beta-1)j}$
  and $\xi = R_{j,\ell} \, \xi'\vphantom{\sum_j}$.
  Therefore, $\varphi := \varphi_0 + \Theta_{j,\ell}$ satisfies
  $\xi = |\xi| \cdot e^{i \varphi}$ and
  $|\varphi - \Theta_{j,\ell}| = |\varphi_0| \leq 4(1+\varepsilon) \cdot 2^{(\beta-1)j}$.
\end{proof}

We now turn to the proof of the admissibility of the covering from Lemma~\ref{lem:CoveringCovers}.

\begin{lem}\label{lem:Admissibility}
  Let $0 \leq \beta \leq \alpha \leq 1$.
  Then the covering $\CalQ := \CalQ^{(\alpha,\beta)} := (Q_i)_{i \in I}$
  from Definition~\ref{defn:CoveringSets} is admissible.

  \medskip{}

  \noindent More specifically,
  \begin{enumerate}[label=\alph*)]
    \item  \label{enu:IntersectionScaleRelation}
               for any given
               $(j,m,\ell), (j',m',\ell') \in I_0^{(\alpha,\beta)}$,
               \begin{equation}
                 Q_{j,m,\ell} \cap Q_{j',m',\ell'} = \emptyset
                 \quad \text{unless} \quad
                 |j - j'| \leq 3 \, ;
                 \label{eq:IntersectionScaleRelation}
               \end{equation}

    \item \label{enu:IntersectionShiftRelation}
              for any given $(j,m,\ell) \in I_0^{(\alpha,\beta)}$
              and $j' \in \N$, there are at most five different values of
              $m' \in \N_0$ such that there is $\ell' \in \N_0$ with
              $(j',m',\ell') \in I_0^{(\alpha,\beta)}$ and
              $Q_{j,m,\ell} \cap Q_{j',m',\ell'} \neq \emptyset$;

    \item \label{enu:IntersectionAngleControl}
              for any given
              $(j,m,\ell), (j',m',\ell') \in I_0^{(\alpha,\beta)}$,
              \begin{equation}
                Q_{j,m,\ell} \, \cap \, Q_{j',m',\ell'} = \emptyset
                \quad \! \text{unless} \quad \!
                \min_{k \in \{-2,\dots,2\}}
                |\Theta_{j,\ell} - \Theta_{j',\ell'} - 2\pi k|
                \leq 4 (1+\eps) \cdot (2^{(\beta-1)j} + 2^{(\beta - 1) j'})
                ;
                \label{eq:IntersectionAngleControl}
              \end{equation}

    \item \label{enu:IntersectionAngleCount}
          for any given $(j,m,\ell) \in I_0^{(\alpha,\beta)}$
          and $j' \in \N$, there are at most
          $65 N$ different values of
          $\ell' \in \N_0$ such that there is
          $m' \in \N_0$ with $(j',m',\ell') \in I_0^{(\alpha,\beta)}$ and
          $Q_{j,m,\ell} \cap Q_{j',m',\ell'} \neq \emptyset$; and

    \item \label{enu:IntersectionLowFrequencyCount}
          there are at most $135N$ different values of
          $(j',m',\ell') \in I_0^{(\alpha,\beta)}$ such that
          $Q_0 \cap Q_{j',m',\ell'} \neq \emptyset$.
  \end{enumerate}
\end{lem}

\begin{rem*}
  The derived bounds concerning the number of intersections are quite pessimistic,
  but sufficient for our purposes.
  The reason for the unappealing bounds is that we provide uniform bounds that apply simultaneously
  for all values of $0 \leq \beta \leq \alpha \leq 1$.
\end{rem*}

\begin{proof}
  \emph{Proof of \ref{enu:IntersectionScaleRelation}}
  Assume there is some $\xi \in Q_{j,m,\ell} \cap Q_{j',m',\ell'}$.
  We claim that $|j - j'| \leq 3$.
  To show this, let us assume the contrary, i.e. $|j-j'| \geq 4$.
  By symmetry, we can assume that $j \geq j'$, whence $0 \leq j' \leq j - 4$
  and $2^{\alpha j'} \leq 2^{j'} \leq 2^{j - 4}$.
  Thus, we infer from~\eqref{eq:AbsoluteValueEstimate} that
  \[
    2^{j - 1} - \eps \, 2^{j}
    \leq 2^{j - 1} - \eps 2^{\alpha j}
    \leq |\xi|
    \leq 2^{j'} + 2^{\alpha j'} (3 + 2 \eps)
    \leq 2^{j - 4} + 2^{j - 4} (3 + 2 \eps)
    =    2^{j-4} (4 + 2 \eps).
  \]
  Multiplying this estimate by $2^{4 - j}$, we obtain
  $2^3 - 2^4 \eps \leq 4 + 2 \eps$ and hence $\eps \geq \tfrac{2}{9}$,
  which contradicts our choice of $\eps \in (0, \tfrac{1}{32})$.

  \medskip{}

  \emph{Proof of \ref{enu:IntersectionShiftRelation}}
  We assume that $\xi \in Q_{j,m,\ell} \cap Q_{j',m',\ell'}$ and derive
  restrictions for the possible values of $m'$.
  To do this, we distinguish three possible cases:

  \emph{Case 1:} $j=j'$. From Lemma \ref{lem:CoveringGeometry} we infer that
  \[
     2^{j-1} + 2^{\alpha j} (m-\eps)
     \leq |\xi|
     \leq 2^{j'-1} + 2^{\alpha j'} (m' + 2 + 2\eps)
     =    2^{j-1} + 2^{\alpha j} (m' + 2 + 2\eps)
  \]
  and hence $m - m' \leq 2 + 3\eps$.
  By symmetry (interchanging the indices $(j,m,\ell)$ and $(j',m',\ell')$),
  this yields $|m - m'| \leq 2 + 3 \eps < 4$, that is, $|m - m'| \leq 3$.
  Thus, in case $j = j'$, the index $m'$ can take five different values at most.

  \smallskip{}

  \emph{Case 2}: $j' < j$.
  Thanks to \eqref{eq:IntersectionScaleRelation},
  we can write $j = j' + \kappa$ where $\kappa \in \{1,\dots,3\}$.
  From Lemma~\ref{lem:CoveringGeometry} we infer that
  $2^{j-1} - \eps 2^{\alpha j} \leq |\xi| \leq 2^{j' - 1} + 2^{\alpha j'} (m' + 2 + 2\eps)$
  and hence
  $2^{j-1} - 2^{j' - 1} \leq \eps 2^{\alpha j} + 2^{\alpha j'} (m' + 2 + 2\eps)$.
  Taking into account the possible values of $\kappa$, we conclude that
  $2^{j'-1} \leq 2^{j'-1} \cdot (2^{\kappa} - 1) = 2^{j-1} - 2^{j'-1}$.
  Combining the last two estimates with
  $\eps 2^{\alpha j} = \eps 2^{\alpha (j'+\kappa)} \leq 2^{\alpha j'}$ results in
  \[
    2^{j' - 1}
    \leq \eps 2^{\alpha j} + 2^{\alpha j'} (m' + 2 + 2\eps)
    \leq 2^{\alpha j'} (m' + 3 + 2\eps)
    <    2^{\alpha j'} (m' + 4) \, ,
  \]
  whence $2^{(1-\alpha)j' - 1} - 4 < m' \leq m_{j'}^{\max} \leq 2^{(1-\alpha)j' - 1} + 1$.
  Thus, in case $j' < j$, the index $m'$ can take five different values at most.

  \smallskip{}

  \emph{Case 3}: $j' > j$ and thus $j' \geq j+1$.
  From Lemma \ref{lem:CoveringGeometry} we infer that
  \[
    2^{j' - 1} + 2^{\alpha j'} (m' - \eps)
    \leq |\xi|
    \leq 2^{j} + 2^{\alpha j} (3 + 2\eps)
    \leq 2^{j'-1} + 2^{\alpha j'} (3 + 2\eps)
  \]
  and hence $0 \leq m' \leq 3 (1+\eps) < 4$.
  Thus, in case $j' > j$, the index $m'$ can take four different values at most.

  Combining our conclusions of the three cases completes the proof of b).



  \medskip{}

  \emph{Proof of \ref{enu:IntersectionAngleControl}}
  If $\xi \in Q_{j,m,\ell} \cap Q_{j',m',\ell'}$, then
  \eqref{eq:AngleControl} implies that there are $\varphi, \varphi' \in \R$
  such that $\xi = |\xi| \cdot e^{i\varphi}$ where
  $|\varphi - \Theta_{j,\ell}| \leq 4(1+\eps) \cdot 2^{(\beta-1)j}$ and
  such that $\xi = |\xi| \cdot e^{i \varphi'}$
  where $|\varphi' - \Theta_{j',\ell'}| \leq 4(1+\eps) \cdot 2^{(\beta-1)j'}$.
  Moreover, Equation~\eqref{eq:AbsoluteValueEstimate} shows that $|\xi| > 0$.
  Therefore, $e^{i\varphi} = e^{i\varphi'}$ so that there is
  $k \in \Z$ such that $\varphi - \varphi' = 2\pi k$.

  Taking into account that
  $\ell \leq \ell_j^{\max} \leq 1 + N \cdot 2^{(1-\beta)j}$, $N = 10$, that $\beta \leq 1$ and $\eps \leq \frac{1}{32}$, we conclude that
  \begin{equation}
    0 \leq \Theta_{j,\ell}
    = \frac{2\pi}{N} \cdot 2^{(\beta-1)j} \cdot \ell
    \leq \frac{2\pi}{N} \cdot 2^{(\beta-1)j} + 2\pi \leq \frac{22}{10} \pi
    \label{eq:ThetaJEstimate}
  \end{equation}
  and hence
  \[
    -\frac{14}{10} \pi
    <    - 4 (1+\eps)
    \leq - |\varphi - \Theta_{j,\ell}|
    \leq \varphi - \Theta_{j,\ell} + \Theta_{j,\ell}
    =    \varphi
    \leq |\varphi - \Theta_{j,\ell}| + \Theta_{j,\ell}
    \leq 4 (1 + \eps) + \frac{22}{10} \pi
    <    \frac{36}{10} \pi .
  \]
  In the same way, we also see that $-\tfrac{14}{10} \pi < \varphi' < \tfrac{36}{10} \pi$
  and hence $- 5 \pi < \varphi - \varphi' < 5 \pi$,
  so that $|k| = \frac{|\varphi - \varphi'|}{2\pi} < \frac{5}{2} < 3$,
  or, in other words, $k \in \{-2,-1,0,1,2\}$.
  Finally, we conclude, as claimed, that
  \begin{equation}
    |\Theta_{j,\ell} - \Theta_{j',\ell'} - 2\pi k|
    \leq |\Theta_{j,\ell} - \varphi|
         + |\varphi - \varphi' - 2\pi k|
         + |\varphi' - \Theta_{j',\ell'}|
    \leq 4(1+\eps) \cdot \big(2^{(\beta-1)j} + 2^{(\beta-1)j'}\big) \, .
    \label{eq:AngleControlInProof}
  \end{equation}

  \medskip{}

  \emph{Proof of \ref{enu:IntersectionAngleCount}}
  %
  Given \eqref{eq:AngleControlInProof} and the definition of $\Theta_{j',\ell'}$, we see that
  \[
    \left|
      \frac{2\pi}{N} \cdot 2^{(\beta-1)j'} \cdot (\ell' - \lambda_{j,\ell,k,j'})
    \right|
    \leq 4 (1+\eps) \cdot (2^{(\beta-1)j} + 2^{(\beta-1)j'})
  \]
  where
  \[
    \lambda_{j,\ell,k,j'}
    := \frac{N}{2\pi} \cdot 2^{(1-\beta)j'} \cdot (\Theta_{j,\ell} - 2\pi k)
    \in \R \, .
  \]
  Multiplying this estimate by $\frac{N}{2\pi} \cdot 2^{(1-\beta)j'}$
  and noting that $2^{(1-\beta)(j'-j)} \leq 2^{3(1-\beta)} \leq 8$,
  we conclude that $ |\ell' - \lambda_{j,\ell,k,j'}| \leq 6 N$.
  Since, for given $(j,m,\ell)$ and $j'$, the parameter $k \in \{-2,\dots,2\}$
  can only take up to five different values,
  the index $\ell'$ can take at most $5 \cdot 13 N = 65 N$ different values,
  as claimed.

  \medskip{}

  \emph{Proof of \ref{enu:IntersectionLowFrequencyCount}}
  For $\xi \in Q_0 \cap Q_{j',m',\ell'}$,
  the estimate \eqref{eq:AbsoluteValueEstimate} implies that
  $2^{j'-2} \leq |\xi| < 4 = 2^{2}$ and hence $j' \leq 3$
  if $Q_0 \cap Q_{j',m',\ell'} \neq \emptyset$.
  Furthermore
  \[
    0
    \leq \ell
    \leq \ell_{j'}^{\max}
    = \Big\lceil N \cdot 2^{(1-\beta)j'} \Big\rceil
    \leq 8N
    \qquad \text{and} \qquad
    0
    \leq m'
    \leq m_{j'}^{\max}
    = \Big\lceil 2^{(1-\alpha)j' - 1} \Big\rceil
    \leq 4 \, ,
  \]
  as $j' \leq 3$.
  Hence there can be at most $3 \cdot 5 \cdot 9N = 135N$ different triples
  $(j',m',\ell') \in I_0^{(\alpha,\beta)}$ such that
  $Q_0 \cap Q_{j',m',\ell'} \neq \emptyset$.

  \medskip{}

  Finally, we can prove the admissibility of $\CalQ^{(\alpha,\beta)}$.
  Combining \ref{enu:IntersectionScaleRelation},
  \ref{enu:IntersectionShiftRelation}, and \ref{enu:IntersectionAngleCount}
  we conclude that, for any given $i \in I_0^{(\alpha,\beta)}$, there are at most
  $7 \cdot 5 \cdot 65 \cdot N + 1$ different values of
  $i' \in I^{(\alpha,\beta)}$ such that $Q_i \cap Q_{i'} \neq \emptyset$.
  Part \ref{enu:IntersectionLowFrequencyCount} shows that this also holds for $i = 0$.
\end{proof}

\section{Proving almost-structuredness of the wave packet covering}
\label{sec:Structuredness}

\noindent We now prove that the wave packet covering $\CalQ^{(\alpha,\beta)}$ is almost structured.


%

\begin{lem}\label{lem:CoveringAlmostStructured}
  Let $0 \leq \beta \leq \alpha \leq 1$
  and let us define, with notations as in Definition \ref{defn:CoveringSets},
  \begin{equation}
    Q_1^{(0)} := Q,
    \quad
    T_i := T_{j,m,\ell} := R_{j,\ell} \, A_{j}
    \quad \text{and} \quad
    b_i := b_{j,m,\ell} := R_{j,\ell} \, c_{j,m}
    \quad \text{for} \quad
    i = (j,m,\ell) \in I_0^{(\alpha,\beta)}
    \label{eq:CoveringParametrization}
  \end{equation}
  and $Q_2^{(0)} := B_4 (0)$, $T_0 := \identity$ and $b_0 := 0$.
  Finally, set $k_i := 1$ for $i \in I_0^{(\alpha,\beta)}$ and $k_0 := 2$
  and $Q_i ' := Q_{k_i}^{(0)}$ for $i \in I^{(\alpha,\beta)}$.

  Then the admissible covering
  $\CalQ^{(\alpha,\beta)} = (Q_i)_{i \in I} = (T_i \, Q_i ' + b_i)_{i \in I}$
  with associated family
  $(T_i \mybullet + \, b_i)_{i \in I}$ is almost structured.
\end{lem}

\begin{proof}
  First of all, note that the family $\CalQ^{(\alpha,\beta)} = (Q_i)_{i \in I}$ indeed satisfies
  $Q_0 = T_0 B_4 (0) + b_0 = T_0 Q_0 ' + b_0$ and
  $Q_{j,m,\ell} = T_{j,m,\ell} Q + b_{j,m,\ell} = T_{j,m,\ell} Q_{j,m,\ell} ' + b_{j,m,\ell}$
  and that $Q_1^{(0)}, Q_2^{(0)} \subset \R^2$ are nonempty, open, and bounded.

  Moreover, $Q_1^{(0)} \supset \overline{P_1}$ and $Q_2^{(0)} \supset \overline{P_2}$
  for the non-empty, open, bounded sets
  \[
    P_1 := (-\eps/2, 1 + \eps/2) \times (-1-\eps/2, 1+\eps/2)
    \qquad \text{and} \qquad
    P_2 := B_3 (0).
  \]
  From Lemma \ref{lem:CoveringCovers} we infer that the family
  $(T_i P_{k_i} + b_i)_{i \in I}$ covers the entire frequency plane $\R^2$,
  and Lemma~\ref{lem:Admissibility} shows that $\CalQ^{(\alpha,\beta)}$ is admissible.
  Therefore, to prove that the covering $\CalQ^{(\alpha,\beta)}$ is almost structured,
  it is enough to show
  that there exists a constant $0 < C < \infty$ such that
  \begin{equation}
    \|T_i^{-1} T_{i'}\| \leq C
    \qquad \forall \, i, i' \in I^{(\alpha,\beta)}
                      \text{ for which } Q_i \cap Q_{i'} \neq \emptyset \, .
    \label{eq:StructurednessCondition}
  \end{equation}

  To do so we first consider the case where neither $i$ nor $i'$ are zero, i.e.,
  $i = (j,m,\ell)$ and $i' = (j',m',\ell')$ belong to $I_0^{(\alpha,\beta)}$.
  Note that $T_i^{-1} T_{i'} = A_j^{-1} R_{j,\ell}^{-1} R_{j',\ell'} A_{j'}$,
  so that a direct computation shows that
  \begin{equation}
    T_{j,m,\ell}^{-1} T_{j',m',\ell'}
    = \left(
        \begin{matrix}
            2^{\alpha (j' - j)}
            \cdot \cos (\Theta_{j',\ell'} - \Theta_{j,\ell})
          & - 2^{\beta j' - \alpha j}
            \cdot \sin (\Theta_{j',\ell'} - \Theta_{j,\ell}) \\
            2^{\alpha j' - \beta j} \cdot \sin (\Theta_{j',\ell'} - \Theta_{j,\ell})
          & 2^{\beta(j' - j)} \cdot \cos (\Theta_{j',\ell'} - \Theta_{j,\ell})
        \end{matrix}
      \right)
    =: \left( \begin{matrix} a & b \\ c & d \end{matrix}\right) \, .
    \label{eq:SimpleTransitionMatrixExplicit}
  \end{equation}

  From \eqref{eq:IntersectionScaleRelation} we infer that $|j - j'| \leq 3$,
  since $Q_i \cap Q_{i'} \neq \emptyset$.
  Recalling that $0 \leq \beta \leq \alpha \leq 1$, we thus see that
  \[
    |a| \leq 2^{\alpha (j'-j)} \leq 2^{\alpha |j'-j|} \leq 2^3,
    \qquad
    |b| \leq 2^{\beta j' - \alpha j}
        \leq 2^{\alpha (j' - j)}
        \leq 2^{\alpha |j' - j|}
        \leq 2^3
    \quad \text{and} \quad
    |d| \leq 2^{\beta (j'-j)} \leq 2^3.
  \]
  Furthermore, from \eqref{eq:IntersectionAngleControl} we conclude that
  \(
    |\Theta_{j,\ell} - \Theta_{j',\ell'} - 2\pi k|
    \leq 4(1+\eps) \cdot \big(2^{(\beta-1)j} + 2^{(\beta-1)j'}\big)
  \)
  for some $k \in \{-2,\dots,2\}$.
  Therefore, since the sine is $2\pi$-periodic and $|\sin \phi| \leq |\phi|$, we conclude that
  \begin{align*}
    |c| = 2^{\alpha j' - \beta j} \cdot |\sin (\Theta_{j',\ell'} - \Theta_{j,\ell} + 2\pi k)|
    & \leq 2^{\alpha j' - \beta j}
           \cdot 4(1+\eps)
           \cdot (2^{(\beta-1)j} + 2^{(\beta-1)j'}) \\
    & \leq 5 \cdot (2^{\alpha j' - j} + 2^{(\alpha-1)j' + \beta (j'-j)}) \\
    ({\scriptstyle{\text{since } \alpha j' - j \leq j' - j \leq 3 \text{ and } (\alpha - 1) j' \leq 0}})
    & \leq 5 \cdot (2^{3} + 2^{3\beta})
      \leq 80 \, .
  \end{align*}
  Thus, we have shown that
  \(
    \|T_{j,m,\ell}^{-1} T_{j',m',\ell'} \|
    \leq 2^3 + 2^3 + 2^3 + 80
    = 104
  \).

  %


  \medskip{}

  We now consider the case where $i = 0$ or $i' = 0$. If
  $i = i' = 0$, then $\|T_i^{-1} T_{i'}\| = 1 \leq 104$.
  Furthermore, if $\xi \in Q_{0} \cap Q_{j,m,\ell} \neq \emptyset$, then Lemma
  \ref{lem:CoveringGeometry} shows that $2^{j-2} \leq |\xi| < 4$, and hence
  $j \leq 3$. Therefore, since $\|R_{j,\ell}\| = \| R_{j,\ell}^{-1}\| = 1$
  and since $\| A_{j}^{-1} \| \leq 1$ and $\|A_j\| = 2^{\alpha j} \leq 2^3$,
  we finally deduce that $\| T_{j,m,\ell}^{-1} T_0 \| = \| A_{j}^{-1} \| \leq 1 \leq 104$
  and $\| T_0^{-1} T_{j,m,\ell} \| \leq \|A_j\| \leq 2^3 \leq 104$.
  This completes the proof of Equation~\eqref{eq:StructurednessCondition} with $C = 104$.
\end{proof}

\section{Defining the wave packet smoothness spaces and investigating their properties}
\label{sec:WavePacketSpaces}

\noindent Having proved that $\CalQ^{(\alpha,\beta)}$ is an almost structured and
admissible covering of $\R^2$, we shall now define the
\emph{wave packet smoothness spaces} $\PacketSpace_s^{p,q}(\alpha,\beta)$
as decomposition spaces associated with $\CalQ^{(\alpha,\beta)}$
and investigate their basic properties.
In particular, we shall demonstrate that the spaces
$\PacketSpace_s^{p,q}(\alpha,\beta)$ are embedded in the space of
tempered distributions and, under certain restrictions on its parameters,
in classical function spaces such as Besov and Sobolev spaces.
We shall also investigate the conditions under which
one wave packet smoothness space $\PacketSpace_{s_1}^{p_1,q_1}(\alpha,\beta)$
is embedded in another wave packet smoothness space
$\PacketSpace_{s_2}^{p_2,q_2}(\alpha',\beta')$.
Furthermore, we show that any two wave packet spaces $\PacketSpace_{s_1}^{p_1, q_1} (\alpha,\beta)$
and $\PacketSpace_{s_2}^{p_2, q_2} (\alpha',\beta')$ are distinct,
unless their parameters satisfy $(p_1, q_1, s_1, \alpha, \beta) = (p_2, q_2, s_2, \alpha', \beta')$
or $(p_1, q_1, s_1) = (2, 2, s) = (p_2, q_2, s_2)$ for some $s \in \R$.
Finally, we show that if $\alpha = \beta$, then
$\PacketSpace_s^{p,q}(\alpha,\alpha)$ coincides with the $\alpha$-modulation
space $M^{p,q}_{\alpha,s}(\R^2)$.

\subsection{Defining the wave packet smoothness spaces}
\label{sub:WavePacketSpacesDefinition}

\noindent The $(\alpha,\beta)$ \textbf{wave packet covering}
$\CalQ^{(\alpha,\beta)} = (Q_i)_{i \in I}$
with $I = I^{(\alpha,\beta)} = \{0\} \cup I_0^{(\alpha,\beta)}$ is an almost
structured covering of $\R^2$ as we saw in
Lemma~\ref{lem:CoveringAlmostStructured}.
In Section~\ref{sub:DecompositionDefinition}, we explained that this guarantees
that the associated decomposition spaces
$\DecompSp (\CalQ^{(\alpha,\beta)}, L^p, \ell_w^q)$ are well-defined
quasi-Banach spaces, as long as the weight $w = (w_i)_{i \in I}$ is
$\CalQ^{(\alpha,\beta)}$-moderate.
For the weights we are interested in, this is verified in the following lemma:

\begin{lem}\label{lem:WavePacketWeight}

\noindent For $0 \leq \beta \leq \alpha \leq 1$ and $s \in \R$, define
  \[
    w_i^{s}
    := \begin{cases}
         2^{j s} \, ,
         & \text{if } i = (j,m,\ell) \in I_0^{(\alpha, \beta)} \, , \\
         1 \, ,
         & \text{if } i = 0 \, .
       \end{cases}
  \]
  Then $w^{s} = (w_i^{s})_{i \in I}$ is $\CalQ^{(\alpha,\beta)}$-moderate.
\end{lem}

\begin{proof}
  Let $i, i' \in I$ with $\emptyset \neq Q_i \cap Q_{i'} \ni \xi$.
  Our goal is to show that $w_i^s / w_{i'}^s \leq 2^{3 |s|}$.

  First, let us consider the case where $i = (j, m, \ell) \in I_0$ and $i' = (j', m', \ell') \in I_0$.
  Then Equation~\eqref{eq:IntersectionScaleRelation} shows that $|j - j'| \leq 3$, whence
  $w_i^s / w_{i'}^s = 2^{s (j - j')} \leq 2^{|s| \cdot |j-j'|} \leq 2^{3 |s|}$.

  Second, we consider the case $i = (j, m, \ell) \in I_0$, but $i' = 0$.
  By virtue of Equation~\eqref{eq:AbsoluteValueEstimate}, this entails $2^{j - 2} < |\xi|$.
  Since $Q_0 = B_4 (0)$, this implies that $2^{j-2} \leq |\xi| < 2^{2}$ and hence $j \leq 3$.
  Therefore, $w_i^s / w_{i'}^s = 2^{js} \leq 2^{|j| \cdot |s|} \leq 2^{3 |s|}$.

  Third, if $i = 0$ and $i' = (j', m', \ell') \in I_0$, then we see as in the
  preceding case that $j' \leq 3$, whence
  $w_{i}^s / w_{i'}^{s} = 2^{-s j'} \leq 2^{|s| \cdot |j'|} \leq 2^{3 |s|}$.

  Finally, if $i = i' = 0$, then $w_i^s / w_{i'}^s = 1 \leq 2^{3 |s|}$ as well.
\end{proof}

With the preceding lemma, we know that the spaces introduced below are
well-defined quasi-Banach spaces.

\begin{defn}\label{def:WavePacketSmoothnessSpaces}

\noindent Let $0 \leq \beta \leq \alpha \leq 1$.
  For $s \in \R$ and $p,q \in (0,\infty]$, the
  $(\alpha,\beta)$ \textbf{wave packet smoothness space}
  associated with the parameters $p,q,s$ is the decomposition space
  \[
    \PacketSpace_{s}^{p, q} (\alpha, \beta)
    := \DecompSp (\CalQ^{(\alpha,\beta)}, L^p, \ell_{w^s}^q) \, .
  \]
\end{defn}

\begin{rem*}
  Recall from Lemma \ref{lem:CoveringGeometry} that $1 + |\xi| \asymp 2^{j}$
  for $\xi \in Q_{j,m,\ell}$.
  Therefore, the weight $w_i^s$ satisfies
  \begin{equation}
    w_i^s \asymp (1 + |\xi|)^s
    \qquad \text{for } \quad \xi \in Q_i
    \quad \text{ and } \quad i \in I^{(\alpha,\beta)} \, .
    \label{eq:WeightConsistencyCondition}
  \end{equation}
  Therefore, the weight $w^s$ here is similar to that in Besov- and modulation spaces.
\end{rem*}


\subsection{Investigating the conditions for inclusions between different wave packet smoothness spaces}
\label{sub:WavePacketEmbeddings}

\noindent In order to use the theory of embeddings for decomposition spaces
to establish conditions under which the inclusion
\[
  \PacketSpace_{s_1}^{p_1,q_1} (\alpha,\beta)
  \subset \PacketSpace_{s_2}^{p_2,q_2} (\alpha',\beta')
\]
holds, we first have to determine for which values of $\alpha,\beta$
and $\alpha',\beta'$ the covering
$\CalQ^{(\alpha,\beta)}$ is almost subordinate to
the covering $\CalQ^{(\alpha',\beta')}$.
This will be done in the following lemma.
In proving this lemma, we shall often use arguments similar to those
in the proof of Lemma~\ref{lem:Admissibility}.

In what follows, we shall write $T_i^{(\alpha,\beta)}$ rather than $T_i$
and $Q_i^{(\alpha,\beta)}$ rather than $Q_i$.
This will be done to avoid any confusion when we consider the two coverings
$\CalQ^{(\alpha,\beta)}$ and $\CalQ^{(\alpha',\beta')}$ at the same time.
We also remind the reader of the notations $m_j^{\max,\alpha}$ and $\ell_j^{\max,\beta}$,
$\Theta_{j,\ell}^{(\beta)}$ and $\phi_j^{(\beta)}$ introduced in Definition~\ref{defn:CoveringSets}.


\begin{prop}\label{prop:WavePacketCoveringSubordinateness}
  Let $0 \leq \beta \leq \alpha \leq 1$ and $0 \leq \beta' \leq \alpha' \leq 1$
  and let the coverings $\CalQ^{(\alpha,\beta)}$ and $\CalQ^{(\alpha',\beta')}$
  be as introduced in Definition~\ref{defn:CoveringSets}.
  Then
  \begin{equation}
    \forall \, (j,m,\ell) \in I_0^{(\alpha,\beta)} \quad
      \forall \, (j',m',\ell') \in I_0^{(\alpha',\beta')} \,\, : \,\,
        \text{if }
        Q_{j,m,\ell}^{(\alpha,\beta)} \cap Q_{j',m',\ell'}^{(\alpha',\beta')} \neq \emptyset
        \text{ then }
        |j - j'| \leq 4.
    \label{eq:RelativeModeratenessHelper}
  \end{equation}
  Moreover, $\CalQ^{(\alpha,\beta)}$ is almost subordinate to
  $\CalQ^{(\alpha',\beta')}$ if and only if $\alpha \leq \alpha'$ and $\beta \leq \beta'$.
\end{prop}

\begin{proof}
  First of all, if
  $\xi\in Q_{j,m,\ell}^{(\alpha,\beta)}\cap Q_{j',m',\ell'}^{(\alpha',\beta')} \neq \emptyset$,
  then \eqref{eq:AbsoluteValueEstimate} implies that both
  $2^{j-2} < |\xi| < 2^{j'+3}$ and $2^{j'-2} < |\xi| < 2^{j+3}$.
  Combining these estimates results immediately in
  \eqref{eq:RelativeModeratenessHelper}.

  \medskip{}

  \textbf{Part 1:}
  In this part, we assume that $\alpha \leq \alpha'$ and $\beta \leq \beta'$
  and prove that $\CalQ^{(\alpha,\beta)}$ is almost subordinate to $\CalQ^{(\alpha',\beta')}$.
  To do so, let us define
  \[
    J_i
    := \big\{
             i' \in I^{(\alpha',\beta')}
             \,:\,
             Q_{i'}^{(\alpha',\beta')} \cap Q_i^{(\alpha,\beta)} \neq \emptyset
       \big\}
    \quad \text{for} \quad
    i \in I^{(\alpha,\beta)} \, .
  \]
  Since the coverings $\CalQ^{(\alpha,\beta)}$ and $\CalQ^{(\alpha',\beta')}$
  consist of open \emph{path-connected}, indeed convex, sets,
  Lemma~\ref{lem:WeakAndAlmostSubordinateness} shows that
  $\CalQ^{(\alpha,\beta)}$ is almost subordinate to
  $\CalQ^{(\alpha',\beta')}$ if and only if there is $K > 0$ such that
  $|J_i| \leq K$ for all $i \in I^{(\alpha,\beta)}$.
  To verify this, it will be enough to prove the following claims:
  \begin{enumerate}[label=\alph*)]
    \item For any given $i = (j,m,\ell) \in I_0^{(\alpha,\beta)}$ and
          $j' \in \N$, there are at most five different values of $m' \in \N_0$
          such that there is some $\ell' \in \N_0$ with
          $(j',m',\ell') \in I_0^{(\alpha',\beta')} \cap J_i$;

    \item For any given $i = (j,m,\ell) \in I_0^{(\alpha,\beta)}$ and
          $j' \in \N$, $m' \in \N_0$, there are at most $125 N$ different values
          of $\ell' \in \N_0$ with
          $(j',m',\ell') \in I_0^{(\alpha',\beta')} \cap J_i$; and

    \item $J_0 \cap I_0^{(\alpha',\beta')}$ contains at most $135N$ elements.
  \end{enumerate}
  Indeed, as $I^{(\alpha',\beta')} = \{0\} \cup I_0^{(\alpha',\beta')}$,
  the statements a),b) and c) together with Equation~\eqref{eq:RelativeModeratenessHelper}
  imply that
  \[
    |J_i| \leq \max \{ 1 + 135N \,,\, 1 + 9 \cdot 5 \cdot 125 N \}
    = 1 + 5625 N
    \quad \text{for all} \quad
    i \in I^{(\alpha,\beta)} \, .
  \]

  \emph{Proof of a)} We suppose that
  $Q_{j,m,\ell}^{(\alpha,\beta)} \cap Q_{j',m',\ell'}^{(\alpha',\beta')}
  \neq \emptyset$ and derive restrictions on the possible values of $m'$.
  To do so, we distinguish three possible cases:
  \smallskip{}

  \emph{Case 1:} $j=j'$. Let $m_{\min}'$ and $m_{\max}'$ be respectively the
  minimal and the maximal values of $m'$ such that
  $Q_{j,m,\ell}^{(\alpha,\beta)} \cap Q_{j',m',\ell'}^{(\alpha',\beta')}
  \neq \emptyset$.
  Therefore, there exist $\xi \in
  Q_{j,m,\ell}^{(\alpha,\beta)} \cap Q_{j',m_{\min}',\ell'}^{(\alpha',\beta')}$
  and $\eta \in
  Q_{j,m,\ell}^{(\alpha,\beta)} \cap Q_{j',m_{\max}',\ell'}^{(\alpha',\beta')}$.
  Since $j = j'$, Equation~\eqref{eq:AbsoluteValueEstimate} inplies that
  \[
    2^{j-1} + 2^{\alpha' j} (m_{\max}' - \eps)
    \leq |\eta|
    \leq 2^{j-1} + 2^{\alpha j} (m + 2 + 2 \eps)
  \]
  and
  \[
    2^{j-1} + 2^{\alpha j} (m - \eps)
    \leq |\xi|
    \leq 2^{j-1} + 2^{\alpha' j} (m_{\min}' + 2 + 2 \eps) \, .
  \]
  Combining these estimates results in
  \begin{align*}
    2^{\alpha' j} \cdot (m_{\max}' - m_{\min}')
    & = [2^{j-1} + 2^{\alpha' j} \cdot (m_{\max}' - \eps)]
        - [2^{j-1} + 2^{\alpha' j} \cdot (m_{\min}' + 2 + 2 \eps)]
        + (2 + 3 \eps) \cdot 2^{\alpha' j} \\
    & \leq [2^{j-1} + 2^{\alpha j} \cdot (m + 2 + 2\eps)]
           - [2^{j-1} + 2^{\alpha j} \cdot (m - \eps)]
           + (2 + 3 \eps) \cdot 2^{\alpha' j} \\
    & \leq (2 + 3 \eps) \cdot (2^{\alpha j} + 2^{\alpha' j}) \, .
  \end{align*}
  Since $\alpha \leq \alpha'$ and $\eps < \frac{1}{32}$, this finally implies that
  \(
    m_{\max}' - m_{\min}'
    \leq (2 + 3 \eps) \cdot (2^{(\alpha - \alpha') j} + 1)
    \leq 2 \cdot (2 + 3 \eps) < 5
  \).
  Therefore, $m'$ can take at most five different values if $j = j'$.
  \smallskip{}

  \emph{Case 2:} $j' < j$ and hence $j' \leq j - 1$.
  Since $Q_{j,m,\ell}^{(\alpha,\beta)} \cap Q_{j',m',\ell'}^{(\alpha',\beta')}
  \neq \emptyset$, there is some $\xi \in Q_{j,m,\ell}^{(\alpha,\beta)} \cap
  Q_{j',m',\ell'}^{(\alpha',\beta')}$.
  Therefore, from Equation~\eqref{eq:AbsoluteValueEstimate} we infer that
  \(
    2^{j-1} - \eps \cdot 2^{\alpha j}
   \leq |\xi|
   \leq 2^{j' - 1} + 2^{\alpha' j'} (m' + 2 + 2\eps)
  \)
  and thus
  \begin{align*}
    2^{j' - 1}
      \leq 2^{j - 1} - 2^{j' - 1}
    & \leq 2^{\alpha j} \eps + 2^{\alpha' j'} (m' + 2 + 2\eps) \\
    ({\scriptstyle{\text{since } \eps < 2^{-5} \text{ and } j \leq j' + 4
                   \,\, (\text{see } \eqref{eq:RelativeModeratenessHelper})
                   \text{ and } \alpha \leq \alpha'}})
    & \leq 2^{\alpha' j'} (m' + 3 + 2\eps)
      <    2^{\alpha ' j'} (m' + 4) \, ,
  \end{align*}
  since $2^{j'} \leq 2^{j-1}$.
  From this we infer that
  \(
    2^{(1-\alpha')j' - 1} - 4
    < m'
    \leq m_{j'}^{\max,\alpha'}
    \leq 2^{(1-\alpha') j'} + 1
  \).
  Thus, $m'$ can take at most five different values if $j' < j$.
  \smallskip{}

  \emph{Case 3:} $j' > j$ and thus $j \leq j' - 1$.
  Here there exists again
  $\xi\in Q_{j,m,\ell}^{(\alpha,\beta)}\cap Q_{j',m',\ell'}^{(\alpha',\beta')}$
  and from \eqref{eq:AbsoluteValueEstimate} we infer that
  \[
    2^{j' - 1} + 2^{\alpha' j'} (m' - \eps)
    \leq |\xi|
    \leq 2^j + 2^{\alpha j} (3 + 2 \eps)
    \leq 2^{j' - 1} + 2^{\alpha (j' - 1)} (3 + 2 \eps)
    \leq 2^{j' - 1} + 2^{\alpha j'} (3 + 2 \eps) \, ,
  \]
  and hence $0 \leq m' \leq \eps + 2^{(\alpha - \alpha') j'} (3 + 2 \eps) \leq 3 + 3 \eps < 4$,
  since $\alpha \leq \alpha'$.
  Thus, $m'$ can take at most four different values if $j' > j$.

  \smallskip{}

  Having considered all three possible cases, we conclude that $m'$
  can take at most five different values, as claimed in a).

  \medskip{}

  \emph{Proof of b)}
  Here again there exists
  $\xi \in Q_{j,m,\ell}^{(\alpha,\beta)} \cap Q_{j',\ell',m'}^{(\alpha',\beta')}$
  and from \eqref{eq:AngleControl} we infer that there are $\varphi, \varphi' \in \R$
  such that $|\xi| \cdot e^{i \varphi} = \xi = |\xi| \cdot e^{i \varphi'}$,
  $|\varphi - \Theta_{j,\ell}^{(\beta)}| \leq 4(1+\eps) \cdot 2^{(\beta - 1) j} \leq 4(1+\eps)$
  and furthermore
  $|\varphi' - \Theta_{j',\ell'}^{(\beta')}| \leq 4(1+\eps) \cdot 2^{(\beta' - 1) j'} \leq 4(1+\eps)$.
  Using essentially the same arguments as in the proof of
  Lemma~\ref{lem:Admissibility}, we conclude that there is some
  $k \in \{-2,\dots,2\}$ such that $\varphi - \varphi' = 2\pi k$.

  Finally, defining
  \[
    \lambda_{j,\ell,k,j'}
    := \frac{N}{2\pi} \cdot 2^{(1-\beta')j'}
                      \cdot \big( \Theta_{j,\ell}^{(\beta)} - 2\pi k \big)
    = \frac{1}{2} \cdot \left(\phi_{j'}^{(\beta')}\right)^{-1}
                  \cdot \big( \Theta_{j,\ell}^{(\beta)} - 2\pi k \big),
  \]
  we conclude that
  \begin{align*}
    |\ell' - \lambda_{j,\ell,k,j'}|
    & = \frac{1}{2} \cdot \left(\phi_{j'}^{(\beta')}\right)^{-1}
        \cdot \left|
               2 \ell' \phi_{j'}^{(\beta')}
               - (\Theta_{j,\ell}^{(\beta)} - 2\pi k)
              \right|
      = \frac{1}{2} \cdot \left(\phi_{j'}^{(\beta')}\right)^{-1}
        \left|
         \Theta_{j',\ell'}^{(\beta')}
         - \Theta_{j,\ell}^{(\beta)}
         + \varphi
         - \varphi'
        \right| \\
    & \leq \left(\phi_{j'}^{(\beta')}\right)^{-1} \cdot 2(1+\eps)
           \cdot (2^{(\beta-1)j} + 2^{(\beta' - 1) j'})
      =    \frac{N}{\pi} \cdot 2(1+\eps) \cdot (2^{(\beta-1)j - (\beta' - 1)j'} + 1) \\
    ({\scriptstyle{\text{since } \beta \leq \beta '}})
    & \leq \frac{N}{\pi} \cdot 2(1+\eps) \cdot (2^{(\beta'-1)(j-j')} + 1)
      \leq \frac{34 N}{\pi} \, (1+\eps) \leq 12 N \, ,
  \end{align*}
  since $|j-j'| \leq 5$, according to Equation~\eqref{eq:RelativeModeratenessHelper}.

  Because of $|\ell' - \lambda_{j,\ell,k,j'}| \leq 12 N$ and
  $k \in \{-2,\dots,2\}$, 
  the index $\ell'$ can take at most
  $5 \cdot 25 N = 125 N$ different values, for given $j,\ell$ and $j'$.

  \medskip{}

  \emph{Proof of c)}
  The proof of this part is identical to that of part e) of
  Lemma \ref{lem:Admissibility}, since the set $Q_0^{(\alpha,\beta)} = B_4 (0)$
  is independent of the choice of $\alpha$ and $\beta$.

  \medskip{}

  \textbf{Part 2:}
  In this part, we prove that
  $\CalQ^{(\alpha,\beta)}$ is \emph{not} almost subordinate to
  $\CalQ^{(\alpha',\beta')}$ if $\alpha > \alpha'$ or $\beta > \beta'$.
  To do so, it will be enough to prove the following two properties:
  \begin{enumerate}[label=\alph*)]
    \setcounter{enumi}{3}
    \item If $\alpha' < \alpha$, then
          \[
            \lim_{j\to\infty} |J_{(j,0,0)}|
            \geq \lim_{j \to \infty}
                   \left|
                    \big\{
                          m' \in \N_0
                          \,:\,
                          (j,m',0) \in I_0^{(\alpha',\beta')}
                          \text{ and }
                          Q_{j,m',0}^{(\alpha',\beta')}
                          \cap Q_{j,0,0}^{(\alpha,\beta)}
                          \neq \emptyset
                    \big\}
                   \right|
            = \infty .
          \]

    \item If $\alpha \leq \alpha '$ but $\beta' < \beta$, then
          \[
            \lim_{j \to \infty} |J_{(j,0,0)}|
            \geq \lim_{j \to \infty}
                   \left|
                    \{
                      \ell' \in \N_0
                      \,:\,
                      (j,0,\ell') \in I_0^{(\alpha',\beta')}
                      \text{ and }
                      Q_{j,0,\ell'}^{(\alpha',\beta')}
                      \cap
                      Q_{j,0,0}^{(\alpha,\beta)} \neq \emptyset
                    \}
                   \right|
            = \infty \, .
          \]
  \end{enumerate}
  Indeed, d) and e) show that $\CalQ^{(\alpha,\beta)}$ is not \emph{weakly}
  subordinate to $\CalQ^{(\alpha',\beta')}$.
  Thanks to Lemma~\ref{lem:WeakAndAlmostSubordinateness}, this implies that
  $\CalQ^{(\alpha,\beta)}$ is not \emph{almost} subordinate to $\CalQ^{(\alpha',\beta')}$.

  \medskip{}
  \emph{Proof of d)}
  From the definition of $Q_{j,m,\ell}^{(\alpha,\beta)}$, we infer that
  \[
    Q_{j,0,0}^{(\alpha,\beta)}
    \supset [2^{j-1} \,,\, 2^{j-1} + 2^{\alpha j}] \times \{0\}
    \quad \text{and} \quad
    Q_{j,m',0}^{(\alpha',\beta')}
    \supset [2^{j-1} + m' \cdot 2^{\alpha' j} \,,\,
             2^{j-1} + (m' + 1) \cdot 2^{\alpha' j}]
            \times \{0\} \, .
  \]
  The latter implies that
  $\xi_{j,m'} := (2^{j-1} + m' \cdot 2^{\alpha' j}, \, 0)^t \in Q_{j,m',0}^{(\alpha',\beta')}$.

  Let us now choose $m' \in \N_0$ with $m' \leq 2^{j(\alpha-\alpha') - 1}$.
  Then, on the one hand, $(j,m',0) \in I_0^{(\alpha',\beta')}$ since
  $m' \leq 2^{j(1-\alpha') - 1} \leq m_j^{\max,\alpha'}$.
  On the other hand, $\xi_{j,m'} \in Q_{j,0,0}^{(\alpha,\beta)}$
  since $m' \cdot 2^{\alpha' j} \leq 2^{\alpha j - 1} \leq 2^{\alpha j}$.

  Put together, this implies, as $\alpha > \alpha'$, that
  \begin{align*}
    |J_{(j,0,0)}|
    & \geq |
            \{
              m' \in \N_0
              \,:\,
              (j,m',0) \in I_0^{(\alpha',\beta')}
              \text{ and }
              Q_{j,m',0}^{(\alpha',\beta')} \cap Q_{j,0,0}^{(\alpha,\beta)}
              \neq \emptyset
            \}
           | \\
     & \geq |\{ m' \in \N_0 \,:\, m' \leq 2^{j (\alpha - \alpha') - 1}\}|
       = 1 + \lfloor 2^{j(\alpha - \alpha') - 1} \rfloor
       \xrightarrow[j\to\infty]{} \infty \, .
  \end{align*}

  \medskip{}

  \emph{Proof of e)}
  Here we shall write
  $R_{j,\ell}^{(\beta)}$ instead of $R_{j,\ell}$ to clearly indicate the value of $\beta$
  that determines this matrix.

  From the definition of $Q_{j,m,\ell}^{(\alpha,\beta)}$ we infer that
  \begin{equation}
    Q_{j,0,0}^{(\alpha',\beta')}
    \supset [2^{j-1}, 2^{j-1} + 2^{\alpha' j}]
            \times \{0\}
    \quad \text{and} \quad
    Q_{j,0,0}^{(\alpha,\beta)}
    \supset [2^{j-1}, 2^{j-1} + 2^{\alpha j}]
            \times [-2^{\beta j}, 2^{\beta j}] \, .
    \label{eq:SubordinatenessCounterexampleBasicSet}
  \end{equation}
  For $j \in \N$ define
  $\theta_j := \min \big\{ \tfrac{1}{2} \, 2^{(\beta - 1) j} , 2^{\frac{\alpha' - 1}{2} j} \big\}$.
  Below, we shall prove the following technical auxiliary claim:
  \begin{equation}
    \forall \, j \in \N \,\,
      \forall \, \theta \in [-\theta_j, \theta_j] \,\,
        \exists \, z_{j,\theta} \in [2^{j-1}, 2^{j-1} + 2^{\alpha ' j}] \,\,\, : \,\,\,
          \left(
            \begin{smallmatrix}
              z_{j,\theta} \, \cos \theta \\
              z_{j,\theta} \, \sin \theta
            \end{smallmatrix}
          \right)
          \in [2^{j-1}, 2^{j-1} + 2^{\alpha j}] \times [-2^{\beta j}, 2^{\beta j}] .
    \label{eq:NonSubordinatenessDifficultCaseSpecialStatement}
  \end{equation}
  Accepting this for the moment, we can combine
  Equations~\eqref{eq:NonSubordinatenessDifficultCaseSpecialStatement}
  and \eqref{eq:SubordinatenessCounterexampleBasicSet} to conclude
  that if $j \in \N$ and $\ell ' \in \N_0$ with $\ell' \leq \ell_j^{\max, \beta'}$
  are such that $\theta(j,\ell') := \Theta_{j,\ell'}^{(\beta')}$
  satisfies $|\theta(j,\ell')| \leq \theta_j$, then
  \[
    R_{j,\ell'}^{(\beta')}
    \left(
      \begin{smallmatrix}
        z_{j, \theta(j,\ell')} \\
        0
      \end{smallmatrix}
    \right)
    = \left(
        \begin{smallmatrix}
          z_{j, \theta(j,\ell')} \cdot \cos ( \theta(j,\ell') ) \\
          z_{j, \theta(j,\ell')} \cdot \sin ( \theta(j,\ell') )
        \end{smallmatrix}
      \right)
    \in Q_{j,0,0}^{(\alpha,\beta)} \cap R_{j,\ell'}^{(\beta')} \, Q_{j,0,0}^{(\alpha',\beta')}
    =   Q_{j,0,0}^{(\alpha,\beta)} \cap Q_{j,0,\ell'}^{(\alpha',\beta')}
  \]
  and hence
  \begin{align*}
    |J_{(j,0,0)}|
    & \geq \Big|
             \big\{
               \ell ' \in \N_0
               \,\, : \,\,
               \ell ' \leq \ell_j^{\max, \beta'}
               \text{ and }
               \big| \Theta_{j,\ell'}^{(\beta')} \big| \leq \theta_j
             \big\}
           \Big| \\
    ({\scriptstyle{\text{Def.~of } \Theta_{j,\ell'}^{(\beta')},
                                   \ell_j^{\max,\beta'},
                                   \text{ and } \theta_j}})
    & \geq \Big|
             \Big\{
               \ell ' \in \N_0
               \,\, : \,\,
               \ell ' \leq N \cdot 2^{(1 - \beta') j}
               \text{ and }
               \ell' \leq \tfrac{N}{2 \pi}
                          \cdot \min
                                \big\{
                                  \tfrac{1}{2} 2^{(\beta - \beta') j},
                                  2^{\frac{1 - \beta' + \alpha ' - \beta'}{2} \cdot j}
                                \big\}
             \Big\}
           \Big| \\
    & \xrightarrow[j\to\infty]{} \infty .
  \end{align*}
  Here we noted in the very last step that $\beta' < \beta \leq 1$ and that $\beta' \leq \alpha'$,
  so that $2^{(1-\beta') j}$, $2^{(\beta - \beta') j}$
  and $2^{\frac{1 - \beta' + \alpha ' - \beta'}{2} \cdot j}$ all tend to $\infty$ as $j \to \infty$.
  Thus, we shall prove Claim e), if we prove~\eqref{eq:NonSubordinatenessDifficultCaseSpecialStatement}.

  \medskip{}

  To prove that \eqref{eq:NonSubordinatenessDifficultCaseSpecialStatement} is indeed satisfied,
  let $j \in \N$ and $\theta \in [-\theta_j, \theta_j]$.
  We first show that we can choose $z = z_{j,\theta} \in [2^{j-1}, 2^{j-1} + 2^{\alpha ' j}]$
  such that $z \cdot \cos \theta \in [2^{j-1}, 2^{j-1} + 2^{\alpha j}]$.
  Note that $|\theta| \leq \theta_j \leq 1 < \tfrac{\pi}{2}$ and hence $\cos \theta > 0$.
  Thus, our goal is to show that we can choose
  \begin{equation}
    z_{j,\theta} \in [2^{j-1} , 2^{j-1} + 2^{\alpha' j}]
                     \cap \Big[
                            \frac{2^{j-1}}{\cos \theta},
                            \frac{2^{j-1} + 2^{\alpha j}}{\cos \theta}
                          \Big] .
    \label{eq:NonSubordinatenessZChoice}
  \end{equation}
  This is possible if and only if the first condition
  in the following chain of equivalences is satisfied:
  \begin{align*}
    [2^{j-1} , 2^{j-1} + 2^{\alpha' j}]
    \cap \Big[
           \frac{2^{j-1}}{\cos \theta},
           \frac{2^{j-1} + 2^{\alpha j}}{\cos \theta}
         \Big]
    \neq \emptyset
    & \quad \Longleftrightarrow \quad
    2^{j-1} \leq \frac{2^{j-1} + 2^{\alpha j}}{\cos \theta}
    \text{ and }
    \frac{2^{j-1}}{\cos \theta} \leq 2^{j-1} + 2^{\alpha' j} \\
    & \quad \Longleftrightarrow \quad
    \frac{2^{j-1}}{2^{j-1} + 2^{\alpha' j}}
    \leq \cos \theta
    \leq \frac{2^{j-1} + 2^{\alpha j}}{2^{j-1}} \\
    ({\scriptstyle{\text{since } \cos \theta \leq 1 \leq \frac{2^{j-1} + 2^{\alpha j}}{2^{j-1}}}})
    & \quad \Longleftrightarrow
    \cos \theta \geq 1 - \frac{2^{\alpha' j}}{2^{j-1} + 2^{\alpha' j}} .
  \end{align*}
  To prove that the latter condition is satisfied, we recall
  from Equation~\eqref{eq:CosineQuadraticLowerBound} that $\cos \theta \geq 1 - \tfrac{\theta^2}{2}$
  for all $\theta \in \R$, and hence
  \[
    \cos \theta
    \geq 1 - \frac{\theta^2}{2}
    \geq 1 - \frac{\theta_j^2}{2}
    \geq 1 - \frac{2^{(\alpha' - 1) j}}{2}
    =    1 - \frac{2^{\alpha ' j}}{2 \cdot 2^j}
    \geq 1 - \frac{2^{\alpha ' j}}{2^{j-1} + 2^{\alpha ' j}},
  \]
  as desired.
  Here we noted in the last step that $2^{j-1} + 2^{\alpha' j} \leq 2^j + 2^j$ since $\alpha' \leq 1$.
  Overall, we have shown that one can indeed choose $z_{j,\theta}$
  as in Equation~\eqref{eq:NonSubordinatenessZChoice}.

  Thus, to prove Equation~\eqref{eq:NonSubordinatenessDifficultCaseSpecialStatement},
  it suffices to verify that $|z_{j,\theta} \cdot \sin \theta| \leq 2^{\beta j}$.
  But this is a consequence of the estimate $|\sin \phi| \leq |\phi|$
  combined with $0 \leq z_{j,\theta} \leq 2^{j-1} + 2^{\alpha ' j} \leq 2 \cdot 2^j$
  and ${|\theta| \leq \theta_j \leq \tfrac{1}{2} 2^{(\beta - 1) j}}$;
  indeed, these estimates imply that
  \(
    |z_{j,\theta} \cdot \sin \theta|
    \leq 2 \cdot 2^j \cdot \frac{1}{2} 2^{(\beta - 1) j}
    =    2^{\beta j}
  \).
\end{proof}

In the next corollary, we verify the conditions
concerning relative moderation of coverings and weights that we shall need to apply
Theorems~\ref{thm:EmbeddingPFinerThanQ} and \ref{thm:EmbeddingQFinerThanP}.

\begin{cor}\label{cor:RelativeModerateness}
  Let $0 \leq \beta \leq \alpha \leq 1$ and $0 \leq \beta' \leq \alpha' \leq 1$.
  Then, for any fixed $s \in \R$, the weight $w^s$ --- considered as a weight
  for $\CalQ^{(\alpha,\beta)}$ --- is relatively
  $\CalQ^{(\alpha',\beta')}$-moderate; more specifically,
  \[
    w_i^s \asymp w_{i'}^s
    \qquad
    \text{if }
    Q_i^{(\alpha,\beta)} \cap Q_{i'}^{(\alpha', \beta')} \neq \emptyset \, .
  \]
  Furthermore, the covering $\CalQ^{(\alpha,\beta)}$ is relatively
  $\CalQ^{(\alpha',\beta')}$-moderate, and
  \[
    |\det T_i^{(\alpha,\beta)}| \asymp w_{i'}^{\alpha + \beta}
    \qquad \text{if }
    Q_i^{(\alpha,\beta)} \cap Q_{i'}^{(\alpha', \beta')} \neq \emptyset \, .
  \]
\end{cor}

\begin{proof}
  If $i = (j,m,\ell) \in I_0^{(\alpha,\beta)}$ and
  $i' = (j',m',\ell') \in I_0^{(\alpha',\beta')}$
  satisfy $Q_i^{(\alpha,\beta)} \cap Q_{i'}^{(\alpha',\beta')} \neq \emptyset$,
  then \eqref{eq:RelativeModeratenessHelper} implies that $|j - j'| \leq 4$.
  Therefore,
  \[
    \frac{w_i^s}{w_{i'}^s}
    = 2^{(j - j') \cdot s}
    \leq 2^{|j-j'| \cdot |s|}
    \leq 2^{4 |s|}
    \qquad \text{and} \qquad
    \frac{w_i^s}{w_{i'}^s}
    = 2^{(j - j') \cdot s}
    \geq 2^{- |j-j'| \cdot |s|}
    \geq 2^{-4 |s|} \, .
  \]
  Moreover, if $\emptyset \neq Q_0^{(\alpha,\beta)} \cap Q_{i'}^{(\alpha',\beta')} \ni \xi$
  for $i' = (j',m',\ell') \in I_0^{(\alpha',\beta')}$, then
  \eqref{eq:AbsoluteValueEstimate} implies
  ${2^2 > |\xi| \geq 2^{j' - 2}}$, and hence $j' \leq 4$.
  Therefore,
  $w_0^s / w_{i'}^s = 2^{-s \cdot j'} \leq 2^{|s| \cdot j'} \leq 2^{4 |s|}$
  and
  $w_0^s / w_{i'}^s = 2^{-s \cdot j'} \geq 2^{-|s| \cdot j'} \geq 2^{-4 |s|}$.

  Similarly, if
  $Q_i^{(\alpha,\beta)} \cap Q_{0}^{(\alpha',\beta')} \neq \emptyset$ for
  $i = (j,m,\ell) \in I_0^{(\alpha,\beta)}$, we see precisely as in the
  preceding paragraph that $j \leq 4$ and hence
  $2^{-4 |s|} \leq w_i^s / w_0^s \leq 2^{4 |s|}$.
  Finally, $2^{-4 |s|} \leq 1 = w_0^s / w_0^s = 1 \leq 2^{4 |s|}$.

  These estimates show that $\vphantom{\sum_j}w_i^s \asymp w_{i'}^s$ if
  $Q_i^{(\alpha,\beta)} \cap Q_{i'}^{(\alpha',\beta')} \neq \emptyset$,
  proving that $w^s$ --- considered as a weight for
  $\CalQ^{(\alpha,\beta)}$ --- is relatively $\CalQ^{(\alpha',\beta')}$-moderate.

  \medskip{}

  To prove that $\CalQ^{(\alpha,\beta)}$ is relatively
  $\CalQ^{(\alpha',\beta')}$-moderate, we note that
  \[
    \det T_i
    = \det (R_{j,\ell} \, A_{j})
    = \det A_{j}
    = 2^{\alpha j} \cdot 2^{\beta j}
    = w_i^{\alpha + \beta}
    \qquad \text{for} \quad i = (j,m,\ell) \in I_0^{(\alpha,\beta)} \, .
  \]
  Similarly, $\det T_i = 1 = w_i^{\alpha+\beta}$ for $i = 0$.
  Thus, we conclude that
  \(
    |\det T_i^{(\alpha,\beta)}|
    = w_i^{\alpha+\beta}
    \asymp w_{i'}^{\alpha + \beta}
  \)
  if $Q_i^{(\alpha,\beta)} \cap Q_{i'}^{(\alpha', \beta')} \neq \emptyset$.
\end{proof}

We can now state and prove the main theorem of this subsection.

\begin{thm}\label{thm:WavePacketSpacesEmbeddings}
  Let $0 \leq \beta \leq \alpha \leq 1$ and $0 \leq \beta' \leq \alpha' \leq 1$
  be such that $\alpha \leq \alpha '$ and $\beta \leq \beta '$.
  Let $p_1, p_2, q_1, q_2 \in (0,\infty]$ and $s_1, s_2 \in \R$.

  Then
  \[
    \PacketSpace_{s_1}^{p_1, q_1} (\alpha, \beta)
    \hookrightarrow \PacketSpace_{s_2}^{p_2, q_2} (\alpha',\beta')
  \]
  if and only if $p_1 \leq p_2$ and
  \[
    \begin{cases}
      s_1 > s_2
            + (p_1^{-1} - p_2^{-1})(\alpha + \beta)
            + \mu (\alpha' - \alpha + \beta' - \beta)
            + (2 - \alpha' - \beta') (q_2^{-1} - q_1^{-1}) \, ,
      & \text{if } q_1 > q_2 \, , \\
      s_1 \geq s_2
               + (p_1^{-1} - p_2^{-1})(\alpha + \beta)
               + \mu (\alpha' - \alpha + \beta' - \beta) \, ,
      & \text{if } q_1 \leq q_2 ,
    \end{cases}
  \]
  where $\mu = (p_2^{\ast \ast} - q_1^{-1})_+$ and $p_2^{\ast \ast} = (\min \{ p_2, p_2' \})^{-1}$
  and where the conjugate exponent $p_2' \in [1,\infty]$ of $p_2 \in (0,\infty]$ is defined as in
  Appendix~\ref{sec:Notation}.

  \medskip{}

  Conversely,
  \[
    \PacketSpace_{s_1}^{p_1, q_1} (\alpha', \beta')
    \hookrightarrow \PacketSpace_{s_2}^{p_2, q_2} (\alpha,\beta)
  \]
  if and only if $p_1 \leq p_2$ and
  \[
    \begin{cases}
      s_1 > s_2
            + (p_1^{-1} - p_2^{-1})(\alpha + \beta)
            + \nu (\alpha' - \alpha + \beta' - \beta)
            + (2 - \alpha' - \beta') (q_2^{-1} - q_1^{-1}) \, ,
      & \text{if } q_1 > q_2 \, , \\
      s_1 \geq s_2
               + (p_1^{-1} - p_2^{-1})(\alpha + \beta)
               + \nu (\alpha' - \alpha + \beta' - \beta) \, ,
      & \text{if } q_1 \leq q_2 ,
    \end{cases}
  \]
  where $\nu = (q_2^{-1} - p_1^\ast)_+$ and $p_1^\ast = \min \{ p_1^{-1}, 1 - p_1^{-1} \}$.
\end{thm}

\begin{rem*}
  1) Note that this theorem cannot be applied if one of the conditions
     $\alpha \leq \alpha'$ or $\beta \leq \beta'$ does not hold.
     Nevertheless, \emph{sufficient} conditions for embeddings can still be derived,
     for instance by considering the chain of embeddings
     %
     \[
       \PacketSpace_{s_1}^{p_1, q_1} (\alpha,\beta)
       \hookrightarrow \PacketSpace_{s}^{p,q} (\max\{\alpha,\alpha'\}, \max\{\beta,\beta'\})
       \hookrightarrow \PacketSpace_{s_2}^{p_2, q_2} (\alpha', \beta') \, ,
     \]
     for suitable parameters $p,q,s$ under certain conditions
     on $p_1,p_2, q_1,q_2$ and $s_1,s_2$.
     Alternatively, one can use embedding criteria provided in
     \cite{DecompositionEmbeddings} which are applicable to coverings
     that are not almost subordinate to each other.
     This, however, is outside the scope of the present paper.

  \medskip{}

  2) In Subsection~\ref{sub:WavePacketAlphaModulation}, we shall see that
  the wave packet smoothness spaces $\PacketSpace_s^{p,q}(\alpha,\alpha)$
  are identical to the $\alpha$\textbf{-modulation spaces}
  $M_{p,q}^{s,\alpha}(\R^2)$ introduced in Gröbner's PhD
  thesis \cite{GroebnerAlphaModulationSpaces} and studied further in
  \cite{BorupNielsenAlphaModulationSpaces,
  FornasierFramesForAlphaModulation,SpeckbacherAlphaModulationTransform,
  HanWangAlphaModulationEmbeddings,KatoAlphaModulationSobolev,
  GuoAlphaModulationEmbeddingCharacterization,DecompositionEmbeddings}.
  Therefore, Theorem~\ref{thm:WavePacketSpacesEmbeddings} can be seen as a
  generalisation of the characterisation of the embeddings between
  $\alpha$-modulation spaces, which were first studied in
  \cite{GroebnerAlphaModulationSpaces,HanWangAlphaModulationEmbeddings}
  and fully understood in
  \cite{VoigtlaenderPhDThesis,
        GuoAlphaModulationEmbeddingCharacterization,
        DecompositionEmbeddings}.
\end{rem*}

\begin{proof}
  To characterise the embedding
  \(
    \PacketSpace_{s_1}^{p_1, q_1} (\alpha, \beta)
    \hookrightarrow \PacketSpace_{s_2}^{p_2, q_2} (\alpha',\beta')
  \),
  we shall apply Theorem~\ref{thm:EmbeddingQFinerThanP} to the coverings
  $\CalQ = \CalQ^{(\alpha,\beta)}$ and $\CalP = \CalQ^{(\alpha',\beta')}$
  and the respective weights $w = w^{s_1}$ and $v = v^{s_2}$.
  All assumptions of that theorem are indeed satisfied, as can be seen from
  Lemmas~\ref{lem:CoveringAlmostStructured} and \ref{lem:WavePacketWeight},
  Proposition~\ref{prop:WavePacketCoveringSubordinateness} and
  Corollary~\ref{cor:RelativeModerateness}.
  Furthermore, note that the constant $\mu$ defined in the present theorem
  is identical to the one introduced in Theorem~\ref{thm:EmbeddingQFinerThanP}.
  Finally, let us select, for each $i' \in I^{(\alpha',\beta')}$, such an index
  $i_{i'} \in I^{(\alpha,\beta)}$ that
  $Q_{i_{i'}}^{(\alpha,\beta)} \cap Q_i^{(\alpha,\beta)} \neq \emptyset$.
  Then, Theorem~\ref{thm:EmbeddingQFinerThanP} implies that the embedding
  $\PacketSpace_{s_1}^{p_1, q_1} (\alpha, \beta)
   \hookrightarrow \PacketSpace_{s_2}^{p_2, q_2} (\alpha',\beta')$
  holds if and only if $p_1 \leq p_2$ and
  \begin{equation}
    \begin{split}
      \infty
      & > \left\|
            \left(
              \frac{w_{i'}^{s_2}}{w_{i_{i'}}^{s_1}}
              \cdot \big| \det T_{i_{i'}}^{(\alpha,\beta)} \big|^{p_1^{-1} - p_2^{-1} - \mu}
              \cdot \big| \det T_{i'}^{(\alpha',\beta')} \big|^\mu
            \right)_{i' \in  I^{(\alpha',\beta')}}
          \right\|_{\ell^{q_2 \cdot (q_1 / q_2)'}} \\
      ({\scriptstyle{\text{Corollary }\ref{cor:RelativeModerateness}}})
      & \asymp
        \left\|
          \left(
            w_{i'}^{s_2 - s_1
                    + (\alpha + \beta) (p_1^{-1} - p_2^{-1} - \mu)
                    + \mu (\alpha' + \beta')}
          \right)_{i' \in  I^{(\alpha',\beta')}}
        \right\|_{\ell^{q_2 \cdot (q_1 / q_2)'}} \, .
    \end{split}
    \label{eq:WavePacketEmbeddingFineIntoCoarseCondition}
  \end{equation}
  First, we note that the single term with index $0 \in I^{(\alpha',\beta')}$ alone
  has no influence on whether the norm in
  \eqref{eq:WavePacketEmbeddingFineIntoCoarseCondition} is finite or not.
  Therefore, it is enough to consider only the terms
  $i' \in I_0^{(\alpha',\beta')}$.

  Next, since the set
  \[
    \Omega_{j'} :=
    \{
      (m',\ell') \in \N_0 \times \N_0
      \,:\,
      (j', m', \ell') \in I_0^{(\alpha', \beta')}
    \}
  \]
  satisfies $|\Omega_{j'}| \asymp 2^{(1 - \alpha' + 1 - \beta') j'}$
  and since the weight $w_{i'}^\gamma = 2^{\gamma \, j'}$ is independent of
  $m',\ell'$ for ${i' = (j',m',\ell')}$, we conclude that
  \begin{equation}
    \left\|
      \left(
        w_{i'}^\gamma
      \right)_{i' \in I_0^{(\alpha',\beta')}}
    \right\|_{\ell^q}
    \asymp \Big\|
             \big(
               2^{j' \cdot (\gamma + (2 - \alpha' - \beta')/q)}
             \big)_{j' \in \N}
           \Big\|_{\ell^q}
    \qquad \forall \, \gamma \in \R \text{ and } q \in (0,\infty] \, .
    \label{eq:MainSequenceSpaceNormCharacterization1}
  \end{equation}
  The right-hand side of \eqref{eq:MainSequenceSpaceNormCharacterization1}
  is finite if and only if
  \begin{equation}
    \begin{cases}
      \gamma + (2 - \alpha' - \beta') / q < 0,    & \text{if } q < \infty \, , \\
      \gamma + (2 - \alpha' - \beta') / q \leq 0, & \text{if } q = \infty \, .
    \end{cases}
    \quad \Longleftrightarrow \quad
    \begin{cases}
      \gamma + (2 - \alpha' - \beta') / q < 0,    & \text{if } q < \infty \, , \\
      \gamma                              \leq 0, & \text{if } q = \infty \, .
    \end{cases}
    \label{eq:MainSequenceSpaceNormCharacterization2}
  \end{equation}
  Therefore, by recalling the identity \eqref{eq:EmbeddingSequenceSpaceExponent},
  we infer that \eqref{eq:WavePacketEmbeddingFineIntoCoarseCondition}
  is satisfied if and only if
  \[
    \begin{cases}
      s_2 - s_1
      + (\alpha + \beta)(p_1^{-1} - p_2^{-1} - \mu)
      + \mu (\alpha' + \beta')
      + (2 - \alpha' - \beta') \cdot (q_2^{-1} - q_1^{-1})
      < 0 \, ,
      & \text{if } q_1 >    q_2 \, , \\
      s_2 - s_1
      + (\alpha + \beta)(p_1^{-1} - p_2^{-1} - \mu)
      + \mu (\alpha' + \beta') \leq 0 \, ,
      & \text{if } q_1 \leq q_2 \, ,
    \end{cases}
  \]
  which is equivalent to the conditions stated in the theorem.

  \medskip{}

  To characterise the converse embedding
  $\PacketSpace_{s_1}^{p_1, q_1} (\alpha', \beta')
   \hookrightarrow \PacketSpace_{s_2}^{p_2, q_2} (\alpha,\beta)$,
  we apply Theorem~\ref{thm:EmbeddingPFinerThanQ} to the coverings
  $\CalQ = \CalQ^{(\alpha',\beta')}$ and $\CalP = \CalQ^{(\alpha,\beta)}$
  and the respective weights $w = w^{s_1}$ and $v = v^{s_2}$.
  As before, we see that all assumptions of that theorem are indeed satisfied.
  Furthermore, we note that the constant $\nu$ defined in the present is identical
  to the one introduced in Theorem~\ref{thm:EmbeddingPFinerThanQ}.
  Therefore, we see as above that the desired embedding holds if and only if $p_1 \leq p_2$ and
  \[
    \begin{split}
      \infty
      & > \left\|
            \left(
              \frac{w_{i_{i'}}^{s_2}}{w_{i'}^{s_1}}
              \cdot \big| \det T_{i'}^{(\alpha',\beta')} \big|^{\nu}
              \cdot \big| \det T_{i_{i'}}^{(\alpha,\beta)} \big|^{p_1^{-1} - p_2^{-1} - \nu}
            \right)_{i' \in  I^{(\alpha',\beta')}}
          \right\|_{\ell^{q_2 \cdot (q_1 / q_2)'}} \\
      ({\scriptstyle{\text{Corollary } \ref{cor:RelativeModerateness}}})
      & \asymp \left\|
                 \left(
                   w_{i'}^{s_2 - s_1
                           + (\alpha + \beta) (p_1^{-1} - p_2^{-1} - \nu)
                           + \nu (\alpha' + \beta')}
                 \right)_{i' \in  I^{(\alpha',\beta')}}
               \right\|_{\ell^{q_2 \cdot (q_1 / q_2)'}} \, .
    \end{split}
  \]
  Precisely as before, we thus see that the embedding holds
  if and only if the conditions stated in the theorem are satisfied.
\end{proof}

\subsection{Characterising the coincidence of two wave packet smoothness spaces}
\label{sub:WavePacketCoincidence}

\noindent In this short subsection, we show that two wave packet spaces
$\PacketSpace_{s_1}^{p_1, q_1}(\alpha,\beta)$
and $\PacketSpace_{s_2}^{p_2, q_2}(\alpha',\beta')$ can coincide only
if all their parameters are identical.
This is almost true as stated; a small exception occurs for the case $p_1 = q_1 = p_2 = q_2 = 2$
in which the wave packet smoothness spaces are simply $L^2$-Sobolev spaces,
independently of the parameters $\alpha,\beta$.

\begin{thm}\label{thm:WavePacketCoincidence}
  Let $0 \leq \beta \leq \alpha \leq 1$, $0 \leq \beta' \leq \alpha' \leq 1$, $s_1, s_2 \in \R$ and $p_1, p_2, q_1, q_2 \in (0,\infty]$.
  If $\PacketSpace_{s_1}^{p_1, q_1} (\alpha,\beta) = \PacketSpace_{s_2}^{p_2, q_2}(\alpha', \beta')$,
  then $(p_1, q_1, s_1) = (p_2, q_2, s_2)$.
  If furthermore $(p_1, q_1) \neq (2,2)$, then $(\alpha, \beta) = (\alpha', \beta')$.

  Finally, for arbitrary $s \in \R$,
  \(
    \PacketSpace_{s}^{2,2}(\alpha,\beta) = H^s (\R^2)
  \)
  with equivalent norms, where the $L^2$-Sobolev space $H^s(\R^2)$ is given by
  \({
    H^s (\R^2)
    = \big\{
        f \in \Schwartz'(\R^2)
        \colon
        (1 + |\xi|^2)^{s/2} \cdot \widehat{f} \in L^2 (\R^2)
      \big\}
  }\)
  (see for instance Section~9.3 in \cite{FollandRA}).
\end{thm}

\begin{proof}
  Let us assume that
  $\PacketSpace_{s_1}^{p_1, q_1} (\alpha,\beta) = \PacketSpace_{s_2}^{p_2, q_2}(\alpha', \beta')$.
  Since $\PacketSpace_s^{p,q}(\alpha,\beta) = \DecompSp(\CalQ^{(\alpha,\beta)}, L^p, \ell_{w^s}^q)$,
  Theorem~\ref{thm:RigidityTheorem} implies that $(p_1, q_1) = (p_2, q_2)$ and that there is
  $C > 0$ such that $C^{-1} \cdot w_i^{s_1} \leq w_{i'}^{s_2} \leq C \cdot w_i^{s_1}$ for all
  $i \in I^{(\alpha,\beta)}$ and $i' \in I^{(\alpha',\beta')}$ for which
  $Q_i^{(\alpha,\beta)} \cap Q_{i'}^{(\alpha',\beta')} \neq \emptyset$.
  Because of $(2^{j-1}, 0)^t \in Q_{j,0,0}^{(\alpha,\beta)} \cap Q_{j,0,0}^{(\alpha',\beta')}$
  for arbitrary $j \in \N$, this implies $C^{-1} \cdot 2^{s_1 j} \leq 2^{s_2 j} \leq C \cdot 2^{s_1 j}$
  for all $j \in \N$, which implies that $s_1 = s_2$.

  Furthermore, in case of $(p_1, q_1) \neq (2,2)$, Theorem~\ref{thm:RigidityTheorem} shows that
  $\CalQ^{(\alpha,\beta)}$ and $\CalQ^{(\alpha',\beta')}$ are weakly equivalent.
  Since the coverings $\CalQ^{(\alpha,\beta)}$ and $\CalQ^{(\alpha',\beta')}$ consist of
  open, path-connected sets, Lemma~\ref{lem:WeakAndAlmostSubordinateness} shows that
  $\CalQ^{(\alpha,\beta)}$ and $\CalQ^{(\alpha',\beta')}$ are in fact equivalent coverings.
  Therefore, Proposition~\ref{prop:WavePacketCoveringSubordinateness} shows that
  $(\alpha,\beta) = (\alpha',\beta')$.

  \medskip{}

  Finally, since $w_i^s \asymp (1 + |\xi|)^s \asymp (1 + |\xi|^2)^{s/2}$ for all
  $\xi \in Q_i^{(\alpha,\beta)}$ and $i \in I^{(\alpha,\beta)}$
  (see Equation~\eqref{eq:WeightConsistencyCondition}),
  Lemma~6.10 in \cite{DecompositionEmbeddings} implies that
  \[
    \PacketSpace_{s}^{2,2}(\alpha,\beta)
    = \big\{ f \in Z' \colon (1 + |\xi|^2)^{s/2} \cdot \widehat{f} \in L^2(\R^2) \big\}
    = \big\{ f \in \Schwartz'(\R^2) \colon (1 + |\xi|^2)^{s/2} \cdot \widehat{f} \in L^2(\R^2) \big\}
    = H^s (\R^2) ,
  \]
  where the penultimate equality is justified by the smoothness and growth properties of the weight
  ${\xi \mapsto (1 + |\xi|^2)^{s/2}}$, which imply that if $g = \widehat{f} \in \CalD'(\R^2)$
  satisfies $(1 + |\xi|^2)^{s/2} \cdot g \in L^2(\R^2) \subset \Schwartz'(\R^2)$,
  then $g \in \Schwartz'(\R^2)$ and hence $f \in \Schwartz'(\R^2)$.
\end{proof}

\subsection{Establishing embeddings of wave packet smoothness spaces in classical spaces}
\label{sub:WavePacketClassicalEmbeddings}

\noindent In this subsection, we study the conditions on the parameters
$\alpha,\beta$ and $p,q,s$ under which the wave packet smoothness space
$\PacketSpace_{s}^{p,q}(\alpha,\beta)$ embeds in the Sobolev space
$W^{k,r}(\R^2)$ or the inhomogeneous Besov space $B^s_{p,q}(\R^2)$.
For the Besov spaces, we also study the converse question, that is, whether
the Besov spaces embed in the wave packet smoothness spaces.
As an application, we show that the Besov spaces arise as special cases of the
wave packet smoothness spaces for the case $\alpha = \beta = 1$.

We start by analysing the existence of embeddings between
wave packet smoothness and Besov spaces.

\begin{thm}\label{thm:WavePacketBesov}
  Let $0 \leq \beta \leq \alpha \leq 1$,
  $p_1, p_2, q_1, q_2 \in (0,\infty]$ and $s_1, s_2 \in \R$.
  Let $B_{p,q}^s (\R^2)$ be the inhomogeneous Besov spaces as
  introduced for instance in Definition~2.2.1 in \cite{GrafakosModern} or in
  Definition~2 of Section 2.3.1 in \cite{TriebelTheoryOfFunctionSpaces}.
  Let us define $p_1^\ast := \min \{ p_1^{-1}, 1 - p_1^{-1} \}$ and
  $p_2^{\ast \ast} := ( \min \{ p_2, p_2 ' \})^{-1}$.

  Then,
  \[
    \PacketSpace_{s_1}^{p_1, q_1} (\alpha,\beta)
    \hookrightarrow B_{p_2,q_2}^{s_2} (\R^2)
  \]
  if and only if $p_1 \leq p_2$ and
  \begin{equation}
    \begin{cases}
      s_1 \geq s_2 + (\alpha+\beta)(p_1^{-1} - p_2^{-1} - \mu) + 2 \mu
      \, ,
      & \text{if } q_1 \leq q_2 \, , \\
      s_1 > s_2 + (\alpha+\beta)(p_1^{-1} - p_2^{-1} - \mu) + 2 \mu \, ,
      & \text{if } q_1 >    q_2
    \end{cases}
    \qquad \text{where} \qquad
    \mu := \left( p_2^{\ast \ast} - q_1^{-1} \right)_+
    \,\,\, .
    \label{eq:WavePacketIntoBesovCharacterization}
  \end{equation}

  Conversely,
  \[
    B_{p_1,q_1}^{s_1} (\R^2)
    \hookrightarrow \PacketSpace_{s_2}^{p_2, q_2} (\alpha,\beta)
  \]
  if and only if $p_1 \leq p_2$ and
  \begin{equation}
    \begin{cases}
      s_1 \geq s_2 + (\alpha+\beta)(p_1^{-1} - p_2^{-1} - \nu) + 2 \nu
      \, ,
      & \text{if } q_1 \leq q_2 \, , \\
      s_1 > s_2 + (\alpha+\beta)(p_1^{-1} - p_2^{-1} - \nu) + 2 \nu \, ,
      & \text{if } q_1 >    q_2
    \end{cases}
    \qquad \!\! \text{where} \qquad \!\!
    \nu := \left( q_2^{-1} - p_1^\ast \right)_+ \,\, .
    \label{eq:BesovIntoWavePacketCharacterization}
  \end{equation}
\end{thm}

\begin{rem*}
  Let us somewhat clarify this statement.
  The Besov space $B_{p,q}^s(\R^2)$ is defined as a subspace
  of $\Schwartz'(\R^2)$, while the wave packet smoothness space
  $\PacketSpace_{s}^{p,q} (\alpha,\beta)$ is a subspace of $Z'$
  (see Definition~\ref{def:ReservoirDefinition}).

  Therefore, validity of the embedding $B_{p_1,q_1}^{s_1}(\R^2)
  \hookrightarrow \PacketSpace_{s_2}^{p_2,q_2}(\alpha,\beta)$ means, strictly
  speaking, that the map
  $B_{p_1,q_1}^{s_1}(\R^2) \to \PacketSpace_{s_2}^{p_2,q_2}(\alpha,\beta),
   f \mapsto f|_{Z}$ is well-defined and bounded.
  Likewise, validity of the embedding
  $\PacketSpace_{s_1}^{p_1, q_1}(\alpha,\beta)
   \hookrightarrow B_{p_2, q_2}^{s_2}(\R^2)$ means that each
  $f \in \PacketSpace_{s_1}^{p_1, q_1}(\alpha,\beta) \subset Z'$ can be extended
  to a uniquely determined tempered distribution $f_{\Schwartz'}$ and that
  the map $\PacketSpace_{s_1}^{p_1,q_1}(\alpha,\beta)
  \to B_{p_2,q_2}^{s_2} (\R^2),f \mapsto f_{\Schwartz'}$ is well-defined
  and bounded.
\end{rem*}

\begin{proof}
  It was shown in Lemma~9.15 in \cite{DecompositionEmbeddings} that the map
  \begin{equation}
    B_{p,q}^s (\R^2) \to \DecompSp(\CalB, L^p, \ell_{v^{(s)}}^q),
    f \mapsto f|_{Z}
    \label{eq:BesovAsDecomposition}
  \end{equation}
  is an isomorphism of quasi-Banach spaces.
  Here, the \textbf{inhomogeneous Besov covering} $\CalB = (B_n)_{n \in \N_0}$
  is given by $B_0 = B_4 (0)$ and
  $B_n = B_{2^{n+2}}(0) \setminus \overline{B}_{2^{n-2}} (0)$
  for $n \in \N$ and the weight $v^{(s)}$ is given by
  $v^{(s)}_n = 2^{s n}$ for $n \in \N_0$.
  It was shown in Lemma~9.10 in \cite{DecompositionEmbeddings} for
  \[
    S_n := 2^n \, \identity,
    \quad
    e_n := 0,
    \quad
    B_1^{(0)} := B_4 (0) \setminus \overline{B_{1/4}} (0),
    \quad
    B_2^{(0)} := B_4 (0),
    \quad \text{and} \quad
    k_n := \begin{cases}
             1 , & \text{if } n \in \N, \\
             2 , & \text{if } n = 0
           \end{cases}
  \]
  that $\CalB = (B_n)_{n \in \N_0} = (S_n B_{k_n}^{(0)} + e_n)_{n \in \N_0}$
  is an almost structured covering of $\R^2$ with
  associated family $(S_n \mybullet + e_n)_{n \in \N_0}$.

  Given the isomorphism \eqref{eq:BesovAsDecomposition} and the remark
  we made after the theorem, we need to characterise the existence
  of the embeddings
  \(
    \DecompSp(\CalQ^{(\alpha,\beta)}, L^{p_1}, \ell_{w^{s_1}}^{q_1})
    =               \PacketSpace_{s_1}^{p_1,q_1}(\alpha,\beta)
    \smash{\overset{!}{\hookrightarrow}}
    \DecompSp(\CalB, L^{p_2}, \ell_{v^{(s_2)}}^{q_2})
  \)
  and
  \(
    \DecompSp(\CalB, L^{p_1}, \ell_{v^{(s_1)}}^{q_1})
    \smash{\overset{!}{\hookrightarrow}}
    \PacketSpace_{s_2}^{p_2,q_2}(\alpha,\beta)
    = \DecompSp(\CalQ^{(\alpha,\beta)}, L^{p_2}, \ell_{w^{s_2}}^{q_2})
  \).
  To do so, we shall rely on Theorems~\ref{thm:EmbeddingQFinerThanP}
  and \ref{thm:EmbeddingPFinerThanQ}, respectively.
  The main prerequisite for applying these theorems is that
  $\CalQ^{(\alpha,\beta)} = (Q_i)_{i \in I^{(\alpha,\beta)}}$ be
  almost subordinate to $\CalB$ and that $\CalQ^{(\alpha,\beta)}$
  and $w^{s_1}$ be relatively $\CalB$-moderate.

  Since $\CalQ^{(\alpha,\beta)}$ consists only of open and path-connected sets,
  and since $\CalB$ consists only of open sets,
  Lemma~\ref{lem:WeakAndAlmostSubordinateness} implies that
  $\CalQ^{(\alpha,\beta)}$ is almost subordinate to $\CalB$,
  if it is weakly subordinate; that is, we need to show that
  $\sup_{i \in I^{(\alpha,\beta)}} |J_i| < \infty$ where
  $J_i := \{n \in \N_0 \,:\, B_n \cap Q_i \neq \emptyset\}$ for
  $i \in I^{(\alpha,\beta)}$.
  To see that this is the case, let $i = (j,m,\ell) \in I_0^{(\alpha,\beta)}$ be arbitrary.
  For any $n \in \N$ with $\emptyset \neq B_n \cap Q_i \ni \xi$,
  \eqref{eq:AbsoluteValueEstimate} implies that
  \[
    2^{n-2} < |\xi| < 2^{j + 3}
    \quad \text{and} \quad
    2^{j - 2} < |\xi| < 2^{n + 2} \, ,
  \]
  and hence $j - 3 \leq n \leq j + 4$.
  Thus, $J_i \subset \{0\} \cup \{j-3,\dots,j+4\}$, which implies that
  $|J_i| \leq 9$ for all $i \in I_0^{(\alpha,\beta)}$.
  Finally, if $\emptyset \neq B_n \cap Q_0 \ni \xi$ for some $n \in \N$, then
  $2^{n-2} < |\xi| < 4 = 2^{2}$ and hence $n \leq 3$.
  Therefore, $J_0 \subset \{0,\dots,3\}$ and thus $|J_i| \leq 9$ for all
  $i \in I^{(\alpha,\beta)}$.
  We have thus shown that $\CalQ^{(\alpha,\beta)}$ is almost subordinate to
  $\CalB$.

  To verify that, for arbitrary $\sigma \in \R$, the weight $w^\sigma$ is
  relatively $\CalB$-moderate, we recall from
  \eqref{eq:WeightConsistencyCondition} that $w_i^1 \asymp 1 + |\xi|$
  for arbitrary $\xi \in Q_i$ and $i \in I^{(\alpha,\beta)}$.
  Since $1 + |\xi| \asymp 2^n$ for $\xi \in B_n$ and any $n \in \N_0$,
  this implies that
  \begin{equation}
    w_i^\sigma
    \asymp_\sigma (1 + |\xi|)^\sigma
    \asymp_\sigma 2^{\sigma n}
    \quad \text{ for any } \sigma \in \R
          \text{ and all } i \in I^{(\alpha,\beta)}
          \text{ and } n \in \N_0
          \text{ with } Q_i \cap B_n \neq \emptyset \, .
    \label{eq:BesovEmbeddingWeightModerateness}
  \end{equation}
  In particular, $w_i^\sigma \asymp 2^{\sigma n} \asymp w_{i'}^\sigma$, if
  $Q_i \cap B_n \neq \emptyset \neq Q_{i'} \cap B_n$.
  Hence, $w^\sigma$ is relatively $\CalB$-moderate.

  From this, we conclude that the wave packet covering
  $\CalQ^{(\alpha,\beta)}$ is relatively $\CalP$-moderate.
  Indeed,
  \begin{equation}
    |\det T_i|
    = 2^{(\alpha+\beta) j}
    = w_i^{\alpha+\beta}
    \qquad \forall \, i = (j,m,\ell) \in I_0^{(\alpha,\beta)} \, .
    \label{eq:BesovEmbeddingDeterminantModerateness}
  \end{equation}
  Likewise, $|\det T_0| = 1 = w_0^{\alpha+\beta}$, so that
  \eqref{eq:BesovEmbeddingDeterminantModerateness} is also true for $i = 0$.
  In particular, we see that
  \[
    |\det T_i|
    = w_i^{\alpha+\beta}
    \asymp w_{i'}^{\alpha+\beta}
    = |\det T_{i'}|
    \qquad \text{for} \qquad
    i,i' \in I^{(\alpha,\beta)} \text{ such that }
    Q_i \cap B_n \neq \emptyset \neq Q_{i'} \cap B_n \, ,
  \]
  or, in other words, $\CalQ^{(\alpha,\beta)}$ is relatively $\CalB$-moderate.

  \medskip{}

  Now, let us choose, for each $n \in \N_0$, such an index $i_n \in I^{(\alpha,\beta)}$
  that $Q_{i_n} \cap B_n \neq \emptyset$.
  Then, for $\mu$ as defined in the present theorem,
  Theorem~\ref{thm:EmbeddingQFinerThanP} shows that
  \(
    \DecompSp(\CalQ^{(\alpha,\beta)}, L^{p_1}, \ell_{w^{s_1}}^{q_1})
    \hookrightarrow \DecompSp(\CalB, L^{p_2}, \ell_{v^{(s_2)}}^{q_2})
  \)
  holds if and only if $p_1 \leq p_2$ and
  \begin{align*}
    \infty
    & \overset{!}{>}
      \left\|
        \left(
          \frac{v_n^{(s_2)}}{w_{i_n}^{s_1}}
          \cdot |\det T_{i_n}|^{p_1^{-1} - p_2^{-1} - \tau_1}
          \cdot |\det S_n|^{\tau_1}
        \right)_{n \in \N_0}
      \right\|_{\ell^{q_2 \cdot (q_1 / q_2)'}}
      \\
    ({\scriptstyle{\eqref{eq:BesovEmbeddingWeightModerateness}
                   \text{ and }
                   \eqref{eq:BesovEmbeddingDeterminantModerateness}}})
    & \asymp
      \left\|
        \left(
          2^{(s_2 - s_1) n
             + n (\alpha+\beta) (p_1^{-1} - p_2^{-1} - \tau_1)
             + 2n \tau_1}
        \right)_{n \in \N_0}
      \right\|_{\ell^{q_2 \cdot (q_1 / q_2)'}} \, .
  \end{align*}
  On the other hand, Equation~\eqref{eq:EmbeddingSequenceSpaceExponent} shows that
  $q_2 \cdot (q_1 / q_2)' = \infty$ if and only if $q_1 \leq q_2$.
  Therefore, the norm in the last expression is finite
  if and only if Condition~\eqref{eq:WavePacketIntoBesovCharacterization}
  is satisfied.

  Similarly, Theorem~\ref{thm:EmbeddingPFinerThanQ} shows that
  \(
    \DecompSp(\CalB, L^{p_1}, \ell_{v^{(s_1)}}^{q_1})
    \hookrightarrow \DecompSp(\CalQ^{(\alpha,\beta)}, L^{p_2}, \ell_{w^{s_2}}^{q_2})
  \)
  holds if and only if $p_1 \leq p_2$ and
  \begin{align*}
    \infty
    & \overset{!}{>}
      \left\|
        \left(
          \frac{w_{i_n}^{s_2}}{v_n^{(s_1)}}
          \cdot |\det S_{n}|^{\tau_2}
          \cdot |\det T_{i_n}|^{p_1^{-1} - p_2^{-1} - \tau_2}
        \right)_{n \in \N_0}
      \right\|_{\ell^{q_2 \cdot (q_1 / q_2)'}}
      \\
    ({\scriptstyle{\eqref{eq:BesovEmbeddingWeightModerateness}
                   \text{ and }
                   \eqref{eq:BesovEmbeddingDeterminantModerateness}}})
    & \asymp
      \left\|
        \left(
          2^{n (s_2 - s_1)
             + 2n \tau_2
             + n (\alpha+\beta)(p_1^{-1} - p_2^{-1} - \tau_2)}
        \right)_{n \in \N_0}
      \right\|_{\ell^{q_2 \cdot (q_1 / q_2)'}} \, .
  \end{align*}
  As before, we see that this norm is finite if and only if
  Condition~\eqref{eq:BesovIntoWavePacketCharacterization} holds.
\end{proof}

As a direct application of the preceding theorem to the case $\alpha = \beta = 1$,
we conclude that the inhomogeneous Besov spaces are special examples of the wave packet smoothness spaces.

\begin{cor}\label{cor:BesovAsWavePacketSmoothness}
  \[
    \PacketSpace_s^{p,q}(1,1) = B^s_{p,q} (\R^2)
  \]
  for all $p,q \in (0,\infty]$ and $s \in \R$.
\end{cor}

From Section~2.3.3 in \cite{TriebelTheoryOfFunctionSpaces}, we know that
$\Schwartz(\R^2) \hookrightarrow B_{p,q}^{\sigma} (\R^2) \hookrightarrow \Schwartz' (\R^2)$.
Combining this with the previous theorem, we conclude that,
for arbitrary $p,q \in (0,\infty]$ and $s \in \R$,
\[
  \Schwartz (\R^2)
  \hookrightarrow B_{p,q}^\sigma (\R^2)
  \hookrightarrow \PacketSpace_s^{p,q} (\alpha,\beta)
  \hookrightarrow B_{p,q}^\varrho (\R^2)
  \hookrightarrow \Schwartz' (\R^2) \,
\]
if $\sigma$ is sufficiently large and $\varrho$ sufficiently small (negative).
We have thus established the following corollary.

\begin{cor}\label{cor:WavePacketSmoothnessAndSchwartzSpaces}
  \[
    \Schwartz (\R^2)
    \hookrightarrow \PacketSpace_s^{p,q}(\alpha,\beta)
    \hookrightarrow \Schwartz'(\R^2)
  \]
  for arbitrary $0 \leq \beta \leq \alpha \leq 1$, $p,q \in (0,\infty]$, and $s \in \R$.
\end{cor}

We now turn to studying conditions under which the wave packet smoothness
space $\PacketSpace_s^{p,q}(\alpha,\beta)$ are embedded in the Sobolev space $W^{k,r}(\R^2)$.
The following theorem will also justify the name ``smoothness spaces,'' since it
will show that, if the smoothness parameter $s$ is chosen so that
${s > k + 2 \, (1 + p^{-1})}$, then the wave packet smoothness space
$\PacketSpace_s^{p,q}(\alpha,\beta)$ consists of $C^k$ functions.

\begin{thm}\label{thm:WavePacketIntoSobolev}
  Let $0 \leq \beta \leq \alpha \leq 1$,
  $p,q \in (0,\infty]$, $k \in \N_0$, $s \in \R$ and $r \in [1,\infty]$,
  and let us define $r^{\triangledown} := \min \{r, r'\}$. If
  \begin{equation}
    p \leq r
    \quad \text{and} \quad
    \begin{cases}
      s \geq k + (\alpha+\beta)(p^{-1} - r^{-1}) \, ,
      & \text{if } q \leq r^{\triangledown} \, , \\
      s >    k + (\alpha+\beta)(p^{-1} - r^{-1})
             + (2 - \alpha - \beta)
               \left(\frac{1}{r^{\triangledown}} - \frac{1}{q}\right) \, ,
      & \text{if } q > r^{\triangledown} \, ,
    \end{cases}
    \label{eq:SobolevEmbeddingSufficient}
  \end{equation}
  then $\PacketSpace_s^{p,q}(\alpha,\beta) \hookrightarrow W^{k,r}(\R^2)$;
  that is, there is an injective bounded linear map
  \[
    \iota : \PacketSpace_s^{p,q}(\alpha,\beta) \to W^{k,r}(\R^2)
    \quad \text{such that} \quad
    \iota f = f
    \quad \forall \, f \in \Schwartz(\R^2)
                  \text{ with } \widehat{f} \in C_c^\infty (\R^2) \, .
  \]
  Moreover, if \eqref{eq:SobolevEmbeddingSufficient} is satisfied for
  $r = \infty$, then $\iota f \in C_b^k (\R^2)$ for all
  $f \in \PacketSpace_s^{p,q}(\alpha,\beta)$ where
  \[
    C_b^k (\R^2)
    = \{
        f \in C^k (\R^2; \CC)
        \,:\,
        \partial^\alpha f \in L^\infty (\R^2)
        \quad \forall \, \alpha \in \N_0^2 \text{ with } |\alpha| \leq k
      \} \, .
  \]

  Conversely, assume that there is such a $C > 0$ that
  $\|f\|_{W^{k,r}} \leq C \cdot \|f\|_{\PacketSpace_s^{p,q}(\alpha,\beta)}$
  for all such $f \in \Schwartz (\R^2)$ that
  ${\widehat{f} \in C_c^\infty(\R^2)}$.
  Then
  \begin{equation}
    p \leq r
    \quad \text{and} \quad
    \begin{cases}
      s \geq k + (\alpha+\beta)(p^{-1} - r^{-1}) \, ,
      & \text{if } q \leq r \, , \\
      s >    k + (\alpha+\beta)(p^{-1} - r^{-1})
               + (2-\alpha-\beta) (r^{-1} - q^{-1}) \, ,
      & \text{if } q > r \, .
    \end{cases}
    \label{eq:SobolevEmbeddingNecessary}
  \end{equation}
  Furthermore, if $r = \infty$, then \eqref{eq:SobolevEmbeddingSufficient} is
  satisfied and if $r \in (2,\infty)$, then
  \begin{equation}
    \begin{cases}
      s \geq k + (\alpha+\beta)(p^{-1} - 2^{-1})_+ \,\, ,
      & \text{if } q \leq 2 \, , \\
      s > k + (\alpha+\beta)(p^{-1} - 2^{-1})_+
            + (2 - \alpha - \beta)(2^{-1} - q^{-1}) \, ,
      & \text{if } q >    2 \, .
    \end{cases}
    \label{eq:SobolevSpecialCondition}
  \end{equation}
\end{thm}

\begin{rem*}
  Note that this theorem gives a \emph{complete characterisation} of
  the existence of the embedding for $r \in [1,2] \cup \{\infty\}$.
  Indeed, for $r = \infty$, this results from the theorem statement;
  moreover, $r^{\triangledown} = r$ for $r \in [1,2]$, so that
  \eqref{eq:SobolevEmbeddingSufficient} and \eqref{eq:SobolevEmbeddingNecessary}
  are identical.

  For $r \in (2,\infty)$, on the other hand, there is a gap
  between the necessary and the sufficient conditions.
\end{rem*}

\begin{proof}
  Let us use the notations of Lemma~\ref{lem:CoveringAlmostStructured} and additionally define
  \[
    v_i := |\det T_i|^{p^{-1} - r^{-1}}
    \qquad \text{and} \qquad
    u_i := |\det T_i|^{p^{-1} - r^{-1}} \cdot \big(|b_i|^k + \|T_i\|^{k} \big)
    \quad \text{for } i \in I := I^{(\alpha,\beta)} \, .
  \]
  Next, remember that
  $\PacketSpace_s^{p,q}(\alpha,\beta) = \DecompSp(\CalQ^{(\alpha,\beta)}, L^p, \ell_{w^s}^q)$
  where $\CalQ^{(\alpha,\beta)} = (Q_i)_{i \in I}$
  with $Q_i = T_i \, Q_i ' + b_i$ is an almost
  structured covering, and thus --- according to Theorem~2.8 in
  \cite{DecompositionIntoSobolev} --- a regular covering of $\R^2$.
  We can thus apply \cite[Corollary 3.5]{DecompositionIntoSobolev} to conclude that
  the embedding
  $\PacketSpace_s^{p,q} (\alpha,\beta) \hookrightarrow W^{k,r}(\R^2)$
  --- which is to be understood as in the statement of the theorem --- holds
  as long as
  \begin{equation}
    p \leq r
    \quad \text{and} \quad
    \ell_{w^s}^q (I^{(\alpha,\beta)})
    \hookrightarrow \ell_{v}^{r^{\triangledown}} (I^{(\alpha,\beta)}) ,
    \quad \text{as well as} \quad
    \ell_{w^s}^q (I^{(\alpha,\beta)})
    \hookrightarrow \ell_{u}^{r^{\triangledown}} (I^{(\alpha,\beta)}) \, .
    \label{eq:SobolevEmbeddingSufficientPreSimplification}
  \end{equation}

  To verify \eqref{eq:SobolevEmbeddingSufficientPreSimplification},
  we simplify the weights $v = (v_i)_{i \in I}$ and $u = (u_i)_{i \in I}$.
  First, for $i = (j,m,\ell) \in I_0^{(\alpha,\beta)}$, we infer from the
  definition of $b_i$ in Lemma~\ref{lem:CoveringAlmostStructured} and that of
  $c_{j,m}$ in Definition~\ref{defn:CoveringSets}, that
  \[
    |b_i|
    = |c_{j,m}|
    = \left|
        \left(
          \begin{smallmatrix}
            2^{j-1} + m \cdot 2^{\alpha j} \\
            0
          \end{smallmatrix}
        \right)
      \right|
    \asymp 2^j \, ,
    \quad \text{since} \quad
    0 \leq m \leq m_j^{\max} = \lceil 2^{(1-\alpha)j - 1} \rceil
      \lesssim 2^{(1 - \alpha) j} \, .
  \]
  Second, since $T_i = R_{j,\ell} \, A_{j}$ and $A_{j} = \mathrm{diag} (2^{\alpha j}, 2^{\beta j})$
  and since $\beta \leq \alpha$,
  we conclude that
  \[
    \|T_i\|
    = \|A_{j}\|
    = \max \{2^{\alpha j}, 2^{\beta j} \}
    =      2^{\alpha j} \, .
  \]
  Combining these two results leads to
  $|b_i|^k + \|T_i\|^k \asymp 2^{j k} + 2^{\alpha j k} \asymp 2^{j k} = w_i^k$
  for $i = (j,m,\ell) \in I_0^{(\alpha,\beta)}$ where the weight
  $w^k$ is as introduced in Lemma~\ref{lem:WavePacketWeight}.
  On the other hand, if $i = 0$, then $b_i = 0$ and $T_i = \identity$ and thus
  $|b_i|^k + \|T_i\|^k = 1 = w_i^k$ as well.
  Therefore,
  \begin{equation}
    |b_i|^k + \|T_i\|^k \asymp w_i^k \qquad \forall \, i \in I^{(\alpha,\beta)}
    \, .
    \label{eq:SpecialSobolevWeightEstimate}
  \end{equation}

  Furthermore, we note that
  $|\det T_i| = w_i^{\alpha + \beta}$.
  Therefore,
  \begin{equation}
    v_i \asymp w_i^{(\alpha+\beta)(p^{-1} - r^{-1})}
    \quad \text{and} \quad
    u_i \asymp w_i^{k + (\alpha + \beta)(p^{-1} - r^{-1})} \, .
    \label{eq:SobolevWeightsSimplified}
  \end{equation}

  Finally, since $w_i \geq 1$ for all $i \in I$ and since $k \geq 0$,
  we conclude that \eqref{eq:SobolevEmbeddingSufficientPreSimplification}
  holds if and only if
  \[
    p \leq r
    \quad \text{and} \quad
    \ell_{w^s}^q (I^{(\alpha,\beta)})
    \hookrightarrow
    \ell_{w^{k + (\alpha + \beta)(p^{-1} - r^{-1})}}^{r^{\triangledown}}
    (I^{(\alpha,\beta)}) \, .
  \]
  This, according to Lemma~5.1 in \cite{DecompositionIntoSobolev}
  and the remark that follows it, holds if and only if
  \begin{equation}
    p \leq r
    \quad \text{and} \quad
        \frac{w^{k + (\alpha+\beta)(p^{-1} - r^{-1})}}{w^s}
    \in \ell^{r^{\triangledown} \cdot (q / r^{\triangledown})'}
        (I^{(\alpha,\beta)})
    \quad \text{where} \quad
    \frac{1}{r^{\triangledown} \cdot (q / r^{\triangledown})'}
    = \left(\frac{1}{r^{\triangledown}} - \frac{1}{q}\right)_+ \,\, .
    \label{eq:SobolevEmbeddingSufficientAlmostDone}
  \end{equation}
  In particular, this shows that
  $r^{\triangledown} \cdot (q / r^{\triangledown}) = \infty$ if and only if
  $q \leq r^{\triangledown}$.
  Now, using the same arguments as in the proof of
  Theorem~\ref{thm:WavePacketSpacesEmbeddings} --- see especially
  \eqref{eq:MainSequenceSpaceNormCharacterization1}
  and \eqref{eq:MainSequenceSpaceNormCharacterization2} --- we conclude that
  \eqref{eq:SobolevEmbeddingSufficientAlmostDone} holds
  if and only if \eqref{eq:SobolevEmbeddingSufficient} does.

  \medskip{}

  It remains to prove the converse statement.
  To do so, we shall apply Theorem~4.7 in \cite{DecompositionIntoSobolev},
  which is fully applicable only if there is such an
  $M > 0$ that $\|T_i^{-1}\| \leq M$ for all $i \in I$.
  This can be easily verified in our case, since $T_0^{-1} = \identity$
  and since
  \(
     \|T_i^{-1}\|
    = \|A_{j}^{-1}\|
    = \max \{ 2^{-\alpha j}, 2^{-\beta j} \}
    \leq 1
  \)
  for all $i = (j,m,\ell) \in I_0^{(\alpha,\beta)}$.

  Now, let us define
  \[
    u_i^{(\sigma,\tau)}
    := |\det T_i|^{\sigma^{-1} - \tau^{-1}}
       \cdot \big( |b_i|^k + \|T_i\|^k \big)
    \quad \text{for} \quad
    i \in I
    \quad \text{and} \quad
    \sigma, \tau \in (0,\infty] \,
  \]
  and note that
  $u_i^{(\sigma,\tau)} \asymp w_i^{k + (\alpha+\beta)(\sigma^{-1} - \tau^{-1})}$
  as a consequence of $|\det T_i| = w_i^{\alpha+\beta}$ and of
  \eqref{eq:SpecialSobolevWeightEstimate}.
  Since by assumption
  $\|f\|_{W^{k,r}} \lesssim \|f\|_{\PacketSpace_s^{p,q}(\alpha,\beta)}$
  for all $f \in \Schwartz(\R^2)$ with $\widehat{f} \in C_c^\infty (\R^2)$,
  we can combine Theorems~4.4 and 4.7 and Lemma~5.1 in \cite{DecompositionIntoSobolev}
  to conclude that
  $p \leq r$, and that
  \begin{equation}
    w^{k - s + (\alpha+\beta)(\sigma^{-1} - \tau^{-1})}
    \asymp \frac{u^{(\sigma,\tau)}}{w^s}
    \in \ell^{\varrho \cdot (q / \varrho)'} (I^{(\alpha,\beta)})
    \label{eq:SobolevNecessaryAlmostDone}
  \end{equation}
  holds for the following choices of $\sigma,\tau,\varrho$:
  \begin{enumerate}[label=(\arabic*)]
    \item 
          $(\sigma,\tau,\varrho) = (p,r,r)$;

    \item 
          $(\sigma,\tau,\varrho) = (p,r,1) = (p,r,r^{\triangledown})$
          if $r = \infty$;

    \item 
          $(\sigma,\tau,\varrho) = (p,2,2)$ if $r \in (2,\infty)$; and

    \item 
          $(\sigma,\tau,\varrho) = (p,p,2)$ if $r \in (2,\infty)$.
  \end{enumerate}
  Using the same arguments as in the proof of
  Theorem~\ref{thm:WavePacketSpacesEmbeddings}, we conclude that
  (1) implies \eqref{eq:SobolevEmbeddingNecessary}, while
  (2) implies \eqref{eq:SobolevEmbeddingSufficient}.
  Similarly, (3) and (4) imply, respectively, that
  \[
    \begin{cases}
      s \geq k + (\alpha+\beta)(p^{-1} - 2^{-1}) \, ,
      & \text{if } q \leq 2 \, , \\
      s > k + (\alpha+\beta)(p^{-1} - 2^{-1})
            + (2 - \alpha - \beta)(2^{-1} - q^{-1}) \, ,
      & \text{if } q >    2
    \end{cases}
  \]
  and
  \[
    \begin{cases}
      s \geq k \, ,
      & \text{if } q \leq 2 \, , \\
      s > k + (2 - \alpha - \beta)(2^{-1} - q^{-1}) \, ,
      & \text{if } q >    2 \, .
    \end{cases}
  \]
  Combining these shows that \eqref{eq:SobolevSpecialCondition} is satisfied.
\end{proof}

\subsection{Identifying \texorpdfstring{$\alpha$-modulation}{α-modulation}
spaces as wave-packet smoothness spaces}
\label{sub:WavePacketAlphaModulation}

\noindent In this subsection, we show that, for arbitrary $\alpha \in [0,1)$,
the $\alpha$-modulation spaces $M_{p,q}^{s,\alpha}(\R^2)$
\cite{GroebnerAlphaModulationSpaces,BorupNielsenAlphaModulationSpaces}
are identical --- up to canonical identifications --- to the wave packet
smoothness spaces $\PacketSpace_s^{p,q}(\alpha,\alpha)$.
In particular, we show that the wave packet smoothness spaces
$\PacketSpace_s^{p,q}(0,0)$ are identical to the modulation spaces
$M_{p,q}^s (\R^2)$, which play a crucial role in time-frequency analysis
\cite{GroechenigTimeFrequencyAnalysis,ModulationSpacesTimeFrequency}.
Precisely, we prove the following theorem.

\begin{thm}\label{thm:AlphaModulationAsWavePacket}
  Let $\alpha \in [0,1)$, $p,q \in (0,\infty]$ and $s \in \R$.
  Then, for the $\alpha$-modulation space $M_{p,q}^{s,\alpha}(\R^2)$
  as defined in Definition~2.4 in \cite{BorupNielsenAlphaModulationSpaces}
  and the space $Z$ as introduced in Definition~\ref{def:ReservoirDefinition},
  the map
  \[
    M_{p,q}^{s,\alpha}(\R^2) \to \PacketSpace_s^{p,q} (\alpha, \alpha),
    f \mapsto f |_{Z}
  \]
  is an isomorphism of quasi-Banach spaces. In other words,
  $M_{p,q}^{s,\alpha}(\R^2) = \PacketSpace_s^{p,q} (\alpha, \alpha)$,
  up to canonical identifications.
\end{thm}

\begin{rem*}
  Usually $\alpha$-modulation spaces for $\alpha = 1$
  are understood as inhomogeneous Besov spaces.
  With this interpretation, Corollary~\ref{cor:BesovAsWavePacketSmoothness} shows
  that the preceding theorem also remains valid for $\alpha = 1$.
\end{rem*}

\begin{proof}
  It was shown in Corollary~9.16 in \cite{DecompositionEmbeddings} that the map
  \begin{equation}
    M_{p,q}^{s,\alpha}(\R^2) \to
    \DecompSp \big( \CalP^{(\alpha)}, L^p, \ell^q_{v^{(s/(1-\alpha))}} \big),
    f \mapsto f|_{Z}
    \label{eq:AlphaModulationAsDecomposition}
  \end{equation}
  is an isomorphism of quasi-Banach spaces.
  Here, the covering $\CalP^{(\alpha)}$ is given by
  \[
    \CalP^{(\alpha)} = \big( P_k \big)_{k \in \Z^2 \setminus \{0\}}
    \quad \text{where} \quad
    P_k := B_{r |k|^{\alpha_0}} \, \big( |k|^{\alpha_0} \, k \big)
    \quad \text{and} \quad
    \alpha_0 := \frac{\alpha}{1 - \alpha} \, ,
  \]
  and where $r > 0$ is chosen large enough so that $\CalP^{(\alpha)}$ is a structured
  admissible covering of $\R^2$; this is possible due to Lemma~9.3 in
  \cite{DecompositionEmbeddings}.
  Furthermore, for arbitrary $\theta \in \R$, the weight
  $v^{(\theta)} = (v_k^{(\theta)})_{k \in \Z^2 \setminus \{0\}}$
  is given by $v_k^{(\theta)} = (1 + |k|^2)^{\theta/2}$.

  Given \eqref{eq:AlphaModulationAsDecomposition}, it is enough to prove that
  \(
    \DecompSp \big( \CalP^{(\alpha)}, L^p, \ell^q_{v^{(\gamma/(1-\alpha))}} \big)
    = \PacketSpace_s^{p,q}(\alpha,\alpha)
  \).
  To do so, Lemma~6.11 in \cite{DecompositionEmbeddings} and the identity
  $\PacketSpace_s^{p,q}(\alpha,\alpha) = \DecompSp(\CalQ^{(\alpha,\alpha)}, L^p, \ell_{w^s}^q)$
  show that it is enough to prove that
  the covering $\CalP^{(\alpha)}$ is equivalent to the
  $(\alpha,\alpha)$-wave packet covering
  $\CalQ^{(\alpha,\alpha)} = (Q_i)_{i \in I^{(\alpha,\alpha)}}$ and that
  \begin{equation}
    v_k^{(s/(1-\alpha))} \asymp w_i^s
    \quad \text{for all} \quad i \in I^{(\alpha,\alpha)}
          \text{ and } k \in \Z^2 \setminus \{0\}
          \text{ with } P_k \cap Q_i \neq \emptyset \, .
    \label{eq:AlphaModulationWeightEquivalence}
  \end{equation}
  We start by proving the latter.
  From \eqref{eq:AbsoluteValueEstimate}, we conclude that
  $1 + |\xi| \asymp |\xi| \asymp 2^j = w_i$ for all $\xi \in Q_i$ and
  $i = (j,m,\ell) \in I_0^{(\alpha,\alpha)}$.
  On the other hand, $Q_0 = B_4 (0)$
  and thus $1 + |\xi| \asymp 1 = w_0$ for all $\xi \in Q_0$ as well.
  Furthermore, Lemma~9.2 in \cite{DecompositionEmbeddings} shows that
  $1 + |\xi| \asymp (1 + |\xi|^2)^{1/2} \asymp v_k^{(1/(1-\alpha))}$
  for all $k \in \Z^2 \setminus \{0\}$ and $\xi \in P_k$.
  Therefore, if there is some $\xi \in P_k \cap Q_i \neq \emptyset$, then
  \[
    v_k^{(s / (1-\alpha))}
    = \big( v_k^{(1/(1-\alpha))} \big)^s
    \asymp (1 + |\xi|)^s
    \asymp w_i^s \, .
  \]

  \medskip{}

  It remains to prove that the coverings $\CalQ^{(\alpha,\alpha)}$ and
  $\CalP^{(\alpha)}$ are equivalent.
  Since both coverings consist of open, path-connected sets,
  Lemma~\ref{lem:WeakAndAlmostSubordinateness} shows that it
  suffices to prove that the two coverings are \emph{weakly} equivalent.
  To prove this, we shall use Lemma~B.2 in
  \cite{BorupNielsenAlphaModulationSpaces}, which implies that any
  two $\alpha$\textbf{-coverings} of $\R^d$ are weakly equivalent.
  Therefore, it suffices to show that both $\CalQ^{(\alpha,\alpha)}$ and
  $\CalP^{(\alpha)}$ are $\alpha$-coverings of $\R^2$.

  To present the notion of $\alpha$-coverings as introduced in
  Definition~2.1 in \cite{BorupNielsenAlphaModulationSpaces}, we need
  yet another notation: For a bounded open set $\Omega \subset \R^d$,
  we shall write
  \[
    R_\Omega
    := \inf \{
              R > 0
              \,:\,
              \exists \, \xi \in \R^d : \Omega \subset \xi + B_R (0)
            \}
    \quad \text{and} \quad
    r_{\Omega}
    := \sup \{
              r > 0
              \,:\,
              \exists \, \xi \in \R^d : \xi + B_r (0) \subset \Omega
            \} \, .
  \]
  With this, a family $(\Omega_\ell)_{\ell \in L}$ is called an
  $\alpha$-covering if it satisfies the following:
  \begin{enumerate}[label=(\arabic*)]
    \item $(\Omega_\ell)_{\ell \in L}$ is an admissible covering of $\R^d$
          consisting of open bounded sets;

    \item there is such a constant $K \geq 1$ that
          $R_{\Omega_\ell} / r_{\Omega_\ell} \leq K$ for all $\ell \in L$; and

    \item $\lambda(\Omega_\ell) \asymp (1 + |\xi|)^{d \alpha}$
          for all $\ell \in L$ and $\xi \in \Omega_\ell$ where the implied
          constant is independent of $\ell$ and $\xi$.
          Here, $\lambda$ denotes the Lebesgue measure.
  \end{enumerate}
  The first of these three conditions is satisfied for $\CalQ^{(\alpha,\alpha)}$
  and $\CalP^{(\alpha)}$, as shown by Lemmas~\ref{lem:CoveringCovers} and \ref{lem:Admissibility},
  and by Theorem~2.6 in \cite{BorupNielsenAlphaModulationSpaces}, respectively.

  Now, since the covering
  $\CalP^{(\alpha)} = (P_k)_{k \in \Z^2 \setminus \{0\}}$ consists
  of open balls, $R_{P_k} = r_{P_k}$ for all $k \in \Z^2 \setminus \{0\}$.
  Hence, $\CalP^{(\alpha)}$ also satisfies the second condition from above.
  Finally, since we are dealing with coverings of $\R^2$ and hence $d = 2$,
  \begin{align*}
    \lambda (P_k)
    & = \lambda \big(B_{r |k|^{\alpha_0}} (|k|^{\alpha_0} \, k) \big)
      \asymp (r |k|^{\alpha_0})^2 \\
    ({\scriptstyle{|k| \asymp (1 + |k|^2)^{1/2}
                   \text{ since } k \in \Z^2 \setminus \{0\}}})
    & \asymp \big[ (1 + |k|^2)^{1/2} \, \big]^{2 \alpha_0}
      =      \big[ v_k^{(1 / (1 - \alpha))} \, \big]^{2 \alpha}
      \asymp (1 + |\xi|)^{2 \alpha}
      =      (1 + |\xi|)^{\alpha d}
  \end{align*}
  for all $k \in \Z^2 \setminus \{0\}$ and $\xi \in P_k$.
  Here, we noted that $1 + |\xi| \asymp v_k^{(1 / (1-\alpha))}$ for
  $\xi \in P_k$, as seen above.

  \medskip{}

  It remains to show that $\CalQ^{(\alpha,\alpha)}$ satisfies conditions
  (2) and (3).
  To do so, we first note that if $\Omega = T Q + b$ with $T \in \GL(\R^d)$ and
  $b \in \R^d$ and with an open bounded set $Q$ such that
  $Q \subset B_\varrho (\xi)$, then
  $\Omega \subset b + T \xi + T B_\varrho (0)
          \subset b + T \xi + B_{\varrho \|T\|} (0)$
  and hence $R_\Omega \leq \varrho \|T\|$.
  Conversely, if $\Omega = T Q + b$ with $Q \supset B_\varrho (\xi)$, then
  \[
    b + T \xi + B_{\varrho / \|T^{-1}\|} (0)
    = b + T \xi + T T^{-1} B_{\varrho / \|T^{-1}\|} (0)
    \subset b + T B_\varrho (\xi)
    \subset T Q + b
    = \Omega \, ,
  \]
  and hence $r_\Omega \geq \varrho / \|T^{-1}\|$.

  Moreover, we note that $Q_i = T_i Q + b_i$ for
  $i = (j,m,\ell) \in I_0^{(\alpha,\alpha)}$ with $T_i,b_i$ as in
  Lemma~\ref{lem:CoveringAlmostStructured} and with
  $Q = (-\eps, 1 + \eps) \times (-1-\eps, 1 + \eps)$ for a certain
  $\eps \in (0,1)$.
  From the definition of $Q$, we conclude that $Q \subset B_4 (0)$ and
  $Q \supset B_{1/2}(\xi_0)$ for $\xi_0 = (1/2, \, 0)^t \in \R^2$.
  Finally, recalling that $\beta = \alpha$, we see that
  \[
    \|T_i\|
    = \|A_{j}\|
    = \max \{ 2^{\alpha j}, 2^{\beta j} \}
    = 2^{\alpha j}
    \quad \text{and} \quad
    \|T_i^{-1}\|
    = \|A_{j}^{-1}\|
    = \max \{ 2^{-\alpha j}, 2^{-\beta j} \}
    = 2^{-\alpha j} \, .
  \]
  All these considerations imply that
  \[
    \frac{1}{2} \cdot 2^{\alpha j}
    = \frac{1}{2} \cdot \|T_i^{-1}\|^{-1}
    \leq r_{Q_i}
    \leq R_{Q_i}
    \leq 4 \cdot \|T_i\|
    \leq 4 \cdot 2^{\alpha j} \, ,
  \]
  and hence $R_{Q_i} / r_{Q_i} \leq 8$ for all
  $i = (j,m,\ell) \in I_0^{(\alpha,\alpha)}$.
  Furthermore, since $Q_0 = B_4 (0)$, we also see that
  $R_{Q_0} / r_{Q_0} = 1 \leq 8$.
  Put together, this shows that the $(\alpha,\alpha)$-wave packet covering
  $\CalQ^{(\alpha,\alpha)}$ satisfies Condition~(2).

  \smallskip{}

  To verify Condition~(3), we recall that $1 + |\xi| \asymp 2^j$ for
  all $\xi \in Q_i$ and $i = (j,m,\ell) \in I_0^{(\alpha,\alpha)}$.
  Noting that $Q_i = T_i Q + b_i$ and $\alpha = \beta$,
  this shows that, for $i = (j,m,\ell) \in I_0^{(\alpha,\alpha)}$,
  \[
    \lambda (Q_i)
    \asymp | \det T_i|
    =      2^{\alpha j} \cdot 2^{\beta j}
    =      2^{2 \alpha j}
    \asymp (1 + |\xi|)^{2 \alpha}
    =      (1 + |\xi|)^{d \alpha}
    \qquad \forall \, \xi \in Q_i .
  \]
  Finally, $1 + |\xi| \asymp 1$ and hence $\lambda (Q_0) \asymp 1 \asymp (1 + |\xi|)^{d \alpha}$
  for all $\xi \in Q_0 = B_4 (0)$.
  Therefore, $\CalQ^{(\alpha,\alpha)}$ satisfies Condition~(3)
  and thus is an $\alpha$-covering of $\R^2$.
\end{proof}

\section{Universality of the wave packet coverings}
\label{sec:Universality}

\noindent Even though we showed in Section~\ref{sub:WavePacketAlphaModulation} that the
$\alpha$-modulation spaces arise as special cases of the wave packet smoothness spaces,
it might still be objected that the construction of the coverings
$\CalQ^{(\alpha,\beta)}$ involves a lot of arbitrariness, so that these
coverings and the decomposition spaces associated with them are rather esoteric.

In the present section, we show that this is not the case.
Generalising the concept of $\alpha$-coverings
\cite{GroebnerAlphaModulationSpaces,BorupNielsenAlphaModulationSpaces},
we introduce the \emph{natural} class of $(\alpha,\beta)$ coverings of $\R^2$.
We then prove that any two $(\alpha,\beta)$ coverings determine the same class
of decomposition spaces.
Finally, we show that the wave packet coverings $\CalQ^{(\alpha,\beta)}$
are indeed $(\alpha,\beta)$ coverings.
In summary, this shows that the wave packet coverings $\CalQ^{(\alpha,\beta)}$
and the wave packet smoothness spaces
$\PacketSpace_s^{p,q}(\alpha,\beta)$ are natural objects universal among,
respectively, all coverings and function spaces with a similar frequency concentration.

We begin by introducing the class of $(\alpha,\beta)$ coverings,
drawing on the intuition that an element $Q_i$ of an
$(\alpha,\beta)$-covering $\CalQ = (Q_i)_{i \in I}$ should essentially
be a set or a union of two sets symmetric with respect to the origin of the frequency plane of length
$\approx (1 + |\xi|)^\alpha$ in the radial direction
and of width $\approx (1 + |\xi|)^{\beta}$ in the angular direction
where $\xi \in Q_i$ is chosen arbitrarily.

In the rest of this section, we shall identify vectors
$(x,y) \in \R^2$ with corresponding complex numbers $x + i y$.

\begin{defn}\label{def:AlphaBetaCoverings}
  Let $\alpha, \beta \in [0,1]$.
  A family $\CalQ = (Q_i)_{i \in I}$ of open bounded subsets of $\R^2$
  is called an $(\alpha,\beta)$-\textbf{covering} of $\R^2$, if it
  satisfies the following conditions:
  \begin{enumerate}[label=(\arabic*)]
    \item $\CalQ$ is an admissible covering of $\R^2$;

    \item $\lambda(Q_i) \asymp (1 + |\xi|)^{\alpha + \beta}$ for all $i \in I$
          and $\xi \in Q_i$ where $\lambda$ denotes the Lebesgue measure;

    \item for each $i \in I$, there is an interval $L_i \subset [0,\infty)$
          such that
          \begin{enumerate}[label=(\alph*)]
            \item $\lambda (L_i) \leq (1 + |\xi|)^{\alpha}$
                  for all $\xi \in Q_i$ and

            \item $|\xi| \in L_i$ for all $\xi \in Q_i$; and
          \end{enumerate}

    \item for each $i \in I$, there is an angle $\phi_i \in \R$ such that
          \begin{equation}
            \forall \, \xi \in Q_i \quad
              \exists \, \phi \in \R \, : \quad
                \min_{\omega \in \{0,\pi\}}
                  |\phi - (\phi_i + \omega)|
                \lesssim (1 + |\xi|)^{\beta - 1}
                \quad \text{and} \quad
                \xi = |\xi| \cdot e^{i \phi} \, .
            \label{eq:AlphaBetaCoveringAngleControl}
          \end{equation}
  \end{enumerate}
  Here, all the constants must be independent of the choice of
  $i \in I$ and $\xi \in Q_i$.
\end{defn}

Our first goal is to prove that any two $(\alpha,\beta)$ coverings are
weakly equivalent.
To do so, the following lemma will be helpful.

\begin{lem}\label{lem:AlphaBetaCoveringNormEstimate}
  Let $\alpha,\beta \in [0,1]$ and let $\CalQ = (Q_i)_{i \in I}$ be an
  $(\alpha,\beta)$ covering of $\R^2$.
  Then $1 + |\xi| \asymp 1 + |\eta|$ for arbitrary $i \in I$ and
  $\xi, \eta \in Q_i$.
\end{lem}

\begin{proof}
  Let us first consider the case where $\alpha = 0$.
  Then, the intervals $L_i$ introduced in Condition~(3) of
  Definition~\ref{def:AlphaBetaCoverings} satisfy
  $\lambda (L_i) \leq C \cdot (1 + |\xi|)^{\alpha} = C$ for all $i \in I$.
  Since, on the other hand, the Lebesgue measure of an interval is simply
  its length and since Definition~\ref{def:AlphaBetaCoverings} implies that
  $|\xi|, |\eta| \in L_i$ if $\xi, \eta \in Q_i$, we conclude that
  $\big|\, |\xi| - |\eta| \,\big| \leq C$.
  Therefore,
  \[
    1 + |\xi|
    \leq 1 + |\eta| + \big| |\xi| - |\eta| \big|
    \leq 1 + C + |\eta|
    \leq (1 + C) \cdot (1 + |\eta|) \, .
  \]
  By symmetry, we also infer that $1 + |\eta| \leq (1 + C) \cdot (1 + |\xi|)$,
  thereby proving the claim for the case $\alpha = 0$.

  \medskip{}

  If $\alpha > 0$ then $\alpha + \beta > 0$, and so
  $1 + |\xi| \asymp [\lambda(Q_i)]^{1 / (\alpha+\beta)} \asymp 1 + |\eta|$
  for arbitrary $\xi, \eta \in Q_i$, according to Condition~(2) in
  Definition~\ref{def:AlphaBetaCoverings}.
\end{proof}

We can now prove that any two $(\alpha,\beta)$ coverings are weakly equivalent.

\begin{thm}\label{thm:AlphaBetaCoveringsAreEquivalent}
  Let $\alpha, \beta \in [0,1]$ and let
  $\CalQ = (Q_i)_{i \in I}$ and $\CalP = (P_j)_{j \in J}$ be two
  $(\alpha,\beta)$ coverings of $\R^2$.

  Then $\CalQ$ and $\CalP$ are weakly equivalent.
  Furthermore, if $\beta \leq \alpha$, then
  \[
    \lambda \big( \, \overline{Q_i} - \overline{\strut P_j} \, \big)
    \lesssim (1 + |\xi|)^{\alpha + \beta}
    \asymp   \min \{ \lambda(Q_i), \, \lambda(P_j) \}
    \quad \text{for all} \quad i \in I \text{ and } j \in J
          \text{ such that } \emptyset \neq Q_i \cap P_j \ni \xi \,
  \]
  where we write $A - B = \{a - b \,:\, a \in A, \, b \in B\}$ for
  $A, B \subset \R^2$.
\end{thm}

\begin{proof}
  Let us fix, for each $j \in J$, some $\zeta_j \in P_j$.
  By symmetry, it suffices to prove that $\CalP$ is weakly subordinate to $\CalQ$.
  Therefore, defining
  \[
    I_j
    := \{
         i \in I
         \,:\,
         Q_i \cap P_j \neq \emptyset
       \}
    \qquad \text{for } j \in J,
  \]
  we have to find such an $N \in \N$ that $|I_j| \leq N$ for all $j \in J$.

  \medskip{}

  \noindent
  \textbf{Step 1:}
  Our first goal is to show that there are $C_1, C_2 > 0$ such that, for each
  $j \in J$, there is an interval $\Lambda_j \subset [0,\infty)$ of length
  $\lambda (\Lambda_j) \leq C_1 \cdot (1 + |\zeta_j|)^{\alpha}$ and such that
  \begin{equation}
    Q_i
    \subset \Big\{
              r \cdot e^{i \phi}
              \,:\,
              r \in \Lambda_j
              \, \text{ and }
              \min_{\omega \in \{0, \pi\}}
                |\phi - (\phi_j + \omega)| \leq C_2 \cdot (1 + r)^{\beta - 1}
            \Big\}
    =: \Omega_j
    \qquad \forall \, i \in I_j \,
    \label{eq:AlphaBetaCoveringEquivalenceMainInclusion}
  \end{equation}
  where $\phi_j \in \R$ is the angle associated with $P_j$ according
  to Condition~(4) in Definition~\ref{def:AlphaBetaCoverings}.

  To prove \eqref{eq:AlphaBetaCoveringEquivalenceMainInclusion}, we recall
  from Lemma~\ref{lem:AlphaBetaCoveringNormEstimate} that there is such a $C_3 \geq 1$ that $1 + |\xi| \leq C_3 \cdot (1 + |\eta|)$ for all $\xi, \eta \in Q_i$
  and for all $\xi, \eta \in P_j$.
  Furthermore, according to Condition~(3) in Definition~\ref{def:AlphaBetaCoverings},
  there is such a $C_4 > 0$ that
  \[
    \lambda (L_i) \leq C_4 \cdot (1 + |\xi|)^\alpha
    \quad \forall \, \xi \in Q_i
    \qquad \text{and} \quad
    \lambda(L_j) \leq C_4 \cdot (1 + |\xi|)^\alpha
    \quad \forall \, \xi \in P_j \, .
  \]

  Now, let $j \in J$ and $i \in I_j$ be arbitrary and let us fix some
  $\xi_i \in Q_i \cap P_j$.
  Then, $|\xi|, |\xi_i| \in L_i$ for arbitrary $\xi \in Q_i$.
  Since the Lebesgue measure $\lambda(L_i)$ is the length of the interval $L_i$
  and since $\xi_i , \zeta_j \in P_j$, this implies that
  \[
    \big|\, |\xi| - |\xi_i| \,\big|
    \leq \lambda(L_i)
    \leq C_4 \cdot (1 + |\xi_i|)^{\alpha}
    \leq C_3^\alpha C_4 \cdot (1 + |\zeta_j|)^{\alpha} \, .
  \]
  Likewise, $|\xi_i|, |\zeta_j| \in L_j$ and hence
  $\big|\, |\xi_i| - |\zeta_j| \,\big| \leq \lambda(L_j)
   \leq C_4 \cdot (1 + |\zeta_j|)^\alpha$.
  Combining these estimates results in
  $\big|\, |\xi| - |\zeta_j| \,\big|
   \leq \big|\, |\xi| - |\xi_i| \,\big| + \big|\, |\xi_i| - |\zeta_j| \,\big|
   \leq ( C_3^\alpha C_4 + C_4 ) \cdot (1 + |\zeta_j|)^{\alpha}$.
  Therefore, defining $C_1 := 2 C_4 \cdot (1 + C_3^\alpha)$, we see that
  \begin{equation}
    \begin{split}
      & \Lambda_j
        := [0,\infty)
           \cap \left[
                  |\zeta_j| - \frac{C_1}{2} \cdot (1 + |\zeta_j|)^\alpha , \,
                  |\zeta_j| + \frac{C_1}{2} \cdot (1 + |\zeta_j|)^\alpha
                \right] \\[0.2cm]
      \text{satisfies} \quad
      & \lambda(\Lambda_j) \leq C_1 \cdot (1 + |\zeta_j|)^{\alpha}
        \quad \text{and} \quad
        |\xi| \in \Lambda_j \quad \forall \, \xi \in Q_i \text{ and } i \in I_j
        \, .
    \end{split}
    \label{eq:AlphaBetaCoveringEquivalenceNormInterval}
  \end{equation}

  \medskip{}

  Having estimated the Euclidean norm of $\xi \in Q_i$ for $i \in I_j$,
  we now estimate the angle of $\xi$.
  To do so, let us choose a $C_5 > 0$ larger than the
  constant in \eqref{eq:AlphaBetaCoveringAngleControl} for both coverings
  $\CalQ$ and $\CalP$.
  Furthermore, for $\phi, \psi \in \R$, let us define
  \[
    d (\phi, \psi)
    := \min_{\omega \in \pi \Z}
         |\phi - (\psi + \omega)| \in [0,\infty) \, .
  \]
  It is not hard to see that the minimum is indeed attained and that
  \[
    d(\phi,\psi) = d(\psi,\phi)
    \quad \text{and} \quad
    d(\phi, \vartheta) \leq d(\phi, \psi) + d(\psi, \vartheta)
    \qquad \forall \, \phi,\psi,\vartheta \in \R \, .
  \]

  Now, let $j \in J$ and $i \in I_j$ be arbitrary.
  Then $Q_i \cap P_j$ is a non-empty open set,
  so that we can find a \emph{non-zero} $\xi_i \in Q_i \cap P_j$.
  According to \eqref{eq:AlphaBetaCoveringAngleControl},
  we can find $\phi,\psi \in \R$ and $\omega_1, \omega_2 \in \{0, \pi\}$
  such that
  \[
    |\xi_i| \cdot e^{i \psi} = \xi_i = |\xi_i| \cdot e^{i \phi}
    \qquad \text{and} \qquad
    \max
    \big\{
      |\phi - (\phi_i + \omega_1)| , \,
      |\psi - (\phi_j + \omega_2)|
    \big\}
    \leq C_5 \cdot (1 + |\xi_i|)^{\beta - 1} \, .
  \]
  On the one hand, since $|\xi_i| \neq 0$, this entails
  $e^{i \psi} = e^{i \phi}$ and hence $\phi = \psi + 2 \pi k$
  for some $k \in \Z$; therefore, $d (\phi,\psi) = 0$.
  On the other hand, the preceding estimate implies that
  $d(\phi, \phi_i) \leq C_5 \cdot (1 + |\xi_i|)^{\beta - 1}$
  and $d (\psi, \phi_j) \leq C_5 \cdot (1 + |\xi_i|)^{\beta - 1}$.

  Finally, from \eqref{eq:AlphaBetaCoveringAngleControl} we conclude that,
  for arbitrary $\xi \in Q_i$, there are $\vartheta \in \R$
  and $\omega \in \{0, \pi\}$ such that
  \[
    \xi = |\xi| \cdot e^{i \vartheta}
    \qquad \text{and} \qquad
    |\vartheta - (\phi_i + \omega)|
    \leq C_5 \cdot (1 + |\xi|)^{\beta - 1}
    \leq C_3^{1 - \beta} C_5 \cdot (1 + |\xi_i|)^{\beta - 1} \, .
  \]
  In particular, $d (\vartheta, \phi_i)
  \leq C_3^{1 - \beta} C_5 \cdot (1 + |\xi_i|)^{\beta - 1}$.
  By combining our observations, we see that
  \begin{equation}
    \begin{split}
      d(\vartheta, \phi_j)
      & \leq d(\vartheta, \phi_i)
             + d(\phi_i, \phi)
             + d(\phi, \psi)
             + d(\psi, \phi_j)
        \leq C_5 \cdot \big( C_3^{1 - \beta} + 1 + 0 + 1 \big)
             \cdot (1 + |\xi_i|)^{\beta - 1} \\
      & \leq C_3^{1 - \beta} C_5 \cdot \big( 2 + C_3^{1 - \beta} \big)
             \cdot (1 + |\xi|)^{\beta - 1}
        =:   C_2 \cdot (1 + |\xi|)^{\beta - 1} \, .
    \end{split}
    \label{eq:AlphaBetaCoveringProofAngleControl}
  \end{equation}
  Here, in the penultimate step, we noted that $\xi, \xi_i \in Q_i$ and hence
  $1 + |\xi| \leq C_3 \cdot (1 + |\xi_i|)$.
  The estimate~\eqref{eq:AlphaBetaCoveringProofAngleControl}
  shows that there is some $k \in \Z$ such that
  $|(\vartheta - k \pi) - \phi_j| \leq C_2 \cdot (1 + |\xi|)^{\beta - 1}$.
  Writing $k = 2 \ell + m$ with $m \in \{0,1\}$ and $\ell \in \Z$, we thus see
  that $|(\vartheta - 2 \ell \pi) - (\phi_j + m \pi)|
        \leq C_2 \cdot (1 + |\xi|)^{\beta - 1}$.

  All in all, defining $r := |\xi|$ and $\varphi := \vartheta - 2 \ell \pi$,
  we have shown that
  $\xi = r \cdot e^{i \, \vartheta} = r \cdot e^{i \, \varphi}$ where
  $\min_{\omega \in \{0, \pi\}}
    |\varphi - (\phi_j + \omega)| \leq C_2 \cdot (1 + r)^{\beta - 1}$
  and $r \in \Lambda_j$, according to
  \eqref{eq:AlphaBetaCoveringEquivalenceNormInterval}.
  Therefore, \eqref{eq:AlphaBetaCoveringEquivalenceMainInclusion} holds.

  \medskip{}

  \noindent
  \textbf{Step 2:} We now estimate the measure of the set $\Omega_j$ on the right-hand side of
  Equation~\eqref{eq:AlphaBetaCoveringEquivalenceMainInclusion}.
  To do so, let
  \[
    W_j
    := \{
         r \cdot e^{i \phi}
         \,:\,
         r \in \Lambda_j
         \text{ and }
         |\phi| \leq C_2 \cdot (1 + r)^{\beta - 1}
       \} \, .
  \]
  Then, since $-1 = e^{-i \pi}$, we see that
  \[
    - \{
        r \cdot e^{i \phi}
        \,:\,
        r \in \Lambda_j
        \text{ and }
        |\phi - \pi| \leq C_2 \cdot (1+r)^{\beta - 1}
      \}
    = \{
        r \cdot e^{i (\phi - \pi)}
        \,:\,
        r \in \Lambda_j
        \text{ and }
        |\phi - \pi| \leq C_2 \cdot (1+r)^{\beta - 1}
      \}
    = W_j \, .
  \]
  Therefore,
  \begin{equation}
    \begin{split}
      e^{- i \phi_j} \, \Omega_j
      & = \big\{
            r \cdot e^{i (\phi - \phi_j)}
            \,:\,
            r \in \Lambda_j
            \text{ and }
            \min_{\omega \in \{0, \pi\}}
              |\phi - (\phi_j + \omega)| \leq C_2 \cdot (1+r)^{\beta - 1}
          \big\} \\
      & = \bigcup_{\omega \in \{0,\pi\}}
            \big\{
              r \cdot e^{i \varphi}
              \,:\,
              r \in \Lambda_j
              \text{ and }
              |\varphi - \omega| \leq C_2 \cdot (1+r)^{\beta - 1}
            \big\}
        = W_j \cup (- W_j)
    \end{split}
    \label{eq:AlphaBetaCoveringOmegaJAsWJUnion}
  \end{equation}
  and hence
  \begin{equation}
    \lambda(\Omega_j) = \lambda(e^{-i \phi_j} \Omega_j)
    = \lambda\big(W_j \cup (-W_j) \big)
    \leq 2 \lambda (W_j) .
    \label{eq:UniversalityOmegaJMeasureEstimate}
  \end{equation}

  Now, let us define
  \begin{equation}
    a_j
    := \max \Big\{
                  0,
                  |\zeta_j| - \frac{C_1}{2} (1 + |\zeta_j|)^\alpha
            \Big\}
    \quad \text{and} \quad
    b_j := |\zeta_j| + \frac{C_1}{2} \cdot (1 + |\zeta_j|)^{\alpha} \, ,
    \label{eq:AlphaBetaCoveringLambdaJIntervalBoundaries}
  \end{equation}
  so that $\Lambda_j = [a_j, b_j]$
  according to \eqref{eq:AlphaBetaCoveringEquivalenceNormInterval}.
  Introducing polar coordinates allows us to write
  \[
    \lambda (W_j)
    = \int_{(a_j, b_j)}
        r \cdot
        \int_{-\pi}^{\pi}
          \Indicator_{W_j} (r \cdot e^{i \phi})
        \, d \phi
      \, dr \ ,
  \]
  where $\Indicator_{W_j}$ denotes the indicator function of the set $W_j$.
  On the other hand, if $\Indicator_{W_j} (r \cdot e^{i \phi}) = 1$ for some
  $r \in (a_j, b_j) \subset (0,\infty)$ and $\phi \in [-\pi,\pi]$, then
  $r \cdot e^{i \phi} = s \cdot e^{i \psi}$ for some $s \in \Lambda_j$ and
  $\psi \in \R$ satisfies $|\psi| \leq C_2 \cdot (1 + s)^{\beta - 1}$.
  Since $r \cdot e^{i \phi} = s \cdot e^{i \psi}$ where $r > 0$,
  we conclude that $s = r$ and $\phi - \psi \in 2 \pi \Z$.
  We claim that this implies $|\phi| \leq C_2 \cdot (1 + r)^{\beta - 1}$.
  As $C_2 \cdot (1 + r)^{\beta - 1} \geq \pi$ this is trivial. Therefore
  we shall assume that $C_2 \cdot (1 + r)^{\beta - 1} < \pi$.
  This implies that $|\psi| < \pi$ and hence $|\phi - \psi| \leq |\phi| + |\psi| < 2 \pi$,
  since $\phi \in [-\pi, \pi]$.
  Since $\phi - \psi \in 2 \pi \Z$, this entails $\phi = \psi$ and thus
  $|\phi| = |\psi| \leq C_2 \cdot (1 + r)^{\beta - 1}$ also in this case.
  In combination with the estimate
  $1 + b_j = 1 + |\zeta_j| + \tfrac{C_1}{2} (1 + |\zeta_j|)^{\alpha} \leq (1 + C_1) (1 + |\zeta_j|)$,
  these considerations show that
  \begin{equation}
    \begin{split}
      \lambda (W_j)
      & \leq \int_{(a_j, b_j)}
               r \cdot 2 C_2 \cdot (1 + r)^{\beta - 1}
             \, dr
        \leq 2 C_2
             \int_{(a_j, b_j)}
               (1 + r)^{\beta}
             \, d r \\
      & \leq 2 C_2 \cdot (1 + b_j)^{\beta} \cdot (b_j - a_j)
        \leq 2 C_2 \cdot (1 + C_1)^\beta \cdot (1 + |\zeta_j|)^\beta
                   \cdot C_1 \cdot (1 + |\zeta_j|)^\alpha \\
      & \leq 2 C_1 C_2 (1 + C_1) \cdot (1 + |\zeta_j|)^{\alpha + \beta}.
    \end{split}
    \label{eq:AlphaBetaWjMeasureEstimate}
  \end{equation}

  \medskip{}

  \noindent
  \textbf{Step 3:} We now show that there is such an $N > 0$ that
  $|I_j| \leq N$ for all $j \in J$, or, in other words,
  $\CalP$ is weakly subordinate to $\CalQ$. Since $\CalQ$ is an admissible covering, there is such an $N_0 \in \N$ that
  $\sum_{i \in I} \Indicator_{Q_i} \leq N_0$.
  Combined with \eqref{eq:AlphaBetaCoveringEquivalenceMainInclusion}, this
  implies $\sum_{i \in I_j} \Indicator_{Q_i} \leq N_0 \cdot \Indicator_{\Omega_j}$.
  Now, from Condition~(2) in Definition~\ref{def:AlphaBetaCoverings}
  we infer that there is such an $C_6 > 0$ that
  \[
    \lambda(Q_i)
    \geq C_6 \cdot (1 + |\xi_i|)^{\alpha + \beta}
    \geq C_6 C_3^{-(\alpha+\beta)} \cdot (1 + |\zeta_j|)^{\alpha+\beta}
    \qquad \forall \, i \in I_j
  \]
  where $\xi_i \in Q_i \cap P_j$ can be chosen arbitrarily.

  Therefore, if $\Gamma \subset I_j$ is an arbitrary finite subset, then
  \begin{align*}
    C_6 C_3^{-(\alpha+\beta)}
    & \cdot (1 + |\zeta_j|)^{\alpha+\beta}
      \cdot |\Gamma|
      \leq \sum_{i \in \Gamma} \lambda(Q_i)
      = \int_{\R^2} \sum_{i \in \Gamma} \Indicator_{Q_i}(\xi) \, d \xi
      \leq \int_{\R^2} N_0 \cdot \Indicator_{\Omega_j} (\xi) \, d \xi \\
    ({\scriptstyle{\text{Eqs. } \eqref{eq:UniversalityOmegaJMeasureEstimate}
                   \text{ and } \eqref{eq:AlphaBetaWjMeasureEstimate}}})
    & \leq 4 C_1 C_2 (1 + C_1) N_0 \cdot (1 + |\zeta_j|)^{\alpha + \beta} \,
  \end{align*}
  and hence $|\Gamma| \leq 4 C_1 C_2 (1 + C_1) C_3^{\alpha+\beta} N_0 / C_6 =: N$.
  Since $\Gamma \subset I_j$ was an arbitrary finite subset and the right-hand side
  of the last estimate is independent of $j \in J$,
  this completes the proof of the statement of this step.

  \medskip{}

  \noindent
  \textbf{Step 4:} We assume that $\beta \leq \alpha$ and estimate
  $\lambda \big(\, \overline{Q_i} - \overline{\strut P_j} \, \big)$.
  To do so, we first estimate the set $W_j$ introduced in Step~2.
  Let us recall the definition
  \eqref{eq:AlphaBetaCoveringLambdaJIntervalBoundaries} of $a_j$ and $b_j$,
  where we saw that $\Lambda_j = [a_j, b_j]$.
  Any $\xi \in W_j$ is given by $\xi = r \cdot e^{i \phi}$ for
  $r \in [a_j, b_j]$ and $|\phi| \leq C_2 \cdot (1 + r)^{\beta - 1}$.

  \smallskip{}

  Now, we note that $1 + r \leq 1 + b_j \leq (1 + C_1) \cdot (1 + |\zeta_j|)$ and so
  \[
    r \cdot (1 + r)^{\beta - 1} \leq (1 + r)^\beta
   \leq (1 + C_1)^\beta \cdot (1 + |\zeta_j|)^\beta
   \leq (1 + C_1) \cdot (1 + |\zeta_j|)^\alpha,
  \]
  where our assumption $\beta \leq \alpha$ was used in the last step.
  On the other hand, $\cos \phi \geq 1 - |\phi|$,
  since the cosine is $1$-Lipschitz with $\cos (0) = 1$.
  Therefore, defining $C_7 := (1 + C_1) (1 + C_2)$ and recalling the definition of $a_j$, we see that
  \[
    \xi_1
    = r \cdot \cos \phi
    \geq r \cdot (1 - |\phi|)
    \geq r - C_2 \cdot r \cdot (1 + r)^{\beta - 1}
    \geq a_j - (1 + C_1) C_2 \cdot (1 + |\zeta_j|)^\alpha
    \geq |\zeta_j| - C_7 \, (1 + |\zeta_j|)^{\alpha} \, .
  \]
  Conversely,
  \[
    \xi_1
    \leq r
    \leq b_j
    \leq |\zeta_j| + C_7 \cdot (1 + |\zeta_j|)^{\alpha} \, .
  \]
  Finally, noting that
  $|\sin \phi| \leq |\phi| \leq C_2 \cdot (1 + r)^{\beta - 1}$
  and $1 + r \leq 1 + b_j \leq (1 + C_1) \, (1 + |\zeta_j|)$,
  we conclude that
  \[
    |\xi_2|
    \leq r \cdot C_2 \cdot (1 + r)^{\beta - 1}
    \leq C_2 \cdot (1 + r)^{\beta}
    \leq (1 + C_1)^\beta C_2 \cdot (1 + |\zeta_j|)^\beta .
  \]
  Defining $C_8 := (1 + C_1)^\beta \, C_2$ and $\gamma_j := (|\zeta_j|, 0)^t \in \R^2$,
  we thus see that
  \begin{equation}
    \overline{\strut W_j} \subset \gamma_j + R_j
    \quad \text{where} \quad
    R_j = \Big( C_7 \cdot (1 + |\zeta_j|)^\alpha \cdot [-1,1] \Big)
          \times \Big( C_8 \cdot (1 + |\zeta_j|)^\beta \cdot [-1,1] \Big) \, .
    \label{eq:AlphaBetaCoveringWjMainInclusion}
  \end{equation}
  Using the symmetry of the rectangle $R_j$, we see that
  $\eps_1 \overline{W_j} + \eps_2 \overline{W_j}
   \subset \eps_1 \gamma_j + \eps_2 \gamma_j + (R_j + R_j)$
  for any $\eps_1, \eps_2 \in \{\pm 1\}$,
  from where we deduce that
  $\lambda(R_j + R_j) \leq C_9 \cdot (1 + |\zeta_j|)^{\alpha+\beta}$.
  Combining this with Equation~\eqref{eq:AlphaBetaCoveringOmegaJAsWJUnion} results in
  \begin{align*}
    \lambda \big(\, \overline{\Omega_j} - \overline{\Omega_j} \,\big)
    & = \lambda \big(\,
                  e^{-i\phi_j} \, \overline{\Omega_j}
                  - e^{-i\phi_j} \, \overline{\Omega_j}
                \,\big)
      \leq \lambda \big( [W_j \cup (-W_j)] - [W_j \cup (-W_j)] \big) \\[0.3cm]
    & \leq \sum_{\eps_1, \eps_2 \in \{\pm 1\}}
             \lambda
             \big(\,
               \eps_1 \, \overline{\strut W_j} + \eps_2 \, \overline{\strut W_j}
             \,\big)
      \leq 4 C_9 \cdot (1 + |\zeta_j|)^{\alpha + \beta} \, .
  \end{align*}
  Finally, since $\CalQ$ covers the whole $\R^2$,
  we see that $P_j \subset \bigcup_{i \in I_j} Q_i$, and hence
  $P_j \subset \Omega_j$, according to
  \eqref{eq:AlphaBetaCoveringEquivalenceMainInclusion}.
  Therefore,
  \(
    \lambda \big(\, \overline{\strut Q_i} - \overline{\strut P_j} \,\big)
    \leq \lambda \big( \, \overline{\strut \Omega_j} - \overline{\strut \Omega_j} \, \big)
    \leq 4 C_9 \cdot (1 + |\zeta_j|)^{\alpha + \beta}
  \)
  for all $j \in J$ and $i \in I_j$.
  Since $\lambda(Q_i) \asymp (1 + |\xi|)^{\alpha+\beta} \asymp \lambda(P_j)$
  for all $\xi \in Q_i \cap P_j$,
  according to Condition~(2) in Definition~\ref{def:AlphaBetaCoverings}
  and since $1 + |\xi| \asymp 1 + |\zeta_j|$ for such $\xi$, the proof
  is complete.
\end{proof}

We now show that the wave packet covering $\CalQ^{(\alpha,\beta)}$
fits the framework of $(\alpha,\beta)$ coverings.

\begin{lem}\label{lem:WavePacketCoveringIsAlphaBetaCovering}
  For any $0 \leq \beta \leq \alpha \leq 1$,
  the wave packet covering $\CalQ^{(\alpha,\beta)}$ is an
  $(\alpha,\beta)$ covering of $\R^2$.
\end{lem}

\begin{rem*}
  This shows in particular that almost structured $(\alpha,\beta)$ coverings exist
  for $0 \leq \beta \leq \alpha \leq 1$.

  By considering coverings that consist of sectors of rings, one can show that
  $(\alpha,\beta)$ coverings also exist for $\beta > \alpha$.
  It seems to be an open questions, however,
  whether such coverings can be chosen to be almost structured.
  To the best of our knowledge, no coverings
  that would satisfy this condition have been reported.
\end{rem*}

\begin{proof}
  As shown by Lemmas~\ref{lem:CoveringCovers} and \ref{lem:Admissibility},
  $\CalQ^{(\alpha,\beta)} = (Q_i)_{i \in I^{(\alpha,\beta)}}$
  is an admissible covering of $\R^2$.
  Thus, it remains to verify conditions (2)-(4) of
  Definition~\ref{def:AlphaBetaCoverings}.
  First, we shall do so only for $i = (j,m,\ell) \in I_0^{(\alpha,\beta)}$;
  the case $i = 0$ will be considered afterwards.

  Recall from Equation~\eqref{eq:AbsoluteValueEstimate} that
  $|\xi| \asymp 1 + |\xi| \asymp 2^j$ for all $\xi \in Q_i = T_i \, Q + b_i$.
  Therefore, we see that Condition~(2) in Definition~\ref{def:AlphaBetaCoverings} is satisfied;
  indeed,
  \[
    \lambda (Q_i)
    \asymp |\det T_i|
    =      |\det A_{j}|
    =      2^{\alpha j} \, 2^{\beta j}
    \asymp (1 + |\xi|)^{\alpha + \beta}
    \quad \forall \, i = (j,m,\ell) \in I_0^{(\alpha,\beta)}
                     \text{ and } \xi \in Q_i \, .
  \]

  Now, let us define
  \[
    L_i
    := [0,\infty) \cap \big[
                            2^{j-1} + 2^{\alpha j}(m - \eps) , \,
                            2^{j-1} + 2^{\alpha j} (m + 2 + 2 \eps)
                       \big]
    \quad \text{for} \quad
    i = (j,m,\ell) \in I_0^{(\alpha,\beta)} \, ,
  \]
  note that $\lambda(L_i) \leq (2+3 \eps) \cdot 2^{\alpha j} \lesssim (1 + |\xi|)^\alpha$
  for all $\xi \in Q_i$ and recall from \eqref{eq:AbsoluteValueEstimate}
  that $|\xi| \in L_i$ for all $\xi \in Q_i$.
  Hence, Condition~(3) in Definition~\ref{def:AlphaBetaCoverings} is satisfied.

  Finally, let $\phi_i := \Theta_{j,\ell}$ with $\Theta_{j,\ell}$ as in \eqref{eq:RotationMatrix}.
  With this, Equation~\eqref{eq:AngleControl} shows that --- for
  an arbitrary $\xi \in Q_i$ --- there is $\phi \in \R$ such that
  $\xi = |\xi| \cdot e^{i \phi}$ and
  \[
    \min_{\omega \in \{0, \pi \}}
      |\phi - (\phi_j + \omega)|
    \leq |\phi - \Theta_{j,\ell}|
    \leq 5 \cdot 2^{(\beta - 1) j}
    \asymp (1 + |\xi|)^{\beta - 1} \, ,
  \]
  as required in \eqref{eq:AlphaBetaCoveringAngleControl}.

  Since $Q_0 = B_4 (0)$, we have $|\xi| \in L_0 := [0,4]$
  and furthermore $1 + |\xi| \asymp 1$, $\lambda(Q_0) \asymp 1 \asymp (1 + |\xi|)^{\alpha+\beta}$,
  and $\lambda(L_0) \lesssim 1 \lesssim (1 + |\xi|)^\alpha$ for all $\xi \in Q_0$.
  Thus, it is not hard to see that Conditions (2)-(4) of
  Definition~\ref{def:AlphaBetaCoverings} remain valid for $i = 0$ as well,
  after adjusting the implicit constants.
\end{proof}

As our final result in this section, we shall now show that any two
$(\alpha,\beta)$ coverings induce the same family of decomposition spaces.
To make this result as general as possible, we first introduce slightly broader
classes of coverings than the almost-structured coverings
that we introduced in Section~\ref{sub:DecompositionDefinition}.

\begin{defn}\label{def:SemiStructuredCovering}
  (Definition~2.5 in \cite{DecompositionEmbeddings};
  inspired by \cite{BorupNielsenDecomposition})

\noindent A family $\CalQ = (Q_i)_{i \in I}$ is called a \textbf{semi-structured
  covering} of $\R^d$, if there is an \textbf{associated family}
  $(T_i \bullet + b_i)_{i \in I}$ of invertible affine-linear maps
  such that the following properties hold:
  \begin{enumerate}
    \item $\CalQ$ covers $\R^d$, that is, $\R^d = \bigcup_{i \in I} Q_i$;

    \item $\CalQ$ is \textbf{admissible}, that is, the sets $i^\ast$
          introduced in \eqref{eq:IndexClusterDefinition} have uniformly
          bounded cardinality;

    \item there is a family $(Q_i ')_{i \in I}$
          of non-empty open sets $Q_i ' \subset \R^d$ such that
          $Q_i = T_i \, Q_i ' + b_i$ for all $i \in I$,
          and such that the $Q_i '$ are uniformly bounded, that is,
          $\sup_{i \in I} \sup_{\xi \in Q_i '} |\xi| < \infty$; and


    \item there is such a $C > 0$ that $\|T_i^{-1}T_\ell\| \leq C$ for all such
          $i, \ell \in I$ that $Q_i \cap Q_\ell \neq \emptyset$.
  \end{enumerate}
\end{defn}

We also need to impose less restrictive conditions on the partitions of unity
than those imposed on regular partitions of unity.

\begin{defn}\label{def:BAPUs}
  (Definitions~3.5 and 3.6 in \cite{DecompositionEmbeddings};
  inspired by Definition~2.2 in \cite{DecompositionSpaces1}
  and by Definition~2 in \cite{BorupNielsenDecomposition})

\noindent Let $\CalQ = (Q_i)_{i \in I}$ be an admissible covering of $\R^d$.
  A family of functions $\Phi = (\varphi_i)_{i \in I}$ is called an
  $L^p$-\textbf{bounded admissible partition of unity} ($L^p$-BAPU) subordinate
  to $\CalQ$ for all $1 \leq p \leq \infty$, if
  \begin{enumerate}[label=(\arabic*)]
    \item $\varphi_i \in C_c^\infty (\R^d)$ with
          $\varphi_i (\xi) = 0$ for all $\xi \in \R^d \setminus Q_i$ and
          any $i \in I$;

    \item $\sum_{i \in I} \varphi_i \equiv 1$ on $\R^d$; and

    \item $\sup_{i \in I}
               \| \, \Fourier^{-1} \varphi_1 \, \|_{L^1}
            < \infty$.
  \end{enumerate}
  If there is such an $L^p$-BAPU, the covering $\CalQ$ is called an
  $L^p$-\textbf{decomposition covering} of $\R^d$, for all
  $1 \leq p \leq \infty$.

  \medskip{}

  Now, let $p \in (0,1)$ and let us assume that $\CalQ$ is semi-structured
  with associated family $(T_i \mybullet + b_i)_{i \in I}$.
  A family $\Phi = (\varphi_i)_{i \in I}$ is called an $L^p$-\textbf{BAPU}
  subordinate to $\CalQ$, if it is an $L^q$-BAPU for all $1 \leq q \leq \infty$,
  and
  \(
    \sup_{i \in I}
      |\det T_i|^{p^{-1} - 1} \cdot \|\Fourier^{-1} \varphi_i\|_{L^p}
    < \infty
  \).
  If there is such an $L^p$-BAPU, we shall say that $\CalQ$ is an
  $L^p$-\textbf{decomposition covering} of $\R^d$.
\end{defn}

Replacing the regular partition of unity by an $L^p$-BAPU, one can define
decomposition spaces associated with $L^p$-decomposition coverings,
proceeding exactly as in Definition~\ref{def:DecompositionSpace}.

Now, we state the main result of this section, whose proof is slightly deferred.

\begin{thm}\label{thm:AlphaBetaCoveringUniversality}
  Let $p,q \in (0,\infty]$, $\alpha, \beta \in [0,1]$, $\CalQ = (Q_i)_{i \in I}$ and $\CalP = (P_j)_{j \in J}$
  be two $(\alpha,\beta)$ coverings of $\R^2$,
  and $w = (w_i)_{i \in I}$ and $v = (v_j)_{j \in J}$ be $\CalQ$-moderate
  and $\CalP$-moderate, respectively.
  Let us assume that
  \[
    w_i \asymp v_j
    \quad \text{for any} \quad
    i \in I \text{ and } j \in J \text{ satisfying } Q_i \cap P_j \neq \emptyset
    \, .
  \]
  Then the following holds:
  \begin{enumerate}[label=(\arabic*)]
    \item If $p \in [1,\infty]$ and if $\CalQ,\CalP$ are $L^p$-decomposition
          coverings, then
          \[
            \DecompSp (\CalQ, L^p, \ell_w^q)
            = \DecompSp (\CalP, L^p, \ell_v^q)
            \quad \text{with equivalent quasi-norms.}
          \]

    \item If $p \in (0,1)$, if $\beta \leq \alpha$, and if $\CalQ,\CalP$
          are semi-structured $L^p$-decomposition coverings, then
          \[
            \DecompSp (\CalQ, L^p, \ell_w^q)
            = \DecompSp (\CalP, L^p, \ell_v^q)
            \quad \text{with equivalent quasi-norms.}
          \]
  \end{enumerate}
\end{thm}

The wave packet covering $\CalQ^{(\alpha,\beta)} = (Q_i)_{i \in I^{(\alpha,\beta)}}$
is an $(\alpha,\beta)$ covering and $\CalQ^{(\alpha,\beta)}$ is almost-structured,
so that it is a semi-structured $L^p$-decomposition covering for all $p \in (0,\infty]$
and therefore satisfies the assumptions of the preceding theorem. Moreover, Equation~\eqref{eq:WeightConsistencyCondition} and
Lemma~\ref{lem:AlphaBetaCoveringNormEstimate} imply that
$w_i^s \asymp (1 + |\xi|)^s \asymp (1 + |\xi_j|)^s$ for
$\xi \in Q_i \cap P_j$ and any fixed $\xi_j \in P_j$, if $\CalP = (P_j)_{j \in J}$ is any $(\alpha,\beta)$-covering. From this we immediately infer the following corollary.

\begin{cor}\label{cor:WavePacketSmoothnessUniversality}
  Let $p,q \in (0,\infty]$, $0 \leq \beta \leq \alpha \leq 1$
  and $\CalP = (P_j)_{j \in J}$ be an arbitrary $(\alpha,\beta)$ covering.
  Moreover let $s \in \R$ and the weight $v = (v_j)_{j \in J}$ be given by
  \[
    v_j = (1 + |\xi_j|)^s
    \quad \text{with} \quad
    \xi_j \in P_j \text{ arbitrary} \, .
  \]
  Then the following holds:
  \begin{enumerate}[label=(\arabic*)]
    \item If $p \in [1,\infty]$ and if $\CalP$ is an $L^p$-decomposition covering, then
          \[
            \DecompSp (\CalP, L^p, \ell_v^q)
            = \PacketSpace_s^{p,q}(\alpha,\beta)
            \quad \text{with equivalent quasi-norms.}
          \]

    \item If $p \in (0,1)$, and if $\CalP$ is a semi-structured
          $L^p$-decomposition covering, then
          \[
            \DecompSp (\CalP, L^p, \ell_v^q)
            = \PacketSpace_s^{p,q}(\alpha,\beta)
            \quad \text{with equivalent quasi-norms.}
          \]
  \end{enumerate}
\end{cor}

\begin{proof}[Proof of Theorem~\ref{thm:AlphaBetaCoveringUniversality}]
  The first statement is an immediate consequence of
  Lemma~6.11 in \cite{DecompositionEmbeddings}, since
  Theorem~\ref{thm:AlphaBetaCoveringsAreEquivalent} shows that $\CalQ$
  and $\CalP$ are weakly equivalent.

  \medskip{}

  For the second part, where $p \in (0,1)$, the definition of an
  $L^p$-decomposition covering shows that $\CalQ$ and $\CalP$ are
  semi-structured coverings, say, with associated families
  $(T_i \mybullet + b_i)_{i \in I}$ and $(S_j \mybullet + c_j)_{j \in J}$,
  respectively.
  From the uniform boundedness of the sets $Q_i^{\natural}$ and
  $P_j^{\natural}$ that satisfy $Q_i = T_i \, Q_i^{\natural} + b_i$ and
  $P_j = S_j \, P_j^{\natural} + c_j$, respectively,
  we conclude that $\lambda(Q_i) \lesssim |\det T_i|$ and
  $\lambda(P_j) \lesssim |\det S_j|$.

  Since $\CalQ$ and $\CalP$ are weakly equivalent and since
  $w_i \asymp v_j$ for any $i \in I$ and $j \in J$ such that
  $Q_i \cap P_j \neq \emptyset$, the remark that follows
  Lemma~6.11 in \cite{DecompositionEmbeddings} shows that
  $\DecompSp (\CalQ, L^p, \ell_w^q) = \DecompSp (\CalP, L^p, \ell_v^q)$
  provided that
  $\lambda\big(\, \overline{Q_i} - \overline{\strut P_j} \,\big)
   \lesssim \min \big\{ |\det T_i|, |\det S_j| \big\}$ for any $i \in I$ and
  $j \in J$ such that $Q_i \cap P_j \neq \emptyset$.
  However, the last part of Theorem~\ref{thm:AlphaBetaCoveringsAreEquivalent}
  shows indeed that
  \[
    \lambda\big(\, \overline{\strut Q_i} - \overline{\strut P_j} \,\big)
    \lesssim \min \big\{ \lambda(Q_i) \,, \lambda(P_j) \big\}
    \lesssim \min \{ |\det T_i|, |\det S_j| \}
  \]
  for any $i \in I$ and $j \in J$ satisfying $Q_i \cap P_j \neq \emptyset$.
\end{proof}

\section{Dilation invariance of the wave packet smoothness spaces}
\label{sec:DilationInvariance}

\noindent In this section, we shall use the universality of the wave packet coverings
to show that the wave packet smoothness spaces are invariant under
dilation with arbitrary invertible matrices $B \in \GL(\R^2)$.

To do so, we shall first show that it suffices to establish an embedding
between certain decomposition spaces to derive the dilation invariance.
To show this, let us fix $B \in \GL(\R^2)$ and define the dilation $f \circ B \in Z'$ of an
element $f$ of the reservoir $Z'$ as usual%
\footnote{This definition is natural, since
$\langle f \circ B, g \rangle = |\det B|^{-1} \, \langle f, g \circ B^{-1} \rangle$
if $f : \R^2 \to \CC$ is of moderate growth and $g \in \Schwartz(\R^2)$.}
as
\begin{equation}
  f \circ B : Z \to \CC,
              \varphi \mapsto |\det B|^{-1} \cdot \langle f, \varphi \circ B^{-1} \rangle_{Z'}.
  \label{eq:DilationDefinitionOnReservoir}
\end{equation}
We now show that this indeed defines a well-defined element $f \circ B \in Z'$
and, at the same time, compute the Fourier transform $\Fourier [f \circ B]$.
Namely, $f \circ B \in Z'$ if and only if
$\Fourier[f \circ B] = (f \circ B) \circ \Fourier \in \CalD'(\R^2)$.
This is indeed the case, since for $\psi \in C_c^\infty (\R^2)$
\[
  \langle \Fourier [f \circ B], \psi \rangle_{\CalD'}
  = \langle f \circ B, \widehat{\psi} \rangle_{Z'}
  = |\det B|^{-1} \, \langle f, \widehat{\psi} \circ B^{-1} \rangle_{Z'}
  = |\det B|^{-1} \, \langle \widehat{f}, \Fourier^{-1} [\widehat{\psi} \circ B^{-1}] \rangle_{\CalD'}
  = \langle \widehat{f}, \psi \circ B^t \rangle_{\CalD'},
\]
from where it is easy to see that the map
$\psi \mapsto \langle \widehat{f}, \psi \circ B^t \rangle_{\CalD'}$ is a well-defined distribution
on $\R^2$.

Now, let us fix $0 \leq \beta \leq \alpha \leq 1$ and a regular partition of unity
$(\varphi_i)_{i \in I}$ subordinate to the $(\alpha,\beta)$ wave packet covering
$\CalQ^{(\alpha,\beta)} = (Q_i)_{i \in I}$.
By using the computation of $\Fourier[f \circ B]$  and
recalling how the Fourier transform is computed for compactly supported distributions
(see Theorem~7.23 in \cite{RudinFA}), we conclude that
\begin{align*}
  \Fourier^{-1} (\varphi_i \cdot \Fourier[f \circ B]) (x)
  & = \langle
        \Fourier[f \circ B],
        e^{2 \pi i \langle x, \bullet \rangle} \cdot \varphi_i
      \rangle_{\CalD'}
    = \langle
        \widehat{f},
        e^{2 \pi i \langle x, B^t \bullet \rangle} \cdot (\varphi_i \circ B^t)
      \rangle_{\CalD'} \\
  & = \langle
        \widehat{f},
        e^{2 \pi i \langle B x, \bullet \rangle} \cdot (\varphi_i \circ B^t)
      \rangle_{\CalD'}
    = \Fourier^{-1} \big[ (\varphi_i \circ B^t) \cdot \widehat{f} \,\, \big] \, (B x)
\end{align*}
for all $i \in I$ and $x \in \R^2$.
Therefore, for all $i \in I$,
\begin{equation}
  \|\Fourier^{-1} (\varphi_i \cdot \Fourier[f \circ B])\|_{L^p}
  = \|
      \Fourier^{-1} \big[ (\varphi_i \circ B^t) \cdot \widehat{f} \big] \circ B
    \|_{L^p}
  = |\det B|^{-1/p}
    \cdot \|
            \Fourier^{-1} \big[ (\varphi_i \circ B^t) \cdot \widehat{f} \,\, \big]
          \|_{L^p} .
  \label{eq:DilationInvarianceMainIdentity}
\end{equation}

It is straightforward to verify that the family
$B^{-t} \CalQ^{(\alpha,\beta)} := (B^{-t} Q_i)_{i \in I}$ is an almost structured admissible covering of
$\R^2$ with associated family $(B^{-t} T_i \bullet + B^{-t} b_i)_{i \in I}$
where $T_i$ and $b_i$ are as defined in Lemma~\ref{lem:CoveringAlmostStructured}.
Likewise, it follows directly from the definitions that $(\varphi_i \circ B^t)_{i \in I}$
is a regular partition of unity subordinate to $B^{-t} \CalQ^{(\alpha,\beta)}$.
Furthermore, $B^{-t} Q_i \cap B^{-t} Q_j \neq \emptyset$ if and only if $Q_i \cap Q_j \neq \emptyset$,
so that the weight $w^s$ introduced in Lemma~\ref{lem:WavePacketWeight}
is $B^{-t} \CalQ^{(\alpha,\beta)}$-moderate.
Thus, the decomposition spaces $\DecompSp(B^{-t} \CalQ^{(\alpha,\beta)},L^p,\ell_{w^s}^q)$
are well-defined and Equation~\eqref{eq:DilationInvarianceMainIdentity} implies that, for $f \in Z'$,
\[
  \|f \circ B\|_{\PacketSpace_s^{p,q}(\alpha,\beta)} \!
  \asymp \Big\|
           \big(
             \|\Fourier^{-1} (\varphi_i \cdot \widehat{f \circ B})\|_{L^p}
           \big)_{i \in I}
         \Big\|_{\ell_{w^s}^q} \!\!\!
  \asymp \Big\|
           \big(
             \|\Fourier^{-1} ( [\varphi_i \circ B^t] \cdot \widehat{f} \, )\|_{L^p}
           \big)_{i \in I}
         \Big\|_{\ell_{w^s}^q} \!\!\!
  \asymp \! \|f\|_{\DecompSp(B^{-t} \CalQ^{(\alpha,\beta)}, L^p, \ell_{w^s}^q)} .
\]
If we knew that
$\DecompSp(B^{-t} \CalQ^{(\alpha,\beta)}, L^p, \ell_{w^s}^q) = \PacketSpace_s^{p,q} (\alpha,\beta)$
with equivalent quasi-norms, then the preceding equation would show that every
$f \in \PacketSpace_s^{p,q} (\alpha,\beta) = \DecompSp(B^{-t} \CalQ^{(\alpha,\beta)}, L^p, \ell_{w^s}^q)$
satisfies ${f \circ B \in \PacketSpace_s^{p,q}(\alpha,\beta)}$ and
\(
  \|f \circ B\|_{\PacketSpace_s^{p,q}(\alpha,\beta)}
  \asymp \|f\|_{\DecompSp(B^{-t} \CalQ^{(\alpha,\beta)}, L^p, \ell_{w^s}^q)}
  \asymp \|f\|_{\PacketSpace_s^{p,q} (\alpha,\beta)}.
\)
Thus, if we establish the identity
${\DecompSp(B^{-t} \CalQ^{(\alpha,\beta)}, L^p, \ell_{w^s}^q) = \PacketSpace_s^{p,q} (\alpha,\beta)}$,
we shall prove the following theorem concerning the dilation invariance of the
wave packet spaces $\PacketSpace_s^{p,q}(\alpha,\beta)$:

\begin{thm}\label{thm:DilationInvariance}
  Let $0 \leq \beta \leq \alpha \leq 1$, $p,q \in (0,\infty]$, $s \in \R$, and $B \in \GL(\R^2)$.
  With $f \circ B$ as defined in Equation~\eqref{eq:DilationDefinitionOnReservoir}, the linear map
  \[
    \PacketSpace_s^{p,q}(\alpha,\beta) \to \PacketSpace_s^{p,q}(\alpha,\beta), f \mapsto f \circ B
  \]
  is well-defined and bounded.
\end{thm}


Corollary~\ref{cor:WavePacketSmoothnessUniversality} shows that in order to prove
${\DecompSp(B^{-t} \CalQ^{(\alpha,\beta)}, L^p, \ell_{w^s}^q) = \PacketSpace_s^{p,q} (\alpha,\beta)}$,
it suffices to show that $B^{-t} \CalQ^{(\alpha,\beta)}$ is an $(\alpha,\beta)$-covering
and that $w_i^s \asymp (1 + |\xi_i|)^s$ where $\xi_i \in B^{-t} Q_i$ is chosen arbitrarily.
The latter point is straightforward:
For arbitrary $\eta_i \in Q_i$ we have $\xi_i = B^{-t} \eta_i$,
whence Equation~\eqref{eq:WeightConsistencyCondition} shows that $(1 + |\eta_i|)^s \asymp w_i^s$;
but since $B$ is invertible, $|\xi_i| \asymp |B^{t} \xi_i| = |\eta_i|$,
which finally shows that $(1 + |\xi_i|)^s \asymp (1 + |\eta_i|)^s \asymp w_i^s$
for all $i \in I$ and $\xi_i \in B^{-t} Q_i$.

To prove that $B^{-t} \CalQ^{(\alpha,\beta)}$ is indeed an $(\alpha,\beta)$-covering,
we start with a geometric lemma which shows that if two vectors $x,y$ are
pointing essentially in the same direction (the angle between the vectors is small),
then the angle between the vectors $B x, B y$ is also small where $B \in \GL(\R^2)$ is fixed.
As in Section~\ref{sec:Universality}, we identify the complex number $x + i y$ with
the vector $(x,y)^t \in \R^2$ throughout this section.

\begin{lem}\label{lem:LinearMapEffectOnAngle}
  Let $B \in \GL(\R^2)$ and $\varphi_0 \in \R$ be arbitrary.
  Then there is an angle $\psi_0 \in \R$ with the following property:

  For arbitrary $r \geq 0$ and $\varphi \in \R$ there are $r' \geq 0$ and $\varphi ' \in \R$
  such that $B(r \, e^{i \varphi}) = r' \, e^{i \varphi'}$ and
  \[
    \min_{\omega \in \{0,\pi\}}
      |\varphi ' - (\psi_0 + \omega)|
    \leq \frac{\pi}{2} \cdot \|B^{-1}\|^2 \cdot |\det B|
         \cdot \min_{\theta \in \{0,\pi\}}
                 |\varphi - (\varphi_0 + \theta)| .
  \]
\end{lem}

\begin{rem*}
  Roughly speaking, the lemma shows that if $x = \pm \, r \, e^{i \varphi}$
  with $\varphi \approx \varphi_0$, then also $B x = \pm \, s \, e^{i \varphi '}$
  with $\varphi' \approx \psi_0$.
  Here, the angle $\psi_0$ is the one satisfying $B ( e^{i \varphi_0}) = r_B \, e^{i \psi_0}$
  for some $r_B > 0$.
\end{rem*}

\begin{proof}
  \textbf{Step 1:} In this step, we prove that
  \begin{equation}
    \frac{2}{\pi} \cdot \min_{\ell \in \Z}
                          |x - \pi \ell|
    \leq |\sin x|
    \leq \min_{\ell \in \Z}
           |x - \pi \ell|
    \qquad \forall x \in \R .
    \label{eq:SineEquivalentToIntegerDistance}
  \end{equation}
  To prove the upper bound, simply note that $x \mapsto |\sin x|$ is a $1$-Lipschitz
  and $\pi$-periodic function and that $\sin(0) = 0$, whence
  $|\sin x| = |\sin(x - \pi \ell) - \sin(0)| \leq |x - \pi \ell|$ for all $\ell \in \Z$.

  To prove the lower bound, note that the functions $x \mapsto |\sin x|$
  and $x \mapsto \min_{\ell \in \Z} |x - \pi \ell|$ are both $\pi$-periodic,
  so that it is enough to prove the claim for $x \in [-\tfrac{\pi}{2},\tfrac{\pi}{2}]$.
  On this interval, $|\sin x| = \sin (|x|)$ and $\min_{\ell \in \Z} |x - \pi \ell| = |x|$,
  so that it is enough to show that $\tfrac{2}{\pi} x \leq \sin x$ for $x \in [0, \tfrac{\pi}{2}]$.
  To see this, note that the sine is concave on $[0, \tfrac{\pi}{2}]$, since
  $\sin'' = - \sin \leq 0$ on this interval.
  Now, if $x \in [0, \tfrac{\pi}{2}]$, then $t := \tfrac{2}{\pi} x \in [0,1]$ and
  $x = (1-t) \cdot 0 + t \cdot \tfrac{\pi}{2}$, so that
  \(
    \sin(x)
    \geq (1-t) \sin(0) + t \sin(\pi/2)
    = t
    = \frac{2}{\pi} \cdot x
  \),
  completing the proof of Equation~\eqref{eq:SineEquivalentToIntegerDistance}.

  \medskip{}

  \textbf{Step 2:} In this step, we show that if $r, s, r_B, s_B > 0$
  and $\varphi, \psi, \varphi_B, \psi_B \in \R$ are arbitrary with
  $B (r e^{i \varphi}) = r_B \, e^{i \varphi_B}$ and $B(s e^{i \psi}) = s_B \, e^{i \psi_B}$, then
  \begin{equation}
    |\sin (\psi_B - \varphi_B)|
    \leq \|B^{-1}\|^2 \cdot |\det B| \cdot |\sin (\psi - \varphi)| .
    \label{eq:SineOfAngleDifferenceLinearTransform}
  \end{equation}

  To see this, let $\rho, \sigma > 0$ and $\theta, \omega \in \R$ be arbitrary and define
  \[
    \Delta_{\rho,\sigma}^{\theta,\omega}
    := A_{\rho,\sigma}^{\theta,\omega} (\Delta)
    \quad \text{where} \quad
    A_{\rho,\sigma}^{\theta,\omega}
    := \big( \rho \, e^{i \theta} \,\big|\, \sigma \, e^{i \omega} \big)
     = \Big(
         \begin{matrix}
           \rho \cos \theta & \sigma \cos \omega \\
           \rho \sin \theta & \sigma \sin \omega
         \end{matrix}
       \Big) \in \R^{2 \times 2}
  \]
  and $\Delta := \{ (\mu, \nu) \in [0,1]^2 \colon \mu + \nu \leq 1 \}$.
  Since the Lebesgue measure of $\Delta$ is given by $\lambda(\Delta) = \tfrac{1}{2}$, we see that
  \begin{equation}
    \lambda(\Delta_{\rho,\sigma}^{\theta,\omega})
    = \frac{1}{2} \cdot |\det A_{\rho,\sigma}^{\theta,\omega}|
    = \frac{\rho \sigma}{2} \cdot |\cos \theta \sin \omega - \sin \theta \cos \omega|
    = \frac{\rho \sigma}{2} \cdot |\sin (\omega - \theta)|.
    \label{eq:SineLaw}
  \end{equation}

  Now, note that
  \[
    A_{r_B, s_B}^{\varphi_B, \psi_B}
    = \big( r_B \, e^{i \varphi_B} \big| s_B \, e^{i \psi_B} \big)
    = \big( B (r \, e^{i \varphi}) \big| B (s \, e^{i \psi}) \big)
    = B A_{r, s}^{\varphi,\psi}
  \]
  and hence
  \(
    \lambda(\Delta_{r_B, s_B}^{\varphi_B, \psi_B})
    = \lambda(B \Delta_{r,s}^{\varphi,\psi})
    = |\det B| \cdot \lambda(\Delta_{r,s}^{\varphi,\psi})
  \).
  Therefore, by applying Equation~\eqref{eq:SineLaw} twice, we conclude that
  \begin{equation}
    \frac{r_B \cdot s_B}{2} \cdot |\sin (\psi_B - \varphi_B)|
    = \lambda(\Delta_{r_B, s_B}^{\varphi_B,\psi_B})
    = |\det B| \cdot \lambda(\Delta_{r,s}^{\varphi,\psi})
    = |\det B| \cdot \frac{r \cdot s}{2} \cdot |\sin(\psi - \varphi)|.
    \label{eq:SineOfAngleDifferenceAlmostDone}
  \end{equation}
  Now, note that $r = |B^{-1} x_B| \leq \|B^{-1}\| \cdot |x_B| = \|B^{-1}\| \cdot r_B$ since of $x_B := B(r \, e^{i \varphi}) = r_B \, e^{i \varphi_B}$
  and similarly $s \leq \|B^{-1}\| \cdot s_B$.
  Introducing these estimates into Equation~\eqref{eq:SineOfAngleDifferenceAlmostDone},
  resluts in \eqref{eq:SineOfAngleDifferenceLinearTransform}.

  \medskip{}

  \textbf{Step 3:} In this step, we complete the proof.
  With $\varphi_0$ as in the statement of the lemma, let $x_0 := e^{i \varphi_0}$
  and write $B x_0 = r_0 \, e^{i \psi_0}$ for suitable $r_0 > 0$ and $\psi_0 \in \R$.

  Let $r \geq 0$ and $\varphi \in \R$ be arbitrary.
  If $r = 0$, we can simply choose $\varphi' = \psi_0$,
  so that the claimed estimate trivially holds.
  Thus, we shall assume that $r > 0$, define $x := r \, e^{i \varphi}$
  and write $B x = r_B \, e^{i \varphi_B}$ for suitable $r_B > 0$ and $\varphi_B \in \R$.
  Let us choose $k \in \Z$ such that
  $|\varphi_B - (\psi_0 + \pi k)| = \min_{\ell \in \Z} |\varphi_B - (\psi_0 + \pi \ell)|$.
  Finally, we write $k = 2n + \kappa$ with $n \in \Z$ and $\kappa \in \{0,1\}$ and define
  $\varphi' := \varphi_B - 2 \pi n$.
  Then we have $B (r e^{i \varphi}) = r_B \, e^{i \varphi_B} = r_B \, e^{i \varphi'}$ and
  \begin{align*}
    \min_{\omega \in \{0,\pi\}}
      |\varphi' - (\psi_0 + \omega)|
    & \leq |(\varphi_B - 2 \pi n) - (\psi_0 + \pi \kappa)|
      =    |\varphi_B - (\psi_0 + \pi k)|
      =    \min_{\ell \in \Z}
             |\varphi_B - (\psi_0 + \pi \ell)| \\
    ({\scriptstyle{\text{Step 1 and then Step 2}}})
    & \leq \frac{\pi}{2} \, |\sin(\varphi_B - \psi_0)|
      \leq \frac{\pi}{2} \cdot \|B^{-1}\|^2 \cdot |\det B| \cdot |\sin(\varphi - \varphi_0)| \\
    ({\scriptstyle{\text{Step 1}}})
    & \leq \frac{\pi}{2} \cdot \|B^{-1}\|^2 \cdot |\det B|
           \cdot \min_{\theta \in \{0,\pi\}}
                   |\varphi - (\varphi_0 + \theta)|.
    \qedhere
  \end{align*}
\end{proof}

Using Lemma~\ref{lem:LinearMapEffectOnAngle}, we can now show that $B^{-t} \CalQ^{(\alpha,\beta)}$
is an $(\alpha,\beta)$-covering.
As explained above, this will also prove Theorem~\ref{thm:DilationInvariance}.

\begin{lem}\label{lem:DilatedCoveringIsAlphaBetaCovering}
  Let $0 \leq \beta \leq \alpha \leq 1$ and $B \in \GL(\R^2)$.
  Then $B^{-t} \CalQ^{(\alpha,\beta)}$ is an $(\alpha,\beta)$-covering of $\R^2$.
\end{lem}

\begin{proof}
  It is straightforward to verify that $B^{-t} \CalQ^{(\alpha,\beta)} = (B^{-t} Q_i)_{i \in I}$
  is an almost structured admissible covering of $\R^2$
  with associated family $(B^{-t} T_i \bullet + B^{-t} b_i)_{i \in I}$.
  We now verify the remaining three properties in Definition~\ref{def:AlphaBetaCoverings}.

  \begin{enumerate}
    \item[(2)] Since $\CalQ^{(\alpha,\beta)}$ is an $(\alpha,\beta)$-covering
               (see Lemma~\ref{lem:WavePacketCoveringIsAlphaBetaCovering}),
               $\lambda(Q_i) \asymp (1 + |\eta|)^{\alpha + \beta}$
               for all $i \in I$ and $\eta \in Q_i$.
               Since $B$ is invertible and since $\eta = B^t \xi \in Q_i$ for $\xi \in B^{-t} Q_i$,
               we conclude that, for all $i \in I$ and $\xi \in B^{-t} Q_i$,
               \(
                 \lambda(B^{-t} Q_i)
                 = |\det B|^{-1} \lambda(Q_i)
                 \asymp (1 + |B^t \xi|)^{\alpha + \beta}
                 \asymp (1 + |\xi|)^{\alpha + \beta}
               \).

    \item[(3)] 
               Let us define
               \(
                 L_i
                 := [0,\infty) \cap [\gamma_i, \lambda_i]
                 := [0,\infty) \cap \big[
                                      |B^{-t} b_i| - 4 \|B^{-1}\| \cdot 2^{\alpha j},
                                      |B^{-t} b_i| + 4 \|B^{-1}\| \cdot 2^{\alpha j}
                                    \big]
               \)
               for $i = (j,m,\ell) \in I_0^{(\alpha,\beta)}$.
               It is not hard to see that $|\eta| \leq 4$ for all
               $\eta \in Q = (-\eps,1+\eps) \times (-1-\eps, 1+\eps)$
               and $\|T_i\| = \|A_j\| = 2^{\alpha j}$.
               Therefore, we see for $\xi \in B^{-t} Q_i = B^{-t} (T_i Q + b_i)$ that
               \[
                      \gamma_i
                 =    |B^{-t} b_i| - 4 \, \|B^{-t}\| \cdot \|T_i\|
                 \leq |\xi|
                 \leq |B^{-t} b_i| + 4 \, \|B^{-t}\| \cdot \|T_i\|
                 =    \lambda_i
               \]
               and hence $|\xi| \in L_i$ for all $\xi \in B^{-t} Q_i$ and $i \in I_0^{(\alpha,\beta)}$.
               Finally, Equation~\eqref{eq:AbsoluteValueEstimate} indicates that, for $\xi \in B^{-t} Q_i$,
               $2^{j-2} \leq |B^t \xi| \leq \|B\| \cdot |\xi|$ and hence
               \(
                 \lambda(L_i)
                 \leq 8 \|B^{-1}\| \cdot 2^{\alpha j}
                 \lesssim (\|B\|^{-1} \cdot 2^{j-2})^\alpha
                 \leq (1 + |\xi|)^\alpha
               \)
               for all $\xi \in Q_i$ and $i \in I_0^{(\alpha,\beta)}$.

               For the remaining case $i = 0$, define $L_0 := [0, 4 \, \|B^{-1}\|]$
               and note that, if $\xi \in B^{-t} Q_0 = B^{-t} B_4 (0)$ is arbitrary,
               then $|\xi| \in L_0$ and $\lambda(L_0) \lesssim 1 \leq (1 + |\xi|)^\alpha$.

    \item[(4)] Lemma~\ref{lem:WavePacketCoveringIsAlphaBetaCovering} shows that $\CalQ^{(\alpha,\beta)}$
               is an $(\alpha,\beta)$-covering.
               Thus, there is such a constant $C > 0$ and, for each $i \in I$, such an angle $\phi_i \in \R$
               that, for each $\eta \in Q_i$, there is another angle $\phi \in \R$
               that satisfies $\eta = |\eta| \cdot e^{i \phi}$ and
               \(
                 \min_{\omega \in \{0,\pi\}} |\phi - (\phi_i + \omega)|
                 \leq C \cdot (1 + |\eta|)^{\beta - 1}
               \).
               Therefore, Lemma~\ref{lem:LinearMapEffectOnAngle} (applied to $B^{-t}$ instead of $B$)
               yields a constant $C' = C'(B) > 0$ and for each $i \in I$ an angle $\psi_i \in \R$
               such that for each $\xi = B^{-t} \eta \in B^{-t} Q_i$ there is another angle $\psi \in \R$
               satisfying $\xi = B^{-t} \eta = |\xi| \cdot e^{i \psi}$ and
               \[
                 \min_{\omega \in \{0,\pi\}}
                   |\psi - (\psi_i + \omega)|
                 \leq C' \cdot \min_{\theta \in \{0,\pi\}}
                                 |\phi - (\phi_i + \theta)|
                 \leq C' C \cdot (1 + |\eta|)^{\beta - 1}
                 \leq C' C \cdot (1 + \|B^{-1}\|)^{1 - \beta} \cdot (1 + |\xi|)^{\beta - 1},
               \]
               since $1 + |\xi| \leq (1 + \|B^{-1}\|) \cdot (1 + |\eta|)$.
               \qedhere
  \end{enumerate}
\end{proof}

\section{Constructing Banach frame decompositions of the wave packet smoothness spaces}
\label{sec:SeriesConvergence}

\noindent In this section, we introduce so-called \emph{wave packet systems}
(as informally introduced in \cite{DemanetWaveAtoms}).
Furthermore, we prove that if the generators of such a system are nice enough,
then the system constitutes an atomic decomposition and
a Banach frame for certain wave packet smoothness spaces.
To do so, we shall concentrate on the case where $\alpha < 1$;
results for the case $\alpha = 1$ can be found in \cite{AlphaShearletSparsity}.

\begin{defn}\label{def:WavePacketSystem}
  Let $0 \leq \beta \leq \alpha < 1$ $\delta > 0$
  and $\varphi, \gamma \in L^1 (\R^2)$.
  Moreover let $I_0 = I_0^{(\alpha,\beta)}$ and $I = I^{(\alpha,\beta)}$
  be as defined in Equation~\eqref{eq:IndexSet}.
  Finally, let $T_i := R_{j,\ell} \, A_{j}$ and $b_i = R_{j,\ell} \, c_{j,m}$
  for $i = (j,m,\ell) \in I_0$ and $T_0 := \identity$ and $b_0 := 0$.
  Here, $A_{j}, R_{j,\ell}$, and $c_{j,m}$ are as introduced in Definition~\ref{defn:CoveringSets}.
  \smallskip{}

  The family $(L_{\delta \cdot T_i^{-t} \cdot k} \, \gamma^{[i]})_{i \in I, k \in \Z^2}$ with
  \[
    \gamma^{[i]}
    := \begin{cases}
         |\det A_{j}|^{1/2}
         \cdot M_{R_{j,\ell}
                  \left(
                    \begin{smallmatrix}
                      2^{j-1} + 2^{\alpha j} \, m \\ 0
                    \end{smallmatrix}
                  \right)}
                  \Big[
                   \gamma \circ A_{j}^{t} \circ R_{j,\ell}^t
                  \Big] \, ,
         & \text{if } i = (j,m,\ell) \in I_0 \, , \\
         \varphi \, , & \text{if } i = 0
       \end{cases}
  \]
  is called the $(\alpha,\beta)$-\emph{wave packet system with generators}
  $\varphi, \gamma$ \emph{and sampling density} $\delta > 0$.
\end{defn}

\begin{rem*}
  If $\gamma \in C_c^\infty (\R^2)$, then
  $\supp \gamma^{[i]} = R_{j,\ell}^{-t} A_{j}^{-t} \supp \gamma$ is
  essentially a rectangle that can be obtained by rotating the axis-aligned
  rectangle of the dimensions $2^{-\alpha j} \times 2^{-\beta j}$
  with its centre at the origin through the angle
  $\Theta_{j,\ell} \asymp \frac{2\pi}{N} \cdot 2^{(\beta - 1) j} \cdot \ell$.

  Furthermore, 
  $\gamma \circ A_{j}^{t} \circ R_{j,\ell}^t$
  oscillates roughly at frequencies $\xi$ such that
  $|\xi| \lesssim \|A_{j}\| = 2^{\alpha j} \ll 2^j$.
  Since $|2^{j-1} + 2^{\alpha j} \, m| \geq 2^{j-1} \asymp 2^j$,
  the behaviour of $\gamma^{[i]}$ will be largely determined by the oscillations
  at frequency $\approx 2^j$ in direction
  $R_{j,\ell} \left(\begin{smallmatrix}1 \\ 0\end{smallmatrix}\right)$
  that are caused by applying the modulation
  $M_{R_{j,\ell}
           \left(
             \begin{smallmatrix}
               2^{j-1} + 2^{\alpha j} \, m \\ 0
             \end{smallmatrix}
           \right)}$
  to the function $\gamma \circ A_{j}^{t} \circ R_{j,\ell}^t$.

  Therefore, our $(\alpha,\beta)$-wave packet systems are similar
  to those introduced in \cite[Section 1.1]{DemanetWaveAtoms}, with one notable
  exception. Namely, the frequency support of the elements of our wave packet
  systems are not symmetric with respect to the origin,
  while those in \cite{DemanetWaveAtoms} are.
\end{rem*}

To formulate our discretisation results for the wave packet
smoothness spaces, we need the following definition.

\begin{defn}\label{def:WavePacketCoefficientSpace}
  Let $0 \leq \beta \leq \alpha < 1$, $p,q \in (0,\infty]$ and
  $s \in \R$.
  The set of sequences of complex numbers
  \[
    \mathcal{C}_{s}^{p,q}(\alpha,\beta)
    := \left\{
         c = (c_k^{(i)})_{i \in I^{(\alpha,\beta)}, k \in \Z^2}
         \in \CC^{I^{(\alpha,\beta)} \times \Z^2}
         \,:\,
         \|c\|_{\mathcal{C}_s^{p,q}} < \infty
       \right\} \,
  \]
  where
  \[
    \|(c_k^{(i)})_{i \in I^{(\alpha,\beta)}, k \in \Z^2}\|_{\mathcal{C}_s^{p,q}}
    := \left\|
         \left(
           w_i^{s + (\alpha + \beta) \cdot (\frac{1}{2} - \frac{1}{p})}
           \cdot \big\| (c_k^{(i)})_{k \in \Z^2} \big\|_{\ell^p}
         \right)_{i \in I^{(\alpha,\beta)}}
       \right\|_{\ell^q}
    \in [0,\infty]
  \]
  is called the \emph{coefficient space} associated with the
  wave packet smoothness space $\PacketSpace_s^{p,q}(\alpha,\beta)$.
\end{defn}


The main goal of the present section is to prove the following two theorems that provide condition on
the functions $\varphi,\gamma$ that --- if satisfied --- guarantee that the
$(\alpha,\beta)$ wave packet system generated by $\varphi,\gamma$ forms an atomic decomposition
or a Banach frame for the wave packet smoothness spaces $\PacketSpace_s^{p,q} (\alpha,\beta)$.

\begin{thm}\label{thm:WavePacketAtomicDecomposition}
  Let $0 \leq \beta \leq \alpha < 1$, $s_0 \geq 0$ and $\omega, p_0, q_0 \in (0,1]$.
  Moreover let $\varphi,\gamma \in L^1 (\R^2)$ be such that:
  \begin{enumerate}
    \item $\widehat{\varphi}, \widehat{\gamma} \in C^\infty (\R^2)$ and all
          partial derivatives of $\widehat{\varphi}$ and $\widehat{\gamma}$ are
          of polynomial growth at most;

    \item $\widehat{\varphi}(\xi) \neq 0$ for all $\xi \in \overline{B}_4 (0)$
          and $\widehat{\gamma} (\xi) \neq 0$ for all
          $\xi \in [-\eps, 1+\eps] \times [-1-\eps, 1+\eps]$; and

    \item $\sup_{x \in \R^2} (1 + |x|)^{1 + 2 \cdot p_0^{-1}} |\varphi (x)| < \infty$ and
          $\sup_{x \in \R^2} (1 + |x|)^{1 + 2 \cdot p_0^{-1}} |\gamma (x)| < \infty$.

  \end{enumerate}
  Finally, let us assume that there is such a constant $C > 0$ that
  \begin{equation}
    \begin{split}
      & \big|
          (\partial^\theta \, \widehat{\varphi}) (\xi)
        \big|
        \leq C \cdot (1 + |\xi|)^{-(4 + \kappa_0)}
               \cdot (1 + |\xi_1|)^{- \kappa_1}
               \cdot (1 + |\xi_2|)^{-\kappa_2} \\
      \text{and} \quad
      & \big|
          (\partial^\theta \, \widehat{\gamma}) (\xi)
        \big|
        \leq C \cdot (1 + |\xi|)^{-(4 + \kappa_0)}
               \cdot (1 + |\xi_1|)^{- \kappa_1}
               \cdot (1 + |\xi_2|)^{-\kappa_2}
    \end{split}
    \label{eq:WavePacketAtomicDecompositionCondition}
  \end{equation}
  for all $\xi \in \R^2$ and all such $\theta \in \N_0^2$ that $|\theta| \leq N_0$ where
  \[
    N_0 := \lceil p_0^{-1} (2 + \omega) \rceil \, ,
    \qquad
    \kappa_1 := \frac{2}{\min \{p_0, q_0\}} \, ,
    \qquad
    \kappa_2 := 3 + \frac{2}{(1 - \beta)\min \{p_0, q_0\}} + \frac{5}{p_0} \, ,
  \]
  and
  \[
    \kappa_0
    := (1 - \alpha)^{-1}
       \cdot \left(
               3
               + s_0
               + \frac{3 + \alpha}{\min\{p_0,q_0\}}
               + \frac{6\alpha + 9 \beta}{p_0}
               + \frac{2\beta}{(1-\beta) \min \{p_0,q_0\}}
             \right)
    \, .
  \]
  Then there is such a
  $\delta_0 = \delta_0(\alpha,\beta,\omega,p_0,q_0,s_0,\varphi,\gamma,C) > 0$,
  that, for each $\delta \in (0, \delta_0]$, $p \in [p_0,\infty]$,
  $q \in [q_0, \infty]$ and $s \in [-s_0,s_0]$,
  the $(\alpha,\beta)$-wave packet system with
  generators $\varphi,\gamma$ and sampling density $\delta$ is an atomic
  decomposition for $\PacketSpace_s^{p,q} (\alpha,\beta)$ with coefficient space
  $\mathcal{C}_s^{p,q}(\alpha,\beta)$.
\end{thm}

\begin{rem*}

  Note that Conditions (1) and (3) in the theorem are satisfied as long as
  $\varphi, \gamma \in C_c(\R^2)$.
  Furthermore, Condition~\eqref{eq:WavePacketAtomicDecompositionCondition} is satisfied if
  $\varphi, \gamma \in C_c^k (\R^2)$ where
  $k \geq 4 + \kappa_0 + \kappa_1 + \kappa_2$.
  This last observation is due to the fact that
  $\partial^\alpha \widehat{f} (\xi)
   = [\Fourier ((-2 \pi i x)^\alpha \cdot f)](\xi)$
  where $(-2 \pi i x)^\alpha \cdot f \in C_c^k (\R^2)$ if
  $f \in C_c^k (\R^2)$.
\end{rem*}


\begin{thm}\label{thm:WavePacketBanachFrames}
  Let $0 \leq \beta \leq \alpha < 1$, $s_0 \geq 0$ and $\omega, p_0, q_0 \in (0,1]$.
  Moreover let $\varphi,\gamma \in L^1 (\R^2)$ be such that Properties (1)--(2) from
  Theorem~\ref{thm:WavePacketAtomicDecomposition} are satisfied and that
  \begin{enumerate}
    \item[(3')] $\varphi, \gamma \in C^1 (\R^2)$ and
                $\partial^\mu \varphi, \partial^\mu \gamma \in L^1 (\R^2) \cap L^\infty (\R^2)$
                for all $\mu \in \N_0^2$ with $|\mu| \leq 1$.
  \end{enumerate}
  %
  %
  Finally, let us assume that there is such a constant $C > 0$ that
  \begin{equation}
    \begin{split}
      & \big|
          (\partial^\nu \, \widehat{\partial^\mu \varphi}) (\xi)
        \big|
        \leq C \cdot (1 + |\xi|)^{- \kappa_0 '}
               \cdot (1 + |\xi_1|)^{- \kappa_1}
               \cdot (1 + |\xi_2|)^{-\kappa_2} \\
      \text{and} \quad
      & \big|
          (\partial^\nu \, \widehat{\partial^\mu \gamma}) (\xi)
        \big|
        \leq C \cdot (1 + |\xi|)^{- \kappa_0 '}
               \cdot (1 + |\xi_1|)^{- \kappa_1}
               \cdot (1 + |\xi_2|)^{-\kappa_2}
    \end{split}
    \label{eq:WavePacketBanachFrameCondition}
  \end{equation}
  for all $\xi \in \R^2$ and all such $\mu, \nu \in \N_0^2$ that
  $|\nu| \leq N_0$ and $|\mu| \leq 1$ where $N_0, \kappa_1, \kappa_2$ are as in
  Theorem~\ref{thm:WavePacketAtomicDecomposition} and where
  \[
    \kappa_0'
    := (1 - \alpha)^{-1}
       \cdot \left(
               3
               + s_0
               + \frac{1 + \alpha}{\min\{p_0,q_0\}}
               + \frac{5 \alpha + 10 \beta}{p_0}
               + \frac{2}{(1-\beta)\min\{p_0,q_0\}}
             \right)
    \, .
  \]
  Then there exists
  $\delta_0 = \delta_0(\alpha,\beta,\omega,p_0,q_0,s_0,\varphi,\gamma,C) > 0$
  such that, for any $\delta \in (0, \delta_0]$, $p \in [p_0,\infty]$,
  ${q \in [q_0, \infty]}$ and $s \in [-s_0,s_0]$,
  the $(\alpha,\beta)$-wave packet system with generators $\varphi,\gamma$
  and sampling density $\delta > 0$ is a Banach frame for
  $\PacketSpace_s^{p,q} (\alpha,\beta)$ with coefficient space
  $\mathcal{C}_s^{p,q}(\alpha,\beta)$.

  More specifically, there is such a bounded \emph{analysis map}
  $A^{(\delta)} : \PacketSpace_s^{p,q}(\alpha,\beta) \to \mathcal{C}_s^{p,q}$
  that
  \begin{equation}
    A^{(\delta)} f
    = \left(
        \big\langle
          f \mid L_{\delta \cdot T_i^{-t} \cdot k} \, \gamma^{[i]}
        \big\rangle_{L^2}
      \right)_{i \in I, k \in \Z^2}
    \qquad \forall \, f \in L^2 (\R^2) \cap \PacketSpace_s^{p,q}(\alpha,\beta)
    \, ,
    \label{eq:WavePacketAnalysisOperatorExplicit}
  \end{equation}
  and $A^{(\delta)}$ has a bounded linear left inverse whose action is
  independent of the choice of $p \in [p_0,\infty]$, $q \in [q_0, \infty]$ and $s \in [-s_0, s_0]$.
\end{thm}

\begin{rem*}
  Conditions (1) and (3') are satisfied as long as $\varphi,\gamma \in C_c^{1}(\R^2)$.
  Furthermore, Condition~\eqref{eq:WavePacketBanachFrameCondition} is satisfied
  if $\varphi, \gamma \in C_c^k (\R^2)$ where $k \geq 1 + \kappa_0 ' + \kappa_1 + \kappa_2$.
\end{rem*}

%

To prove Theorems~\ref{thm:WavePacketAtomicDecomposition}
and \ref{thm:WavePacketBanachFrames}, we shall use
Theorems~\ref{thm:StructuredBFDAtomicDecomposition} and
\ref{thm:StructuredBFDBanachFrame}, respectively.
To do so, we need to show that the constants
$K_1, K_2$ and $C_1, C_2$, as they were introduced in those theorems,
are finite.
Given that these constants differ only marginally from each other, we shall
slightly reformulate this problem and by doing so prove that
$K_1, K_2, C_1, C_2$ are all finite at once.
Specifically, let us define
\begin{equation}
  \psi :
  \R^2 \to (0,\infty),
  \xi \mapsto (1 + |\xi|)^{-\kappa_0}
              \cdot (1 + |\xi_1|)^{-\kappa_1}
              \cdot (1 + |\xi_2|)^{- \kappa_2} \, ,
  \label{eq:PsiDefinition}
\end{equation}
where $\kappa_0, \kappa_1, \kappa_2 \geq 0$ are fixed, but arbitrary.
In the remainder of this section, we shall establish conditions on
$\kappa_0, \kappa_1, \kappa_2$ and $B$ so that
\begin{equation}
  \sup_{i' \in I} \,\,
    \sum_{i \in I}
      M_{i,i'}^{(1)}
  \leq B < \infty
  \qquad \text{and} \qquad
  \sup_{i \in I} \,\,
    \sum_{i' \in I}
      M_{i,i'}^{(1)}
  \leq B < \infty \, ,
  \label{eq:TargetEstimate}
\end{equation}
where, for $i = (j,m,\ell) \in I_0$ and $i' = (j',m',\ell') \in I_0$, the
quantity $M_{i,i'}^{(1)}$ is given by
\begin{equation}
  M_{i,i'}^{(1)}
  := 2^{(j-j')s}
     \cdot \big( 1 + \|T_i^{-1} T_{i'}\| \big)^\sigma
     \cdot \Big(
             \fint_{Q_{i'}}
               \psi \big( T_i^{-1} (\xi - b_i) \big)
             \, d\xi
           \Big)^\tau
     \, ,
  \label{eq:MainTerm}
\end{equation}
with $s \in \R$ and $\sigma, \tau \in (0,\infty)$ fixed, but arbitrary
and with $T_i,b_i$ 
as in Lemma \ref{lem:CoveringAlmostStructured}.
Here, we used the notation
$\fint_M f(x) \, dx := \frac{1}{\lambda(M)} \int_M f(x) \, dx$ where $\lambda$
denotes the Lebesgue measure.

Similarly, we define
\begin{equation}
  \begin{split}
    & M_{0, i'}^{(1)}
      := 2^{-j' s}
         \cdot \big( 1 + \|T_0^{-1} T_{i'}\| \big)^{\sigma}
         \cdot \Big(
                 \fint_{Q_{i'}}
                   \psi \big( T_{0}^{-1} (\xi - b_0) \big)
                 \, d \xi
               \Big)^{\tau} \, , \\
    & M_{i, 0}^{(1)}
      := 2^{j s}
         \cdot \big( 1 + \|T_i^{-1} T_{0}\| \big)^{\sigma}
         \cdot \Big(
                 \fint_{Q_{0}}
                   \psi \big( T_{i}^{-1} (\xi - b_i) \big)
                 \, d \xi
               \Big)^{\tau} \, , \\
    \text{and} \quad
    & M_{0, 0}^{(1)}
      := \big( 1 + \| T_0^{-1} T_{0} \| \big)^{\sigma}
         \cdot \Big(
                 \fint_{Q_{0}}
                   \psi \big( T_{0}^{-1} (\xi - b_0) \big)
                 \, d \xi
               \Big)^{\tau} \,
  \end{split}
  \label{eq:MainTermLowPass}
\end{equation}
where, again $i = (j,m,\ell) \in I_0$ and $i' = (j',m',\ell') \in I_0$.
Precisely, we shall prove the following theorem, from which we shall then deduce
Theorems~\ref{thm:WavePacketAtomicDecomposition} and \ref{thm:WavePacketBanachFrames}.

\begin{thm}\label{thm:MainTechnicalResult}
  If $\sigma \geq 0$, $\tau > 0$,
  \begin{equation}
    \kappa_1 \geq \max \left\{2 , \tfrac{2}{\tau} \right\}
    \quad \text{and} \quad
    \kappa_2 \geq \max \left\{
                         1 + \tfrac{\sigma + 2}{\tau}
                         \,,\,
                         2 + \kappa_2^{(0)}
                       \right\}
    \quad \text{with} \quad
    \kappa_2^{(0)} := \tfrac{2 - \alpha - \beta + \sigma (\alpha - \beta)}
                           {(1-\beta) \tau}
    \label{eq:KappaConditions1}
  \end{equation}
  and if
  \begin{equation}
    \kappa_0
    \geq \max \left\{
                \tfrac{3 + |s| + \tau + \alpha + (\alpha + \beta) \sigma}
                     {(1-\alpha)\tau}
                \, , \,
                \tfrac{2 + |s|
                      + \tau \beta \, \kappa_2^{(0)}
                      + \max \{\tau, \sigma \} (\alpha + \beta)
                      }
                     {(1-\alpha)\tau}
              \right\}
    \label{eq:KappaConditions2}
  \end{equation}
  and
  \begin{equation}
    B := N
         \cdot 2^{37
                  + 8 \sigma
                  + \tau (10 + 5 \kappa_0 + 6 \kappa_2^{(0)} + \kappa_2)} ,
    \label{eq:BChoice}
  \end{equation}
  then \eqref{eq:TargetEstimate} holds.
\end{thm}


\paragraph{Structure of the section}

To prove Theorem~\ref{thm:MainTechnicalResult}, we first estimate the
different terms occurring in \eqref{eq:TargetEstimate}; this will be done in
Subsection~\ref{sub:GeneralEstimates}.
In Subsections~\ref{sub:SummingOverI} and \ref{sub:SummingOverIPrime},
we estimate respectively the former and the latter series in
\eqref{eq:TargetEstimate} for $i,i' \in I_0$.
In Subsection~\ref{sub:LowPassContribution}, we estimate these series
for $i' = 0$ or $i = 0$, respectively.
Finally, in Subsection~\ref{sub:WavePacketDecompositionProof}, we prove
Theorems~\ref{thm:WavePacketAtomicDecomposition}
and \ref{thm:WavePacketBanachFrames}
by using Theorem~\ref{thm:MainTechnicalResult}.

\medskip{}

For $i \in I_0$ or $i' \in I_0$, we shall use the convention
$i = (j,m,\ell)$ and $i' = (j',m',\ell')$ throughout this section,
without mentioning it explicitly.

\subsection[Estimating various terms occurring in (\ref*{eq:MainTerm})]
           {Estimating various terms occurring in $M_{i,i'}^{(1)}$ for $i,i' \in I_0$}
\label{sub:GeneralEstimates}

Let $i = (j,m,\ell) \in I_0$ and $i' = (j',m',\ell') \in I_0$ and let us define
\begin{equation}
  i_0 := (j,\ell) \, ,
  \quad
  i_0 ' := (j',\ell') \, ,
  \quad
  i_\ast := (j,m) \, ,
  \quad
  i_\ast ' := (j',m') \, ,
  \quad \text{and} \quad
  \vartheta_{i_0,i_0 '}^{(0)}
  := \Theta_{j',\ell'} - \Theta_{j,\ell} \, .
  \label{eq:AngleDifference}
\end{equation}
Since $0 \leq \Theta_{j,\ell} < 3 \pi$ according to \eqref{eq:ThetaJEstimate},
$\vartheta_{i_0,i_0'}^{(0)} \in (-3\pi, 3\pi)$.
Thus there exists ${k = k_{i_0,i_0 '} \in \{-1,0,1,2\}}$ such that
\begin{equation}
  \vartheta_{i_0, i_0 '}
  := \vartheta_{i_0, i_0 '}^{(0)} + 2 \pi k
  \in [0, 2 \pi) \, .
  \label{eq:NormalizedAngleDifference}
\end{equation}

With the change of variables $\eta = T_i^{-1}(\xi - b_i)$, we obtain
\begin{equation}
  M_{i,i'}^{(1)}
  = M_{i,i'}^{(2)}
  := 2^{(j-j')s}
     \cdot \big( 1 + \|T_i^{-1} T_{i'} \| \big)^\sigma
     \cdot \Big( \fint_{\Omega_{i,i'}} \psi(\eta) \, d\eta \Big)^\tau \,
  \quad \text{where} \quad
  \Omega_{i,i'} := T_i^{-1} (Q_{i'} - b_i) \, .
  \label{eq:MainDomain}
\end{equation}
To estimate the integral in \eqref{eq:MainDomain},
we first estimate the Euclidean norm $|\xi|$ of $\xi \in \Omega_{i,i'}$.

\begin{lem}\label{lem:SeriesEuclideanNormEstimate}
  Let $i,i' \in I_0$ and $\xi \in \Omega_{i,i'}$, then:
  \begin{enumerate}[label=\alph*)]
    \item \label{enu:SeriesNormEstimate}
          $1 + |\xi| \geq 2^{-5} \cdot 2^{(1-\alpha) |j-j'|}$; and

    \item \label{enu:SeriesNormEstimateLargeScaleDifference} if $|j-j'| \geq 5$,
          then $|\xi| \geq 2^{-5} \cdot 2^{\max\{j,j'\} - \alpha j}
                      \geq 2^{-5} \cdot 2^{(1-\alpha) \max \{j,j'\}}$.
  \end{enumerate}
\end{lem}

\begin{proof}
  Since $\xi \in \Omega_{i,i'}$, we see that $\eta := T_i \xi + b_i \in Q_{i'}$.
  Thus, Equation~\eqref{eq:AbsoluteValueEstimate} implies
  $2^{j' - 2} < |\eta| < 2^{j' + 3}$.
  Furthermore, for $c_{j,m}$ as defined in \eqref{eq:DilationMatrixAndTranslation},
  \[
         2^{j-1}
    \leq |c_{j,m}|
    =    |b_i|
    =    |c_{j,m}|
    \leq 2^{j-1} + (1 + 2^{(1-\alpha)j - 1}) \cdot 2^{\alpha j}
    = 2^{j-1} + 2^{\alpha j} + 2^{j-1}
    \leq 2^{j+1} \, .
  \]
  Finally, we see that
  \(
    \|T_i\|
    = \|A_{j}\|
    = \max \{2^{\alpha j}, 2^{\beta j} \}
    = 2^{\alpha j}
  \)
  since $\beta \leq \alpha$ and $A_j = \mathrm{diag}(2^{\alpha j}, 2^{\beta j})$.

  After this preparation, we first prove
  Part~\ref{enu:SeriesNormEstimateLargeScaleDifference},
  so that we are working under the assumption $|j-j'| \geq 5$.
  Thus, there are two cases:
  \smallskip{}

  \noindent
  \emph{Case 1:} $j' \geq j + 5$.
  Then $2^{j' - 2} < |\eta| = |T_i \xi + b_i| \leq 2^{\alpha j} |\xi| + 2^{j+1}$
  and hence
  \[
    \qquad \qquad
    |\xi|
    \geq 2^{-\alpha j}
         \cdot \big( 2^{j' - 2} - 2^{j+1} \big)
    \geq 2^{-\alpha j}
         \cdot \big( 2^{j' - 2} - 2^{j'-3} \big)
    =    2^{-3} \cdot 2^{j' - \alpha j}
    \geq 2^{-5} \cdot 2^{\max\{j,j'\} - \alpha j} \, .
  \]

  \noindent
  \emph{Case 2:} $j \geq j' + 5$. Then
  \[
    2^{\alpha j} \cdot |\xi|
    \geq |T_i \xi|
    =    |\eta - b_i|
    \geq |b_i| - |\eta|
    \geq 2^{j-1} - 2^{j'+3}
    \geq 2^{j-1} - 2^{j-2}
    =    2^{j-2} \, ,
  \]
  and therefore
  $|\xi| \geq 2^{-2} \cdot 2^{j - \alpha j}
         \geq 2^{-5} \cdot 2^{\max\{j,j'\} - \alpha j}$.

  Combining these two cases proves the first estimate in
  Part~\ref{enu:SeriesNormEstimateLargeScaleDifference}.
  To prove the second, we note that
  $\max\{j,j'\} - \alpha j \geq \max\{j,j'\} - \alpha \max\{j,j'\}
   = (1-\alpha) \max \{j,j'\}$.

  \medskip{}

  Finally, to prove Part~\ref{enu:SeriesNormEstimate}, we note that
  $2^{-5} \, 2^{(1-\alpha) |j-j'|} \leq 1 \leq 1 + |\xi|$ if $|j-j'| \leq 5$.
  If otherwise $|j-j'| \geq 5$, then
  Part~\ref{enu:SeriesNormEstimateLargeScaleDifference} implies that
  \[
    1 + |\xi| \geq |\xi| \geq 2^{-5} \cdot 2^{(1-\alpha) \max\{j,j'\}}
             \geq 2^{-5} \cdot 2^{(1-\alpha)
                  \cdot (\max\{j,j'\} - \min\{j,j'\})}
             =    2^{-5} \cdot 2^{(1-\alpha) |j-j'|} \, .
     \qedhere
  \]
\end{proof}

To prove \eqref{eq:TargetEstimate}, we shall rely on the following two lemmata.

\begin{lem}\label{lem:AnneLemma}
  (see Lemma~C.1 and ensuing remark in \cite{AlphaShearletSparsity})

\noindent Let $N \in [0,\infty)$, $\tau,\beta_0,L \in (0,\infty)$ and $M \in \R$
  and let $f : \R \to \CC$ be measurable and such that
  \[
    |f(x)| \leq C \cdot (1+|x|)^{-(N+2)/\tau}
    \quad \forall \, x \in \R \, .
  \]
  Then,
  \[
    \sum_{k \in \Z}
      \big( 1 + |\beta_0 k + M| \big)^N
      \left(
        \int_{\beta_0 k + M - L}^{\beta_0 k + M + L}
          |f(x)|
        \, dx
      \right)^\tau
    \leq 2^{3+\tau+N}
         \cdot 10^{N+3}
         \cdot C^\tau
         \cdot L^\tau
         \cdot (1 + L^N)
         \cdot \left(1 + \frac{L+1}{\beta_0}\right) .
  \]
\end{lem}

\begin{lem}\label{lem:MainLemma}
  Let $N, \gamma \in [0,\infty)$, $L,\tau \in (0,\infty)$, $M \in \R$ and $0 < \beta_1 \leq \beta_2$.
  Furthermore, let $f: \R \to \CC$ be measurable and assume that there is
  \begin{equation}
    q \geq 1 + \frac{N+2}{\tau}
    \quad
    \text{such that}
    \qquad
    \left|f(x)\right|
    \leq C_{0} \cdot \left(1+|x|\right)^{-q}
    \qquad \text{for all } x \in \R \, ,
    \label{eq:MainLemmaFDecayAssumption}
  \end{equation}
  and let $C := 2^{4 + N + \tau + \tau q}$.
  Then
  \begin{align*}
      \sum_{\substack{k \in \Z \text{ with}\\
                      k + M \geq 0}} \!
        \left[
          \left|\gamma \!\cdot\! (k+M)\right|^{N}
          \!\cdot\!
          \left(\!
            \int_{\beta_{1} \cdot (k+M)-L}^{\beta_{2} \cdot (k+M) + L}
              \!\! \left|f(x)\right|
            \, d x
          \right)^{\!\!\tau}\,
        \right] 
      \!\leq\! 
               C \!\cdot\! \left(\frac{\beta_{2}}{\beta_{1}}\right)^{\!\!\tau} \!
               \left(\frac{\gamma}{\beta_{1}}\right)^{\!\!N}
               \!\!\cdot \, C_{0}^{\tau} \cdot \left(1\!+\!L^{\tau + N}\right)
               \cdot \left(\!1\!+\!\frac{L \!+\! 1}{\beta_{1}}\right)\!.
    \qedhere
  \end{align*}
\end{lem}

\begin{proof}
  See Appendix~\ref{sec:MainLemmaProof}.
\end{proof}

We shall also need a slightly reformulated form of this lemma.

\begin{cor}\label{cor:MainLemmaNegativVersion}
  Let $N, \gamma \in [0,\infty)$, $L,\tau \in (0,\infty)$, $M \in \R$, and $0 < \beta_2 \leq \beta_1$.
  Furthermore, let $f : \R \to \CC$ be measurable and such that
  \eqref{eq:MainLemmaFDecayAssumption} is satisfied.

  Then, with the same constant $C$ as in Lemma~\ref{lem:MainLemma},
  \[
    \sum_{\substack{k \in \Z \text{ with} \\ k + M \leq 0}}
    \left[
      |\gamma \!\cdot\! (k+M)|^N
      \!\cdot\! \left(
                  \int_{\beta_1 \cdot (k+M) - L}^{\beta_2 \cdot (k+M) + L}
                    |f(x)|
                  \, dx
                \right)^{\!\!\tau} \,
    \right]
    \leq C
         \cdot \left(\frac{\beta_1}{\beta_2}\right)^\tau \!
               \left(\frac{\gamma}{\beta_2}\right)^N \!\!\!
         \cdot \, C_0^\tau
         \cdot (1 + L^{\tau + N}) \,
         \cdot \left(\! 1 \!+\! \frac{L \!+\! 1}{\beta_2}\right) .
  \]
\end{cor}
\begin{proof}
  See Appendix~\ref{sec:MainLemmaProof}.
\end{proof}


To use the preceding results for proving \eqref{eq:TargetEstimate}, we have to
verify that the domain of integration $\Omega_{i,i'}$ in
\eqref{eq:MainDomain} is contained in a Cartesian product of intervals
that comply with the requirements of the lemmata.
To this end, let us define
\begin{equation}
  R_{i_0,i_0 '}
  :=  \left(
        \begin{matrix}
              \cos \vartheta_{i_0, i_0 '}^{(0)}
          & - \sin \vartheta_{i_0, i_0 '}^{(0)} \\[0.2cm]
              \sin \vartheta_{i_0, i_0 '}^{(0)}
          &   \cos \vartheta_{i_0, i_0 '}^{(0)}
        \end{matrix}
      \right)
  =  \left(
       \begin{matrix}
         \cos \vartheta_{i_0, i_0 '} & - \sin \vartheta_{i_0, i_0 '} \\
         \sin \vartheta_{i_0, i_0 '} & \cos \vartheta_{i_0, i_0 '}
       \end{matrix}
     \right) \, ,
  \label{eq:TransitionRotation}
\end{equation}
and recall that $T_i = R_{j,\ell} A_{j}$, to conclude that
\begin{equation}
  \begin{split}
    \Omega_{i,i'}
    = T_i^{-1} (Q_{i'} - b_i)
    & = A_{j}^{-1}
        R_{j,\ell}^{-1}
        \big(
          R_{j',\ell'} (A_{j'} Q
          + c_{j',m'})
          - R_{j,\ell} \, c_{j,m}
        \big) \\
    & = A_{j}^{-1} \big( R_{i_0,i_0'} Q_{j',m',0} - c_{j,m} \big) \, .
  \end{split}
  \label{eq:SeriesMainDomainEstimate1}
\end{equation}
Next, we investigate the set $Q_{j',m',0}$ and, in doing so, introduce
a convenient notation:

\begin{lem}\label{lem:InclusionLemma1}
  For $i' = (j',m',\ell') \in I_0$, let us define
  \[
    x_{i_\ast '}^{-} := 2^{j' - 1} + (m' - \eps) \cdot 2^{\alpha j'},
    \qquad
    x_{i_\ast '}^{+} := 2^{j' - 1} + (m' + 1 + \eps) \cdot 2^{\alpha j'} \, ,
    \quad \text{and} \quad
    y_{j'} := 2^{\beta j' + 1} \, .
  \]
  Then
  \[
    Q_{j',m',0}
    \subset [x_{i_\ast '}^{-} \,,\, x_{i_\ast '}^{+}]
            \times [-y_{j'} \,,\, y_{j'}]
    \qquad \text{and} \qquad
    \frac{1}{4} \cdot 2^{j'}
    \leq x_{i_\ast '}^{-}
    \leq x_{i_\ast '}^{+}
    \leq 4 \cdot 2^{j'} \, .
  \]
\end{lem}

\begin{proof}
  Since
  \(
    Q =       (-\eps, 1+\eps) \times (-1-\eps, 1+\eps)
      \subset (-\eps,1+\eps) \times (-2 \,,\, 2)\vphantom{\sum_j}
  \)
  and $A_{j'} = \mathrm{diag}(2^{\alpha j'}, 2^{\beta j'})$,
  we conclude that
  \begin{align*}
    Q_{j',m',0}
    = A_{j'} Q + c_{j',m'}
    & = [
         2^{\alpha j'} \cdot (-\eps, 1+\eps)
         \times 2^{\beta j'} \cdot (- 2, 2)
        ]
        + \left(
            \begin{smallmatrix}
              \raisebox{0.2cm}{$\scriptstyle 2^{j' - 1} \,+\, m' \cdot 2^{\alpha j'}$} \\
              0
            \end{smallmatrix}
          \right) \\
    & \subset [2^{j' - 1} + 2^{\alpha j'} (m' - \eps),
               2^{j'-1} + 2^{\alpha j'} (m' + 1 + \eps)]
              \times [-2^{\beta j' + 1}, 2^{\beta j' + 1}] \, ,
  \end{align*}
  which completes the proof of the first claim of the lemma.

  To prove the second claim, let us recall that
  $m' \leq m_{j'}^{\max} \leq 1 + 2^{(1-\alpha) j' - 1}$, whence
  \[
    x_{i_\ast '}^{+}
    \leq 2^{j' - 1} + (2^{(1-\alpha) j' - 1} + 2 + \eps) \cdot 2^{\alpha j'}
    = 2^{j'} + (2 + \eps) \cdot 2^{\alpha j'}
    \leq 4 \cdot 2^{j'} \, .
  \]
  Clearly, $x_{i_\ast '}^{-} \leq x_{i_\ast '}^{+}$.
  Finally, since $m' \geq 0$ and $\eps \leq 1/32 \leq 1/4$, we see that
  \[
    x_{i_\ast '}^{-}
    \geq 2^{j' - 1} - \eps \cdot 2^{\alpha j'}
    \geq 2^{j' - 1} - \eps \cdot 2^{j'}
    =    2^{j'} \cdot \big( \tfrac{1}{2} - \eps \big)
    \geq \tfrac{1}{4} \cdot 2^{j'} \, .
    \qedhere
  \]
\end{proof}

We now investigate the set $\Omega_{i,i'}$ defined in \eqref{eq:MainDomain}.


\begin{lem}\label{lem:InclusionLemma2}
  Recall that $\vartheta_{i_0, i_0 '} \in [0,2\pi)$ and define
  \begin{equation}
    \theta_{i_0 , i_0 '}
    :=  \vartheta_{i_0, i_0 '} - \iota \cdot \tfrac{\pi}{2}
    \in \big[0,\tfrac{\pi}{2}\big)
    \quad \text{if} \quad
    \vartheta_{i_0, i_0 '}
    \in \iota \cdot \tfrac{\pi}{2} + \big[0,\tfrac{\pi}{2}\big)
    \quad \text{for some} \quad \iota \in \{0,1,2,3\}
    \, .
    \label{eq:RenormalizedAngleDifference}
  \end{equation}
  Let us also define
  \begin{equation}
    u_{i_0, i'}^{\pm}
    := \begin{cases}
         \phantom{-} x_{i_\ast '}^{\pm} \cdot \cos \theta_{i_0, i_0 '}
         \pm y_{j'} \cdot \sin \theta_{i_0, i_0 '} \,\, ,
         & \text{ if } \vartheta_{i_0, i_0 '} \in [0,\frac{\pi}{2}) \, , \\
         - x_{i_\ast '}^{\mp} \cdot \sin \theta_{i_0, i_0 '}
         \pm y_{j'} \cdot \cos \theta_{i_0, i_0 '} \,\, ,
         & \text{ if } \vartheta_{i_0, i_0 '} \in [\frac{\pi}{2}, \pi) \, , \\
         - x_{i_\ast '}^{\mp} \cdot \cos \theta_{i_0, i_0 '}
         \pm y_{j'} \cdot \sin \theta_{i_0, i_0 '} \,\, ,
         & \text{ if } \vartheta_{i_0, i_0 '} \in [\pi,\frac{3}{2} \pi) \, , \\
         \phantom{-} x_{i_\ast '}^{\pm} \cdot \sin \theta_{i_0, i_0'}
         \pm y_{j'} \cdot \cos \theta_{i_0, i_0 '} \,\, ,
         & \text{ if } \vartheta_{i_0, i_0 '} \in [\frac{3}{2} \pi , 2\pi) \, ,
       \end{cases}
    \label{eq:UBoundsDefinition}
  \end{equation}
  and
  \begin{equation}
    v_{i_0, i'}^{\pm}
    := \begin{cases}
         \phantom{-}
             x_{i_\ast '}^{\pm} \cdot \sin \theta_{i_0, i_0 '}
         \pm y_{j'} \cdot \cos \theta_{i_0, i_0 '} \,\, ,
         & \text{ if } \vartheta_{i_0, i_0 '} \in [0,\frac{\pi}{2}) \, , \\
         \phantom{-}
             x_{i_\ast '}^{\pm} \cdot \cos \theta_{i_0, i_0 '}
         \pm y_{j'} \cdot \sin \theta_{i_0, i_0 '} \, ,
         & \text{ if } \vartheta_{i_0, i_0 '} \in [\frac{\pi}{2}, \pi) \, , \\
         -   x_{i_\ast '}^{\mp} \cdot \sin \theta_{i_0, i_0 '}
         \pm y_{j'} \cdot \cos \theta_{i_0, i_0 '} \, ,
         & \text{ if } \vartheta_{i_0, i_0 '} \in [\pi,\frac{3}{2} \pi) \, , \\
         -   x_{i_\ast '}^{\mp} \cdot \cos \theta_{i_0, i_0 '}
         \pm y_{j'} \cdot \sin \theta_{i_0, i_0 '} \, ,
         & \text{ if } \vartheta_{i_0, i_0 '} \in [\frac{3}{2} \pi , 2\pi) \, .
       \end{cases}
    \label{eq:VBoundsDefinition}
  \end{equation}
  Then
  \begin{equation}
    R_{i_0, i_0'} \, Q_{j',m',0}
    \subset [u_{i_0, i'}^{-} \,,\, u_{i_0, i'}^{+}]
            \times [v_{i_0, i'}^{-}, v_{i_0, i'}^{+}]
    \label{eq:SpeciallyRotatedBaseSet}
  \end{equation}
  and
  \begin{equation}
    \Omega_{i,i'}
    \subset I_1^{(i,i')} \times I_2^{(i,i')} \, ,
    \label{eq:MainDomainCartesianInclusion}
  \end{equation}
  where
  \begin{equation}
    \begin{split}
      & I_1^{(i,i')}
        := I_1
        := \big[
             2^{-\alpha j} \cdot (
                                  u_{i_0, i'}^{-}
                                  - 2^{j-1}
                                  - m \cdot 2^{\alpha j}
                                 )
             \,,\,
             2^{-\alpha j} \cdot (
                                  u_{i_0, i'}^{+}
                                  - 2^{j-1}
                                  - m \cdot 2^{\alpha j}
                                 )
           \big] \\
      \quad \text{and} \quad
      & I_2^{(i,i')}
        := I_2
        := [
            2^{-\beta j} \cdot v_{i_0, i'}^{-}
            \,\,,\,\,
            2^{-\beta j} \cdot v_{i_0, i'}^{+}
           ] \, .
    \end{split}
    \label{eq:MainDomainIntervalDefinition}
  \end{equation}
\end{lem}

\begin{rem*}
  Note that the angle $\theta_{i_0, i_0 '}$ was introduced to ensure that
  $\cos \theta_{i_0, i_0 '} \geq 0$ and $\sin \theta_{i_0, i_0 '} \geq 0$,
  which will prove convenient.
\end{rem*}

\begin{proof}[Proof of Lemma~\ref{lem:InclusionLemma2}]
  See Appendix~\ref{sec:IntervalInclusionProofs}.
\end{proof}


\medskip{}

Having estimated the domain of integration $\Omega_{i,i'}$ in
\eqref{eq:MainDomain}, we still need to estimate the quantities
$(1 + \|T_{i}^{-1} T_{i'}\|)^\sigma$ and $\psi(\eta)$,
for $\eta \in \Omega_{i,i'}$, in such a way that Lemmas~\ref{lem:AnneLemma}
or \ref{lem:MainLemma} can be readily applied.

First, from the definition of $\psi$ and from
Lemma~\ref{lem:SeriesEuclideanNormEstimate}, we infer that
\begin{equation}
  \psi(\eta)
  \leq 2^{5 \kappa_0}
       \cdot 2^{-(1-\alpha) \kappa_0 |j-j'|}
       \cdot (1 + |\eta_1|)^{-\kappa_1}
       \cdot (1 + |\eta_2|)^{-\kappa_2}
  \qquad \forall \, \eta \in \Omega_{i,i'} \, .
  \label{eq:PsiEstimateOnMainDomain}
\end{equation}

Second, by recalling Equation~\eqref{eq:SimpleTransitionMatrixExplicit} from the proof of
Lemma~\ref{lem:CoveringAlmostStructured} and by recalling the definitions
of $\vartheta_{i_0, i_0'}^{(0)}$ and $\vartheta_{i_0, i_0'}$
(see Equations~\eqref{eq:AngleDifference} and \eqref{eq:NormalizedAngleDifference}),
we conclude that
\[
  T_i^{-1} T_{i'}
  = E_{i,i'}
  := \left(
       \begin{matrix}
             2^{\alpha (j' - j)}     \cdot \cos \vartheta_{i_0, i_0 '}
         & - 2^{\beta j' - \alpha j} \cdot \sin \vartheta_{i_0, i_0 '} \\[0.1cm]
             2^{\alpha j' - \beta j} \cdot \sin \vartheta_{i_0, i_0 '}
         &   2^{\beta (j' - j)}      \cdot \cos \vartheta_{i_0, i_0 '}
       \end{matrix}
     \right)
  =: \left(
       \begin{matrix}
         E_{i,i'}^{(1)} & E_{i,i'}^{(2)} \\[0.1cm]
         E_{i,i'}^{(3)} & E_{i,i'}^{(4)}
       \end{matrix}
     \right) \, .
\]
To estimate the matrix elements of $E_{i,i'}$, we recall that $\beta \leq \alpha$, whence
\begin{align*}
  |E_{i,i'}^{(1)}| \leq 2^{\alpha (j' - j)}
                   \leq 2^{\alpha (j' - j)_{+}} \, ,
  & \qquad
    |E_{i,i'}^{(2)}| \leq 2^{\beta j' - \alpha j}
                     \leq 2^{\alpha (j' - j)}
                     \leq 2^{\alpha (j' - j)_{+}} \, , \\
  \text{and}
  & \qquad
  |E_{i,i'}^{(4)}| \leq 2^{\beta (j' - j)}
                   \leq 2^{\beta (j' - j)_{+}}
                   \leq 2^{\alpha (j' - j)_{+}} \, .
\end{align*}
Therefore,
\[
  1 + \|T_{i}^{-1} T_{i'}\|
  =    1 + \|E_{i,i'}\|
  \leq 4 \cdot 2^{\alpha (j' - j)_{+}}
       + 2^{\alpha j' - \beta j} |\sin \vartheta_{i_0, i_0'}|
  \leq 4 \cdot 2^{\alpha (j' - j)_{+}}
         \cdot \big(
                 1 + 2^{\alpha j' - \beta j} \, |\sin \vartheta_{i_0, i_0'}|
               \big) \,
\]
and hence
\begin{equation}
  \big( 1 + \|T_{i}^{-1} T_{i'}\| \big)^{\sigma}
  \leq 4^{\sigma} \cdot 2^{\alpha \sigma (j' - j)_{+}}
                  \cdot \big(
                          1
                          + 2^{\alpha j' - \beta j}
                            \, |\sin \vartheta_{i_0, i_0'}|
                        \big)^{\sigma} \, .
  \label{eq:TransitionMatrixNormEstimate}
\end{equation}
Finally, since we need to convert the mean integral in \eqref{eq:MainDomain}
to an ordinary integral, we need to estabish a lower bound the measure of $\Omega_{i,i'}$.
Given \eqref{eq:SeriesMainDomainEstimate1} and recalling that
$A_{j} = \mathrm{diag} (2^{\alpha j}, 2^{\beta j})$, we conclude that
\[
  \lambda(\Omega_{i,i'})
  =    |\det A_{j}^{-1}| \cdot \lambda(Q_{j',m',0})
  =    |\det A_{j}|^{-1} \cdot |\det A_{j'}| \cdot \lambda(Q)
  \geq 2^{(\alpha + \beta) (j' - j)},
\]
and hence
\begin{equation}
  [\lambda(\Omega_{i,i'})]^{-1} \leq 2^{(\alpha+\beta) (j-j')} \, .
  \label{eq:MainDomainMeasureLowerBound}
\end{equation}

Combining the estimates \eqref{eq:PsiEstimateOnMainDomain}-\eqref{eq:MainDomainMeasureLowerBound}
and recalling that $\Omega_{i,i'} \subset I_1^{(i,i')} \times I_2^{(i,i')}$, we conclude that
\begin{equation}
  M_{i,i'}^{(2)}
  \leq 4^{\sigma} \cdot 2^{5 \tau \kappa_0} \cdot 2^{\omega_{j,j'}} \cdot M_{i,i'}^{(3)} \,
  \label{eq:MainTermMainEstimate}
\end{equation}
where
\begin{equation}
  \omega_{j,j'} := (s + \tau (\alpha + \beta)) \cdot (j-j')
                   + \alpha \, \sigma \cdot (j' - j)_{+}
                   - (1-\alpha) \, \kappa_0 \, \tau \cdot |j - j'|
  \label{eq:OmegaJDefinition}
\end{equation}
and
\begin{equation}
  M_{i,i'}^{(3)}
  := (1 + 2^{\alpha j' - \beta j} \, |\sin \vartheta_{i_0, i_0 '}|)^{\sigma}
     \cdot \bigg(
             \int_{I^{(i,i')}_{1}}
               (1 + |\eta_1|)^{-\kappa_1} \,
             \, d\eta_1
             \cdot
             \int_{I^{(i,i')}_{2}}
               (1 + |\eta_2|)^{-\kappa_2}
             \, d \eta_2
           \bigg)^\tau \, .
  \label{eq:MainTermLemmaVersion}
\end{equation}

The estimate \eqref{eq:MainTermMainEstimate} and the inclusion
$\Omega_{i,i'} \subset I_1^{(i,i')} \times I_2^{(i,i')}$ from
Lemma~\ref{lem:InclusionLemma2} are the main ingredients for applying
Lemmas~\ref{lem:AnneLemma} and \ref{lem:MainLemma}.
This will be done in the next two subsections.

\subsection{Estimating the sum over \texorpdfstring{$i \in I_0$}{i ∈ I₀}}
\label{sub:SummingOverI}

We fix $i' = (j',m',\ell') \in I_0$ for this whole subsection.
To be able to apply Lemmas~\ref{lem:AnneLemma} and \ref{lem:MainLemma},
we investigate the intervals $I_1^{(i,i')}$ and $I_2^{(i,i')}$ a little further.


\begin{lem}\label{lem:IntervalEstimateMSummation}
  Let $j \in \N$ and $\ell \in \N_0$ such that $\ell \leq \ell_j^{\max}$.
  Then there is a number $S_{j,\ell} \in \R$ such that
  \[
    I_1^{(i,i')}
    \subset \big[
              - m + S_{j,\ell} - 2^{2 + \alpha (j' - j)}
              \,,\,
              - m + S_{j,\ell} + 2^{2 + \alpha (j' - j)} \,
            \big]
    \quad \text{for all } m \in \N_0
    \text{ for which } i = (j, m, \ell) \in I_0 .
  \]
\end{lem}

\begin{proof}
  Let us define $x := 2^{j' - 1} + m' \cdot 2^{\alpha j'}$.
  From the definition of $x_{i_\ast '}^{\pm}$ in
  Lemma~\ref{lem:InclusionLemma1}, we infer that
  \[
    |x_{i_\ast '}^{\pm} - x|
    \leq (1 + \eps) \cdot 2^{\alpha j'}
    \leq 2^{1 + \alpha j'} \, .
  \]

  Let $m \in \N_0$ be arbitrary with $i = (j,m,\ell) \in I_0$,
  and let us define
  \[
    S_{j,\ell}^{(0)}
    := \begin{cases}
         \phantom{-} x \cdot \cos \theta_{i_0, i_0'},
         & \text{if } \vartheta_{i_0, i_0 '} \in [0, \frac{\pi}{2} ) \, , \\
         - x \cdot \sin \theta_{i_0, i_0 '},
         & \text{if } \vartheta_{i_0, i_0 '} \in [\frac{\pi}{2}, \pi ) \, , \\
         - x \cdot \cos \theta_{i_0, i_0'},
         & \text{if } \vartheta_{i_0, i_0 '} \in [\pi, \frac{3}{2} \pi ) \, , \\
         \phantom{-} x \cdot \sin \theta_{i_0, i_0'},
         & \text{if } \vartheta_{i_0, i_0'} \in [\frac{3}{2} \pi, \pi )
       \end{cases}
    \qquad \text{and} \qquad
    S_{j,\ell} := 2^{-\alpha j} \cdot (S_{j,\ell}^{(0)} - 2^{j-1}) \, ,
  \]
  and compare the definition of $S_{j,\ell}^{(0)}$
  with the definition of $u_{i_0, i'}^{\pm}$ in \eqref{eq:UBoundsDefinition}.
  Recalling that $\beta \leq \alpha$ and $y_{j'} = 2^{\beta j' + 1}$,
  and that $|\sin \theta_{i_0, i_0 '}|, |\cos \theta_{i_0, i_0 '}| \leq 1$,
  we conclude that
  \[
    | u_{i_0, i'}^{\pm} - S_{j,\ell}^{(0)} |
    \leq |y_{j'}| + |x - x_{i_\ast '}^{\pm}|
    \leq 2^{1 + \beta j'} + 2^{1 + \alpha j'}
    \leq 2^{2 + \alpha j'}
  \]
  and
  \[
    \big|
      2^{-\alpha j} \cdot (u_{i_0,i'}^{\pm} - 2^{j-1} - m \cdot 2^{\alpha j})
      - (S_{j,\ell} - m)
    \big|
    = \big| 2^{-\alpha j} \cdot (u_{i_0,i'}^{\pm} - S_{j,\ell}^{(0)}) \big|
    \leq 2^{2 + \alpha (j' - j)} \, .
  \]
  Combining this and the definition of $I_1^{(i,i')}$ in
  \eqref{eq:MainDomainIntervalDefinition} results in
  \[
    I_1^{(i,i')} \subset \big[
                           - m + S_{j,\ell} - 2^{2 + \alpha (j' - j)}
                           \,,\,
                           - m + S_{j,\ell} + 2^{2 + \alpha (j' - j)} \,
                         \big]
    \, . \qedhere
  \]
\end{proof}

\begin{lem}\label{lem:IntervalEstimateLSummation}
  Let $\iota \in \{0,1,2,3\}$, $k \in \{-1,0,1,2\}$, and $j \in \N$.
  Define
  \[
    J_{j}^{\iota,k}
    := \left\{
         \ell \in \N_0
         \,:\,
         \ell \leq \ell_{j}^{\max}
         \text{ and }
         \vartheta_{i_0, i_0 '}
         \in \iota \cdot \tfrac{\pi}{2} + \big[0, \tfrac{\pi}{2} \big)
         \text{ and }
         k_{i_0, i_0 '} = k
       \right\} \, ,
  \]
  with $k_{i_0, i_0 '}$ is as defined in \eqref{eq:NormalizedAngleDifference}.
  Furthermore, let
  \[
    \beta_1 := N^{-1} \cdot 2^{j' - j} \, ,
    \qquad
    \beta_2 := \frac{8 \pi}{N} \cdot 2^{j' - j} \, ,
    \quad \text{and} \quad
    L_j := 2 \cdot 2^{\beta (j' - j)} \, .
  \]
  Then there are $S_{k,\iota,j} \in \R$ and $\nu_{\iota} \in \{\pm 1\}$
  such that:
  \begin{enumerate}
    \item for any $m \in \N_0$ and $\ell \in J_j^{\iota,k}$ satisfying
          $i = (j,m,\ell) \in I_0$, we have
          \[
            \begin{cases}
              I_2^{(i,i')}
              \subset
              \big[
                \beta_1 \cdot (S_{k,\iota,j} + \nu_{\iota} \, \ell) - L_{j}
                \,,\,
                \beta_2 \cdot (S_{k,\iota,j} + \nu_{\iota} \, \ell) + L_{j}
              \big] \, ,
              & \text{if } \iota \in \{0,1\} \, , \\[0.2cm]
              I_2^{(i,i')}
              \subset
              \big[
                \beta_2 \cdot (S_{k,\iota,j} + \nu_{\iota} \, \ell) - L_{j}
                \,,\,
                \beta_1 \cdot (S_{k,\iota,j} + \nu_{\iota} \, \ell) + L_{j}
              \big] \, ,
              & \text{if } \iota \in \{2,3\} \, ;
            \end{cases}
          \]

    \item for any $\ell \in J_{j}^{\iota, k}$, we have
          \[
            \begin{cases}
              S_{k,\iota,j} + \nu_{\iota} \, \ell \geq 0 \, ,
              & \text{if } \iota \in \{0,1\} \, , \\[0.2cm]
              S_{k,\iota,j} + \nu_{\iota} \, \ell \leq 0 \, ,
              & \text{if } \iota \in \{2,3\} \, ;
            \end{cases}
          \]

    \item $2^{\alpha j' - \beta j} \cdot |\sin \vartheta_{i_0,i_0 '}|
          \leq 2\pi \cdot \beta_1 \cdot |S_{k,\iota,j} + \nu_\iota \, \ell|$
          for all $\ell \in J_{j}^{\iota,k}$.
  \end{enumerate}
\end{lem}

\begin{proof}
  See Appendix~\ref{sec:IntervalInclusionProofs}.
\end{proof}

Given Lemmas~\ref{lem:IntervalEstimateMSummation} and
\ref{lem:IntervalEstimateLSummation}, we can finally show that the first
supremum in \eqref{eq:TargetEstimate} is finite,
provided that \eqref{eq:KappaConditions1} and \eqref{eq:KappaConditions2} are satisfied.

To show this, let us define $\langle x \rangle := 1 + |x|$ for $x \in \R$
and fix $j \in \N$ for the moment.
Since
\[
  \sum_{\ell = 0}^{\ell_j^{\max}} \,
    \sum_{m = 0}^{m_j^{\max}}
      M_{i,i'}^{(3)}
  \leq \sum_{\iota = 0}^3
         \sum_{k = -1}^{2}
           \sum_{\ell \in J_{j}^{\iota,k}} \,
             \sum_{m = 0}^{m_j^{\max}}
               M_{i,i'}^{(3)} \, ,
\]
it suffices to estimate the inner double sum
for fixed $\iota \in \{0,\dots,3\}$ and $k \in \{-1,\dots,2\}$.
Let $\beta_1, \beta_2, S_{k,\iota,j}, \nu_\iota, L_j$ as in Lemma~\ref{lem:IntervalEstimateLSummation}.
If $\iota \in \{0,1\}$, let us define $\widetilde{\beta_1} := \beta_1$
and $\widetilde{\beta_2} := \beta_2$;
otherwise, if $\iota \in \{2,3\}$, let us define $\widetilde{\beta_1} := \beta_2$
and $\widetilde{\beta_2} := \beta_1$  instead.
Now, by definition of $M_{i,i'}^{(3)}$ and by
Lemmas~\ref{lem:IntervalEstimateMSummation}
and \ref{lem:IntervalEstimateLSummation}, we see that, for $\ell \in J_{j}^{\iota,k}$
and $0 \leq m \leq m_j^{\max}$ with $S_{j,\ell}$ as in
Lemma~\ref{lem:IntervalEstimateMSummation},
\begin{align*}
  M_{i,i'}^{(3)}
  \leq \big(
         1 + 2 \pi \beta_1 \, |S_{k,\iota,j} + \nu_\iota \, \ell|
       \big)^\sigma
       \! \cdot \!
       \left(
         \int_{\widetilde{\beta_1} \cdot (S_{k,\iota,j} + \nu_\iota \, \ell) - L_j}
             ^{\widetilde{\beta_2} \cdot (S_{k,\iota,j} + \nu_\iota \, \ell) + L_j}
           \langle \eta_2 \rangle^{-\kappa_2}
         \, d \eta_2
       \right)^{ \! \tau}
       \!\! \cdot \!
       \left(
         \int_{S_{j,\ell} - m - 2^{2 + \alpha (j' - j)}}
             ^{S_{j,\ell} - m + 2^{2 + \alpha (j' - j)}}
            \langle \eta_1 \rangle^{-\kappa_1}
         \, d \eta_1
       \right)^{\! \tau} \! .
\end{align*}

On the other hand, since $\kappa_1 \geq \frac{2}{\tau}$,
Lemma \ref{lem:AnneLemma}
(applied with $N = 0$, $\beta_0 = 1$, $L = 2^{2 + \alpha (j' - j)}$ and $M = - S_{j,\ell}$)
yields the estimate
\begin{align*}
  \sum_{m \in \Z}
    \left(
      \int_{S_{j,\ell} - m - 2^{2 + \alpha (j' - j)}}
          ^{S_{j,\ell} - m + 2^{2 + \alpha (j' - j)}}
         (1 + |\eta_1|)^{- \kappa_1}
      \, d \eta_1
    \right)^{\tau}
  & =    \sum_{m \in \Z}
           \left(
             \int_{m - S_{j,\ell} - 2^{2 + \alpha (j' - j)}}
                 ^{m - S_{j,\ell} + 2^{2 + \alpha (j' - j)}}
                (1 + |\xi_1|)^{- \kappa_1}
             \, d \xi_1
           \right)^{\tau} \\
  & \leq 2^{3 + \tau}
         \cdot 10^{3}
         \cdot 2^{2\tau + \tau \alpha (j' - j)}
         \cdot 2
         \cdot (2 + 2^{2 + \alpha (j' - j)}) \\
  & \leq 2^{17 + 3\tau}
         \cdot 2^{\tau \alpha (j' - j)}
         \cdot 2^{\alpha (j' - j)_{+}}
    =: \Gamma_j^{(1)} \, .
\end{align*}

Next, we note that
\[
  (1 + 2 \pi \beta_1 \cdot |S_{k,\iota,j} + \nu_\iota \, \ell|)^{\sigma}
  \leq 2^{\sigma}
       \cdot \sum_{E \in \{0,\sigma\}}
               |2 \pi \beta_1 \cdot (S_{k,\iota,j} + \nu_\iota \, \ell)|^{E} \, .
\]
Let us fix $E \in \{0, \sigma\}$ for the moment.

Set $\varepsilon := 1$ if $\iota \in \{0,1\}$ and $\varepsilon := -1$ otherwise.
Then Lemma \ref{lem:IntervalEstimateLSummation} shows that
$\varepsilon \cdot (S_{k,\iota,j} + \nu_{\iota} \, \ell) \geq 0$ for
$\ell \in J_{j}^{\iota,k}$.
Furthermore, the change of variable $\lambda = \nu_\iota \, \ell$ combined with
an application of Lemma \ref{lem:MainLemma} or
Corollary \ref{cor:MainLemmaNegativVersion} (depending on whether
$\varepsilon = 1$ or $\varepsilon = -1$) shows that,
for $\Gamma_j^{(2)} := 2^{4 + \sigma + \tau + \tau \kappa_2}$,
\begin{align*}
  & \sum_{\substack{\ell \in \Z, \\
                    \varepsilon \cdot (S_{k,\iota,j} + \nu_\iota \, \ell) \geq 0}}
      |2 \pi \beta_1 (S_{k,\iota,j} + \nu_\iota \, \ell)|^E
      \left(
        \int_{\widetilde{\beta_1} \cdot (S_{k,\iota,j} + \nu_\iota \, \ell) - L_j}
            ^{\widetilde{\beta_2} \cdot (S_{k,\iota,j} + \nu_\iota \, \ell) + L_j}
          \langle \eta_2 \rangle^{-\kappa_2}
        \, d \eta_2
      \right)^{ \! \tau} \\
  & = \sum_{\substack{\lambda \in \Z, \\
                      \varepsilon \cdot (S_{k,\iota,j} + \lambda) \geq 0}}
      |2 \pi \beta_1 (S_{k,\iota,j} + \lambda)|^E
      \left(
        \int_{\widetilde{\beta_1} \cdot (S_{k,\iota,j} + \lambda) - L_j}
            ^{\widetilde{\beta_2} \cdot (S_{k,\iota,j} + \lambda) + L_j}
          \langle \eta_2 \rangle^{-\kappa_2}
        \, d \eta_2
      \right)^{ \! \tau} \\
  & \leq \Gamma_j^{(2)}
         \cdot (8 \pi)^\tau
         \cdot (2\pi)^\sigma
         \cdot ( 1 + \max\{1, L_j\}^{\tau + \sigma} )
         \cdot (1 + 2N \cdot 2^{(1-\beta)(j-j')} + N \cdot 2^{j-j'}) \\
  & \leq \Gamma_j^{(2)}
         \cdot 2^{5 \tau + 3 \sigma}
         \cdot 2 \cdot (2 \cdot 2^{\beta (j' - j)_{+}})^{\tau + \sigma}
         \cdot 4N \cdot 2^{(j-j')_{+}} \\
  & \leq N \cdot \Gamma_j^{(2)}
         \cdot 2^{3 + 6 \tau + 4 \sigma}
         \cdot 2^{(j-j')_{+} + \beta (\tau + \sigma) (j' - j)_{+}}
    =:   \Gamma_j^{(3)} \, .
\end{align*}
Here, we noted in the penultimate estimate that
$L_j = 2 \cdot 2^{\beta (j' - j)} \leq 2 \cdot 2^{\beta (j' - j)_{+}}$.
Furthermore, we noted that $\kappa_2 \geq 1 + \frac{\sigma + 2}{\tau}$
(see Equation~\eqref{eq:KappaConditions1}), so that Lemma~\ref{lem:MainLemma}
and Corollary~\ref{cor:MainLemmaNegativVersion} are indeed applicable.

Summarising all these estimates, we finally conclude that
\[
  \sum_{\ell = 0}^{\ell_{j}^{\max}} \,
    \sum_{m = 0}^{m_{j}^{\max}} \,
      M_{i,i'}^{(3)}
  \leq 4 \cdot 4 \cdot 2^{\sigma} \cdot 2
       \cdot \Gamma_j^{(1)} \cdot \Gamma_j^{(3)}
  \leq \Gamma_{j}^{(4)} \, ,
\]
where a straightforward but tedious calculation shows that one can choose
\[
  \Gamma_{j}^{(4)}
  := 2^{30 + 6 \sigma + (10 + \kappa_2) \tau}
     \cdot N
     \cdot 2^{(\alpha + \beta (\tau + \sigma)) (j' - j)_{+}
              + (j - j')_{+}
              + \tau \alpha (j' - j)} \, .
\]

Finally, by recalling Equations~\eqref{eq:MainDomain} and \eqref{eq:MainTermMainEstimate}, we see that
\[
  \sum_{i = (j,m,\ell) \in I_0}
    M_{i,i'}^{(1)}
  \leq 4^{\sigma} \cdot 2^{5 \tau \kappa_0} \cdot
       \sum_{j = 1}^{\infty} 2^{\omega_{j,j'}} \cdot \Gamma_{j}^{(4)}
  \leq \Gamma^{(5)} \cdot \sum_{j=1}^{\infty} 2^{\widetilde{\omega}_{j,j'}}
\]
with
\[
  \widetilde{\omega}_{j,j'}
  = \tau \alpha (j' - j)
    + (j - j')_{+}
    + (\alpha + \beta (\tau + \sigma)) (j' - j)_{+}
    + \omega_{j,j'}
  \quad \! \text{and} \quad \!
  \Gamma^{(5)} := 2^{30 + 5 \tau \kappa_0 + 8 \sigma + (10 + \kappa_2) \tau} N \, .
\]
A direct calculation shows that
\[
  \widetilde{\omega}_{j,j'}
  = \begin{cases}
      |j' - j|
      \cdot (1-\alpha) \tau
      \cdot \big(
              \frac{\alpha + \sigma(\alpha + \beta) - s}{(1-\alpha)\tau}
              - \kappa_0
            \big) \, ,
      & \text{if } j \leq j' \\[0.2cm]
      |j' - j| \cdot (1 - \alpha) \tau
      \cdot \big(
              \frac{1 + s + \tau \beta}{(1-\alpha) \tau}
              - \kappa_0
            \big) \, ,
      & \text{otherwise} \, .
    \end{cases}
\]
Given our choice of $\kappa_0$ (see Equation \eqref{eq:KappaConditions2}),
we therefore conclude that $\widetilde{\omega}_{j,j'} \leq -|j' - j|$ and
\begin{equation}
  \sum_{i = (j,m,\ell) \in I_0} M_{i,i'}^{(1)}
  \leq \Gamma^{(5)} \cdot \sum_{j \in \Z} 2^{-|j' - j|}
  \leq 3 \cdot \Gamma^{(5)} < \infty
  \qquad \forall \, i' \in I_0 .
  \label{eq:ISummationFinalEstimate}
\end{equation}
Since $\Gamma^{(5)}$ is independent of the choice of
$i' = (j',m',\ell') \in I_0$, we have thus shown that the first supremum
in \eqref{eq:TargetEstimate} is finite, as long as both $i$ and $i'$ are restricted
to $I_0$ instead of to $I = \{0\} \cup I_0$.
\hfill $\square$

\subsection{Estimating the sum over \texorpdfstring{$i' \in I_0$}{i’ ∈ I₀}}
\label{sub:SummingOverIPrime}

For this whole subsection, we fix $i = (j,m,\ell) \in I_0$.

In the preceding subsection, we used the inclusion
\({
  I_1^{(i,i')}
  \subset [S_{j,\ell} \!-\! m \!-\! 2^{2 + \alpha(j' - j)}, S_{j,\ell} \!-\! m \!+\! 2^{2 + \alpha(j' - j)}]
}\)
to then apply Lemma~\ref{lem:AnneLemma}.
Observe that the parameter $m$ appears on the right-hand side
of this inclusion \emph{without a factor in front}.
In contrast, we shall prove in Lemma~\ref{lem:IntervalEstimateMPrimeSummation}
that the inclusion
\({
  I_1^{(i,i')}
  \subset [s_{j',\ell'} \cdot \beta_{j',\ell'} \cdot m' + S_{j',\ell'} - L_{j'}^\ast \,\, , \,\,
           s_{j',\ell'} \cdot \beta_{j',\ell'} \cdot m' + S_{j',\ell'} + L_{j'}^\ast]
}\)
holds where $s_{j',\ell'} \in \{\pm 1\}$ and
the factor $\beta_{j',\ell'}$ depends on $\cos \theta_{i_0, i_0 '}$
or $\sin \theta_{i_0, i_0 '}$, so that possibly $\beta_{j',\ell'} \approx 0$.
If this happens, the bound provided by Lemma~\ref{lem:AnneLemma} will be ineffective.
Therefore, we first deal with this special case using another method.
Precisely, let us define
\[
  t_{j', \ell'}
  := \begin{cases}
       \cos \theta_{i_0, i_0 '}
       & \text{if } \vartheta_{i_0, i_0 '} \in [0, \frac{\pi}{2} ) \, , \\
       \sin \theta_{i_0, i_0 '}
       & \text{if } \vartheta_{i_0, i_0 '} \in [\frac{\pi}{2}, \pi) \, , \\
       \cos \theta_{i_0, i_0 '}
       & \text{if } \vartheta_{i_0, i_0 '} \in [\pi, \frac{3}{2} \pi ) \, , \\
       \sin \theta_{i_0, i_0 '}
       & \text{if } \vartheta_{i_0, i_0 '} \in [\frac{3}{2} \pi , 2\pi ) \, .
     \end{cases}
\]
Note that this is well-defined, i.e. the right-hand side indeed only depends on $j',\ell'$),
since ${i = (j,m,\ell)}$ is fixed and $i_0 ' = (j',\ell')$.
Furthermore, we note that $t_{j',\ell'}$ is, up to a sign,
the coefficient of $x_{i_\ast '}^{\pm}$ in the definition of $u_{i_0, i'}^{\pm}$;
see Equation \eqref{eq:UBoundsDefinition}.
Finally, for fixed $j' \in \N$, let us define
\begin{align*}
  J^{(j')}_{\mathrm{special}}
  & := \left\{
         l' \in \N_0
         \,:\,
         l' \leq l_{j'}^{\max} \text{ and } t_{j',l'} \leq \frac{1}{10}
       \right\} \\
  \quad \text{and} \quad
  J^{(j')}_{\mathrm{normal}}
  & := \left\{
         l' \in \N_0
         \,:\,
         l' \leq l_{j'}^{\max} \text{ and } t_{j',l'} > \frac{1}{10}
       \right\} \, .
\end{align*}

The following lemma provides the crucial ingredient for estimating the
contribution of the terms with problematic indices
$l' \in J^{(j')}_{\mathrm{special}}$.

\begin{lem}\label{lem:CosineVanishingSpecialEstimate}
  For $j' \in \N$, $\ell' \in J^{(j')}_{\mathrm{special}}$ and $m' \in \N_0$
  with $i' = (j',m',\ell') \in I_0$, any $\eta_2 \in I_{2}^{(i,i')}$
  (see Equation~\eqref{eq:MainDomainIntervalDefinition}) satisfies
  \[
    1 + |\eta_2| \geq \frac{2^{j' - \beta j}}{40} \, .
  \]
\end{lem}

\begin{proof}
  let us define
  \[
    \widetilde{t}_{j',\ell'}
    := \begin{cases}
         \sin \theta_{i_0, i_0 '}
         & \text{if } \vartheta_{i_0, i_0 '} \in [0, \frac{\pi}{2} ) \, , \\
         \cos \theta_{i_0, i_0 '}
         & \text{if } \vartheta_{i_0, i_0 '} \in [\frac{\pi}{2}, \pi) \, , \\
         \sin \theta_{i_0, i_0 '}
         & \text{if } \vartheta_{i_0, i_0 '} \in [\pi, \frac{3}{2} \pi ) \, , \\
         \cos \theta_{i_0, i_0 '}
         & \text{if } \vartheta_{i_0, i_0 '} \in [\frac{3}{2} \pi , 2\pi ) \, ,
       \end{cases}
  \]
  and note that
  \begin{equation}
    \widetilde{t}_{j', \ell'} \geq 1 - t_{j', \ell'} \,
    \label{eq:SpecialEstimateTTTilde}
  \end{equation} as a consequence of Equation~\eqref{eq:CosineSineBound} and since
  $t_{j', \ell'}, \widetilde{t}_{j', \ell'} \geq 0$.

  According to Equation~\eqref{eq:MainDomainIntervalDefinition}, 
  $I_2^{(i,i')} = 2^{-\beta j} \cdot [v_{i_0, i'}^{-}, v_{i_0, i'}^{+}]$.
  To derive from this the desired estimate, we distinguish two cases,
  depending on $\vartheta_{i_0, i_0'}$.
  \medskip{}

  \emph{Case 1:} $\vartheta_{i_0, i_0 '} \in [0, \pi)$.
  In this case, $v_{i_0, i'}^{-}
  = x_{i_\ast '}^{-} \cdot \widetilde{t}_{j',\ell'} - y_{j'} \cdot t_{j',\ell'}$
  where $y_{j'} = 2^{\beta j' + 1} \leq 2^{j' + 1}$ and
  $x_{i_\ast '}^{-} \geq 2^{j'} / 4$; see Lemma \ref{lem:InclusionLemma1}.
  This and $t_{j', \ell'} \leq 1/10$ for
  $\ell' \in J^{(j')}_{\mathrm{special}}$ allows us to conculude that
  \[
    v_{i_0, i'}^{-}
    \geq 2^{j' - 2} \cdot \widetilde{t}_{j', \ell'}
         - 2^{j' + 1} \cdot t_{j', \ell'}
    \,\, \overset{\text{Eq. } \eqref{eq:SpecialEstimateTTTilde}}{\geq} \,\,
         2^{j' - 2} \cdot \big( 1 - t_{j', \ell'} - 2^{3} \cdot t_{j', \ell'} \big)
    \geq 2^{j' - 2} \cdot \frac{1}{10}
    >    0 \, .
  \]
  Therefore, 
  we see that
  \(
    1 + |\eta_2|
    \geq \eta_2
    \geq 2^{-\beta j} \cdot v_{i_0, i'}^{-}
    \geq 2^{j' - \beta j} / 40
  \),
  as desired.

  \medskip{}

  \emph{Case 2:} $\vartheta_{i_0, i_0 '} \in [\pi, 2\pi)$.
  In this case, $- v_{i_0, i'}^{+}
  = x_{i_\ast '}^{-} \cdot \widetilde{t}_{j', \ell'} - y_{j'} \cdot t_{j', \ell'}$.
  Precisely as in the previous case, we note that
  $- v_{i_0, i'}^{+} \geq 2^{j' - 2} / 10 > 0$.
  This and $\eta_2 \leq 2^{-\beta j} \cdot v_{i_0, i'}^{+}$, leads to
  \[
    1 + |\eta_2|
    \geq - \eta_2
    \geq 2^{-\beta j} \cdot (- v_{i_0, i'}^{+})
    \geq 2^{j' - \beta j} / 40 \, .
    \qedhere
  \]
\end{proof}

Since $\kappa_2 \geq 2 + \kappa_2^{(0)}$ (see Equation \eqref{eq:KappaConditions1}),
Lemma \ref{lem:CosineVanishingSpecialEstimate} shows that, for
$\ell' \in J^{(j')}_{\mathrm{special}}$ and $m' \in \N_0$ with $m' \leq m_{j'}^{\max}$,
\[
  \int_{I_2^{(i,i')}} (1+|\eta_2|)^{-\kappa_2} \, d \eta_2
  \leq \left(\frac{2^{j' - \beta j}}{40}\right)^{-\kappa_2^{(0)}}
       \cdot \int_{\R} (1 + |\eta_2|)^{-2} \, d \eta_2
  \leq 2 \cdot \left(\frac{2^{j' - \beta j}}{40}\right)^{-\kappa_2^{(0)}} \, .
\]
Likewise, since $\kappa_1 \geq 2$,
\(
  \int_{I_{1}^{(i,i')}} (1 + |\eta_1|)^{-\kappa_1} \, d \eta_1
  \leq \int_{\R} (1 + |\eta_1|)^{-2} \, d \eta_1
  \leq 2
\).
Furthermore,
\[
  1 + 2^{\alpha j' - \beta j} |\sin \vartheta_{i_0, i_0 '}|
  \leq 1 + 2^{(\alpha - \beta) j'} \cdot 2^{\beta (j' - j)}
  \leq 2 \cdot 2^{\beta (j' - j)_{+}} \cdot 2^{(\alpha - \beta) j'} \, .
\]
Finally,
\[
  2^{j' - \beta j}
  = 2^{(1 - \beta) j'} \cdot 2^{\beta (j' - j)}
  \geq 2^{(1 - \beta) j'} \cdot 2^{- \beta (j - j')_{+}} \, .
\]
Combined with the definition of $M_{i,i'}^{(3)}$
(see Equation~\eqref{eq:MainTermLemmaVersion}),
these estimates imply that, for any index ${i' = (j',m',\ell') \in I_0}$
with $\ell' \in J^{(j')}_{\mathrm{special}}$,
\[
  M_{i,i'}^{(3)}
  \leq 2^{\sigma + 2 \tau} \cdot 40^{\tau \kappa_2^{(0)}}
       \cdot 2^{\beta \sigma (j' - j)_{+}}
       \cdot 2^{\beta \tau \kappa_2^{(0)} (j - j')_{+}}
       \cdot 2^{\sigma (\alpha - \beta) j'}
       \cdot 2^{- \tau \kappa_2^{(0)} (1 - \beta) j'} \, .
\]

Since
$1 + \ell_{j'}^{\max} \leq 3N \cdot 2^{(1 - \beta) j'}$
and $1 + m_{j'}^{\max} \leq 3 \cdot 2^{(1 - \alpha) j'}$ as well,
combining this with Equations~\eqref{eq:MainDomain} and \eqref{eq:MainTermMainEstimate} results in
\[
  \sum_{\ell' \in J_{\mathrm{special}}^{(j')}}
    \sum_{m' = 0}^{m_{j'}^{\max}} \!\!
      M_{i,i'}^{(1)}
  \leq \! N \cdot 2^{4 + 3 \sigma + 2 \tau + 5 \tau \kappa_0 + 6 \tau \kappa_2^{(0)}}
       \!\cdot\, 2^{\omega_{j,j'}}
       \cdot 2^{\beta \sigma (j' - j)_{+}}
       \cdot 2^{\beta \tau \kappa_2^{(0)} (j - j')_{+}}
       \cdot 2^{j' \cdot [
                          (1-\beta)(1 - \tau \kappa_2^{(0)})
                          + \sigma(\alpha - \beta)
                          + 1 - \alpha
                         ]}
\]
for each fixed $j' \in \N$.
But by definition of $\kappa_2^{(0)}$ and by our choice of $\kappa_0$
(see Equations \eqref{eq:KappaConditions1} and \eqref{eq:KappaConditions2}),
we see that
$\omega_{j,j'}
 + \beta \sigma (j' - j)_{+}
 + \beta \tau \kappa_2^{(0)} (j - j')_{+}
 \leq -|j-j'|$
and $(1-\beta)(1-\tau \kappa_2^{(0)}) + \sigma(\alpha - \beta) + 1 - \alpha \leq 0$.
Therefore,
\begin{equation}
  \sum_{j' = 1}^{\infty} \,
  \sum_{\ell' \in J_{\mathrm{special}}^{(j')}} \,
    \sum_{m' = 0}^{m_{j'}^{\max}} \,
      M_{i,i'}^{(1)}
  \leq N \cdot 2^{4 + 3 \sigma + 2 \tau + 5 \tau \kappa_0 + 6 \tau \kappa_2^{(0)}}
       \cdot \sum_{j' = 1}^{\infty} 2^{-|j' - j|}
  \leq N \cdot 2^{6 + 3 \sigma + 2 \tau + 5 \tau \kappa_0 + 6 \tau \kappa_2^{(0)}}
  \, .
  \label{eq:SpecialCaseOverallEstimate}
\end{equation}

Having taken care of the special case as
$\ell' \in J^{(j')}_{\mathrm{special}}$, the first step for estimating the
remaining series is to further
estimate the intervals $I_1^{(i,i')}$ and $I_2^{(i,i')}$,
which we shall do in the following two lemmata.

\begin{lem}\label{lem:IntervalEstimateMPrimeSummation}
  Let $j' \in \N$ and $\ell' \in J^{(j')}_{\mathrm{normal}}$ be arbitrary
  and define
  \[
    \beta_{j',\ell'} := t_{j',\ell'} \cdot 2^{\alpha (j' - j)} \,
    \qquad \text{and} \qquad
    L_{j'}^\ast := 4 \cdot 2^{\alpha (j' - j)} \, .
  \]
  Then there are such $s_{j', \ell'} \in \{\pm 1\}$ and $S_{j',\ell'} \in \R$
  that
  \[
    I_1^{(i,i')}
    \subset \big[
              s_{j',\ell'} \cdot \beta_{j',\ell'} \cdot m'
              + S_{j', \ell'}
              - L_{j'}^\ast
              \,,\,
              s_{j',\ell'} \cdot \beta_{j',\ell'} \cdot m'
              + S_{j', \ell'}
              + L_{j'}^\ast
            \big]
  \]
  for all $m' \in \N_0$ with $i' = (j',m',\ell') \in I_0$.
\end{lem}

\begin{proof}
  Directly from the definition of $x_{i_\ast '}^{\pm}$, we see that
  $|x_{i_\ast '}^{\pm} - (2^{j' - 1} + m' \cdot 2^{\alpha j'})|
   \leq 2 \cdot 2^{\alpha j'}$.
  Therefore, defining
  \[
    s_{j', \ell'}
    := \begin{cases}
         -1 &
         \text{if } \vartheta_{i_0, i_0 '} \in [\frac{\pi}{2}, \frac{3}{2} \pi), \\
          1 &
         \text{otherwise}
       \end{cases}
  \]
  and recalling the definitions of $u_{i_0, i'}^{\pm}$ and $t_{j', \ell'}$,
  we conclude that
  \[
    |
     u_{i_0, i'}^{\pm}
     - s_{j',\ell'}
       \cdot t_{j', \ell'}
       \cdot (2^{j' - 1} + m' \cdot 2^{\alpha j'})
    |
    \leq 2 \cdot 2^{\alpha j'} + |y_{j'}|
    \leq 4 \cdot 2^{\alpha j'} \, ,
  \]
  where we noted in the last step that $\beta \leq \alpha$ and that $y_{j'} = 2^{\beta j' + 1}$;
  see Lemma~\ref{lem:InclusionLemma1}.

  By the definition of $I_1^{(i,i')}$, we thus see that, for
  \[
    S_{j', \ell'}
    := - 2^{(1-\alpha) j - 1}
       - m
       + s_{j',\ell'} \cdot t_{j',\ell'} \cdot 2^{j' - \alpha j - 1} \, ,
  \]
  \begin{align*}
    I_{1}^{(i,i')}
    & = 2^{-\alpha j} \cdot [u_{i_0,i'}^{-} - 2^{j-1} - m \cdot 2^{\alpha j} \,\, , \,\,
                             u_{i_0,i'}^{+} - 2^{j-1} - m \cdot 2^{\alpha j}] \\
    & \subset 2^{-\alpha j}
              \cdot \Big(
                      - 2^{j-1}
                      - m \cdot 2^{\alpha j}
                      + s_{j',\ell'}
                        \cdot t_{j',\ell'}
                        \cdot (2^{j' - 1} + m' \cdot 2^{\alpha j'})
                      + [-4 \cdot 2^{\alpha j'}, 4 \cdot 2^{\alpha j'}]
                    \Big) \\
    & = [s_{j',\ell'} \cdot \beta_{j',\ell'} \cdot m' + S_{j',\ell'} - L_{j'}^\ast
         \,,\,
         s_{j',\ell'} \cdot \beta_{j',\ell'} \cdot m' + S_{j',\ell'} + L_{j'}^\ast]
    \, .
    \qedhere
  \end{align*}
\end{proof}

\begin{lem}\label{lem:IntervalEstimateLPrimeSummation}
  Let us fix $\iota \in \{0,1,2,3\}$, $k \in \{-1,0,1,2\}$ and $j' \in \N$ and
  define
  \[
    J_{j'}^{\iota,k}
    := \left\{
         \ell' \in \N_0
         \,:\,
         \ell ' \leq \ell_{j'}^{\max}
         \text{ and }
         \vartheta_{i_0, i_0 '} \in \iota \cdot \tfrac{\pi}{2} + \big[0, \tfrac{\pi}{2} \big)
         \text{ and }
         k_{i_0, i_0 '} = k
       \right\} \, ,
  \]
  with $k_{i_0, i_0 '}$ as in Equation~\eqref{eq:NormalizedAngleDifference}.

  Furthermore, let us define
  \[
    \beta_1 := N^{-1} \cdot 2^{\beta (j' - j)} \,
    \quad \text{ } \quad
    \beta_2 := \frac{8 \pi}{N} \cdot 2^{\beta (j' - j)} \,
    \quad \text{and} \quad
    L_{j'} := 2 \cdot 2^{\beta (j' - j)} \, .
  \]
  Then there are such $S_{k,\iota,j'} \in \R$ and $\nu_{\iota} \in \{\pm 1\}$
  that:
  \begin{enumerate}
    \item For any $m' \in \N_0$ and $\ell' \in J_{j'}^{\iota,k}$ with
          $i' = (j', m', \ell') \in I_0$,
          \[
            \begin{cases}
              I_2^{(i,i')}
              \subset
              \big[
                \beta_1 \cdot (S_{k,\iota,j'} + \nu_{\iota} \, \ell') - L_{j'}
                \,,\,
                \beta_2 \cdot (S_{k,\iota,j'} + \nu_{\iota} \, \ell') + L_{j'}
              \big] \, ,
              & \text{if } \iota \in \{0,1\} \, , \\[0.2cm]
              I_2^{(i,i')}
              \subset
              \big[
                \beta_2 \cdot (S_{k,\iota,j'} + \nu_{\iota} \, \ell') - L_{j'}
                \,,\,
                \beta_1 \cdot (S_{k,\iota,j'} + \nu_{\iota} \, \ell') + L_{j'}
              \big] \, ,
              & \text{if } \iota \in \{2,3\} \, .
            \end{cases}
          \]

    \item For any $\ell' \in J_{j'}^{\iota, k}$,
          \[
            \begin{cases}
              S_{k,\iota,j'} + \nu_{\iota} \, \ell' \geq 0 \, ,
              & \text{if } \iota \in \{0,1\} \, , \\[0.2cm]
              S_{k,\iota,j'} + \nu_{\iota} \, \ell' \leq 0 \, ,
              & \text{if } \iota \in \{2,3\} \, .
            \end{cases}
          \]

    \item Finally
          \(
            2^{\alpha j' - \beta j} \cdot |\sin \vartheta_{i_0, i_0 '}|
           \leq 2\pi \cdot \beta_1 \cdot |S_{k,\iota,j'} + \nu_\iota \, \ell'|
          \)
          for all $\ell' \in J_{j'}^{\iota,k}$.
  \end{enumerate}
\end{lem}

\begin{proof}
  See Appendix~\ref{sec:IntervalInclusionProofs}.
\end{proof}

\noindent Given Lemmas~\ref{lem:IntervalEstimateMPrimeSummation} and
\ref{lem:IntervalEstimateLPrimeSummation}, we can finally show that the second
supremum in \eqref{eq:TargetEstimate} is finite
provided that Conditions~\eqref{eq:KappaConditions1} and \eqref{eq:KappaConditions2} are satisfied.

For the moment, we fix $j' \in \N$ and  define $\langle x \rangle := 1 + |x|$ for $x \in \R$.
Recall that we already estimated the series as
$\ell' \in J^{(j')}_{\mathrm{special}}$; see Equation~\eqref{eq:SpecialCaseOverallEstimate}.
Therefore, it suffices to consider
\[
  \sum_{\ell' \in J^{(j')}_{\mathrm{normal}}} \,
    \sum_{m' = 0}^{m_{j'}^{\max}} \,
      M_{i,i'}^{(3)}
  = \sum_{\iota = 0}^{3} \,
      \sum_{k = -1}^{2} \,\,
        \sum_{\ell' \in J^{(j')}_{\mathrm{normal}} \cap J_{j'}^{\iota,k}} \,\,
          \sum_{m' = 0}^{m_{j'}^{\max}} \,
            M_{i,i'}^{(3)} \, .
\]
Therefore, let us fix $\iota \in \{0,\dots,3\}$ and $k \in \{-1,\dots,2\}$ for the moment.
Let $\beta_1, \beta_2, S_{k,\iota,j'}, \nu_\iota, L_{j'}$
as in Lemma~\ref{lem:IntervalEstimateLPrimeSummation}.
If $\iota \in \{0,1\}$, let us define $\widetilde{\beta_1} := \beta_1$ and
$\widetilde{\beta_2} := \beta_2$;
otherwise, if $\iota \in \{2,3\}$, let us define $\widetilde{\beta_1} := \beta_2$
and $\widetilde{\beta_2} := \beta_1$  instead.
By definition of $M_{i,i'}^{(3)}$ and by
Lemmas~\ref{lem:IntervalEstimateMPrimeSummation} and
\ref{lem:IntervalEstimateLPrimeSummation}, we see that, for
$\ell' \in J^{(j')}_{\mathrm{normal}} \cap J_{j'}^{\iota,k}$ and
$0 \leq m' \leq m_{j'}^{\max}$ with $\beta_{j',\ell'}, L_{j'}^\ast$ and $s_{j',\ell'}, S_{j',\ell'}$
as in Lemma \ref{lem:IntervalEstimateMPrimeSummation},
\begin{align*}
  M_{i,i'}^{(3)}
  & \leq \big(
           1 + 2\pi \beta_1 \cdot |S_{k,\iota,j'} + \nu_\iota \, \ell'|
         \big)^{\sigma}
         \cdot
         \left(
           \int_{\widetilde{\beta_1} \cdot (S_{k,\iota,j'} + \nu_\iota \, \ell') - L_{j'}}
               ^{\widetilde{\beta_2} \cdot (S_{k,\iota,j'} + \nu_\iota \, \ell') + L_{j'}}
            \langle \eta_2 \rangle^{-\kappa_2}
           \, d \eta_2
         \right)^{\! \tau} \\
       & \quad \cdot
         \left(
           \int_{s_{j',\ell'} \cdot \beta_{j',\ell'} \cdot m' + S_{j',\ell'} - L_{j'}^\ast}
               ^{s_{j',\ell'} \cdot \beta_{j',\ell'} \cdot m' + S_{j',\ell'} + L_{j'}^\ast}
             \langle \eta_1 \rangle^{-\kappa_1}
           \, d \eta_1
         \right)^{\! \tau}
         \! .
\end{align*}
But with $S_{j',\ell'}^{\ast} := s_{j',\ell'} S_{j',\ell'}$,
the change of variables $\xi = s_{j',\ell'} \eta_1$ results in
\[
  \int_{s_{j',\ell'} \cdot \beta_{j',\ell'} \cdot m' + S_{j',\ell'} - L_{j'}^\ast}
      ^{s_{j',\ell'} \cdot \beta_{j',\ell'} \cdot m' + S_{j',\ell'} + L_{j'}^\ast}
    \langle \eta_1 \rangle^{-\kappa_1}
  \, d \eta_1
  = \int_{\beta_{j',\ell'} \cdot m' + S_{j',\ell'}^\ast - L_{j'}^\ast}
        ^{\beta_{j',\ell'} \cdot m' + S_{j',\ell'}^\ast + L_{j'}^\ast}
      \langle \xi \rangle^{-\kappa_1}
    \, d \xi \, .
\]
Since $\kappa_1 \geq \frac{2}{\tau}$, Lemma~\ref{lem:AnneLemma} --- applied with
$N = 0$, $\beta_0 = \beta_{j',\ell'}$,
and $L = L_{j'}^{\ast} = 4 \cdot 2^{\alpha (j' - j)}$) --- gives the estimate
\begin{align*}
  \sum_{m' \in \Z}
    \left(
      \int_{\beta_{j',\ell'} \cdot m' + S_{j',\ell'}^\ast - L_{j'}^\ast}
          ^{\beta_{j',\ell'} \cdot m' + S_{j',\ell'}^\ast + L_{j'}^\ast}
        \langle \xi \rangle^{-\kappa_1}
      \, d \xi
    \right)^{\tau}
  & \leq 2^{3 + \tau} \cdot 10^3
         \cdot (L_{j'}^\ast)^\tau \cdot 2
         \cdot \big(1 + \beta_{j',\ell'}^{-1} \cdot (1 + L_{j'}^{\ast}) \big) \\
  & \leq 2^{20 + 3\tau}
         \cdot 2^{\alpha \tau (j' - j)}
         \cdot 2^{\alpha (j - j')_{+}}
    := \Gamma_{j'}^{(1)} \, .
\end{align*}
Here, we noted that $\beta_{j',\ell'} = t_{j',\ell'} \cdot 2^{\alpha(j' - j)}$
where $t_{j',\ell'} \geq 1/10$ since $\ell' \in J^{(j')}_{\mathrm{normal}}$.

Next, we note that
\[
  (1 + 2 \pi \cdot \beta_1 \cdot |S_{k,\iota,j'} + \nu_\iota \, \ell' |)^{\sigma}
  \leq 2^{\sigma}
       \cdot \sum_{E \in \{0,\sigma\}}
               |2 \pi \beta_1 (S_{k,\iota,j'} + \nu_\iota \, \ell')|^{E} \, .
\]
Let us fix $E \in \{0, \sigma\}$ for the moment and define $\varepsilon := 1$ if $\iota \in \{0,1\}$
and $\varepsilon := -1$ otherwise.
Then Lemma~\ref{lem:IntervalEstimateLPrimeSummation} shows that
$\varepsilon \cdot (S_{k,\iota,j'} + \nu_{\iota} \, \ell ') \geq 0$ for
$\ell' \in J_{j'}^{\iota,k}$.
Furthermore, the change of variable $\lambda = \nu_\iota \, \ell '$ combined
with an application of Lemma~\ref{lem:MainLemma} or
Corollary~\ref{cor:MainLemmaNegativVersion}
--- depending on whether $\varepsilon = 1$ or $\varepsilon = -1$ ---
with $N = E \leq \sigma$, $\gamma = 2 \pi \beta_1$, and $L = L_{j'}$ shows that,
for $\Gamma_{j'}^{(2)} := 2^{4 + \sigma + \tau + \tau \kappa_2}$,
\begin{align*}
  & \sum_{\substack{\ell' \in \Z, \\
                    \varepsilon (S_{k,\iota,j'} + \nu_\iota \, \ell') \geq 0}}
      |2 \pi \beta_1 (S_{k,\iota,j'} + \nu_\iota \, \ell')|^E
      \left(
        \int_{\widetilde{\beta_1} \cdot (S_{k,\iota,j'} + \nu_\iota \, \ell') - L_{j'}}
            ^{\widetilde{\beta_2} \cdot (S_{k,\iota,j'} + \nu_\iota \, \ell') + L_{j'}}
          \langle \eta_2 \rangle^{-\kappa_2}
        \, d \eta_2
      \right)^{ \! \tau} \\
  & = \sum_{\substack{\lambda \in \Z, \\
                      \varepsilon (S_{k,\iota,j'} + \lambda) \geq 0}}
      |2 \pi \beta_1 (S_{k,\iota,j'} + \lambda)|^E
      \left(
        \int_{\widetilde{\beta_1} \cdot (S_{k,\iota,j'} + \lambda) - L_j}
            ^{\widetilde{\beta_2} \cdot (S_{k,\iota,j'} + \lambda) + L_j}
          \langle \eta_2 \rangle^{-\kappa_2}
        \, d \eta_2
      \right)^{ \! \tau} \\
  & \leq \Gamma_{j'}^{(2)}
         \cdot (8 \pi)^\tau
         \cdot (2\pi)^\sigma
         \cdot \big( 1 + (\max\{1, L_{j'}\})^{\tau + \sigma} \big)
         \cdot (1 + 2N + N \cdot 2^{\beta(j-j')}) \\
  & \leq \Gamma_{j'}^{(2)}
         \cdot 2^{5 \tau + 3 \sigma}
         \cdot 2 \cdot (2 \cdot 2^{\beta (j' - j)_{+}})^{\tau + \sigma}
         \cdot 5N \cdot 2^{\beta (j-j')_{+}} \\
  & \leq N \cdot \Gamma_{j'}^{(2)}
         \cdot 2^{4 + 6 \tau + 4 \sigma}
         \cdot 2^{\beta [(j-j')_{+} + (\tau + \sigma) (j' - j)_{+}]}
    =:   \Gamma_{j'}^{(3)} \, .
\end{align*}
Here, we noted that $\kappa_2 \geq 1 + \frac{\sigma + 2}{\tau}$,
so that Lemma~\ref{lem:MainLemma}
and Corollary~\ref{cor:MainLemmaNegativVersion} are indeed applicable.

Summarising all these estimates, we finally see that
\[
  \sum_{\ell' \in J_{\mathrm{normal}}^{(j')}} \,
    \sum_{m' = 0}^{m_{j'}^{\max}} \,
      M_{i,i'}^{(3)}
  \leq 4 \cdot 4 \cdot 2^{\sigma} \cdot 2
       \cdot \Gamma_{j'}^{(1)} \cdot \Gamma_{j'}^{(3)}
  \leq \Gamma_{j'}^{(4)} \, ,
\]
from where a straightforward but tedious calculation shows that one can choose
\[
  \Gamma_{j'}^{(4)}
  := 2^{33 + 6 \sigma + (10 + \kappa_2) \tau} \cdot N
     \cdot 2^{\alpha [(j - j')_{+} + \tau (j' - j)]
              + \beta [(j - j')_{+} + (\tau + \sigma) (j' - j)_{+}]} \, .
\]

Finally, by recalling Equations~\eqref{eq:MainDomain} and
\eqref{eq:MainTermMainEstimate}, we see that
\[
  \sum_{j' = 1}^{\infty} \,
    \sum_{\ell ' \in J_{\mathrm{normal}}^{(j')}} \,
      \sum_{m' = 0}^{m_{j'}^{\max}} \,
        M_{i,i'}^{(1)}
  \leq 4^{\sigma} \cdot 2^{5 \tau \kappa_0} \cdot
       \sum_{j' = 1}^{\infty} 2^{\omega_{j,j'}} \cdot \Gamma_{j'}^{(4)}
  \leq \Gamma^{(6)} \cdot \sum_{j'=1}^{\infty} 2^{\widetilde{\omega}_{j,j'}}
\]
with
\[
  \widetilde{\omega}_{j,j'}
  = \alpha [(j - j')_{+} + \tau (j' - j)]
    + \beta [(j - j')_{+} + (\tau + \sigma) (j' - j)_{+}]
    + \omega_{j,j'}
  \quad \! \text{and} \quad \!
  \Gamma^{(6)} := 2^{33 + 8 \sigma + (10 + \kappa_0 + \kappa_2) \tau} N \, .
\]
A direct calculation shows that
\[
  \widetilde{\omega}_{j,j'}
  = \begin{cases}
      |j' - j|
      \cdot (1-\alpha) \tau
      \cdot \big(
              \frac{\sigma (\alpha + \beta) - s}{(1-\alpha)\tau}
              - \kappa_0
            \big) \, ,
      & \text{if } j \leq j' \\[0.2cm]
      |j' - j| \cdot (1 - \alpha) \tau
      \cdot \big(
              \frac{s + \tau \beta + \alpha + \beta}{(1-\alpha) \tau}
              - \kappa_0
            \big) \, ,
      & \text{otherwise} \, .
    \end{cases}
\]
Given our choice of $\kappa_0$ (see Equation~\eqref{eq:KappaConditions2})
and recalling that $\alpha,\beta \leq 1$,
we therefore see that ${\widetilde{\omega}_{j,j'} \leq -|j' - j|}$, and thus
\[
  \sum_{j' = 1}^{\infty} \,
    \sum_{\ell ' \in J_{\mathrm{normal}}^{(j')}} \,
      \sum_{m' = 0}^{m_{j'}^{\max}} \,
        M_{i,i'}^{(1)}
  \leq \Gamma^{(6)} \cdot \sum_{j' \in \Z} 2^{-|j' - j|}
  \leq 3 \cdot \Gamma^{(6)} < \infty \, .
\]

Combining this with Equation~\eqref{eq:SpecialCaseOverallEstimate} finally results in
\begin{equation}
  \sum_{i' \in I_0}
    M_{i,i'}^{(1)}
  \leq 3 \, \Gamma^{(6)}
       + N \cdot 2^{6 + 3 \sigma + 2 \tau + 5 \tau \kappa_0 + 6 \tau \kappa_2^{(0)}}
  := \Gamma^{(7)}
  \qquad \forall \, i \in I_0 \, .
  \label{eq:IPrimeSummationFinalEstimate}
\end{equation}

\subsection{Estimating the contribution of the low-pass part}
\label{sub:LowPassContribution}

\noindent In this subsection, we estimate the series $\sum_{i \in I} M_{i,0}^{(1)}$
and $\sum_{i' \in I} M_{0, i'}^{(1)}$ where $M_{i,0}^{(1)}$ and
$M_{0,i'}^{(1)}$ are as defined in Equation~\eqref{eq:MainTermLowPass}.

\medskip{}

We first estimate $\sum_{i' \in I_0} M_{0,i'}^{(1)}$.
To this end, we recall that, for $i' = (j',m',\ell') \in I_0$,
$T_{i'} = R_{j',\ell'} A_{j'}$.
Since the rotation matrix $R_{j',\ell'}$ does not change the norm and since
$A_{j'} = \mathrm{diag} (2^{\alpha j'}, 2^{\beta j'})$
we conclude that $\|T_{0}^{-1} T_{i'}\| = \|A_{j'}\| = 2^{\alpha j'}$, since $\beta \leq \alpha$.
Furthermore, Lemma~\ref{lem:CoveringGeometry} shows that
$1 + |\xi| \geq |\xi| \geq 2^{j' - 2}$ for all $\xi \in Q_{i'}$ and hence
\[
  \psi(T_0^{-1} (\xi - b_0))
  = \psi(\xi)
  \leq (2^{j' - 2})^{- \kappa_0}
  \leq 2^{2\kappa_0} \cdot 2^{- \kappa_0 j'}
  \leq 2^{2 \kappa_0} \cdot 2^{-(1-\alpha) \kappa_0 j'}
  \qquad \forall \, \xi \in Q_{i'} \, .
\]
Overall, these considerations imply
\[
  M_{0, i'}^{(1)}
  \leq 2^{\sigma + 2 \kappa_0 \tau}
       \cdot 2^{j' (\alpha \sigma - s - (1-\alpha) \tau \kappa_0)} \, .
\]

Finally, we note that
$1 + m_{j'}^{\max} \leq 3 \cdot 2^{(1 - \alpha) j'} \leq 3 \cdot 2^{j'}$ and
 $1 + \ell_{j'}^{\max} \leq 3N \cdot 2^{(1-\beta) j'} \leq 3N \cdot 2^{j'}$ and thus
\[
  \sum_{i' \in I_0}
    M_{0,i'}^{(1)}
  \leq 9N \cdot 2^{\sigma + 2 \kappa_0 \tau} \cdot
       \sum_{j' = 1}^{\infty}
         2^{j' (2 + \alpha \sigma - s - (1-\alpha) \tau \kappa_0)}
  \leq N \cdot 2^{4 + \sigma + 2 \kappa_0 \tau} \, .
\]
Here, in the last step we noted that
$2 + \alpha \sigma - s - (1 - \alpha) \tau \kappa_0 \leq -1$
thanks to our assumptions regarding $\kappa_0$; see Equation~\eqref{eq:KappaConditions2}.

\medskip{}

Next, we estimate $\sum_{i \in I_0} M_{i,0}^{(1)}$.
As above, we see that, for $i = (j,m,\ell) \in I_0$,
\[
  \|T_{i}^{-1} T_0\|
  = \|T_i^{-1}\|
  = \|A_{j}^{-1}\|
  = \max \{2^{-\alpha j}, 2^{- \beta j} \}
  \leq 1 \, .
\]
In particular, this implies that, for $\xi \in Q_0 = B_4 (0)$,
$|T_i^{-1} \xi| \leq \|T_{i}^{-1}\| \cdot |\xi| \leq 4$.
Furthermore,
\[
  T_{i}^{-1} b_i
  = A_{j}^{-1} R_{j,\ell}^{-1} R_{j,\ell} c_{j,m}
  = A_{j}^{-1} c_{j,m}
  = \mathrm{diag} (2^{-\alpha j}, 2^{-\beta j})
      \cdot (2^{j-1} + m \cdot 2^{\alpha j} \,,\, 0)^t
  = (2^{(1 - \alpha) j - 1} + m \,,\, 0)^t \, .
\]
Combining this results in
\[
       2^{(1 - \alpha) j - 1}
  \leq 1 + |T_{i}^{-1} b_i|
  \leq 1 + |T_{i}^{-1} (b_i - \xi)| + |T_{i}^{-1} \xi|
  \leq 5 \cdot (1 + |T_{i}^{-1} (\xi - b_i)|) \,
\]
and thus
\[
  \psi \big( T_{i}^{-1} (\xi - b_i) \big)
  \leq (1 + |T_i^{-1} (\xi - b_i)|)^{-\kappa_0}
  \leq 2^{4 \kappa_0} \cdot 2^{-(1 - \alpha) \kappa_0 j} \, .
\]

Overall, we have thus shown that
\[
  M_{i, 0}^{(1)}
  \leq 2^{\sigma + 4 \tau \kappa_0}
       \cdot 2^{j (s - \tau \kappa_0 (1 - \alpha))} \, .
\]
Noting once again that
$1 + m_{j}^{\max} \leq 3 \cdot 2^{(1 - \alpha) j} \leq 3 \cdot 2^{j}$
and $1 + \ell_{j}^{\max} \leq 3N \cdot 2^{(1 - \beta) j} \leq 3N \cdot 2^{j}$,
we finally see that
\[
  \sum_{i \in I_0}
    M_{i,0}^{(1)}
  \leq 9N \cdot 2^{\sigma + 4 \tau \kappa_0} \cdot
       \sum_{j=1}^{\infty}
         2^{j (2 + s - \tau \kappa_0 (1 - \alpha))}
  \leq N \cdot 2^{4 + \sigma + 4 \tau \kappa_0} \,
\]
since our assumptions regarding $\kappa_0$ imply that
$2 + s - \tau \kappa_0 (1 - \alpha) \leq -1$; see Equation~\eqref{eq:KappaConditions2}.

\medskip{}

Finally, since $\psi \leq 1$, we see $M_{0,0}^{(1)} \leq 2^{\sigma}$.
Combining this with the preceding estimates from this subsection, we conclude that
\begin{equation}
  \sum_{i \in I}
    M_{i,0}^{(1)}
  \leq N \cdot 2^{5 + \sigma + 4 \tau \kappa_0}
  \quad \text{and} \quad
  \sum_{i' \in I}
    M_{0,i'}^{(1)}
  \leq  N \cdot 2^{5 + \sigma + 2 \tau \kappa_0} \, .
  \label{eq:EndEstimateLowPass}
\end{equation}

\paragraph{Concluding the proof of Equation~\eqref{eq:TargetEstimate}}

Combining the estimates \eqref{eq:ISummationFinalEstimate},
\eqref{eq:IPrimeSummationFinalEstimate} and \eqref{eq:EndEstimateLowPass},
we finally see that Equation~\eqref{eq:TargetEstimate} is satisfied for
$B$ as in \eqref{eq:BChoice}.

\subsection{Proving Theorems \ref{thm:WavePacketAtomicDecomposition}
            and \ref{thm:WavePacketBanachFrames}}
\label{sub:WavePacketDecompositionProof}

\begin{proof}[Proof of Theorem \ref{thm:WavePacketAtomicDecomposition}]
  We shall derive the claims by applying
  Theorem~\ref{thm:StructuredBFDAtomicDecomposition}
  (with $\omega$ instead of $\varepsilon$).
  To this end, let us choose $Q_1^{(0)} := Q = (-\eps, 1+\eps) \times (-1-\eps,1+\eps)$
  and $Q_2^{(0)} := B_4 (0)$.
  Furthermore, let $k_i := 1$ for $i \in I_0^{(\alpha,\beta)}$ and $k_0 := 2$.
  With $T_i, b_i$ as in Lemma~\ref{lem:CoveringAlmostStructured},
  $Q_i = T_i Q_{k_i}^{(0)} + b_i$ for all $i \in I^{(\alpha,\beta)}$,
  as required at the beginning of Section~\ref{sub:BFD}.

  Furthermore, in the notation of Theorem~\ref{thm:StructuredBFDAtomicDecomposition},
 let us define $\gamma_1^{(0)} := \gamma$ and $\gamma_2^{(0)} := \varphi$.
  Still in the notation of Theorem~\ref{thm:StructuredBFDAtomicDecomposition},
  let us define
  \[
    \varrho_k :
    \R^2 \to (0,\infty),
    \xi \mapsto C
                \cdot (1 + |\xi|)^{-\kappa_0}
                \cdot (1 + |\xi_1|)^{-\kappa_1}
                \cdot (1 + |\xi_2|)^{-\kappa_2}
  \]
  for $k \in \{1,2\}$, noting that $\varrho_k = C \cdot \psi$
  with $\psi$ as in Equation~\eqref{eq:PsiDefinition}.
  Finally, we choose $\tau, \vartheta$, and $\sigma$ as in
  Theorem~\ref{thm:StructuredBFDAtomicDecomposition} and note
  that $N_0$ as defined in Theorem~\ref{thm:WavePacketAtomicDecomposition}
  satisfies $N_0 \geq N$ for $N$ as in Theorem~\ref{thm:StructuredBFDAtomicDecomposition}.

  It is then not hard to see that the wave packet system
  $(L_{\delta T_i^{-t} k} \, \gamma^{[i]})_{i \in I, k \in \Z^2}$
  introduced in Definition~\ref{def:WavePacketSystem} coincides with
  the system $\Gamma^{(\delta)}$ from Equation~\eqref{eq:GSISystemDefinition}.
  Furthermore, the assumptions of Theorem~\ref{thm:WavePacketAtomicDecomposition}
  imply that the first three assumptions of Theorem~\ref{thm:StructuredBFDAtomicDecomposition}
  are satisfied.
  In addition, since we are working in dimension $d = 2$, so that
  $d + 1 + \omega \leq 4$ and given our choice of $\varrho_1, \varrho_2$,
  Equation~\eqref{eq:WavePacketAtomicDecompositionCondition} shows that the
  fourth assumption of Theorem~\ref{thm:StructuredBFDAtomicDecomposition}
  is fulfilled.

  Therefore, it remains to verify the last condition in
  Theorem~\ref{thm:StructuredBFDAtomicDecomposition},
  namely that the constants $K_1, K_2$ introduced in
  Equation~\eqref{eq:AtomicDecompositionConstantDefinition} are finite.
  To this end, we first show that the entries $N_{i,j}$ of the infinite matrix
  $(N_{i,j})_{i,j \in I}$ can be estimated in terms of the numbers
  $M_{i,j}^{(1)}$ from Equations \eqref{eq:MainTerm} and \eqref{eq:MainTermLowPass}.
  To see this, first recall that
  \begin{equation}
    |\det T_{j,m,\ell}|
    = |\det \mathrm{diag(2^{\alpha j}, 2^{\beta j})}|
    = 2^{(\alpha + \beta) j}
    \qquad \forall \, (j,m,\ell) \in I_0^{(\alpha,\beta)} \,
    \label{eq:DeterminantEstimate}
  \end{equation}
  and hence
  \[
    \frac{|\det T_{i'}|}{|\det T_{i}|} = 2^{(\alpha + \beta) (j' - j)}
    \qquad \forall \, i = (j,m,\ell) , i' = (j',m',\ell') \in I_0^{(\alpha,\beta)}
    \, .
  \]
  It should be noted that Equation~\eqref{eq:DeterminantEstimate} also implies that
  the coefficient space introduced in Definition~\ref{def:WavePacketCoefficientSpace}
  coincides with the one in Definition~\ref{def:StructuredBFDSequenceSpace}, with
  identical quasi-norms.

  Furthermore, since
  $Q = (-\eps, 1+\eps) \times (-1-\eps, 1+\eps) \subset (-2,2)^2$, we see that
  \[
    \lambda(Q_i)
    = \lambda (T_i Q + b_i)
    = |\det T_i| \cdot \lambda(Q)
    \leq 2^6 \cdot |\det T_i|
    \qquad \forall \, i \in I_0^{(\alpha,\beta)}.
  \]
  Therefore,
  \(
    |\det T_i|^{-1} \int_{Q_i} f \, d\xi
    = \frac{\lambda(Q_i)}{|\det T_i|} \fint_{Q_i} f \, d\xi
    \leq 2^6 \cdot \fint_{Q_i} f \, d\xi
  \)
  for every non-negative measurable function $f : \R^2 \to [0,\infty)$ and each
  $i \in I_0^{(\alpha,\beta)}$.

  By comparing the definition of $N_{i,i'}$ (where $w = w^s$)
  in Theorem~\ref{thm:StructuredBFDAtomicDecomposition}
  with that of $M_{i',i}^{(1)}$, and by using the observations from
  the two preceding paragraphs, it is not hard to see that
  \[
    N_{i,i'}
    \leq C^\tau \cdot 2^{6 \tau} \cdot M_{i',i}^{(1)}
    \quad \forall \, i,i' \in I_0^{(\alpha,\beta)} \,
    \quad \text{where } M_{i',i}^{(1)} \text{ is defined using }
          \tau [\vartheta(\alpha+\beta) - s] \text{ instead of } s \, .
  \]
  Since $\lambda(B_4 (0)) = \pi \cdot 4^2 \leq 2^{6}$,
  it is not hard to see that this estimate remains valid for all
  $i,i' \in I^{(\alpha,\beta)}$.
  Let us define
  $\widetilde{s} := \tau [\vartheta \cdot (\alpha + \beta) - s]$.

  Therefore, if we prove that
  $\kappa_0, \kappa_1, \kappa_2$ in Theorem~\ref{thm:WavePacketAtomicDecomposition}
  satisfy the conditions of Theorem~\ref{thm:MainTechnicalResult}, we shall also prove that,
  for any $\ell \in \{1,2\}$, the constant $K_\ell$ defined in
  Equation~\eqref{eq:AtomicDecompositionConstantDefinition} satisfies
  \begin{align*}
    K_\ell^{1/\tau}
    \leq 2^6 C
         \cdot \left(
                 \sup_{i \in I^{(\alpha,\beta)}}
                   \sum_{i' \in I^{(\alpha,\beta)}}
                     M_{i',i}^{(1)}
               \right)^{1/\tau}
    & \leq 2^6 C
           \cdot \left(
                   N \cdot 2^{37
                              + 8 \sigma
                              + \tau (10
                                      + 5 \kappa_0
                                      + 6 \kappa_2^{(0)}
                                      + \kappa_2)}
                 \right)^{1/\tau} \\
    & \leq C \cdot 2^{24
                      + 41 / \tau_0
                      + 5 \kappa_0
                      + 6 \sigma_0
                      + 12 / ( (1-\beta) \tau_0)
                      + \kappa_2
                      + 40 / p_0} \, ,
  \end{align*}
  where we noted that
  \begin{equation}
    \tau = \min \{1,p,q\} \geq \tau_0 := \min\{p_0,q_0\}
    \qquad \text{and} \qquad
    \vartheta = (p^{-1} - 1)_{+} \leq p_0^{-1} \leq \tau_0^{-1} \, ,
    \label{eq:TauNaughtDefinition}
  \end{equation}
  so that $1/\tau \leq 1/\tau_0$.
  Furthermore, we noted that
  \begin{equation}
    \frac{\sigma}{\tau}
    \leq \frac{2}{p_0} + \lceil p_0^{-1} (2 + \omega) \rceil
    =: \sigma_0
    \leq 1 + \frac{5}{p_0} \, ,
    \label{eq:SigmaTauQuotientEstimate}
  \end{equation}
  which follows directly from the definition of $\sigma$ in
  Theorem \ref{thm:StructuredBFDAtomicDecomposition} by recalling that we use
  $\omega$ instead of $\varepsilon$ and that $d=2$, $p \geq p_0$ and $\omega \leq 1$.
  Finally, we also invoked the estimate
  \begin{equation}
    \kappa_2^{(0)}
    \leq \frac{2}{(1-\beta) \tau_0} + \sigma_0
    \, ,
    \label{eq:KappaTwoNaughtEstimate}
  \end{equation}
  which can be obtained directly from \eqref{eq:KappaConditions1},
  given \eqref{eq:SigmaTauQuotientEstimate} and recalling that
  $\alpha - \beta \leq 1 - \beta$.
  Since the right-hand side of the estimate above only depends on
  $\alpha, \beta, p_0, q_0, s_0, C$,
  Theorem \ref{thm:StructuredBFDAtomicDecomposition} finally yields the claim.

  Overall, it remains to verify that the choice of
  $\kappa_0, \kappa_1, \kappa_2$ in Theorem \ref{thm:WavePacketAtomicDecomposition}
  satisfy the assumptions of Theorem \ref{thm:MainTechnicalResult} where
  $\widetilde{s}$ is used instead of $s$.
  To see this, we first of all note that indeed
  ${\kappa_1 = \frac{2}{\tau_0} \geq \max \{ 2, \frac{2}{\tau} \}}$.
  Second, from Equations~\eqref{eq:KappaTwoNaughtEstimate}
  and \eqref{eq:SigmaTauQuotientEstimate} we infer that
  \[
    2 + \kappa_2^{(0)}
    \leq 2 + \frac{2}{(1-\beta) \tau_0} + \sigma_0
    \leq 3 + \frac{5}{p_0} + \frac{2}{(1-\beta) \tau_0}
    =    \kappa_2 \,
  \]
  and furthermore
  \(
    1 + \frac{\sigma + 2}{\tau}
    \leq 1 + \frac{2}{\tau_0} + \sigma_0
    \leq 2 + \frac{2}{(1-\beta) \tau_0} + \frac{5}{p_0} \leq \kappa_2
  \),
  as required in Theorem \ref{thm:MainTechnicalResult}.

  Regarding $\kappa_0$, we note
  $|\frac{\widetilde{s}}{\tau}| = |\vartheta(\alpha+\beta) - s|
  \leq s_0 + p_0^{-1} (\alpha + \beta)$, which implies
  \begin{align*}
    \frac{3 + |\widetilde{s}| + \tau + \alpha + (\alpha + \beta) \sigma}
         {(1-\alpha)\tau}
    & \leq (1-\alpha)^{-1}
           \left(
             \frac{3}{\tau_0}
             + s_0
             + \frac{\alpha + \beta}{p_0}
             + 1
             + \frac{\alpha}{\tau_0}
             + (\alpha + \beta) \sigma_0
           \right) \\
    & \leq (1-\alpha)^{-1}
           \left(
             3 +
             \frac{3 + \alpha}{\tau_0}
             + s_0
             + \frac{6 \alpha + 6 \beta}{p_0}
           \right)
    \leq \kappa_0 \, .
  \end{align*}
  Finally, we see that
  \(
    \frac{\max\{\tau,\sigma\}}{\tau}
    \leq \max\{1, \sigma_0\} \leq 1 + \frac{5}{p_0}
  \)
  and hence
  \begin{align*}
    & \frac{2
            + |\widetilde{s}|
            + \tau \beta \, \kappa_2^{(0)}
            + \max \{\tau, \sigma \} (\alpha + \beta)
            }
           {(1-\alpha)\tau} \\
    & \leq (1 - \alpha)^{-1}
           \left(
             \frac{2}{\tau_0}
             + s_0
             + \frac{\alpha + \beta}{p_0}
             + \beta \cdot \left(
                             \frac{2}{(1-\beta)\tau_0} + 1 + \frac{5}{p_0}
                           \right)
             + (\alpha + \beta) \cdot \left(1 + \frac{5}{p_0}\right)
           \right) \\
    & \leq (1-\alpha)^{-1} \left(
                             3
                             + \frac{2}{\tau_0}
                             + s_0
                             + \frac{6 \alpha + 11\beta}{p_0}
                             + \frac{2\beta}{(1-\beta)\tau_0}
                           \right)
      \leq \kappa_0 \, ,
  \end{align*}
  as required in \eqref{eq:KappaConditions2}.
\end{proof}

\begin{proof}[Proof of Theorem~\ref{thm:WavePacketBanachFrames}]
The proof is very similar to the one of Theorem~\ref{thm:WavePacketAtomicDecomposition}
and therefore only sketched here.

Instead of Theorem~\ref{thm:StructuredBFDAtomicDecomposition}, we use
Theorem~\ref{thm:StructuredBFDBanachFrame}, but with
$\gamma_1^{(0)} := \overline{\widetilde{\gamma}}$ and
$\gamma_2^{(0)} := \overline{\widetilde{\varphi}}$
instead of $\gamma_1^{(0)} = \gamma$ and $\gamma_2^{(0)} = \varphi$ in the
preceding proof; here we recall the notation $\widetilde{f}(x) = f(-x)$.
To justify this choice of $\gamma_1^{(0)}, \gamma_2^{(0)}$, we recall the elementary identity
$\Fourier \overline{\widetilde{f} \,\,} = \overline{\widehat{f} \,\,}$.
With this, it is not hard to see that $\overline{\widetilde{\gamma}\,}$
and $\overline{\widetilde{\varphi}\,}$ also satisfy the assumptions
(1)-(2) of Theorem~\ref{thm:WavePacketAtomicDecomposition} and (3')
of Theorem~\ref{thm:WavePacketBanachFrames},
and thus assumptions (1)-(3) of Theorem~\ref{thm:StructuredBFDBanachFrame}.

Given our assumptions, it is not hard to verify
--- as in the proof of Theorem~\ref{thm:WavePacketAtomicDecomposition} ---
that the matrix elements $M_{j,i}$ introduced in Theorem~\ref{thm:StructuredBFDBanachFrame} satisfy
\[
  M_{i',i} \leq C^{\tau} \cdot 2^{6 \tau} \cdot M_{i',i}^{(1)}
  \quad \text{where } M_{i',i}^{(1)} \text{ is defined using }
        s \tau \text{ instead of } s
        \text{ and } \theta \text{ instead of } \sigma
\]
where $\theta$ is as defined in Theorem~\ref{thm:StructuredBFDBanachFrame}.
The remainder of the proof is then almost identical to that of
Theorem~\ref{thm:WavePacketAtomicDecomposition} with one exception
Namely we still need to verify Equation~\eqref{eq:WavePacketAnalysisOperatorExplicit}.

For this, to avoid confusion in notations of  the family $\gamma^{[i]}$ defined
in Equation~\eqref{eq:GeneratorDefinition} with those introduced in
Definition~\ref{def:WavePacketSystem}, let us write
$\gamma^{\{i\}} := |\det T_i|^{1/2} \cdot M_{b_i} [\gamma_i (T_i^t \bullet)]$
for the family defined in Equation~\eqref{eq:GeneratorDefinition}.
Now, we recall that
$\gamma_i = \gamma_{k_i}^{(0)} = \gamma_1^{(0)} = \overline{\widetilde{\gamma}}$
for $i \in I_0^{(\alpha,\beta)}$ and
$\gamma_0 = \gamma_{k_0}^{(0)} = \overline{\widetilde{\varphi}}$.
Finally, we note that
\[
  \overline{\widetilde{M_\xi g}} = M_\xi \overline{\widetilde{g}}
  \quad \text{for any measurable } g : \R^2 \to \CC
        \text{ and any } \xi \in \R^2 \, ,
\]
which shows that
\[
  \big\langle
    f , \, L_{\delta T_i^{-t} k} \, \widetilde{\gamma^{\{i\}}}
  \big\rangle
  = \big\langle
      f \mid L_{\delta T_{i}^{-t} k} \, \overline{\widetilde{\gamma^{\{i\}}}}
    \big\rangle_{L^2}
  = \Big\langle
      f
      \,\Big|\,
      |\det T_i|^{1/2}
      \cdot L_{\delta T_i^{-t} k} \,
            \big[
              M_{b_i} ( \overline{\widetilde{\gamma_i \circ T_i^t}} \, )
            \big]
    \Big\rangle_{L^2}
  = \big\langle f \mid L_{\delta T_i^{-t} k} \, \gamma^{[i]} \big\rangle_{L^2}
\]
for all $f \in L^2 (\R^2)$, $i \in I^{(\alpha,\beta)}$ and $k \in \Z^2$.

Therefore, Equation \eqref{eq:AnalysisOperatorConsistency} finally shows that, for
$f \in L^2 (\R^2) \cap \PacketSpace_{s}^{p,q} (\alpha,\beta)$, the analysis
map $A^{(\delta)} := A_{\widetilde{\Gamma}^{(\delta)}}$ constructed in
Theorem~\ref{thm:StructuredBFDBanachFrame} satisfies
\[
  A_{\widetilde{\Gamma}^{(\delta)}} f
  = \left(
      \big\langle
        f , L_{\delta \cdot T_i^{-t} k} \, \widetilde{\gamma^{\{i\}}}
      \big\rangle
    \right)_{i \in I, k \in \Z^2}
  = \left(
      \big\langle
        f
        \,\mid\,
        L_{\delta \cdot T_i^{-t} k} \, \gamma^{[i]}
      \big\rangle_{L^2}
    \right)_{i \in I, k \in \Z^2} \, ,
\]
as desired.
\end{proof}

\appendix

\section{Proof of Lemma \ref{lem:MainLemma} and Corollary \ref{cor:MainLemmaNegativVersion}}
\label{sec:MainLemmaProof}

\noindent For proving Lemma~\ref{lem:MainLemma}, we shall use the following auxiliary result.

\begin{lem}\label{lem:SumEstimate}
  For any $\beta > 0$ and $x \in \R$,
  \[
    \sum_{k\in\Z}
      \big(1 + \left|\beta (k+x)\right|\big)^{-2}
      \leq 2 + \frac{10}{\beta}
      \leq 2^{4} \cdot \left(1+\beta^{-1}\right) .
    \qedhere
  \]
\end{lem}

\begin{proof}
  First of all, we note that the function
  \[
    g: \R \to     [0,\infty],
       x  \mapsto \sum_{k \in \Z} \big(1+|\beta(k+x)|\big)^{-2}
  \]
  is periodic with period one.
  Therefore, to prove that $g(x) \leq 2 + \frac{10}{\beta}$, it is enough to
  consider only the case when $x \in [0,1]$.
  We now distinguish three cases determined by the value of $k \in \Z$.

  \medskip{}

  \emph{Case 1:} $k \geq \frac{1}{\beta}$.
  This implies that $\beta (k+x) \geq \beta k \geq1$, whence
  \[
    \sum_{k \geq \frac{1}{\beta}}
      \left(1 + \left|\beta (k+x)\right|\right)^{-2}
    \leq \sum_{k\geq\frac{1}{\beta}}
           \left(\beta k\right)^{-2}
    =    \beta^{-2}
         \cdot \sum_{k\geq\frac{1}{\beta}}
                 k^{-2}.
  \]
  Furthermore, if $y>0$ and $n \in \Z_{\geq y}$,
  then $n \geq y > 0$, whence $n \geq 1$ and $n+1 \leq 2n$.
  Therefore, if $z \in [n,n+1]$, then
  $z^{-2} \geq (n+1)^{-2} \geq (2n)^{-2} = n^{-2}/4$ and thus
  \begin{equation}
         \sum_{n\in\Z_{\geq y}}
           n^{-2}
    =    \sum_{n\geq y}
           \int_{n}^{n+1}n^{-2} \, d z
    \leq 4 \sum_{n\geq y}
             \int_{n}^{n+1}
               z^{-2}
             \, d z
    \leq 4 \cdot \int_{y}^{\infty}
                    z^{-2}
                 \, d z
    =    4 \cdot \frac{z^{-1}}{-1}\bigg|_{z=y}^{\infty}
    =    \frac{4}{y} 
    \label{eq:IntegralComparisonTest}
  \end{equation}
  and
  \(
    \sum_{k\geq1/\beta} \left( 1+|\beta (k+x)| \right)^{-2}
    \leq \beta^{-2} \cdot \sum_{k \geq 1/\beta} k^{-2}
    \leq \beta^{-2} \cdot \frac{4}{1/\beta}
    = \frac{4}{\beta}
  \).

  \medskip{}

  \emph{Case 2:} $k\leq-\frac{1}{\beta}-1$, and hence
  $-(k+1) \geq \frac{1}{\beta}$.
  For any $x \in [0,1]$, this implies that
  \[
    \beta(k+x)
    \leq \beta(k+1)
    \leq \beta \cdot \left(-\tfrac{1}{\beta}\right)
    = -1 < 0
    \quad \text{ and hence } \quad
         \left|\beta (k+x)\right|
         =-\beta (k+x)
    \geq -\beta (k+1) > 0.
  \]
  Therefore, we can again apply Equation~\eqref{eq:IntegralComparisonTest} to obtain
  \begin{align*}
    \sum_{k \in \Z_{\leq -\frac{1}{\beta}-1}}
      \left( 1 + \left|\beta (k+x)\right|\right)^{-2}
    & \leq \sum_{k \in \Z_{\leq -\frac{1}{\beta} - 1}}
             \left(-\beta (k+1)\right)^{-2}\\
    \left({\scriptstyle \text{with }\ell=-\left(k+1\right)}\right)
    & =    \sum_{\ell \in \Z_{\geq 1/\beta}}
             \left(\beta\ell\right)^{-2}
      \leq \beta^{-2} \cdot \frac{4}{1/\beta}
      =    \frac{4}{\beta}.
  \end{align*}

  \medskip{}

  \emph{Case 3:} $-\frac{1}{\beta} - 1 \leq k \leq \frac{1}{\beta}$.
  This implies that
  $k \in \Z \cap \big[-\frac{1}{\beta}-1,\frac{1}{\beta}\big]$
  and hence $k$ can take at most $2 + \frac{2}{\beta}$ different values.
  Therefore,
  \[
    \sum_{\substack{k \in \Z\\
                    -\frac{1}{\beta} - 1 \leq k \leq \frac{1}{\beta}}}
      \left(1+|\beta(k+x)|\right)^{-2}
    \leq2 \left( 1+\frac{1}{\beta} \right).
  \]

  Combining the three cases results in
  \(
    g(x) \leq \tfrac{4}{\beta}
             + \tfrac{4}{\beta}
             + 2 \big( 1 + \tfrac{1}{\beta} \big)
        =    2 + \frac{10}{\beta}
  \)
  for all $x \in [0,1]$ and, indeed, for all $x \in \R$,
  since $g$ is periodic with period $1$.
\end{proof}

We can now turn to the proof of Lemma~\ref{lem:MainLemma}.

\begin{proof}[Proof of Lemma~\ref{lem:MainLemma}]
  Let
  \[
    I_{k} := \left[\beta_{1} \cdot (k+M) - L \,,\, \beta_{2} \cdot (k+M) + L\right]
  \]
  for $k \in \Z$ such that $k + M \geq 0$.
  The Lebesgue measure $\lambda (I_{k})$ of $I_{k}$ is given by
  \begin{equation}
    \lambda(I_{k})
    = 2L + \frac{\beta_{2} - \beta_{1}}{\beta_{1}} \cdot \beta_{1} \cdot (k+M).
    \label{eq:MeasureOfInterval}
  \end{equation}
  We now distinguish two cases determined by the value of $k \in \Z$.
  \smallskip{}

  \emph{Case 1:} $k$ is such that $\beta_{1} \cdot (k+M) \geq 2 L > 0$.
  Given \eqref{eq:MeasureOfInterval}, this implies that
  \[
    \lambda(I_{k})
    \leq \beta_{1} \cdot (k+M)
         \cdot \left[ 1 + \frac{\beta_{2} - \beta_{1}}{\beta_{1}} \right]
    =    \frac{\beta_{2}}{\beta_{1}}
         \cdot \beta_{1} \cdot (k+M)
    \leq \frac{\beta_{2}}{\beta_{1}}
         \cdot \left( 1 + |\beta_{1} \cdot (k+M)| \right).
  \]
  Furthermore, we note that each $x \in I_{k}$ satisfies
  \[
    x
    \geq \beta_{1} \cdot (k+M)-L
    \geq \frac{1}{2} \cdot \beta_{1} \cdot (k+M)
    =    \frac{1}{2} \cdot |\beta_{1} \cdot (k+M)|.
  \]
  Therefore,
  \[
    |f(x)|
    \leq C_{0} \cdot \left( 1+|x| \right)^{-q}
    \leq C_{0} \cdot \left(
                       1 + \frac{1}{2} \, |\beta_{1} \cdot (k+M)|
                     \right)^{-q}
    \leq C_{0} \cdot 2^{q} \cdot \left(1 + \left|\beta_{1} \cdot (k+M)\right|\right)^{-q},
  \]
  and
  \[
    \left|\gamma \cdot \left(k+M\right)\right|^{N}
    =    \left(\frac{\gamma}{\beta_{1}}\right)^{N}
         \cdot \left|\beta_{1} \cdot (k+M) \right|^{N}
    \leq \left(\frac{\gamma}{\beta_{1}}\right)^{N}
         \cdot \left( 1 + \left|\beta_{1} \cdot \left(k+M\right)\right|\right)^{N}.
  \]

  Put together, this results in
  \begin{align*}
           \left|\gamma \cdot (k+M)\right|^{N}
           \cdot\left(\int_{I_{k}} | f(x) | \, d x\right)^{\!\tau}
    & \leq \left(\frac{\gamma}{\beta_{1}}\right)^{N} \,
           \left(\frac{\beta_{2}}{\beta_{1}}\right)^{\tau}
           \cdot C_{0}^{\tau}
           \cdot 2^{\tau q}
           \cdot \left(1 + \left| \beta_{1} \cdot (k+M) \right|\right)^{N+\tau(1-q)}\\
    & \leq \left(\frac{\gamma}{\beta_{1}}\right)^{N} \,
           \left(\frac{\beta_{2}}{\beta_{1}}\right)^{\tau}
           \cdot C_{0}^{\tau}
           \cdot 2^{\tau q}
           \cdot \left(1+\left| \beta_{1} \cdot (k+M) \right|\right)^{-2},
  \end{align*}
  where, in the last step, we recalled that $q \geq 1 + \tau^{-1} (N+2)$,
  whence $N + \tau(1-q) \leq -2$.

  Finally, applying Lemma~\ref{lem:SumEstimate}, we conclude that
  \begin{align*}
     \sum_{\substack{k \in \Z \text{ with}\\
                     \beta_{1} \cdot (k+M) \geq 2L}}
       \left[
         \left|\gamma \cdot \left(k+M\right)\right|^{N}
         \! \cdot \left(
                    \int_{I_{k}} \left| f(x) \right| \, d x
                  \right)^{\tau}
       \right] 
   & \leq \left(\frac{\gamma}{\beta_{1}}\right)^{N}
          \left(\frac{\beta_{2}}{\beta_{1}}\right)^{\tau}
          \cdot C_{0}^{\tau}
          \cdot 2^{\tau q}
          \cdot \sum_{k \in \Z}
                  \left(
                    1 + \left|\beta_{1} \cdot (k+M)\right|
                  \right)^{-2}\\
   & \leq \left(\frac{\gamma}{\beta_{1}}\right)^{N} \,
          \left(\frac{\beta_{2}}{\beta_{1}}\right)^{\tau}
          \cdot C_{0}^{\tau}
          \cdot 2^{\tau q+4}
          \cdot \left(1 + \frac{1}{\beta_{1}}\right) \, .
  \end{align*}

  \medskip{}

  \emph{Case 2:} $k$ is such that $0 \leq \beta_{1} \cdot (k+M) \leq 2 L$,
  or equivalently $0 \leq k+M \leq 2 \cdot \frac{L}{\beta_{1}}$.
  Therefore, $k$ can take at most $1+2\cdot\frac{L}{\beta_{1}}$ different values.

  Furthermore, from Equation~\eqref{eq:MeasureOfInterval} we infer that
  \[
    \lambda (I_{k})
    =    2L + \frac{\beta_{2}-\beta_{1}}{\beta_{1}} \cdot \beta_{1} \cdot (k+M)
    \leq 2L + \frac{\beta_{2} - \beta_{1}}{\beta_{1}} \cdot 2L
    = 2 \cdot \frac{\beta_2}{\beta_1} \cdot L.
  \]
  Moreover, $|f(x)| \leq C_{0}$ and
  \[
    \left| \gamma \cdot (k+M) \right|^{N}
    \leq \left( \frac{\gamma}{\beta_{1}} \cdot \beta_{1} \cdot (k+M) \right)^{N}
    \leq \left( \frac{\gamma}{\beta_{1}} \right)^{N} \cdot (2 L)^{N}.
  \]

  Put together, this results in
  \[
     \sum_{\substack{k \in \Z \text{ with} \\ 0 \leq \beta_{1} \cdot (k+M) \leq 2L}} \!\!
       \left[
         \left|\gamma \cdot (k+M)\right|^{N}
         \cdot \left(
                 \int_{I_{k}}
                   \left| f(x) \right|
                 \, d x
               \right)^{\!\tau}
       \,\right] 
     \leq \left(1 + 2 \frac{L}{\beta_{1}}\right)
          \cdot \left(\frac{\gamma}{\beta_{1}}\right)^{N}
          \cdot \left(2 L\right)^{N} \, C_{0}^{\tau}
          \cdot \left(2\cdot\frac{\beta_{2}}{\beta_{1}}\cdot L\right)^{\tau}.
  \]

  Finally, combining the estimates obtained in Cases 1 and 2 results in the claimed estimate
  \begin{align*}
   & \sum_{\substack{k \in \Z \\ \text{with } k + M \geq 0}}
       \left[
         \left|\gamma \cdot (k+M)\right|^{N}
         \cdot \left(
                 \int_{I_{k}}
                   \left|f\left(x\right)\right|
                 \, d x
               \right)^{\!\tau}
       \,\right]\\
   & \leq \left(\frac{\gamma}{\beta_{1}}\right)^{N} \,
            \left(\frac{\beta_{2}}{\beta_{1}}\right)^{\tau}
            \cdot C_{0}^{\tau}
            \cdot 2^{\tau q+4}
            \cdot \left(1+\frac{1}{\beta_{1}}\right)
          + \left(1 + 2 \frac{L}{\beta_{1}}\right)
            \cdot \left(\frac{\gamma}{\beta_{1}}\right)^{N}
            \cdot \left(2 L\right)^{N}
            \cdot C_{0}^{\tau}
            \cdot \left(2\cdot\frac{\beta_{2}}{\beta_{1}}\cdot L\right)^{\tau}\\
   & \leq \left(\frac{\beta_{2}}{\beta_{1}}\right)^{\tau} \,
            \left(\frac{\gamma}{\beta_{1}}\right)^{N}
            \cdot C_{0}^{\tau}\cdot\left[2^{\tau q+4}
            \cdot \left(1+\frac{1}{\beta_{1}}\right)
          + \left(1+2 \frac{L}{\beta_{1}}\right)
            \cdot \left(2 L\right)^{N+\tau}\right] \\
   & \leq 2^{4+N+\tau+\tau q}
          \cdot \left(\frac{\beta_{2}}{\beta_{1}}\right)^{\tau} \,
          \left(\frac{\gamma}{\beta_{1}}\right)^{N}
          \cdot C_{0}^{\tau}
                \cdot \left[
                        \left(1+\frac{1}{\beta_{1}}\right)
                        + \left(1+\frac{L}{\beta_{1}}\right) \cdot L^{N+\tau}
                      \right] \\
   & \leq 2^{4+N+\tau+\tau q}
          \cdot \left(\frac{\beta_{2}}{\beta_{1}}\right)^{\tau} \,
          \left(\frac{\gamma}{\beta_{1}}\right)^{N}
          \cdot C_{0}^{\tau}
          \cdot \left(1+L^{\tau + N}\right)
                \left(1+\frac{L+1}{\beta_{1}}\right).
  \qedhere
  \end{align*}
\end{proof}

Finally, we also prove Corollary \ref{cor:MainLemmaNegativVersion}.

\begin{proof}[Proof of Corollary \ref{cor:MainLemmaNegativVersion}]
  Let $\widetilde{\beta}_2 := \beta_1$, $\widetilde{\beta}_1 := \beta_2$, $\widetilde{M} := -M$ and
  $\widetilde{f} : \R \to \CC, x \mapsto f(-x)$.
  Together with $f$, $\widetilde{f}$ also satisfies \eqref{eq:MainLemmaFDecayAssumption}.
  Furthermore, $0 < \widetilde{\beta}_1 \leq \widetilde{\beta}_2$.
  Thus, after the substitution $y = -x$ and the change of summation index
  $l = -k$, we can apply Lemma \ref{lem:MainLemma} (with
  $\widetilde{\beta}_1, \widetilde{\beta}_2$ instead of $\beta_1, \beta_2$
  and with $\widetilde{M}$ instead of $M$) to obtain
  \begin{align*}
    & \sum_{\substack{k \in \Z \text{ with} \\ k + M \leq 0}}
      \left[
        |\gamma \cdot (k+M)|^N
        \cdot \left(
                \int_{\beta_1 \cdot (k+M) - L}^{\beta_2 \cdot (k+M) + L}
                  |f(x)|
                \, dx
              \right)^{\!\!\tau} \,
      \right] \\
    & = \sum_{\substack{k \in \Z \text{ with} \\
              - [(-k) + \widetilde{M}] \leq 0}}
        \left[
          |\gamma \cdot \big( (-k) + \widetilde{M}\big)|^N
          \cdot \left(
                  \int_{\beta_2 \cdot ( (-k) + \widetilde{M} \, ) - L}
                      ^{\beta_1 \cdot ( (-k) + \widetilde{M} \, ) + L}
                    |\widetilde{f}(y)|
                  \, dy
                \right)^{\!\!\tau} \,
        \right] \\
    & = \sum_{\substack{\ell \in \Z \text{ with} \\
              \ell + \widetilde{M} \geq 0}}
        \left[
          |\gamma \cdot ( \ell + \widetilde{M})|^N
          \cdot \left(
                  \int_{\widetilde{\beta_1} \cdot ( \ell + \widetilde{M}) - L}
                      ^{\widetilde{\beta_2} \cdot ( \ell + \widetilde{M}) + L}
                    |\widetilde{f}(y)|
                  \, dy
                \right)^{\!\!\tau} \,
        \right] \\
    & \leq C \cdot \left(
                     \widetilde{\beta}_2 \, \big/ \, \widetilde{\beta}_1
                   \right)^\tau
             \cdot \left( \gamma \, \big/ \, \widetilde{\beta}_1 \right)^N
             \cdot C_0^\tau
             \cdot (1 + L^{\tau + N})
             \cdot \left(1 + \frac{L+1}{\widetilde{\beta}_1}\right) \, ,
  \end{align*}
  which easily yields the claim.
\end{proof}

\section{Estimates for Sine and Cosine}
\label{sec:TrigonometricLinearBounds}

\noindent In this appendix, we first state and prove linear bounds of the sine and cosine
on the interval $[0, \frac{\pi}{2}]$.
Second, we prove an elementary relation between the absolute values of sine and cosine.
Finally, we prove a quadratic lower bound for the cosine.
Even though these bounds are probably well-known,
we prefer to provide a proof, since they play an important role
in the proofs of Proposition~\ref{prop:WavePacketCoveringSubordinateness}
and of Lemmas~\ref{lem:IntervalEstimateLSummation} and
\ref{lem:IntervalEstimateLPrimeSummation}
and thus in our proof of the existence of  Banach frames and atomic
decompositions for the wave packet smoothness spaces.

\medskip{}

First, we show that
\begin{equation}
  \frac{2}{\pi} \cdot \phi \leq \sin \phi \leq \phi
  \qquad \forall \, \phi \in \big[0, \tfrac{\pi}{2}\big] \, .
  \label{eq:SineLinearBound}
\end{equation}
Indeed, the upper bound follows from the stronger estimate $|\sin \phi| \leq |\phi|$
for all $\phi \in \R$, which results from
$|\tfrac{d}{d\phi} \sin \phi| = |\cos \phi| \leq 1$ combined with $\sin 0 = 0$.

To estimate the lower bound, we note that
$\frac{d^2}{d \phi^2} \sin \phi = - \sin \phi \leq 0$ on $[0, \frac{\pi}{2}]$,
so that the sine is concave on this interval.
This together with $\lambda := \frac{2}{\pi} \cdot \phi \in [0,1]$ for
$\phi \in [0, \frac{\pi}{2}]$, implies, as claimed, that
\[
       \sin \phi
  =    \sin \Big( (1-\lambda) \cdot 0 + \lambda \cdot \frac{\pi}{2} \Big)
  \geq (1-\lambda) \cdot \sin 0 + \lambda \cdot \sin \frac{\pi}{2}
  =    \lambda
  =    \frac{2}{\pi} \cdot \phi \, .
\]

\medskip{}

Next, we show that
\begin{equation}
  1 - \frac{2}{\pi} \cdot \phi
  \leq \cos \phi
  \leq \frac{\pi}{2} \cdot \Big( 1 - \frac{2}{\pi} \cdot \phi \Big)
  \qquad \forall \, \phi \in \big[ 0, \tfrac{\pi}{2} \big] \, .
  \label{eq:CosineLinearBound}
\end{equation}
To see this, we recall that $\cos \phi = \cos(-\phi) = \sin(\tfrac{\pi}{2} - \phi)$
and then apply Equation~\eqref{eq:SineLinearBound},
noting that $\tfrac{\pi}{2} - \phi \in [0, \tfrac{\pi}{2}]$ since $\phi \in [0, \tfrac{\pi}{2}]$.
%

\medskip{}
Next, we show that
\begin{equation}
  |\sin \phi| \geq 1 - |\cos \phi|
  \quad \text{and} \quad
  |\cos \phi| \geq 1 - |\sin \phi|
  \qquad \forall \, \phi \in \R \, .
  \label{eq:CosineSineBound}
\end{equation}
To see this, we recall that $\sin^2 \phi + \cos^2 \phi = 1$.
Thus, it suffices to prove that
\[
  |x| \geq 1 - |y| \qquad \forall \, x,y \in \R \text{ with } x^2 + y^2 = 1 \, .
\]
We note that $|x|, |y| \leq 1$ because $x^2 + y^2 = 1$.
Thus, both sides of the desired inequality are non-negative, so that the
inequality is equivalent to
\[
  |x|^2 \overset{!}{\geq} (1 - |y|)^2
  \Longleftrightarrow
  x^2 \overset{!}{\geq} 1 - 2 |y| + y^2
  \Longleftrightarrow
  0 \overset{!}{\geq} 1 - x^2 - 2 |y| + y^2
    =                 2 (y^2 - |y|)
    =                 2 |y| \cdot (|y| - 1) \, .
\]
This last estimate is satisfied since $|y| \leq 1$.

\medskip{}

Finally, we establish the quadratic lower bound
\begin{equation}
  \cos \theta \geq 1 - \frac{\theta^2}{2}
  \qquad \forall \, \theta \in \R.
  \label{eq:CosineQuadraticLowerBound}
\end{equation}
To prove this, we first of all note that both sides of the inequality are even functions, so that it is enough
to consider the case $\theta \geq 0$.
Next, we also note that
\(
  \frac{d}{d \theta} (\frac{\theta^2}{2} + \cos \theta - 1)
  = \theta - \sin \theta
  \geq \theta - |\sin \theta|
  \geq \theta - |\theta|
  = 0
\)
for all $\theta \geq 0$.
Hence, we see that, as claimed,
$\frac{\theta^2}{2} + \cos \theta - 1 \geq \frac{0^2}{2} + \cos 0 - 1 = 0$ for all $\theta \geq 0$.

\section{Proof of Lemmas \ref{lem:InclusionLemma2},
         \ref{lem:IntervalEstimateLSummation},
         and \ref{lem:IntervalEstimateLPrimeSummation}}
\label{sec:IntervalInclusionProofs}

\begin{proof}[Proof of Lemma~\ref{lem:InclusionLemma2}]
  First of all, elementary properties of the sine and cosine imply that
  \[
    \cos \vartheta_{i_0 ,i_0 '} =
    \begin{cases}
      \phantom{-} \cos \theta_{i_0, i_0 '} \, , 
      & \text{ if } \vartheta_{i_0, i_0 '} \in [0,\frac{\pi}{2}) \, , \\
      - \sin \theta_{i_0, i_0 '} \, , 
      & \text{ if } \vartheta_{i_0, i_0 '} \in [\frac{\pi}{2}, \pi) \, , \\
      - \cos \theta_{i_0, i_0 '} \, , 
      & \text{ if } \vartheta_{i_0, i_0 '} \in [\pi, \frac{3}{2} \pi) \, , \\
      \phantom{-} \sin \theta_{i_0, i_0 '} \, , 
      & \text{ if } \vartheta_{i_0, i_0 '} \in [\frac{3}{2} \pi, 2 \pi)
    \end{cases}
    \! \quad \text{and} \quad \!
    \sin \vartheta_{i_0, i_0 '} =
    \begin{cases}
      \phantom{-} \sin \theta_{i_0, i_0 '} \, , 
      & \text{ if } \vartheta_{i_0, i_0 '} \in [0,\frac{\pi}{2}) \, , \\
      \phantom{-} \cos \theta_{i_0, i_0 '} \, , 
      & \text{ if } \vartheta_{i_0, i_0 '} \in [\frac{\pi}{2}, \pi) \, , \\
      - \sin \theta_{i_0, i_0 '} \, , 
      & \text{ if } \vartheta_{i_0, i_0 '} \in [\pi, \frac{3}{2} \pi) \, , \\
      - \cos \theta_{i_0, i_0 '} \, , 
      & \text{ if } \vartheta_{i_0, i_0 '} \in [\frac{3}{2} \pi, 2 \pi) \, .
    \end{cases}
  \]
  Next, Lemma \ref{lem:InclusionLemma1} shows that, for $\xi \in Q_{j',m',0}$,
  $x_{i_\ast '}^{-} \leq \xi_1 \leq x_{i_\ast '}^{+}$ and $|\xi_2| \leq y_{j'}$.
  Furthermore, by definition of $R_{i_0, i_0 '}$,
  \[
    (R_{i_0, i_0 '} \, \xi)_{1}
    = \xi_1 \cdot \cos \vartheta_{i_0, i_0 '}
      - \xi_2 \cdot \sin \vartheta_{i_0, i_0 '}
    \quad \text{and} \quad
    (R_{i_0, i_0 '} \, \xi)_{2}
    = \xi_1 \cdot \sin \vartheta_{i_0, i_0 '}
      + \xi_2 \cdot \cos \vartheta_{i_0, i_0 '} \, .
  \]
  Finally, since $\theta_{i_0, i_0 '} \in [0, \pi/2)$,
  $\sin \theta_{i_0, i_0 '} \geq 0$ and $\cos \theta_{i_0, i_0 '} \geq 0$. By combining these observations,
  we see that \eqref{eq:SpeciallyRotatedBaseSet} is true.
  Indeed, we distinguish four cases:

  \noindent
  \emph{Case 1:} $\vartheta_{i_0, i_0 '} \in [0,\frac{\pi}{2})$.
  In this case, $\vartheta_{i_0, i_0 '} = \theta_{i_0, i_0 '}$ and hence
  \[
    u_{i_0, i'}^{-}
    = x_{i_\ast '}^{-} \cdot \cos \theta_{i_0, i_0 '}
      - y_{j'} \cdot \sin \theta_{i_0, i_0 '}
    \leq (R_{i_0, i_0 '} \, \xi)_1
    \leq x_{i_\ast '}^{+} \cdot \cos \theta_{i_0, i_0 '}
         + y_{j'} \cdot \sin \theta_{i_0, i_0 '}
    =    u_{i_0, i'}^{+}
  \]
  and
  \[
    v_{i_0, i'}^{-}
    = x_{i_\ast '}^{-} \cdot \sin \theta_{i_0, i_0 '}
      - y_{j'} \cdot \cos \theta_{i_0, i_0 '}
    \leq (R_{i_0, i_0 '} \, \xi)_2
    \leq x_{i_\ast '}^{+} \cdot \sin \theta_{i_0, i_0 '}
         + y_{j'} \cdot \cos \theta_{i_0, i_0 '}
    =    v_{i_0, i'}^{+} \, .
  \]

  \medskip{}

  \noindent
  \emph{Case 2:} $\vartheta_{i_0, i_0 '} \in [\frac{\pi}{2}, \pi)$.
  In this case,
  $(R_{i_0, i_0 '} \, \xi)_1 = - \xi_1 \cdot \sin \theta_{i_0, i_0 '}
                               - \xi_2 \cdot \cos \theta_{i_0, i_0 '}$
  and thus
  \[
         u_{i_0, i'}^{-}
    =    - x_{i_\ast '}^{+} \cdot \sin \theta_{i_0, i_0 '}
         - y_{j'} \cdot \cos \theta_{i_0, i_0 '}
    \leq (R_{i_0, i_0 '} \, \xi)_1
    \leq - x_{i_\ast '}^{-} \cdot \sin \theta_{i_0, i_0 '}
         + y_{j'} \cdot \cos \theta_{i_0, i_0 '}
    =    u_{i_0, i'}^{+} \, .
  \]
  Likewise,
  $(R_{i_0, i_0 '} \, \xi)_2 = \xi_1 \cdot \cos \theta_{i_0, i_0 '}
                               - \xi_2 \cdot \sin \theta_{i_0, i_0 '}$
  and thus
  \[
         v_{i_0, i'}^{-}
    =    x_{i_\ast '}^{-} \cdot \cos \theta_{i_0, i_0 '}
         - y_{j'} \cdot \sin \theta_{i_0, i_0 '}
    \leq (R_{i_0, i_0 '} \, \xi)_2
    \leq x_{i_\ast '}^{+} \cdot \cos \theta_{i_0, i_0 '}
         + y_{j'} \cdot \sin \theta_{i_0, i_0 '}
    =    v_{i_0, i'}^{+} \, .
  \]

  \medskip{}

  \noindent
  \emph{Case 3:} $\vartheta_{i_0, i_0 '} \in [\pi, \frac{3}{2} \pi)$.
  In this case,
  $(R_{i_0, i_0 '} \, \xi)_{1} = - \xi_1 \cdot \cos \theta_{i_0, i_0 '}
                                 + \xi_2 \cdot \sin \theta_{i_0, i_0 '}$
  and thus
  \[
         u_{i_0, i'}^{-}
    =    - x_{i_\ast '}^{+} \cdot \cos \theta_{i_0, i_0 '}
         - y_{j'} \cdot \sin \theta_{i_0, i_0 '}
    \leq (R_{i_0, i_0 '} \, \xi)_1
    \leq - x_{i_\ast '}^{-} \cdot \cos \theta_{i_0, i_0 '}
         + y_{j'} \cdot \sin \theta_{i_0, i_0 '}
    =    u_{i_0 ,i'}^{+} \, .
  \]
  Likewise, $(R_{i_0, i_0 '} \, \xi)_2 = - \xi_1 \cdot \sin \theta_{i_0, i_0 '}
                                         - \xi_2 \cdot \cos \theta_{i_0, i_0 '}$
  and thus
  \[
         v_{i_0, i'}^{-}
    =    - x_{i_\ast '}^{+} \cdot \sin \theta_{i_0, i_0 '}
         - y_{j'} \cdot \cos \theta_{i_0, i_0 '}
    \leq (R_{i_0, i_0 '} \, \xi)_2
    \leq - x_{i_\ast '}^{-} \cdot \sin \theta_{i_0, i_0 '}
         + y_{j'} \cdot \cos \theta_{i_0, i_0 '}
    =    v_{i_0, i'}^{+} \, .
  \]

  \medskip{}

  \noindent
  \emph{Case 4:} $\vartheta_{i_0, i_0 '} \in [\frac{3}{2} \pi, 2\pi)$.
  In this case, $(R_{i_0, i_0 '} \, \xi)_{1}
  = \xi_1 \cdot \sin \theta_{i_0, i_0 '} + \xi_2 \cdot \cos \theta_{i_0, i_0 '}$,
  and thus
  \[
         u_{i_0, i'}^{-}
    =    x_{i_\ast '}^{-} \cdot \sin \theta_{i_0, i_0 '}
         - y_{j'} \cdot \cos \theta_{i_0, i_0 '}
    \leq (R_{i_0, i_0 '} \, \xi)_{1}
    \leq x_{i_\ast '}^{+} \cdot \sin \theta_{i_0, i_0 '}
         + y_{j'} \cdot \cos \theta_{i_0, i_0 '}
    =    u_{i_0, i'}^{+} \, .
  \]
  Likewise, $(R_{i_0, i_0 '} \, \xi)_{2}
  = - \xi_1 \cdot \cos \theta_{i_0, i_0 '} + \xi_2 \cdot \sin \theta_{i_0, i_0 '}$
  and thus
  \[
         v_{i_0, i'}^{-}
    =    - x_{i_\ast '}^{+} \cdot \cos \theta_{i_0, i_0 '}
         - y_{j'} \cdot \sin \theta_{i_0, i_0 '}
    \leq (R_{i_0, i_0 '} \, \xi)_{2}
    \leq - x_{i_\ast '}^{-} \cdot \cos \theta_{i_0, i_0 '}
          + y_{j'} \cdot \sin \theta_{i_0, i_0 '}
    =    v_{i_0, i'}^{+} \, .
  \]

  Finally, \eqref{eq:MainDomainCartesianInclusion} results from combination of
  \eqref{eq:SpeciallyRotatedBaseSet} and \eqref{eq:SeriesMainDomainEstimate1},
  since $A_{j} = \mathrm{diag} (2^{\alpha j}, 2^{\beta j})$.
\end{proof}

\begin{proof}[Proof of Lemma~\ref{lem:IntervalEstimateLSummation}]
  First we recall from Lemma~\ref{lem:InclusionLemma1} that
  \begin{equation}
    \frac{1}{4} \cdot 2^{j'}
    \leq x_{i_\ast '}^{-}
    \leq x_{i_\ast '}^{+}
    \leq 4 \cdot 2^{j'}
    \quad \text{and} \quad
    y_{j'} = 2^{\beta j' + 1} \, .
    \label{eq:XiBound}
  \end{equation}
  Next, we recall Equations \eqref{eq:AngleDifference},
  \eqref{eq:NormalizedAngleDifference}, \eqref{eq:RenormalizedAngleDifference}
  and \eqref{eq:RotationMatrix} to see that, for $\ell \in J_{j}^{\iota,k}$,
  \begin{equation}
    \begin{split}
          \theta_{i_0, i_0'}
        = \vartheta_{i_0, i_0'} - \iota \cdot \frac{\pi}{2}
      & = \vartheta_{i_0, i_0'}^{(0)} + 2\pi k - \iota \cdot \frac{\pi}{2}
        = 2\pi k
          + \Theta_{j',\ell'}
          - \iota \cdot \frac{\pi}{2}
          - \Theta_{j,\ell} \\
      & = 2\pi k
          + \Theta_{j',\ell'}
          - \iota \cdot \frac{\pi}{2}
          - \frac{2\pi}{N} \cdot 2^{(\beta - 1) j} \, \ell
        = \frac{2\pi}{N}
          \cdot 2^{(\beta - 1) j}
          \cdot \big(S_{k,\iota,j}^{(0)} - \ell \big) \, ,
    \end{split}
    \label{eq:AngleDifferenceInTermsOfL}
  \end{equation}
  and hence
  \begin{equation}
    1 - \frac{2}{\pi} \cdot \theta_{i_0, i_0 '}
    = \frac{4}{N} \cdot 2^{(\beta - 1) j}
                  \cdot \big( S_{k,\iota,j}^{(1)} + \ell \big) \,
    \label{eq:OneMinusAngleDifferenceInTermsOfL}
  \end{equation}
  where $S_{k,\iota,j}^{(0)} := \frac{N}{2\pi}
                                \cdot 2^{(1-\beta) j}
                                \cdot (
                                       2\pi k + \Theta_{j',\ell'}
                                       - \iota \cdot \frac{\pi}{2}
                                      )$
  and $S_{k,\iota,j}^{(1)}
  := \frac{N}{4} \cdot 2^{(1-\beta) j} - S_{k,\iota,j}^{(0)}$.
  In particular, since $0 \leq \theta_{i_0, i_0'} < \frac{\pi}{2}$,
  we see that $S_{k,\iota,j}^{(0)} - \ell \geq 0$ and
  $S_{k,\iota,j}^{(1)} + \ell \geq 0$ for all $\ell \in J_{j}^{\iota,k}$.

  As a further preparation, we recall from Appendix~\ref{sec:TrigonometricLinearBounds} the estimates
  \begin{equation}
    \frac{2}{\pi} \cdot \phi \leq \sin \phi \leq \phi
    \quad \text{and} \quad
    1 - \frac{2}{\pi} \cdot \phi
    \leq \cos \phi
    \leq \frac{\pi}{2} \cdot \Big( 1 - \frac{2}{\pi} \cdot \phi \Big)
    \qquad \forall \, \phi \in \Big[ 0, \frac{\pi}{2} \Big] \, .
    \label{eq:TrigonometricLinearBounds}
  \end{equation}

  Finally, to actually prove the claim, we distinguish the four possible values of $\iota$.
  \medskip{}

  \emph{Case 1:} $\iota = 0$.
  Let $S_{k,\iota,j} := S_{k,\iota,j}^{(0)}$ and $\nu_\iota := -1$.
  With this definition, Equation~\eqref{eq:AngleDifferenceInTermsOfL} shows that
  \(
    \theta_{i_0, i_0 '}
    = \frac{2\pi}{N} \cdot 2^{(\beta - 1) j} \cdot (S_{k,\iota,j} + \nu_\iota \, \ell)
  \)
  and $S_{k,\iota,j} + \nu_{\iota} \, \ell \geq 0$ for all $\ell \in J_{j}^{\iota,k}$.
  %

  Next, recalling the definition of $v_{i_0, i'}^{\pm}$ (see Equation~\eqref{eq:VBoundsDefinition})
  and combining Equations~\eqref{eq:XiBound} , \eqref{eq:TrigonometricLinearBounds}
  and the identity
  \(
    \theta_{i_0, i_0 '}
    = \frac{2\pi}{N} \cdot 2^{(\beta - 1) j}
                     \cdot (S_{k,\iota,j} + \nu_\iota \, \ell)
  \),
  we see that, for any $\ell \in J_j^{\iota,k}$,
  \[
    v_{i_0, i'}^{+}
    \leq 2^{\beta j' + 1} + 4 \cdot 2^{j'} \cdot \sin \theta_{i_0, i_0 '}
    \leq 2^{\beta j' + 1} + 4 \cdot 2^{j'} \cdot \theta_{i_0, i_0'}
    =  2^{\beta j' + 1}
       + \frac{8 \pi}{N}
         \cdot 2^{j' + (\beta - 1) j}
         \cdot \big( S_{k,\iota,j} + \nu_\iota \, \ell \big)
  \]
  and
  \[
    v_{i_0, i'}^{-}
    \geq - 2^{\beta j' + 1}
         + \frac{1}{4} \cdot 2^{j'} \cdot \sin \theta_{i_0, i_0'}
    \geq - 2^{\beta j' + 1}
         + \frac{1}{2\pi} \cdot 2^{j'} \cdot \theta_{i_0, i_0'}
    =    - 2^{\beta j' + 1}
         + N^{-1}
           \cdot 2^{j' + (\beta - 1) j}
           \cdot \big( S_{k,\iota,j} + \nu_\iota \, \ell \big) \, .
  \]
  Since
  \(
    I_2^{(i,i')} = [2^{-\beta j} \cdot v_{i_0, i'}^{-} \,,\,
                    2^{-\beta j} \cdot v_{i_0, i'}^{+}]
  \),
  this proves the desired estimate for $I_2^{(i,i')}$.
  Here we noted again that $S_{k,\iota,j} + \nu_\iota \, \ell \geq 0$.

  Finally, \eqref{eq:AngleDifferenceInTermsOfL}, $\alpha \leq 1$
  and $\vartheta_{i_0, i_0'} = \theta_{i_0, i_0 '}$ implies that, for
  $\ell \in J_{j}^{\iota,k}$,
  \[
    2^{\alpha j' - \beta j} \cdot |\sin \vartheta_{i_0, i_0'}|
    \leq 2^{\alpha j' - \beta j} \cdot \theta_{i_0, i_0 '}
    =    \frac{2\pi}{N}
         \cdot 2^{(\alpha - 1) j'}
         \cdot 2^{j' - j}
         \cdot \big( S_{k,\iota,j} + \nu_\iota \, \ell \big)
    \leq 2\pi \cdot \beta_1 \cdot \big( S_{k,\iota,j} + \nu_\iota \, \ell \big)
    \, .
  \]

  \medskip{}

  \emph{Case 2:} $\iota = 1$.
  Let $S_{k,\iota,j} := S_{k,\iota,j}^{(1)}$ and $\nu_\iota := 1$.
  On the one hand, as seen after
  Equation~\eqref{eq:OneMinusAngleDifferenceInTermsOfL}, this ensures that
  $S_{k,\iota,j} + \nu_{\iota} \, \ell \geq 0$ for all
  $\ell \in J_{j}^{\iota,k}$.

  On the other hand, by combining Equations \eqref{eq:XiBound},
  \eqref{eq:OneMinusAngleDifferenceInTermsOfL},
  and \eqref{eq:TrigonometricLinearBounds},
  and by recalling the definition of $v_{i_0, i'}^{\pm}$, we see for any
  $\ell \in J_{j}^{\iota,k}$ that
  \[
    v_{i_0, i'}^{+}
    \leq 2^{\beta j' + 1} + 4 \cdot 2^{j'} \cdot \cos \theta_{i_0, i_0 '}
    \leq 2^{\beta j' + 1}
           + 2\pi \cdot 2^{j'}
             \cdot \Big( 1 - \frac{2}{\pi} \theta_{i_0, i_0 '} \Big)
    =    2^{\beta j' + 1}
           + \frac{8\pi}{N} \cdot 2^{j' + (\beta - 1) j}
             \cdot \big( S_{k,\iota,j} + \nu_\iota \, \ell \big)
  \]
  and
  \[
    v_{i_0, i '}^{-}
    \geq - 2^{\beta j' + 1} + \frac{2^{j'}}{4} \cdot \cos \theta_{i_0, i_0 '}
    \geq - 2^{\beta j' + 1}
         + \frac{2^{j'}}{4}
           \cdot \Big(
                   1 - \frac{2}{\pi} \theta_{i_0, i_0 '}
                 \Big)
    =    - 2^{\beta j' + 1}
         + N^{-1} \cdot 2^{j' + (\beta - 1) j}
                  \cdot \big( S_{k,\iota,j} + \nu_\iota \, \ell \big) \, .
  \]
  Just as in Case 1, this yields the desired estimate for $I_2^{(i,i')}$.

  Finally, since $\vartheta_{i_0, i_0 '} = \theta_{i_0, i_0 '} + \pi / 2$
  for $\ell \in J_j^{\iota,k}$,
  we see that $\sin \vartheta_{i_0, i_0 '} = \cos \theta_{i_0, i_0 '}$.
  This combined with \eqref{eq:TrigonometricLinearBounds}
  and \eqref{eq:OneMinusAngleDifferenceInTermsOfL} results in
  \[
         2^{\alpha j' - \beta j} \cdot |\sin \vartheta_{i_0, i_0 '}|
    \leq \frac{\pi}{2}
         \cdot 2^{\alpha j' - \beta j}
         \cdot \Big( 1 - \frac{2}{\pi} \cdot \theta_{i_0, i_0 '} \Big)
    =    \frac{2\pi}{N} \cdot 2^{(\alpha - 1) j'} \cdot 2^{j' - j}
                        \cdot (S_{k,\iota,j} + \nu_\iota \, \ell) \, .
  \]
  Just as in Case 1, this proves that
  \(
    2^{\alpha j' - \beta j} \cdot |\sin \vartheta_{i_0, i_0 '}|
    \leq 2\pi \cdot \beta_1 \cdot |S_{k,\iota,j} + \nu_\iota \, \ell|
  \)
  for $\ell \in J_{j}^{\iota,k}$.

  \medskip{}

  \emph{Case 3:} $\iota = 2$.
  Let $S_{k,\iota,j} := - S_{k,\iota,j}^{(0)}$
  and $\nu_\iota := 1$.
  On the one hand, as ca be seen from
  Equation \eqref{eq:OneMinusAngleDifferenceInTermsOfL}, this ensures that
  $S_{k,\iota,j} + \nu_\iota \, \ell \leq 0$ for all $\ell \in J_j^{\iota,k}$.

  On the other hand, recalling the definition of $v_{i_0, i'}^{\pm}$
  and combining 
  Equations~\eqref{eq:XiBound} \eqref{eq:TrigonometricLinearBounds},
  and \eqref{eq:AngleDifferenceInTermsOfL},
  we see that, for any $\ell \in J_j^{\iota,k}$,
  \[
    v_{i_0, i'}^{+}
    \leq 2^{\beta j' + 1} - \frac{2^{j'}}{4} \cdot \sin \theta_{i_0, i_0 '}
    \leq 2^{\beta j' + 1} - \frac{2^{j'}}{2 \pi} \cdot \theta_{i_0, i_0 '}
    =    2^{\beta j' + 1}
         + N^{-1}
           \cdot 2^{j' + (\beta - 1) j}
           \cdot \big(S_{k,\iota,j} + \nu_\iota \, \ell\big)
  \]
  and
  \[
    v_{i_0, i'}^{-}
    \geq - 2^{\beta j' + 1} - 4 \cdot 2^{j'} \cdot \sin \theta_{i_0, i_0'}
    \geq - 2^{\beta j' + 1} - 4 \cdot 2^{j'} \cdot \theta_{i_0, i_0'}
    =    - 2^{\beta j' + 1}
         + \frac{8\pi}{N}
           \cdot 2^{j' + (\beta - 1) j}
           \cdot \big( S_{k,\iota,j} + \nu_\iota \, \ell \big) \, .
  \]
  Given the previous two estimates and bearing in mind that
  $S_{k,\iota,j} + \nu_\iota \, \ell \leq 0$,
  we get the inclusion
  $I_2^{(i,i')} \subset \big[\beta_2 (S_{k,\iota,j} + \nu_\iota \, \ell) - L_j,
                             \beta_1 (S_{k,\iota,j} + \nu_\iota \, \ell) + L_j\big]$.

  Finally, we have
  $\vartheta_{i_0, i_0'} = \theta_{i_0, i_0 '} + \pi$ for $\ell \in J_j^{\iota,k}$ and thus
  $\sin \vartheta_{i_0, i_0 '} = - \sin \theta_{i_0, i_0 '}$, whence
  \[
      2^{\alpha j' - \beta j} \cdot |\sin \vartheta_{i_0, i_0 '}|
    \leq 2^{\alpha j' - \beta j} \cdot |\theta_{i_0, i_0 '}|
    =    2^{(\alpha - 1) j'} \cdot 2^{j' - j} \cdot \frac{2\pi}{N}
                             \cdot |S_{k,\iota,j} + \nu_\iota \, \ell| \, .
  \]
  As in the preceding cases, this shows that
  \(
    2^{\alpha j' - \beta j} \cdot |\sin \vartheta_{i_0, i_0 '}|
    \leq 2\pi \cdot \beta_1 \cdot |S_{k,\iota,j} + \nu_\iota \, \ell|
  \)
  for $\ell \in J_j^{\iota,k}$.

  \medskip{}

  \emph{Case 4:} $\iota = 3$. Let
  $S_{k,\iota,j} := - S_{k,\iota,j}^{(1)}$ and $\nu_\iota := -1$.
  On the one hand, this, as can be seen from Equation \eqref{eq:OneMinusAngleDifferenceInTermsOfL}, ensures that $S_{k,\iota,j} + \nu_\iota \, \ell \leq 0$
  for all $\ell \in J_j^{\iota,k}$.

  On the other hand, by recalling the definition of $v_{i_0, i'}^{\pm}$,
  and by combining 
  Equations~\eqref{eq:XiBound}, \eqref{eq:AngleDifferenceInTermsOfL},
  and \eqref{eq:TrigonometricLinearBounds},
  we see for any $\ell \in J_j^{\iota,k}$ that
  \[
    v_{i_0, i'}^{+}
    \leq 2^{\beta j' + 1} - \frac{2^{j'}}{4} \cdot \cos \theta_{i_0, i_0 '}
    \leq 2^{\beta j' + 1} - \frac{2^{j'}}{4}
                            \cdot \Big(
                                    1 - \frac{2}{\pi} \cdot \theta_{i_0, i_0 '}
                                  \Big)
    =    2^{\beta j' + 1}
         + N^{-1} \cdot 2^{j' + (\beta - 1) j}
           \cdot \big(S_{k,\iota,j} + \nu_\iota \, \ell \big)
  \]
  and
  \[
    v_{i_0, i'}^{-}
    \geq - 2^{\beta j' + 1} - 4 \cdot 2^{j'} \cdot \cos \theta_{i_0, i_0 '}
    \geq - 2^{\beta j' + 1}
         - 2\pi \cdot 2^{j'}
           \cdot \Big( 1 - \frac{2}{\pi} \cdot \theta_{i_0, i_0 '} \Big)
    =    - 2^{\beta j' + 1}
         + \frac{8\pi}{N} \cdot 2^{j' + (\beta - 1) j}
                          \cdot \big( S_{k,\iota,j} + \nu_\iota \, \ell \big)
    \, .
  \]
  As in Case 3, this yields the desired estimate for $I_2^{(i,i')}$.

  Finally, since $\vartheta_{i_0, i_0 '} = \theta_{i_0, i_0 '} + \frac{3}{2} \pi$
  for $\ell \in J_j^{\iota,k}$, we see that
  $\sin \vartheta_{i_0, i_0 '} = - \cos \theta_{i_0, i_0 '} \leq 0$.
  Therefore, Equations \eqref{eq:TrigonometricLinearBounds}
  and \eqref{eq:OneMinusAngleDifferenceInTermsOfL} imply that
  \[
    2^{\alpha j' - \beta j} \cdot |\sin \vartheta_{i_0, i_0 '}|
    = 2^{\alpha j' - \beta j} \cdot \cos \theta_{i_0, i_0 '}
    \leq 2^{\alpha j' - \beta j}
         \cdot \frac{\pi}{2}
         \cdot \Big( 1 - \frac{2}{\pi} \cdot \theta_{i_0, i_0 '} \Big)
    \leq - \frac{2\pi}{N}
           \cdot 2^{(\alpha - 1) j'} \cdot 2^{j' - j}
           \cdot \big( S_{k,\iota,j} + \nu_\iota \, \ell \big) \, .
  \]
  As in the previous cases, since $S_{k,\iota,j} + \nu_\iota \, \ell \leq 0$, this proves that
  \(
    2^{\alpha j' - \beta j} \cdot |\sin \vartheta_{i_0, i_0 '}|
    \leq 2\pi \cdot \beta_1 \cdot |S_{k,\iota,j} + \nu_\iota \, \ell|
  \)
  for $\ell \in J_j^{\iota,k}$.
\end{proof}

\begin{proof}[Proof of Lemma~\ref{lem:IntervalEstimateLPrimeSummation}]
  This proof is quite similar to that of
  Lemma~\ref{lem:IntervalEstimateLSummation}. Therefore we shall only outline it briefly.
  First, combining Equations~\eqref{eq:NormalizedAngleDifference},
  \eqref{eq:RenormalizedAngleDifference} and \eqref{eq:RotationMatrix} results in
  \begin{equation}
    \theta_{i_0,i_0'}
    = 2\pi k + \Theta_{j',\ell'} - \iota \cdot \frac{\pi}{2} - \Theta_{j,\ell}
    = 2\pi k
      - \Theta_{j,\ell}
      - \iota \cdot \frac{\pi}{2}
      + \frac{2\pi}{N} \cdot 2^{(\beta - 1) j'} \, \ell'
    = \frac{2\pi}{N} \cdot 2^{(\beta - 1) j'}
                     \cdot \big( S_{k, \iota, j'}^{(0)} + \ell' \big) \,
    \label{eq:AngleDifferenceInTermsOfLPrime}
  \end{equation}
  and hence
  \begin{equation}
    1 - \frac{2}{\pi} \cdot \theta_{i_0, i_0 '}
    = \frac{4}{N}
      \cdot 2^{(\beta - 1) j'}
      \cdot \big( S_{k,\iota,j'}^{(1)} - \ell' \big)
    \label{eq:OneMinusAngleDifferenceInTermsOfLPrime}
  \end{equation}
  where
  $S_{k, \iota, j'}^{(0)}
   := \frac{N}{2\pi} \cdot 2^{(1 - \beta) j'}
                     \cdot (
                            2\pi k
                            - \Theta_{j,\ell}
                            - \iota \cdot \frac{\pi}{2}
                           )$
  and $S_{k,\iota,j'}^{(1)}
  := \frac{N}{4} \cdot 2^{(1 - \beta) j'} - S_{k,\iota,j'}^{(0)}$.

  Since $0 \leq \theta_{i_0, i_0 '} < \frac{\pi}{2}$,
  Equations \eqref{eq:AngleDifferenceInTermsOfLPrime} and
  \eqref{eq:OneMinusAngleDifferenceInTermsOfLPrime} show that
  $S_{k,\iota,j'}^{(0)} + \ell' \geq 0$ and
  $S_{k,\iota,j'}^{(1)} - \ell' \geq 0$ for all $\ell \in J_{j'}^{\iota,k}$.

  To prove the lemma, we now distinguish the four possible values of $\iota$.

  \medskip{}

  \emph{Case 1:} $\iota = 0$.
  Let $S_{k,\iota,j'} := S_{k,\iota, j'}^{(0)}$ and
  $\nu_\iota := 1$.
  On the one hand, this, as can be seen from \eqref{eq:OneMinusAngleDifferenceInTermsOfLPrime},
 ensures that $S_{k,\iota, j'} + \nu_\iota \, \ell' \geq 0$ for all
  $\ell \in J_{j'}^{\iota,k}$.

  On the other hand, precisely as in Case 1 in the proof of
  Lemma \ref{lem:IntervalEstimateLSummation}, we see that,
  for $\ell' \in J_{j'}^{\iota,k}$, $v_{i_0, i'}^{+}
  \leq 2^{\beta j' + 1} + 4 \cdot 2^{j'} \cdot \theta_{i_0, i_0 '}$
  and $v_{i_0, i'}^{-}
  \geq - 2^{\beta j' + 1} + \frac{2^{j'}}{2\pi} \cdot \theta_{i_0, i_0 '}$.
  Given \eqref{eq:AngleDifferenceInTermsOfLPrime}, our choice of
  $S_{k,\iota,j'}$ and $\nu_\iota$, this implies that
  \[
    v_{i_0, i'}^{+}
    \leq 2^{\beta j' + 1}
         + \frac{8\pi}{N}
           \cdot 2^{\beta j'}
           \cdot \big( S_{k, \iota, j'} + \nu_\iota \, \ell' \big)
    \quad \text{and} \quad
    v_{i_0, i'}^{-}
    \geq - 2^{\beta j' + 1}
         + N^{-1}
           \cdot 2^{\beta j'}
           \cdot \big( S_{k, \iota, j'} + \nu_\iota \, \ell' \big) \, .
  \]
  Combining these estimates results in the stated inclusion for $I_2^{(i,i')}$.

  Finally, since $\theta_{i_0, i_0'} = \vartheta_{i_0, i_0 '}$ for
  $\ell' \in J_{j'}^{\iota,k}$ and $\alpha \leq 1$, we see that
  \[
    2^{\alpha j' - \beta j} \cdot |\sin \vartheta_{i_0, i_0'}|
    \leq 2^{\alpha j' - \beta j} \cdot \theta_{i_0, i_0 '}
    =    \frac{2\pi}{N}
         2^{\beta (j' - j)} \cdot 2^{(\alpha - 1) j'}
         \cdot \big( S_{k,\iota,j'} + \nu_\iota \, \ell' \big)
    \leq 2\pi \cdot \beta_1 \cdot \big( S_{k,\iota,j'} + \nu_\iota \, \ell' \big)
    \, .
  \]

  \medskip{}

  \emph{Case 2:} $\iota = 1$.
  Let $S_{k,\iota,j'} := S_{k,\iota,j'}^{(1)}$ and
  $\nu_\iota := -1$.
  On the one hand, this, as can be seen from \eqref{eq:OneMinusAngleDifferenceInTermsOfLPrime},
  ensures that $S_{k,\iota,j'} + \nu_\iota \, \ell' \geq 0$ for all
  $\ell \in J_{j'}^{\iota,k}$.

  On the other hand, precisely as in Case 2 of the proof of Lemma \ref{lem:IntervalEstimateLSummation}, we see that, for
  $\ell' \in J_{j'}^{\iota,k}$, $v_{i_0, i'}^{+}
  \leq 2^{\beta j' + 1} + 2\pi \cdot 2^{j'}
                          \cdot \big(
                                  1 - \frac{2}{\pi} \theta_{i_0, i_0 '}
                                \big)$,
  and $v_{i_0, i'}^{-} \geq - 2^{\beta j' + 1}
                              + \frac{2^{j'}}{4}
                                \cdot \big(
                                        1 - \frac{2}{\pi} \theta_{i_0, i_0 '}
                                      \big)$.
  Given \eqref{eq:OneMinusAngleDifferenceInTermsOfLPrime} and our choice of
  $S_{k,\iota,j'}$ and $\nu_\iota$, this implies that
  \[
    v_{i_0, i '}^{+}
    \leq 2^{\beta j' + 1}
         + \frac{8 \pi}{N}
           \cdot 2^{\beta j'}
           \cdot \big( S_{k,\iota,j'} + \nu_\iota \, \ell' \big)
    \quad \text{and} \quad
    v_{i_0, i'}^{-}
    \geq - 2^{\beta j' + 1}
         + N^{-1} \cdot 2^{\beta j'}
                  \cdot \big( S_{k,\iota,j'} + \nu_\iota \, \ell' \big) \, .
  \]
  This yields the stated inclusion for $I_{2}^{(i,i')}$.

  Finally, we see that, exactly as in Case 2 of the proof of
  Lemma \ref{lem:IntervalEstimateLSummation},
  \[
    2^{\alpha j' - \beta j} \cdot |\sin \vartheta_{i_0, i_0 '}|
    \leq \frac{\pi}{2}
         \cdot 2^{\alpha j' - \beta j}
         \cdot \Big( 1 - \frac{2}{\pi} \cdot \theta_{i_0, i_0 '} \Big) \, ,
  \]
  which, given \eqref{eq:OneMinusAngleDifferenceInTermsOfLPrime} and $\alpha \leq 1$, implies that
  \(
    2^{\alpha j' - \beta j} \cdot |\sin \vartheta_{i_0, i_0 '}|
    \leq 2\pi \cdot \beta_1 \cdot \big(S_{k,\iota,j'} + \nu_\iota \, \ell' \big)
  \)
  as in the previous case.

  \medskip{}

  \emph{Case 3:} $\iota = 2$.
  Let $S_{k,\iota,j'} := - S_{k,\iota,j'}^{(0)}$ and
  $\nu_\iota := -1$.
  On the one hand, this, as can be seen from \eqref{eq:OneMinusAngleDifferenceInTermsOfLPrime},
  ensures that $S_{k,\iota,j'} + \nu_\iota \, \ell' \leq 0$
  for all $\ell' \in J_{j'}^{\iota,k}$.

  On the other hand, precisely as in Case 3 of the proof of
  Lemma \ref{lem:IntervalEstimateLSummation}, we see that, for
  $\ell' \in J_{j'}^{\iota,k}$, $v_{i_0, i'}^{+}
   \leq 2^{\beta j' + 1} - \frac{2^{j'}}{2 \pi} \cdot \theta_{i_0, i_0 '}$
  and $v_{i_0, i'}^{-}
  \geq - 2^{\beta j' + 1} - 4 \cdot 2^{j'} \cdot \theta_{i_0, i_0'}$.
  Given \eqref{eq:AngleDifferenceInTermsOfLPrime} and our choice of
  $S_{k,\iota,j'}$ and $\nu_\iota$, this implies that
  \[
    v_{i_0, i'}^{+}
    \leq 2^{\beta j' + 1} + N^{-1} \cdot 2^{\beta j'}
                                   \cdot \big(
                                           S_{k,\iota,j'} + \nu_\iota \, \ell'
                                         \big)
    \quad \text{and} \quad
    v_{i_0, i'}^{-}
    \geq - 2^{\beta j' + 1}
         + \frac{8\pi}{N}
           \cdot 2^{\beta j'}
           \cdot \big( S_{k,\iota,j'} + \nu_\iota \, \ell' \big) \, .
  \]
  These estimates together imply the stated inclusion for $I_2^{(i,i')}$.

  Finally, as in Case 3 of the proof of
  Lemma \ref{lem:IntervalEstimateLSummation}, we see that
  $2^{\alpha j' - \beta j} \cdot |\sin \vartheta_{i_0, i_0 '}|
   \leq 2^{\alpha j' - \beta j} \cdot |\theta_{i_0, i_0 '}|$.
  Given \eqref{eq:AngleDifferenceInTermsOfLPrime} and our choice of
  $S_{k,\iota,j'}$ and $\nu_\iota$, this implies that
  \(
    2^{\alpha j' - \beta j} \cdot |\sin \vartheta_{i_0, i_0 '}|
    \leq 2\pi \cdot \beta_1 \cdot |S_{k,\iota,j'} + \nu_\iota \, \ell'|
  \),
  since $\alpha \leq 1$.

  \medskip{}

  \emph{Case 4:} $\iota = 3$.
  Let $S_{k,\iota,j'} := - S_{k,\iota,j'}^{(1)}$ and $\nu_\iota := 1$.
  On the one hand, this, as can be seen from \eqref{eq:OneMinusAngleDifferenceInTermsOfLPrime},
  ensures that $S_{k,\iota,j'} + \nu_\iota \, \ell' \leq 0$
  for all $\ell' \in J_{j'}^{\iota,k}$.

  On the other hand, precisely as in Case 4 of the proof of
  Lemma \ref{lem:IntervalEstimateLSummation}, we see that, for
  $\ell' \in J_{j'}^{\iota,k}$,
  $v_{i_0, i'}^{+} \leq 2^{\beta j' + 1}
                        - \frac{2^{j'}}{4}
                          \cdot \big(
                                  1 - \frac{2}{\pi} \cdot \theta_{i_0, i_0 '}
                                \big)$
  and $v_{i_0, i'}^{-} \geq - 2^{\beta j' + 1}
                            - 2\pi \cdot 2^{j'}
                              \cdot \big(
                                      1 - \frac{2}{\pi} \cdot \theta_{i_0, i_0 '}
                                    \big)$.
  Given \eqref{eq:OneMinusAngleDifferenceInTermsOfLPrime} and our choice of
  $S_{k,\iota,j'}$ and $\nu_\iota$, this implies that
  \[
    v_{i_0, i'}^{+}
    \leq 2^{\beta j' + 1}
         + N^{-1} \cdot 2^{\beta j'} \cdot (S_{k,\iota,j'} + \nu_\iota \, \ell')
    \quad \text{and} \quad
    v_{i_0, i'}^{-}
    \geq - 2^{\beta j' + 1}
         + \frac{8\pi}{N}
           \cdot 2^{\beta j'}
           \cdot \big( S_{k,\iota,j'} + \nu_\iota \, \ell' \big) \, .
  \]
  These estimates together imply the stated inclusion for $I_2^{(i,i')}$.

  Finally, as in Case 4 of the proof of
  Lemma \ref{lem:IntervalEstimateLSummation}, we see that
  \[
    2^{\alpha j' - \beta j} \cdot |\sin \vartheta_{i_0, i_0 '}|
    \leq 2^{\alpha j' - \beta j}
         \cdot \frac{\pi}{2}
         \cdot \big( 1 - \frac{2}{\pi} \cdot \theta_{i_0, i_0 '} \big) \, .
  \]
  Given \eqref{eq:OneMinusAngleDifferenceInTermsOfLPrime} and our choice of
  $S_{k,\iota,j'}$ and $\nu_\iota$, this implies
  \(
    2^{\alpha j' - \beta j} \cdot |\sin \vartheta_{i_0, i_0 '}|
    \leq 2\pi \cdot \beta_1 \cdot |S_{k,\iota,j'} + \nu_\iota \, \ell'|
  \)
  since $\alpha \leq 1$.
\end{proof}

\section{Notation}\label{sec:Notation}

\noindent We recall a few elements of the theory of distributions and introduce the
notations that we use throughout this work:

For an open non-empty set $U \subset \R^d$, we define
$C_c^\infty (U) := \left\{ g \in C^\infty (\R^d; \CC) \with
\supp g \subset U \text{ compact} \right\}$.
There is a canonical topology on $C_c^\infty (U)$ that makes this space
into a topological vector space; see Sections~6.3-6.6 in \cite{RudinFA}
for the definition of this topology.
The topological dual space $\CalD ' (U) := [C_c^\infty(U)]'$ of $C_c^\infty(U)$
is called the space of \textbf{distributions} on $U$.
The bilinear pairing between $\CalD'(U)$ and $C_c^\infty (U)$ is denoted by
$\langle \phi, g \rangle_{\CalD'} := \langle \phi, g \rangle := \phi(g)$.
We use the characterisation of
$\CalD'(U)$ given in Theorem~6.8 in \cite{RudinFA}:
a linear functional $\phi : C_c^\infty (U) \to \CC$ belongs to $\CalD' (U)$
if and only if, for every compact set $K \subset U$, there are numbers
$C = C(K,\phi) > 0$ and $N = N(K,\phi) \in \N$ such that
$|\phi (g)| \leq C \cdot \max_{|\alpha|\leq N} \|\partial^\alpha g\|_{L^\infty}$
for all $g \in C_c^\infty (U)$ with $\supp g \subset K$.
Here, we used \textbf{multi-index} notation:
Any $\alpha \in \N_0^d$ is called a multi-index.
We write $\partial^\alpha f =
\frac{\partial^{\alpha_d}}{\partial x_d^{\alpha_d}}
\cdots \frac{\partial^{\alpha_1}}{\partial x_1^{\alpha_1}} f$
if $f : U \subset \R^d \to \CC$ is sufficiently smooth for the derivative to be
defined.
Similarly, we write $x^\alpha = x_1^{\alpha_1} \cdots x_d^{\alpha_d}$
for $x \in \R^d$.
Finally, we write $|\alpha| = \alpha_1 + \dots + \alpha_d$.
Even though we use the same notation $|\xi|$ for the Euclidean
norm of a vector $\xi \in \R^d$,
this, we believe, should not lead to any confusion.

Any locally integrable function $f \in L^1_{\mathrm{loc}} (U)$ induces
a distribution $T_f \in \CalD'(U)$ that is given by
$\langle T_f, g \rangle := \int_{U} f(x) g(x) \, dx$.
A distribution $\phi \in \CalD'(U)$ is called \textbf{regular} if $\phi = T_f$
for some $f \in L_{\mathrm{loc}}^1 (U)$.
For the sake of simplicity, we shall often write $f$ instead of $T_f$, i.e.,
$\langle f, g\rangle := \langle T_f, g \rangle$.

Various operations in the space of distributions $\CalD'(U)$ can be defined
using its duality with the space $C_c^\infty(U)$, i.e., for any
$h \in C^\infty (U)$, $\phi \in \CalD'(U)$ and $\alpha \in \N_0^d$,
the distributions $\partial^\alpha \phi \in \CalD'(U)$
and $h \cdot \phi \in \CalD'(U)$ are defined by
$\langle \partial^\alpha \phi, f \rangle
= (-1)^{|\alpha|} \cdot \langle \phi, \partial^\alpha f \rangle$, and
$\langle h\cdot\phi, f \rangle = \langle \phi, h \cdot f \rangle$, respectively.

\medskip{}

With the topology induced by the family of norms $(\| \bullet \|)_{N \in \N_0}$,
the \textbf{Schwartz function space} $\Schwartz(\R^d)$, defined as
\[
  \Schwartz (\R^d)
  := \left\{
        g \in C^\infty (\R^d; \CC)
        \with
        \forall \, N \in \N_0
        : \| g \|_N := \max_{\alpha \in \N_0^d, |\alpha| \leq N_0} \,
                         \sup_{x \in \R^d} \,
                           (1+|x|)^N \cdot |\partial^\alpha g (x)|
        < \infty
     \right\} \, ,
\]
becomes a topological vector space and its topological dual space
$\Schwartz '(\R^d)$, called the space of \textbf{tempered distributions},
becomes a topological vector space when equipped with the weak-$\ast$-topology.
We write
$\langle \phi, g \rangle_{\Schwartz'} := \langle \phi, g \rangle := \phi (g)$
for the dual pairing between $\Schwartz'(\R^d)$ and $\Schwartz (\R^d)$.
The reader is referred to Sections~8.1 and 9.2 in \cite{FollandRA}
for more details on these spaces.

In contrast to the \emph{bilinear} dual pairings for distributions and
tempered distributions, we write
$\langle f \mid g \rangle_{L^2} := \int f \cdot \overline{g} \, dx$,
which is sesquilinear.

The \textbf{Fourier transform} of $f \in L^1(\R^d)$ is defined by
\[
  \Fourier f (\xi)
  := \widehat{f} (\xi)
  := \int_{\R^d} f(x) \cdot e^{-2 \pi i \langle x, \xi \rangle} \, dx
  \quad \text{for} \quad \xi \in \R^d .
\]
The map $\Fourier : L^1 (\R^d) \to L^\infty(\R^d)$
is restricted to a homeomorphism $\Fourier : \Schwartz(\R^d) \to \Schwartz(\R^d)$
whose inverse $\Fourier^{-1} : \Schwartz(\R^d) \to \Schwartz(\R^d)$ is
given by $\Fourier^{-1} g (x) = \Fourier g (-x)$.
Using the duality between $\Schwartz'(\R^d)$ and $\Schwartz(\R^d)$,
the Fourier transform
$\Fourier : \Schwartz'(\R^d) \to \Schwartz'(\R^d)$
on the space of tempered distributions is defined by
$\langle \Fourier \phi, g \rangle := \langle \phi, \Fourier g\rangle$.
It is a homeomorphism with its inverse given by
$\langle \Fourier^{-1} \phi,g \rangle = \langle \phi, \Fourier^{-1} g \rangle$.

As for ``ordinary'' distributions,
certain functions also induce tempered distributions.
Precisely, if $f : \R^d \to \CC$ is measurable and if
$\int_{\R^d} (1+|x|)^{-N} \, |f(x)| \, dx < \infty$ for some $N \in \N$, then
we say that $f$ is \textbf{of moderate growth}.
In this case, one can verify that $T_f \in \Schwartz'(\R^d)$
where $\langle T_f, g \rangle := \int_{\R^d} f(x) \, g(x) \, dx$.
If $\phi \in \Schwartz '(\R^d)$ satisfies $\phi = T_f$ for some $f$ of moderate
growth, then we say that $\phi$ is given by integration against
the function $f$.

The \textbf{transpose} of a matrix $A \in \R^{k \times n}$ is denoted by $A^t \in \R^{n \times k}$.
If $A \in \R^{d \times d}$ is
invertible, we write $A^{-t} := (A^{-1})^t = (A^t)^{-1}$.

The \textbf{composition of two functions $f$ and $g$} is denoted by
$f \circ g (x) = f(g(x))$.


The \textbf{translation operator} $L_y$ and the
\textbf{modulation operator} $M_\xi$ are defined, respectively,
by $(L_y \, g)(x) = g(x-y)$ and
$(M_\xi \, g)(x) = e^{2\pi i \langle x, \xi \rangle} \cdot g(x)$.
Here and elsewhere in the paper,
${\langle \xi, x \rangle = \sum_{j = 1}^d \xi_j x_j}$ denotes the
\textbf{standard scalar product} on $\R^d$.

If $f : X \to Y$ is a function and $A \subset X$, we denote by
$f|_A$ the \textbf{restriction} of $f$ to $A$.

We denote by $B_r (x) = \{ y \in \R^d \colon |y - x| < \eps \}$
the open \textbf{Euclidean ball} of radius $r > 0$ with its centre in $x \in \R^d$.

For a function $g : \R^d \to \CC$, $\widetilde{g}$ stands for
$\widetilde{g} : \R^d \to \CC, x \mapsto g(-x)$.

We define the \textbf{conjugate exponent} of $p \in (0,\infty]$ as follows:
For $p \in (1,\infty)$, we define $p' := \frac{p}{p-1}$,
if $p' := 1$ if $p = \infty$ and $p' := \infty$ if $p \in (0,1]$.
This definition implies that $\frac{1}{p} + \frac{1}{p'} = 1$
for $p \in [1,\infty]$, but not for $p \in (0,1)$.


\section*{Acknowledgements}
\noindent Dimitri Bytchenkoff thanks the Centre National de la Recherche Scientifique of France
and the Deutscher Akademischer Austauschdienst for their financial support. Felix Voigtlaender acknowledges support from the European Commission through DEDALE
(contract no.~665044) within the H2020 Framework Program.
We both thank Professor Gitta Kutyniok of the Technische Universität Berlin and Professor Götz Pfander of
the Katholische Universität Eichstätt-Ingolstadt
for their support of this work. We are also grateful to Jordy van Velthoven for suggesting relevant references. Finally we thank the referee for his careful reading of our manuscript and making a number of suggestions that, we believe, helped us to improve our work.

\bibliographystyle{amsplain}

\phantomsection

\addcontentsline{toc}{chapter}{References}
\section*{\refname}
{\footnotesize
\bibliography{mybib}}

\end{document}